\documentclass[11pt,reqno]{amsart}
\usepackage[dvipdfm,a4paper,top=30mm,bottom=25mm,left=25mm,right=25mm]{geometry}
\usepackage{amsmath,amssymb,empheq,amsthm,mathtools,ulem,comment,mathrsfs}

\theoremstyle{definition}
\newtheorem{theorem}{Theorem}[section]
\newtheorem{proposition}[theorem]{Proposition}
\newtheorem{corollary}[theorem]{Corollary}
\newtheorem{lemma}[theorem]{Lemma}
\newtheorem{definition}[theorem]{Definition}

\newtheorem{remark}[theorem]{Remark}

\numberwithin{equation}{section} 

\allowdisplaybreaks[4]

\def\rnum#1{\expandafter{\romannumeral #1}} 
\def\Rnum#1{\uppercase\expandafter{\romannumeral #1}}

\makeatletter
\def\part{\@startsection {part}{1}{\z@}{
 -3.5ex plus -1ex minus -.2ex}{2.3 ex plus .2ex}{\centering\large\bf}}
\makeatother

\makeatother
\def\newtitle{Global dynamics below the ground state for the quadratic Schr\"odinger system in 5d}
\def\newauthor{Masaru Hamano}

\def\maketitle{
\begin{center}
{\Large\bf \newtitle \par}
\vspace{5mm}
{\Large \newauthor \par}
\vspace{4mm}
\end{center}
}

\begin{document}

\maketitle

Abstract. In this paper we consider the nonlinear Schr\"odinger system (NLS) with quadratic interaction in five dimensions. We determine the global behavior of the solutions to the system with data below the ground state. Our proof of the scattering result is based on an argument by Kenig--Merle \cite{011}. In particular, the new part of this paper is to deal with asymmetric interaction. A blowing-up or growing-up result is proved by combining the argument by Du--Wu--Zhang in \cite{003} and a variational characterization of minimizers. Moreover, we show a blowing-up result if the data has finite variance or is radial.

\tableofcontents 

\section{Introduction}
\subsection{Background}
 
　\\
We consider the quadratic Schr\"odinger system in five space dimensions:
\begin{equation}
\notag \text{(NLS) }
\begin{cases}
\hspace{-0.4cm}&\displaystyle{i\partial_tu+\Delta u=-2v\bar{u}\qquad(x,t)\in\mathbb{R}^5\times\mathbb{R},}\\
\hspace{-0.4cm}&\displaystyle{i\partial_tv+\frac{1}{2}\Delta v=-u^2\qquad(x,t)\in\mathbb{R}^5\times\mathbb{R},}\\
\hspace{-0.4cm}&\displaystyle{u(x,0)=u_0(x),\ v(x,0)=v_0(x)\qquad x\in\mathbb{R}^5,}
\end{cases}
\end{equation}
where $i=\sqrt{-1}$, $u,v:\mathbb{R}^5\times\mathbb{R}\longrightarrow\mathbb{C}$ are unknown functions, $u_0,v_0:\mathbb{R}^5\longrightarrow\mathbb{C}$ are given functions, $\Delta=\sum_{j=1}^5\frac{\partial^2}{\partial x_j^2}$, and $\overline{u}$ is complex conjugate of $u$.\\

　(NLS) has a physical background, which is deduced from the equation describing the Raman process. This process is a nonlinear instability phenomenon (see \cite{018} for more detail). Furthermore, (NLS) is also derived from a non-relativistic limit of the nonlinear Klein-Gordon system (see \cite{007}). Asymptotic behavior of solutions for this NLKG system was studied in Sunagawa--Kawahara \cite{023} and Hayashi--Naumkin \cite{022}. For (NLS), a well-posedness result in $L^2$ or $H^1$, a blowing-up result with finite variance, and existence of a ground state were studied in Hayashi--Ozawa--Tanaka \cite{007}.\\

　Our aim in the present paper is to determine long time behavior of solutions to the problem (NLS). There are various kinds of solutions depending on the choice of the data, for example, scattering solution, blow-up solution, standing wave solution and so on. Here we are especially interested in investigating this problem under the assumption that the value of the action of the initial data is less than that of the ground state. The ground state is one of the solutions to the following nonlinear elliptic system
\begin{equation}
\notag \text{(gNLS) }
\begin{cases}
\hspace{-0.4cm}&\displaystyle{-\Delta\phi_\omega+\omega\phi_\omega=2\psi_\omega\phi_\omega,}\\
\hspace{-0.4cm}&\displaystyle{-\frac{1}{2}\Delta\psi_\omega+2\omega\psi_\omega=\phi_\omega^2,}
\end{cases}
\end{equation}
where $\omega>0$. It was proved in \cite{007} that this system has a non-trivial solution, which is called the ground state. The ground state attains the infimum of $\mu_\omega^{20,8}$ (see Definition \ref{functional}) and is non-negative (real-valued) and radial. We denote by $(\phi_\omega,\psi_\omega)$ this ground state solution. If we set $(u,v)=(e^{i\omega t}\phi_\omega,e^{2i\omega t}\psi_\omega)$, then $(u,v)$ is the solution to (NLS), and is called the standing wave solution. We remark that the standing wave solution neither scatters nor blows-up. The stability of standing wave solutions was studied in \cite{018}.\\

　In this paper, we give the necessary and sufficient conditions on the data, which clarifies the scattering and blowing-up behavior of solutions to (NLS). Our proof of the scattering is based on the argument by Kenig--Merle in \cite{011}. In the past ten years, global behavior of solutions below the ground state for the single focusing case 
\[
i\partial_tu+\Delta u+|u|^\alpha u=0\ \ \ (x,t)\in\mathbb{R}^N\times\mathbb{R},
\]
where $\alpha>0$, was studied by several authors. Kenig--Merle \cite{011} treated the case with $\alpha=\frac{4}{N-2}$, $N=3,4,5$, under radial symmetry, Holmer--Roudenko \cite{008} treated the case with $\alpha=2$, $N=3$ under radial symmetry, Duyckaerts--Holmer--Roudenko \cite{004} treated the case with $\alpha=2$, $N=3$ without radial symmetry, Fang--Xie--Cazenave \cite{005} showed scattering, and Akahori--Nawa \cite{017} showed scattering and blowing-up in the mass supercritical and energy subcritical case. In the system with symmetric interaction
\begin{equation}
\notag \text{(sNLS) }
\begin{cases}
\hspace{-0.4cm}&\displaystyle{i\partial_tu+\Delta u+(|u|^2+|v|^2)u=0\ \ \ (x,t)\in\mathbb{R}^3\times\mathbb{R},}\\
\hspace{-0.4cm}&\displaystyle{i\partial_tv+\Delta v+(|u|^2+|v|^2)v=0\ \ \ \,(x,t)\in\mathbb{R}^3\times\mathbb{R},}
\end{cases}
\end{equation}
Xu \cite{014} and Farah--Pastor \cite{006} showed the scattering result below the ground state. In our system, Hayashi--Li--Ozawa \cite{019} proved a small data scattering by using the end point Strichartz estimate in $\dot{H}^\frac{1}{2}$ setting. However, it is difficult to remove the smallness condition on the data to prove the scattering. To overcome this difficulty, we use Kenig--Merle type argument (Linear profile decomposition, Long time perturbation theory, Compactness, Rigidity).  Xu \cite{014} and Farah--Pastor \cite{006} also applied the similar method to (sNLS). Unlike theirs, our system has assymmetric interaction, so that we have to deal with the interaction carefully. Moreover, we are considering five space dimensions, so we use different type of exponents for Strichartz estimate (see also \cite{005}). Our proof of the blowing-up or growing-up result is based on the argument by Du--Wu--Zhang \cite{003} and variational characterization (Lemma \ref{estimates for K}). We remark that it is still open whether growing-up occurs or not. We note that Hayashi--Ozawa--Tanaka \cite{007} proved blowing-up result if the data has finite variance and negative energy. In this paper, we extend their result to include certain positive energy initial data by using the ground state (see Theorem \ref{Main theorem 1} for more detail). Our blowing-up result with finite variance or radial symmetry is based on the argument by Xu \cite{014}.

\subsection{Definition and main result}
　\\
In order to state the main result, we introduce some notations and basic facts. Let $(T_\ast,T^\ast)$ be the maximal lifespan of the solution $(u,v)$ to (NLS) (see Theorem \ref{H1 local}).

\begin{definition}[Scattering, Blowing-up, Growing-up]
　\\[-0.6cm]
\begin{itemize}
\item(Scattering)\\
We say that the solution $(u,v)$ to (NLS) scatters in positive time (resp. negative time) if $T^\ast=\infty$ (resp. $T_\ast=-\infty$) and there exists $(\phi_+,\psi_+)\in H^1\!\times\!H^1$ (resp. $(\phi_-,\psi_-)\in H^1\!\times\!H^1$) such that
\[
\lim_{t\rightarrow+\infty}\|(u(t),v(t))-(e^{it\Delta}\phi_+,e^{\frac{1}{2}it\Delta}\psi_+)\|_{H^1\times H^1}=0
\]
\vspace{-0.3cm}
\[
\left(\text{resp}.\ \ \lim_{t\rightarrow-\infty}\|(u(t),v(t))-(e^{it\Delta}\phi_-,e^{\frac{1}{2}it\Delta}\psi_-)\|_{H^1\times H^1}=0\right).\ \ \ \ \ \ \ 
\]
\item(Blowing-up)\\
We say that $(u,v)$ blows up in positive time (resp. negative time) if $T^\ast<\infty$ (resp. $T_\ast>-\infty$). Moreover, Theorem \ref{H1 local} and Theorem \ref{Conservation law} give
\[
\lim_{\substack{t\nearrow T^\ast\\(t\searrow T_\ast)}}\|(u(t),v(t))\|_{\dot{H}^1\times \dot{H}^1}=\infty.
\]
\item(Growing-up)\\
We say that the solution $(u,v)$ grows up in positive time (resp. negative time) if $T^\ast=\infty$ (resp. $T_\ast=-\infty$) and there exists a sequence $\{t_n\}$ with $t_n\rightarrow\infty$ (resp. $-\infty$) as $n\rightarrow\infty$ such that
\[
\lim_{n\rightarrow\infty}\|(u(t_n),v(t_n))\|_{\dot{H}^1\times \dot{H}^1}=\infty.
\]
\end{itemize}
\end{definition}


\begin{definition}\label{functional}
We define the following functionals and quantities for $(u,v)\in H^1\!\times\!H^1$, $\omega>0$ and $\alpha,\,\beta\in\mathbb{R}$.
\begin{align*}
&M(u,v)=\|u\|_{L^2}^2+2\|v\|_{L^2}^2\,,\ \ \ \ \ \ \ \ \ \ \ \ \ \ \ \ \ \,E(u,v)=\|\nabla u\|_{L^2}^2+\frac{1}{2}\|\nabla v\|_{L^2}^2-2\text{Re}(v,u^2),\\
&K(u,v)=\|\nabla u\|_{L^2}^2+\frac{1}{2}\|\nabla v\|_{L^2}^2\,,\ \ \ \ \ \ \ \ \ \ \ \ P(u,v)=\text{Re}(v,u^2)_{L^2},\\
&K_\omega(u,v)=K(u,v)+\omega M(u,v)\,,\ \ \ \ \ \ \ \ \ \ \ L_\omega(u,v)=\frac{\omega}{2}M(u,v)+\frac{1}{10}K(u,v),\\
&I_\omega (u,v)=\frac{\omega}{2}M(u,v)+\frac{1}{2}E(u,v)\,,\ \ \ \ \ \ \ \ \ \,K_\omega^{\alpha,\beta}(u,v)=\left.\partial_\lambda I_\omega(e^{\alpha\lambda} u(e^{\beta\lambda}\cdot),e^{\alpha\lambda} v(e^{\beta\lambda}\cdot))\right|_{\lambda=0},\\
&C_\omega=\{(u,v)\in H^1\times H^1\setminus \{(0,0)\}:K_\omega^{20,8}(u,v)=0\},\\
&\mu_\omega^{20,8}=\inf\{I_\omega(u,v):(u,v)\in C_\omega\}.
\end{align*}
\end{definition}


We state our main result.

\begin{theorem}[Scattering versus blowing-up dichotomy]\label{Main theorem 1}
Let $(u_0,v_0)\in H^1\!\times\!H^1$ and $(u,v)$ be the solution to (NLS) with initial data $(u_0,v_0)$. Moreover, we assume for $\omega>0$,
\[
I_\omega(u_0,v_0)<I_\omega(\phi_\omega,\psi_\omega).
\]
\begin{itemize}
\item[(1)]Let $K_\omega^{20,8}(u_0,v_0)\geq0$. Then $(u,v)$ scatters in positive time in $H^1\!\times\!H^1$.
\item[(2)]Let $K_\omega^{20,8}(u_0,v_0)<0$. Then $(u,v)$ blows up or grows up in positive time in $H^1\!\times\!H^1$. Moreover, if $(xu_0,xv_0)\in L^2\!\times \!L^2$ or $(u_0,v_0)$ is radial, then the solution $(u,v)$ blows up in positive time.
\end{itemize}
\end{theorem}

　The same conclusion of Theorem \ref{Main theorem 1} holds for the negative time direction by taking the complex conjugate of the equation and replace $t$ by $-t$.

\subsection{Organization of the paper}
　\\
　The organization of this paper is as follows. In section 2, we give several properties of the ground state $(\phi_\omega,\psi_\omega)$ for (NLS), Strichartz estimates, Small data scattering result, and Long time perturbation theory. In section 3, we show that the solution exists time-globally under the assumption of Theorem \ref{Main theorem 1} (1) and show Theorem \ref{Main theorem 1} (2). In section 4 and section 5, we complete the proof of Theorem \ref{Main theorem 1} (1) by contradiction.\\

　For the convinience of the reader, we explain the strategy for the proof  of Theorem \ref{Main theorem 1} (1). We assume that the threshold for scattering $I_\omega^c$ (see Definition \ref{I omega c}) is strictly below $I_\omega(\phi_\omega,\psi_\omega)$. In section 5.1, we construct a solution $(u_c,v_c)$ (which is called a critical solution) that stands exactly at the boundary between scattering and non-scattering. The way of the construction of this solution is as follows. We take a sequence $\{(u_n,v_n)\}$ of the solutions to (NLS), which exists between $I_\omega^c$ and $I_\omega(\phi_\omega,\psi_\omega)$, and satisfies $\|(u_n,v_n)\|_{S(\dot{H}^\frac{1}{2})\times S(\dot{H}^\frac{1}{2})}=\infty$ and $I_\omega(u_n,v_n)\longrightarrow I_\omega^c$ (see Definition \ref{Strichartz norm} for the definition of $\|(\cdot,\cdot)\|_{S(\dot{H}^\frac{1}{2})\times S(\dot{H}^\frac{1}{2})}$). Since this sequence is bounded in $H^1\!\times\!H^1$, we can apply Theorem \ref{Linear profile decomposition} (Linear profile decomposition) to $(u_{n}(0),v_{n}(0))$. We take the non-linear profile, which approaches the linear profile by Lemma \ref{Existence of wave operators} (Existence of wave operators). Applying Theorem \ref{Long time perturbation} (Long time perturbation), we prove that there exists only one non-zero linear profile. Using this non-zero linear profile, we construct a critical solution. In section 5.2, we prove that the orbit of the critical solution $K=\{(u_c(\cdot-x(t),t),v_c(\cdot-x(t),t)):t\in[t,\infty)\}$ is precompact in $H^1\times H^1$. First, we define a equivalent relation $\sim$ on $H^1\times H^1$ and the quotient space $H^1\times H^1/{\sim}$, which is constructed by the whole equivalent class. Let $\pi:H^1\!\times\!H^1\longrightarrow H^1\!\times\!H^1/{\sim}$ be the natural projection. In Lemma \ref{Precompactness of the flow of the critical solution 2} and Lemma \ref{Precompactness of the flow of the critical solution 3}, we prove that it is sufficient to prove the orbit of the critical solution $K$ is precompact if $\pi(K)$ is precompact. In Proposition \ref{Precompactness of the flow of the critical solution}, we  show that the orbit of the critical solution $K$ is precompact in $H^1\times H^1$. In section 5.3, we prove that the assumption $(I_\omega^c<I_\omega(\phi_\omega,\psi_\omega))$ of a contradiction argument is fault. To prove this, we first show that the momentum of the critical solution is zero. Next, we prove Theorem \ref{Rigidity} (Rigidity). Because the critical solution satisfies the assumption of Theorem \ref{Rigidity} (Rigidity), we apply Theorem \ref{Rigidity}, then a contradiction occurs.

\section{Preliminaries}

\subsection{Results from \cite{007}}


\begin{definition}[Space of $H^1\!\times\!H^1$\,solution]\label{1}
\[
Y(I)\vcentcolon=(C\cap L^\infty)(I:H^1)\cap L^2(I:W^{1,\frac{10}{3}})\,,\ \ \|u\|_{Y(I)}\vcentcolon=\max\left\{\|u\|_{L^\infty H^1},\|u\|_{L^2W^{1,\frac{10}{3}}}\right\}.
\]
\end{definition}


\begin{theorem}[Unique existence of $H^1\!\times\!H^1$\,time local solution\,\cite{007}]\label{H1 local}
For any $\rho>0$, there exists $T(\rho)>0$ such that for any $(u_0,v_0)\in H^1\times H^1$ with $\max\left\{\|u_0\|_{H^1},\|v_0\|_{H^1}\right\}\leq\rho$, (NLS) has the unique solution $(u,v)\in Y(I)\times Y(I)$ with $I=[-T(\rho),T(\rho)]$.
\end{theorem}


\begin{definition}[Space of weighted $L^2\!\times\!L^2$\,solution]
\[
X(I)\vcentcolon=(C\cap L^\infty)(I:L^2)\cap L^2(I:L^{\frac{10}{3}})\,,\ \ \|u\|_{X(I)}\vcentcolon=\max\left\{\|u\|_{L^\infty L^2},\|u\|_{L^2L^{\frac{10}{3}}}\right\},
\]
\[
Z(I)\vcentcolon=\left\{u\in Y(I):xu\in X(I)\right\}\,,\ \ \|u\|_{Z(I)}\vcentcolon=\max\left\{\|u\|_{Y(I)},\|xu\|_{X(I)}\right\}.
\]
\end{definition}


\begin{theorem}[Unique existence of weighted $L^2\!\times\!L^2$\,time local solution\,\cite{007}]\label{xu}
For any $\rho>0$, there exists $T(\rho)>0$ such that for any $(u_0,v_0)\in H^1\!\times\!H^1$ with $(xu_0,xv_0)\in L^2\!\times\! L^2$ and
\[
\max\left\{\|u_0\|_{H^1},\|v_0\|_{H^1},\|xu_0\|_{L^2},\|xv_0\|_{L^2}\right\}\leq \rho,
\]
(NLS) has the unique solution $(u,v)\in Z(I)\!\times\!Z(I)$ with $I=[-T(\rho),T(\rho)]$.
\end{theorem}


\begin{theorem}[Conservation law\,\cite{007}]\label{Conservation law}
The solution $(u,v)$ to (NLS) satisfies the following conservation laws for all $t\in (T_\ast,T^\ast)$
\begin{align}
M(u(t),v(t))=M(u_0,v_0),\label{01}
\end{align}
\begin{align}
E(u(t),v(t))=E(u_0,v_0).\label{02}
\end{align}
\end{theorem}


\begin{theorem}[Virial identity\,\cite{007}]\label{virial identity}
Let $(u_0,v_0)\in H^1\!\times\!H^1$ satisfy $(xu_0,xv_0)\in L^2\!\times\!L^2$ and let $(u,v)\in(Z(I)\!\times\!Z(I))\cap(Y(I)\!\times\!Y(I))$ be the corresponding local solution given by Theorem \ref{H1 local} and Theorem \ref{xu}.\ Then
\begin{align}
\frac{d^2}{dt^2}(\|xu(t)\|_{L^2}^2+2\|xv(t)\|_{L^2})=10E(u_0,v_0)-2\left(\|\nabla u(t)\|_{L^2}^2+\frac{1}{2}\|\nabla v(t)\|_{L^2}^2\right).\label{37}
\end{align}
\end{theorem}

\begin{remark}
Due to the coefficients $1$, $\frac{1}{2}$ of the Laplacian for (NLS), we can prove that Theorem \ref{virial identity} holds. This condition is called mass resonance condition.
\end{remark}

\subsection{Linear estimates}
　\\
To construct the solution, we convert (NLS) into the following integral system (NLSI) by Duhamel's principle.
\begin{equation}
\notag \text{(NLSI) }
\begin{cases}
\hspace{-0.4cm}&\displaystyle{u(t)=e^{it\Delta}u_0+2i\int_0^te^{i(t-s)\Delta}(v\overline{u})(s)ds,}\\[0.3cm]
\hspace{-0.4cm}&\displaystyle{v(t)=e^{\frac{1}{2}it\Delta}v_0+i\int_0^te^{\frac{1}{2}i(t-s)\Delta}(u^2)(s)ds,}
\end{cases}
\end{equation}
where $e^{it\Delta}f(x)=(e^{-4\pi^2it|\xi|^2}\widehat{f}\,)^\vee(x)$ and $e^{\frac{1}{2}it\Delta}f(x)=(e^{-2\pi^2it|\xi|^2}\widehat{f}\,)^\vee(x)$.


\begin{theorem}\label{Linear estimate}\label{Linear estimate}
If $t\neq0,\,\frac{1}{p}+\frac{1}{p'}=1$ and $p'\in[1,2]$, then it follows that $e^{it\Delta}:L^{p'}\longrightarrow L^p$ is continuous and
\[
\|e^{it\Delta}f\|_{L^p}\leq c|t|^{-\frac{5}{2}\left(\frac{1}{p'}-\frac{1}{p}\right)}\|f\|_{L^{p'}},
\]
where $c$ is independent of $t$ and $f$.
\end{theorem}


\begin{definition}\label{Strichartz norm}
We say that $(q,r)$ is $\dot{H}^s$ admissible (in 5d) if 
\[
\frac{2}{q}+\frac{5}{r}=\frac{5}{2}-s.
\]
Strichartz norm is defined as
\[
\|u\|_{S(L^2)}=\sup_{\substack{(q,r):L^2\text{admissible}\\2\leq q\leq\infty,2\leq r\leq\frac{10}{3}}}\|u\|_{L_t^qL_x^r}\,,\ \ \ \|u\|_{S(\dot{H}^\frac{1}{2})}=\sup_{\substack{(q,r):\dot{H}^\frac{1}{2}\text{admissible}\\4^+\leq q\leq\infty,\frac{5}{2}\leq r\leq\frac{10}{3}^-}}\|u\|_{L_t^qL_x^r}.
\]
Dual Strichartz norm is defined as
\[
\|u\|_{S'(L^2)}=\inf_{\substack{(q,r):L^2\text{admissible}\\2\leq q\leq\infty,2\leq r\leq\frac{10}{3}}}\|u\|_{L_t^{q'}L_x^{r'}}\,,\ \ \ \|u\|_{S'(\dot{H}^{-\frac{1}{2}})}=\inf_{\substack{(q,r):\dot{H}^{-\frac{1}{2}}\text{admissible}\\\frac{4}{3}^+\leq q\leq4,2\leq r\leq\frac{10}{3}^-}}\|u\|_{L_t^{q'}L_x^{r'}},
\]
where $(q',r')$ is the H\"older dual to $(q,r)$ and $4^+$ is an arbitrarily preselected and fixed number larger than $4$, similarly for $\frac{10}{3}^-$ and $\frac{4}{3}^+$.
\end{definition}

\begin{remark}
Dual Strichartz norm is not norm precisely. (Dual Strichartz norm may not satisfy triangle inequality.)
\end{remark}


\begin{theorem}[Strichartz estimates \cite{012}]\label{Strichartz estimates}
We have the following:
\[
\|e^{it\Delta}f\|_{S(L^2)}\leq c\|f\|_{L^2},
\]
\[
\left\|\int_0^te^{i(t-s)\Delta}F(\cdot,s)ds\right\|_{S(L^2)}\leq c\|f\|_{S'(L^2)}.
\]
If $(q,r)$ is $\dot{H}^\frac{1}{2}$ admissible and satisfies $\displaystyle2\leq q\leq\infty,\ \frac{5}{2}\leq r\leq5$, then
\[
\|e^{it\Delta}f\|_{L_t^qL_x^r}\leq c\|f\|_{\dot{H}^\frac{1}{2}},
\]
\[
\left\|\int_0^te^{i(t-s)\Delta}F(\cdot,s)ds\right\|_{L_t^qL_x^r}\leq c\|D^\frac{1}{2}F\|_{S'(L^2)},
\]
and
\[
\left\|\int_0^te^{i(t-s)\Delta}F(\cdot,s)ds\right\|_{S(\dot{H}^\frac{1}{2})}\leq c\|F\|_{S'(\dot{H}^{-\frac{1}{2}})},
\]
where $c$ is independent of $f$ or $F$.
\end{theorem}

We extend our notation $S(\dot{H}^\frac{1}{2})$ as follows: If a time interval is not specified (, that is, if we just write $S(\dot{H}^\frac{1}{2})$), then the $t$-norm is evaluated over $(-\infty,+\infty)$. To indicate a restriction to a time subinterval $I\subset(-\infty,+\infty)$, we will write $S(\dot{H}^\frac{1}{2};I)$. Other Strichartz norm, dual Strichartz norm is also described similarly. Even if time is restricted, Theorem \ref{Strichartz estimates} still holds.

\subsection{Variational characterization}


\begin{proposition}(\cite{007})\label{2}
For $\omega>0$, there exists a nontrivial real valued solution $(\phi_\omega,\psi_\omega)\in H^1\!\times\!H^1$ of the elliptic system
\begin{equation}
\notag \text{(gNLS) }
\begin{cases}
\hspace{-0.4cm}&\displaystyle{-\Delta\phi_\omega+\omega\phi_\omega=2\psi_\omega\phi_\omega,}\\
\hspace{-0.4cm}&\displaystyle{-\frac{1}{2}\Delta\psi_\omega+2\omega\psi_\omega=\phi_\omega^2.}
\end{cases}
\end{equation}
It is well known that $(u,v)=(e^{i\omega t}\phi_\omega,e^{i2\omega t}\psi_\omega)$ is a global solution to (NLS) if $(\phi_\omega,\psi_\omega)$ is the solution of (gNLS). We can see this type solutions does not decay as $t\longrightarrow\infty$.
\end{proposition}

\begin{definition}
Real-valued functions $(\phi,\psi)\in H^1\!\times\!H^1$ is called a solution of (gNLS) if
\begin{equation}
\notag \text{(gNLSI) }
\begin{cases}
\hspace{-0.4cm}&\displaystyle{\int_{\mathbb{R}^5}\nabla \phi\cdot\nabla udx+\omega\int_{\mathbb{R}^5}\phi udx=2\int_{\mathbb{R}^5}\phi\psi udx,}\\[0.3cm]
\hspace{-0.4cm}&\displaystyle{\frac{1}{2}\int_{\mathbb{R}^5}\nabla \psi\cdot\nabla vdx+2\omega\int_{\mathbb{R}^5}\psi vdx=\int_{\mathbb{R}^5}\phi^2 vdx}
\end{cases}
\end{equation}
for any $u,v\in C_0^\infty(\mathbb{R}^5)$.
\end{definition}


\begin{proposition}\label{proposition of ground state}
Let $\omega>0$. Then the solution $(\phi_\omega,\psi_\omega)$ to (gNLS) satisfies
\[
2K(\phi_\omega,\psi_\omega)=5P(\phi_\omega,\psi_\omega)\,,\ 2\omega M(\phi_\omega,\psi_\omega)=P(\phi_\omega,\psi_\omega).
\]
\end{proposition}

\begin{proof}
For simplicity, we give a formal calculation for the proof. Actual proof requires a regularization procedure.\\
Multiplying both side of $-\Delta\phi_\omega+\omega\phi_\omega=2\psi_\omega\phi_\omega$ by $\phi_\omega$ and integrating, we have
\[
-\int_{\mathbb{R}^5}\phi_\omega\Delta\phi_\omega dx+\omega\int_{\mathbb{R}^5}\phi_\omega^2dx=2\int_{\mathbb{R}^5}\psi_\omega\phi_\omega^2dx.
\]
Multiplying both side of $\displaystyle-\frac{1}{2}\Delta\psi_\omega+2\omega\psi_\omega=\phi_\omega^2$ by $\psi_\omega$ and integrating, we have
\[
-\frac{1}{2}\int_{\mathbb{R}^5}\psi_\omega\Delta\psi_\omega dx+2\omega\int_{\mathbb{R}^5}\psi_\omega^2dx=\int_{\mathbb{R}^5}\psi_\omega\phi_\omega^2dx.
\]
Using integration by parts,
\[
\|\nabla\phi_\omega\|_{L^2}^2+\omega\|\phi_\omega\|_{L^2}^2=2\text{Re}(\psi_\omega,\phi_\omega^2)_{L^2}\,,\ \ \frac{1}{2}\|\nabla\psi_\omega\|_{L^2}^2+2\omega\|\psi_\omega\|_{L^2}^2=\text{Re}(\psi_\omega,\phi_\omega^2)_{L^2}.
\]
Thus,
\begin{align}
\omega M(\phi_\omega,\psi_\omega)+K(\phi_\omega,\psi_\omega)=3P(\phi_\omega,\psi_\omega).\label{03}
\end{align}
Multiplying both side of $-\Delta\phi_\omega+\omega\phi_\omega=2\psi_\omega\phi_\omega$ by $x\cdot\nabla\phi_\omega$ and integrating, we have
\begin{align}
-\int_{\mathbb{R}^5}\Delta\phi_\omega x\cdot\nabla\phi_\omega dx+\omega\int_{\mathbb{R}^5}\phi_\omega x\cdot\nabla\phi_\omega dx=2\int_{\mathbb{R}^5}\psi_\omega\phi_\omega x\cdot\nabla\phi_\omega dx.\label{47}
\end{align}
Since
\begin{align*}
\int_{\mathbb{R}^5}\Delta\phi_\omega x\cdot\nabla\phi_\omega dx&=\sum_{k=1}^5\sum_{j=1}^5\int_{\mathbb{R}^5}\partial_k^2\phi_\omega x_j\partial_j\phi_\omega dx\\
&=\sum_{k=1}^5\sum_{j=1}^5\int_{\mathbb{R}^4}\left(\left[\partial_k\phi_\omega x_j\partial_j\phi_\omega\right]_{x_k=-\infty}^{x_k=\infty}-\int_\mathbb{R}\partial_k\phi_\omega\partial_k(x_j\partial_j\phi_\omega)dx_k\right)d\overline{x_k}\\
&=-\sum_{k=1}^5\int_{\mathbb{R}^5}(\partial_k\phi_\omega)^2dx-\sum_{k=1}^5\sum_{j=1}^5\int_{\mathbb{R}^5}\partial_k\phi_\omega x_j\partial_{kj}\phi _\omega dx\\
&=-\|\nabla\phi_\omega\|_{L^2}^2-\sum_{k=1}^5\sum_{j=1}^5\int_{\mathbb{R}^4}\left(\left[x_j(\partial_k\phi_\omega)^2\right]_{x_j=-\infty}^{x_j=\infty}-\int_\mathbb{R}\partial_j(x_j\partial_k\phi_\omega)\partial_k\phi_\omega dx_j\right)d\overline{x_j}\\
&=-\|\nabla\phi_\omega\|_{L^2}^2+\sum_{k=1}^5\sum_{j=1}^5\int_{\mathbb{R}^5}(\partial_k\phi_\omega+x_j\partial_{jk}\phi_\omega)\partial_k\phi_\omega dx\\
&=4\|\nabla\phi_\omega\|_{L^2}^2+\sum_{k=1}^5\sum_{j=1}^5\int_{\mathbb{R}^4}\left(\left[x_j\partial_j\phi_\omega\partial_k\phi_\omega\right]_{x_k=-\infty}^{x_k=\infty}-\int_\mathbb{R}\partial_j\phi_\omega\partial_k(x_j\partial_k\phi_\omega)dx_k\right)d\overline{x_k}\\
&=4\|\nabla\phi_\omega\|_{L^2}^2-\sum_{k=1}^5\int_{\mathbb{R}^5}(\partial_k\phi_\omega)^2dx-\sum_{k=1}^5\sum_{j=1}^5\int_{\mathbb{R}^5}\partial_k^2\phi_\omega x_j\partial_j\phi_\omega dx\\
&=3\|\nabla\phi_\omega\|_{L^2}^2-\int_{\mathbb{R}^5}\Delta\phi_\omega x\cdot\nabla\phi_\omega dx,
\end{align*}
it follows that
\begin{align}
\int_{\mathbb{R}^5}\Delta\phi_\omega x\cdot\nabla\phi_\omega dx=\frac{3}{2}\|\nabla\phi_\omega\|_{L^2}^2.\label{48}
\end{align}
Since
\begin{align*}
\int_{\mathbb{R}^5}\phi_\omega x\cdot\nabla\phi_\omega dx&=\sum_{j=1}^{5}\int_{\mathbb{R}^5}\phi_\omega x_j\partial_j\phi_\omega dx\\
&=\sum_{j=1}^{5}\int_{\mathbb{R}^4}\left(\left[x_j\phi_\omega^2\right]_{x_j=-\infty}^{x_j=\infty}-\int_\mathbb{R}(\phi_\omega+x_j\partial_j\phi_\omega)\phi_\omega dx_j\right)d\overline{x_j}\\
&=-\sum_{j=1}^{5}\int_{\mathbb{R}^5}(\phi_\omega^2+\phi_\omega x_j\partial_j\phi_\omega)dx\\
&=-5\|\phi_\omega\|_{L^2}^2-\int_{\mathbb{R}^5}\phi_\omega x\cdot\nabla\phi _\omega dx,
\end{align*}
it follows that
\begin{align}
\int_{\mathbb{R}^5}\phi_\omega x\cdot\nabla\phi_\omega dx=-\frac{5}{2}\|\phi_\omega\|_{L^2}^2.\label{49}
\end{align}
Since
\begin{align*}
\int_{\mathbb{R}^5}\psi_\omega\phi_\omega x\cdot\nabla\phi_\omega dx&=\sum_{j=1}^5\int_{\mathbb{R}^5}\psi_\omega\phi_\omega x_j\partial_j\phi_\omega dx\\
&=\sum_{j=1}^5\int_{\mathbb{R}^4}\left(\left[x_j\psi_\omega\phi_\omega^2\right]_{x_j=-\infty}^{x_j=\infty}-\int_\mathbb{R}\phi_\omega\partial_j(x_j\psi_\omega\phi_\omega)dx_j\right)d\overline{x_j}\\
&=-\sum_{j=1}^5\int_{\mathbb{R}^5}(\psi_\omega\phi_\omega^2+\phi_\omega^2x_j\partial_j\psi_\omega+\psi_\omega\phi_\omega x_j\partial_j\phi_\omega)dx\\
&=-5\int_{\mathbb{R}^5}\psi_\omega\phi_\omega^2dx-\int_{\mathbb{R}^5}\phi_\omega^2x\cdot\nabla\psi_\omega dx-\int_{\mathbb{R}^5}\psi_\omega\phi_\omega x\cdot\nabla\phi_\omega dx,
\end{align*}
it follows that
\begin{align}
\int_{\mathbb{R}^5}\psi_\omega\phi_\omega x\cdot\nabla\phi_\omega dx=-\frac{5}{2}\int_{\mathbb{R}^5}\psi_\omega\phi_\omega^2dx-\frac{1}{2}\int_{\mathbb{R}^5}\phi_\omega^2x\cdot\nabla\psi_\omega dx.\label{50}
\end{align}
Combining \eqref{47},\,\eqref{48},\,\eqref{49} and \eqref{50}, it follows that
\begin{align}
-\frac{3}{2}\|\nabla\phi_\omega\|_{L^2}^2-\frac{5}{2}\omega\|\phi_\omega\|_{L^2}^2=-5\text{Re}(\psi_\omega,\phi_\omega^2)_{L^2}-\int_{\mathbb{R}^5}\phi_\omega^2x\cdot\nabla\psi_\omega dx.\label{04}
\end{align}
Multiplying both side of $\displaystyle-\frac{1}{2}\Delta\psi_\omega+2\omega\psi_\omega=\phi_\omega^2$ by $x\cdot\nabla\psi_\omega$ and integrating
\[
-\frac{1}{2}\int_{\mathbb{R}^5}\Delta\psi_\omega x\cdot\nabla\psi_\omega dx+2\omega\int_{\mathbb{R}^5}\psi_\omega x\cdot\nabla\psi_\omega dx=\int_{\mathbb{R}^5}\phi_\omega^2x\cdot\nabla\psi_\omega dx.
\]
This formula combined with \eqref{48},\,\eqref{49} gives
\begin{align}
-\frac{3}{4}\|\nabla\psi_\omega\|_{L^2}^2-5\omega\|\psi_\omega\|_{L^2}^2=\int_{\mathbb{R}^5}\phi_\omega^2x\cdot\nabla\psi_\omega dx.\label{51}
\end{align}
Combining \eqref{04} and \eqref{51}, we have
\[
\frac{3}{2}\|\nabla\phi_\omega\|_{L^2}^2+\frac{5}{2}\omega\|\phi_\omega\|_{L^2}^2+\frac{3}{4}\|\nabla\psi_\omega\|_{L^2}^2+5\omega\|\psi_\omega\|_{L^2}^2=5\text{Re}(\psi_\omega,\phi_\omega^2)_{L^2}.
\]
Therefore,
\[
\frac{5}{2}\omega M(\phi_\omega,\psi_\omega)+\frac{3}{2}K(\phi_\omega,\psi_\omega)=5P(\phi_\omega,\psi_\omega).
\]
This formula combined with \eqref{03} gives
\[
2K(\phi_\omega,\psi_\omega)=5P(\phi_\omega,\psi_\omega)\,,\ 2\omega M(\phi_\omega,\psi_\omega)=P(\phi_\omega,\psi_\omega).
\]
\end{proof}

\begin{proposition}\label{K and virial}
Let $(u_0,v_0)\in H^1\!\times\!H^1$ and $(xu_0,xv_0)\in L^2\!\times\!L^2$. Let $(u,v)$ be the solution to (NLS) with initial data $(u_0,v_0)$ on $(T_\ast,T^\ast)$. Then, we have
\[
K_\omega^{20,8}(u,v)=\frac{d^2}{dt^2}(\|xu(t)\|_{L^2}^2+2\|xv(t)\|_{L^2})\ \  \left(=8K(u,v)-20P(u,v)\frac{}{}\right)
\]
for any $t\in(T_\ast,T^\ast)$.
\end{proposition}

\begin{proof}
By the change of variable,
\[
\left\|e^{\alpha\lambda} u(e^{\beta\lambda}\cdot)\right\|_{L^2}^2=\int_{\mathbb{R}^5}|e^{\alpha\lambda} u(e^{\beta\lambda} x)|^2dx=e^{2\alpha\lambda}\int_{\mathbb{R}^5}|u(y)|^2\frac{dy}{e^{5\beta\lambda}}=e^{(2\alpha-5\beta)\lambda}\|u\|_{L^2}^2,
\]
\[
\left\|\nabla e^{\alpha\lambda} u(e^{\beta\lambda}\cdot)\right\|_{L^2}^2=\int_{\mathbb{R}^5}|\nabla e^{\alpha\lambda} u(e^{\beta\lambda} x)|^2dx=e^{2\alpha\lambda}\int_{\mathbb{R}^5}|e^{\beta\lambda} (\nabla u)(y)|^2\frac{dy}{e^{5\beta\lambda}}=e^{(2\alpha-3\beta)\lambda}\|\nabla u\|_{L^2}^2,
\]
\[
\text{Re}\left(e^{\alpha\lambda} v(e^{\beta\lambda}\cdot),(e^{\alpha\lambda} u(e^{\beta\lambda}\cdot))^2\right)_{L^2}=\int_{\mathbb{R}^5}e^{\alpha\lambda} v(e^{\beta\lambda} x)\cdot\overline{e^{\alpha\lambda} u(e^{\beta\lambda} x)}^2dx=e^{(3\alpha-5\beta)\lambda}\text{Re}(v,u^2)_{L^2}.
\]
Thus,
\begin{align*}
K_\omega^{\alpha,\beta}(u,v)&=\partial_\lambda I_\omega(e^{\alpha\lambda} u(e^{\beta\lambda}\cdot),e^{\alpha\lambda} v(e^{\beta\lambda}\cdot))\arrowvert_{\lambda=0}\\
&=\partial_\lambda\left(\frac{\omega}{2}\|e^{\alpha\lambda} u(e^{\beta\lambda}\cdot)\|_{L^2}^2+\omega\|e^{\alpha\lambda} v(e^{\beta\lambda}\cdot)\|_{L^2}^2+\frac{1}{2}\|\nabla e^{\alpha\lambda} u(e^{\beta\lambda}\cdot)\|_{L^2}^2\right.\\
&\ \ \ \ \ \ \ \ \ \ \ \ \ \ \ \ \ \ \ \ \ \ \ \ \ \ \ \ \ \ \ \left.+\frac{1}{4}\|\nabla e^{\alpha\lambda} v(e^{\beta\lambda}\cdot)\|_{L^2}^2-\text{Re}(e^{\alpha\lambda} v(e^{\beta\lambda}\cdot),(e^{\alpha\lambda} u(e^{\beta\lambda}\cdot))^2)_{L^2}\right|_{\lambda=0}\\
&=\frac{\omega}{2}(2\alpha-5\beta)\|u\|_{L^2}^2+\omega(2\alpha-5\beta)\|v\|_{L^2}^2+\frac{1}{2}(2\alpha-3\beta)\|\nabla u\|_{L^2}^2\\
&\ \ \ \ \ \ \ \ \ \ \ \ \ \ \ \ \ \ \ \ \ \ \ \ \ \ \ \ \ \ \ +\frac{1}{4}(2\alpha-3\beta)\|\nabla v\|_{L^2}^2-(3\alpha-5\beta)\text{Re}(v,u^2)_{L^2}.
\end{align*}
By Theorem \ref{virial identity},
\begin{align*}
\frac{d^2}{dt^2}(\|xu(t)\|_{L^2}^2&+2\|xv(t)\|_{L^2})=10E(u_0,v_0)-2\left(\|\nabla u(t)\|_{L^2}^2+\frac{1}{2}\|\nabla v(t)\|_{L^2}^2\right)\\
&=10\left(\|\nabla u(t)\|_{L^2}^2+\frac{1}{2}\|\nabla v(t)\|_{L^2}^2-2\text{Re}(v,u^2)_{L^2}\right)-2\left(\|\nabla u(t)\|_{L^2}^2+\frac{1}{2}\|\nabla v(t)\|_{L^2}^2\right)\\
&=8\|\nabla u(t)\|_{L^2}^2+4\|\nabla v(t)\|_{L^2}^2-20\text{Re}(v,u^2)_{L^2}.
\end{align*}
We set $\alpha=20,\ \beta=8$, which means $2\alpha-5\beta=0,\ 2\alpha-3\beta=16,\ 3\alpha-5\beta=20$ and
\[
K_\omega^{\alpha,\beta}(u,v)=\frac{d^2}{dt^2}(\|xu(t)\|_{L^2}^2+2\|xv(t)\|_{L^2}).
\]
\end{proof}

\begin{lemma}\label{lemma of inferior}
We have
\[
\mu_\omega^{20,8}=\inf_{(u,v)\in H^1\times H^1\setminus\{(0,0)\}}\left\{L_\omega(u,v):K_\omega^{20,8}(u,v)\leq0\right\}.
\]
\end{lemma}

\begin{proof}
We take for any $(u,v)\in H^1\!\times\!H^1\setminus\{(0,0)\}$ with $K_\omega^{20,8}(u,v)<0$.
\begin{align*}
K_\omega^{20,8}(e^{20\lambda}u(e^{8\lambda}\cdot),e^{20\lambda}v(e^{8\lambda}\cdot))&=8\|\nabla e^{20\lambda}u(e^{8\lambda}\cdot)\|_{L^2}^2+4\|\nabla e^{20\lambda}v(e^{8\lambda}\cdot)\|_{L^2}^2-20P(e^{20\lambda}u(e^{8\lambda}\cdot),e^{20\lambda}v(e^{8\lambda}\cdot))\\
&=8e^{16\lambda}\|\nabla u\|_{L^2}^2+4e^{16\lambda}\|\nabla v\|_{L^2}^2-20e^{20\lambda}P(u,v).
\end{align*}
From $K_\omega^{20,8}(u,v)<0$, we have $P(u,v)>0$. Hence, there exists $\lambda_0<0$ such that
\begin{align}
K_\omega^{20,8}(e^{20\lambda_0}u(e^{8\lambda_0}\cdot),e^{20\lambda_0}v(e^{8\lambda_0}\cdot))=0.\label{05}
\end{align}
For such $\lambda_0<0$, we have
\begin{align}
&I_\omega(e^{20\lambda_0}u(e^{8\lambda_0}\cdot),e^{20\lambda_0}v(e^{8\lambda_0}\cdot)) \notag \\
&~~~=\frac{\omega}{2}M(e^{20\lambda_0}u(e^{8\lambda_0}\cdot),e^{20\lambda_0}v(e^{8\lambda_0}\cdot))+\frac{1}{2}K(e^{20\lambda_0}u(e^{8\lambda_0}\cdot),e^{20\lambda_0}v(e^{8\lambda_0}\cdot))-P(e^{20\lambda_0}u(e^{8\lambda_0}\cdot),e^{20\lambda_0}v(e^{8\lambda_0}\cdot)) \notag \\
&~~~=\frac{\omega}{2}M(e^{20\lambda_0}u(e^{8\lambda_0}\cdot),e^{20\lambda_0}v(e^{8\lambda_0}\cdot))+\frac{1}{10}K(e^{20\lambda_0}u(e^{8\lambda_0}\cdot),e^{20\lambda_0}v(e^{8\lambda_0}\cdot)) \notag \\
&~~~=\frac{\omega}{2}M(u,v)+\frac{1}{10}e^{16\lambda_0}K(u,v) \notag \\
&~~~\leq L_\omega(u,v)\label{52}
\end{align}
from Proposition \ref{proposition of ground state}. Also, for any $(u,v)\in H^1\!\times\!H^1\setminus\{(0,0)\}$ with $K_\omega^{20,8}(u,v)=0$, we have $I_\omega(u,v)=L_\omega(u,v)$. Thus,
\begin{align*}
\mu_\omega^{20,8}&=\inf_{(u,v)\in H^1\times H^1\setminus\{(0,0)\}}\left\{I_\omega(u,v):K_\omega^{20,8}(u,v)=0\right\}\\
&\leq\inf_{(u,v)\in H^1\times H^1\setminus\{(0,0)\}}\left\{L_\omega(u,v):K_\omega^{20,8}(u,v)\leq0\right\}.
\end{align*}
On the other hand, it follows that
\[
\mu_\omega^{20,8}=\inf_{(u,v)\in H^1\!\times\!H^1\setminus\{(0,0)\}}\{L_\omega(u,v):K_\omega^{20,8}(u,v)=0\},
\]
and so
\[
\mu_\omega^{20,8}\geq\inf_{(u,v)\in H^1\times H^1\setminus\{(0,0)\}}\left\{L_\omega(u,v):K_\omega^{20,8}(u,v)\leq0\right\}.
\]
\end{proof}


\begin{proposition}\label{minimizer existence}
Let $\omega>0$. Then we have the followings
\begin{itemize}
\item[(1)]$\mu_\omega^{20,8}\geq0$,
\item[(2)]there exists a pair of nonnegative radial functions $(\phi_0,\psi_0)\in C_\omega$ such that $I_\omega(\phi_0,\psi_0)=\mu_\omega^{20,8}$. Moreover, $(\phi_0,\psi_0)$ solves (gNLS).
\end{itemize}
\end{proposition}

\begin{proof}
(1)\ \ We apply Lemma \ref{lemma of inferior} for $\mu_\omega^{20,8}$ and take for any $(\phi,\psi)\in C_\omega$. From Proposition \ref{K and virial} with $K_\omega^{20,8}(\phi,\psi)=0$, we have $20P(\phi,\psi)=8K(\phi,\psi)$. Thus,
\[
I_\omega(\phi,\psi)=\frac{1}{2}K_\omega(\phi,\psi)-P(\phi,\psi)=\frac{1}{2}K_\omega(\phi,\psi)-\frac{2}{5}K(\phi,\psi)\geq\frac{1}{10}K_\omega(\phi,\psi)>0.
\]
Therefore, $\mu_\omega^{20,8}\geq0$.\\
(2)\ \ We take $\{(\phi_j,\psi_j)\}\subset\{(u,v)\in H^1\!\times\!H^1\setminus\{(0,0)\}:K_\omega^{20,8}(u,v)\leq0\}$ with $L_\omega(\phi_j,\psi_j)\searrow\mu_\omega^{20,8}$.\\
We consider the symmetric-decreasing rearrangement $(\phi_j^\ast,\psi_j^\ast)$ of $(\phi_j,\psi_j)$. Then, $0\geq K_\omega^{20,8}(\phi_j,\psi_j)$\\
$\geq K_\omega^{20,8}(\phi_j^\ast,\psi_j^\ast)$, $L_\omega(\phi_j,\psi_j)\geq L_\omega(\phi_j^\ast,\psi_j^\ast)$ by $M(\phi_j,\psi_j)=M(\phi_j^\ast,\psi_j^\ast)$, $K(\phi_j,\psi_j)\geq K(\phi_j^\ast,\psi_j^\ast)$, $P(\phi_j,\psi_j)\leq P(\phi_j^\ast,\psi_j^\ast)$ (see\,\cite{007}, p674). Thus, we may assume that $(\phi_j,\psi_j)$ is non-negative and radially symmetric function in $H^1\!\times\!H^1$. By \eqref{05}, for any $j\in\mathbb{N}$, there exists $\lambda_j\leq0$ such that
\[
K_\omega^{20,8}(e^{20\lambda_j}\phi_j(e^{8\lambda_j}\cdot),e^{20\lambda_j}\psi_j(e^{8\lambda_j}\cdot))=0.
\]
We set $(\widetilde{\phi}_j,\widetilde{\psi}_j)=(e^{20\lambda_j}\phi_j(e^{8\lambda_j}\cdot),e^{20\lambda_j}\psi_j(e^{8\lambda_j}\cdot))$. Also, $L_\omega(\phi_j,\psi_j)\geq L_\omega(\widetilde{\phi}_j,\widetilde{\psi}_j)\searrow\mu_\omega^{20,8}$ by \eqref{52}. Since $L_\omega(\widetilde{\phi}_j,\widetilde{\psi}_j)\leq L_\omega(\phi_1,\psi_1)$, 
$\{(\widetilde{\phi}_j,\widetilde{\psi}_j)\}$ is bounded in $H^1\!\times\!H^1$. There exists a subsequence still denoted by $\{(\widetilde{\phi}_j,\widetilde{\psi}_j)\}$ and $(\widetilde{\phi},\widetilde{\psi})\in H_{\text{rad}}^1\times H_{\text{rad}}^1\ (H_{\text{rad}}^1=\{u\in H^1:u\!:\!\text{radial}\})$ such that
\[
\widetilde{\phi}_j\rightharpoonup\widetilde{\phi}\ ,\ \ \widetilde{\psi}_j\rightharpoonup\widetilde{\psi}\ \ \text{as}\ \ j\rightarrow\infty\ \ \text{in}\ \ H^1,
\]
where $\rightharpoonup$ denotes weak convergence. By Strauss' compactness embedding $H_{\text{rad}}^1\subset L^3$,
\[
\widetilde{\phi}_j\rightarrow\widetilde{\phi}\ ,\ \ \widetilde{\psi}_j\rightarrow\widetilde{\psi}\ \ \text{as}\ \ j\rightarrow\infty\ \ \text{in}\ \ L^3.
\]
Thus,
\begin{align*}
\left|P(\widetilde{\phi},\widetilde{\psi})-P(\widetilde{\phi}_j,\widetilde{\psi}_j)\right|&=\left|\int_{\mathbb{R}^5}\widetilde{\psi}\widetilde{\phi}^2dx-\int_{\mathbb{R}^5}\widetilde{\psi}_j\widetilde{\phi}_j^2dx\right|\\
&\leq\left|\int_{\mathbb{R}^5}\widetilde{\psi}(\widetilde{\phi}^2-\widetilde{\phi}_j^2)dx\right|+\left|\int_{\mathbb{R}^5}(\widetilde{\psi}-\widetilde{\psi}_j)\widetilde{\phi}_j^2dx\right|\\
&\leq\|\widetilde{\psi}\|_{L^3}\|\widetilde{\phi}+\widetilde{\phi}_j\|_{L^3}\|\widetilde{\phi}-\widetilde{\phi}_j\|_{L^3}+\|\widetilde{\psi}-\widetilde{\psi}_j\|_{L^3}\|\widetilde{\phi}_j\|_{L^3}^2\\
&\longrightarrow\,0\ \ \text{as}\ \ j\rightarrow\infty.
\end{align*}
Therefore, $P(\widetilde{\phi},\widetilde{\psi})=\displaystyle\lim_{j\rightarrow\infty}P(\widetilde{\phi}_j,\widetilde{\psi}_j)$, which combined with a property of weak convergence, gives $\displaystyle K_\omega^{20,8}(\widetilde{\phi},\widetilde{\psi})\leq\lim_{j\rightarrow\infty}K_\omega^{20,8}(\widetilde{\phi}_j,\widetilde{\psi}_j)=0$. Here, we prove $(\widetilde{\phi},\widetilde{\psi})\neq(0,0)$ by contradiction. We assume  that $(\widetilde{\phi},\widetilde{\psi})=(0,0)$.
\[
8K(\widetilde{\phi}_j,\widetilde{\psi}_j)=20P(\widetilde{\phi}_j,\widetilde{\psi}_j)\longrightarrow20P(\widetilde{\phi},\widetilde{\psi})=0\ \ \text{as}\ \ n\rightarrow\infty,
\]
and we apply Gagliardo-Nirenberg's inequality\,:$\,\|u\|_{L^3}\leq c\|\nabla u\|_{L^2}^\frac{5}{6}\|u\|_{L^2}^\frac{1}{6}$, then
\[
P(\widetilde{\phi}_j,\widetilde{\psi_j})\leq \|\widetilde{\phi}_j\|_{L^3}^2\|\widetilde{\psi}_j\|_{L^3}\leq c\|\nabla\widetilde{\phi}_j\|_{L^2}^\frac{5}{3}\|\widetilde{\phi}_j\|_{L^2}^\frac{1}{3}\|\nabla\widetilde{\psi}_j\|_{L^2}^\frac{5}{6}\|\widetilde{\psi}_j\|_{L^2}^\frac{1}{6}\leq c\|\widetilde{\phi}_j\|_{L^2}^\frac{1}{3}\|\widetilde{\psi}_j\|_{L^2}^\frac{1}{6}K(\widetilde{\phi}_j,\widetilde{\psi}_j)^\frac{5}{4}.
\]
Since $\|\widetilde{\phi}_j\|_{L^2}\,,\,\|\widetilde{\psi}_j\|_{L^2}$ are bounded with respect to $j$, we have $P(\widetilde{\phi}_j,\widetilde{\psi_j})\leq cK(\widetilde{\phi}_j,\widetilde{\psi}_j)^\frac{5}{4}$. Thus, for sufficiently large $j\in\mathbb{N}$, it follows that
\[
K_\omega^{20,8}(\widetilde{\phi}_j,\widetilde{\psi}_j)=8K(\widetilde{\phi}_j,\widetilde{\psi}_j)-20P(\widetilde{\phi}_j,\widetilde{\psi}_j)\geq(8-cK(\widetilde{\phi}_j,\widetilde{\psi}_j)^\frac{1}{4})K(\widetilde{\phi}_j,\widetilde{\psi}_j)>0.
\]
This is contradiction. Therefore, since $(\widetilde{\phi}_j,\widetilde{\psi}_j)\neq(0,0)$, $\mu_\omega^{20,8}\leq\displaystyle L_\omega(\widetilde{\phi},\widetilde{\psi})\leq\lim_{j\rightarrow\infty}L_\omega(\widetilde{\phi}_j,\widetilde{\psi}_j)=\mu_\omega^{20,8}$, i.e. $L_\omega(\phi_j,\psi_j)=\mu_\omega^{20,8}$. Thus, there exists $\lambda_0\leq0$ such that $K_\omega^{20,8}(e^{20\lambda_0}\widetilde{\phi}(e^{8\lambda_0}\cdot),e^{20\lambda_0}\widetilde{\psi}(e^{8\lambda_0}\cdot))=0$  and $I_\omega(e^{20\lambda_0}\widetilde{\phi}(e^{8\lambda_0}\cdot),e^{20\lambda_0}\widetilde{\psi}(e^{8\lambda_0}\cdot))=L_\omega(e^{20\lambda_0}\widetilde{\phi}(e^{8\lambda_0}\cdot),e^{20\lambda_0}\widetilde{\psi}(e^{8\lambda_0}\cdot))=\mu_\omega^{20,8}$. We set $(\phi_0,\psi_0)=(e^{20\lambda_0}\widetilde{\phi}(e^{8\lambda_0}\cdot),e^{20\lambda_0}\widetilde{\psi}(e^{8\lambda_0}\cdot))$. From here, we will prove that this $(\phi_0,\psi_0)$ is nonnegative radial solution to (gNLS). Because we proved that $(\phi_0,\psi_0)$ is nonnegative and radial, it remain to prove that $(\phi_0,\psi_0)$ is a solution to (gNLS). Since $(\phi_0,\psi_0)$ is a minimizing of $I_\omega$, it is a critical point, i.e. for any $(u,v)\in H^1\!\times\!H^1$
\begin{align*}
\displaystyle\left.\frac{d}{ds}I_\omega(\phi_0+su,\psi_0+sv)\right|_{s=0}=0.
\end{align*}
Since
\begin{align}
\left.\frac{d}{ds}\|\phi_0+su\|_{L^2}^2\right|_{s=0}&=\left.\frac{d}{ds}\int_{\mathbb{R}^5}|\phi_0+su|^2dx\right|_{s=0}=\left.\int_{\mathbb{R}^5}2(\phi_0+su)\cdot udx\right|_{s=0}=2\int_{\mathbb{R}^5}\phi_0udx,\label{84}\\
\left.\frac{d}{ds}P(\phi_0+su,\psi_0+sv)\right|_{s=0}&=\left.\frac{d}{ds}\int_{\mathbb{R}^5}(\psi_0+sv)(\phi_0+su)^2dx\right|_{s=0} \notag \\
&\hspace{-1cm}=\left.\int_{\mathbb{R}^5}v(\phi_0+su)^2+(\psi_0+sv)\cdot2(\phi_0+su)udx\right|_{s=0}=\int_{\mathbb{R}^5}(\phi_0^2v+2\phi_0\psi_0u)dx,\label{85}
\end{align}
it follows that
\begin{align*}
&\left.\frac{d}{ds}I_\omega(\phi_0+su,\psi_0+sv)\right|_{s=0}\\
&\ \ \ \ \ \ \ \ \ \ \ \ \ \ \ \ \ =\frac{d}{ds}\left(\frac{\omega}{2}\|\phi_0+su\|_{L^2}^2+\omega\|\psi_0+sv\|_{L^2}^2+\frac{1}{2}\|\nabla(\phi_0+su)\|_{L^2}^2\right.\\
&\ \ \ \ \ \ \ \ \ \ \ \ \ \ \ \ \ \ \ \ \ \ \ \ \ \ \ \ \ \ \ \ \ \ \ \ \ \ \ \ \ \ \ \ \ \ \ \ \ \ \ \ \ \ \ \ \left.\left.+\frac{1}{4}\|\nabla(\psi_0+sv)\|_{L^2}^2-P(\phi_0+su,\psi_0+sv)\right)\right|_{s=0}\\
&\ \ \ \ \ \ \ \ \ \ \ \ \ \ \ \ \ =\omega\int_{\mathbb{R}^5}\phi_0udx+2\omega\int_{\mathbb{R}^5}\psi_0vdx+\int_{\mathbb{R}^5}\nabla\phi_0\cdot\nabla udx\\
&\ \ \ \ \ \ \ \ \ \ \ \ \ \ \ \ \ \ \ \ \ \ \ \ \ \ \ \ \ \ \ \ \ \ \ \ \ \ \ \ \ \ \ \ \ \ \ \ \ \ \ \ \ \ \ \ +\frac{1}{2}\int_{\mathbb{R}^5}\nabla\psi_0\cdot\nabla vdx-\int_{\mathbb{R}^5}(\phi_0^2v+2\phi_0\psi_0u)dx=0.
\end{align*}
Because $(u,v)\in H^1\!\times\!H^1$ is arbitrary, it follows that
\begin{equation}
\notag
\begin{cases}
\hspace{-0.4cm}&\displaystyle{\omega\int_{\mathbb{R}^5}\phi_0udx+\int_{\mathbb{R}^5}\nabla\phi_0\cdot\nabla udx=2\int_{\mathbb{R}^5}\phi_0\psi_0udx,}\\[0.3cm]
\hspace{-0.4cm}&\displaystyle{2\omega\int_{\mathbb{R}^5}\psi_0vdx+\frac{1}{2}\int_{\mathbb{R}^5}\nabla\psi_0\cdot\nabla vdx=\int_{\mathbb{R}^5}\phi_0^2vdx,}
\end{cases}
\end{equation}
i.e. $(\phi_0,\psi_0)\in H^1\!\times\!H^1$ is a solution to (gNLSI) and hence, the one is a solution to (gNLS).
\end{proof}

\begin{remark}
Combining Proposition \ref{proposition of ground state} and Proposition \ref{minimizer existence} (2), we have $\mu_\omega^{20,8}>0$. Indeed, $\mu_\omega^{20,8}=I_\omega(\phi_\omega,\psi_\omega)=\omega M(\phi_\omega,\psi_\omega)>0$.
\end{remark}

\begin{remark}
From now on, we denote the functions, which attain the infimum of $\mu_\omega^{20,8}$ and solve (gNLS) by $(\phi_\omega,\psi_\omega)$.
\end{remark}

\subsection{Small data scattering}


\begin{theorem}[Small data global existence]\label{Small data globally existence}
There exists $\delta_{sd}>0$ such that if
\[
\|(e^{it\Delta}u_0,e^{\frac{1}{2}it\Delta}v_0)\|_{S(\dot{H}^\frac{1}{2})\times S(\dot{H}^\frac{1}{2})}\leq\delta_{sd},
\]
then there exists the unique global solution $(u(t),v(t))\in\dot{H}^\frac{1}{2}\!\times\!\dot{H}^\frac{1}{2}$ to (NLS), which satisfies
\[
\|(u,v)\|_{S(\dot{H}^\frac{1}{2})\times S(\dot{H}^\frac{1}{2})}\leq4\|(e^{it\Delta}u_0,e^{\frac{1}{2}it\Delta}v_0)\|_{S(\dot{H}^\frac{1}{2})\times S(\dot{H}^\frac{1}{2})}.
\]
\end{theorem}


\begin{proof}
We define a set
\[
E=\left\{(u,v)\in\dot{H}^\frac{1}{2}\times\dot{H}^\frac{1}{2}:\|(u,v)\|_{S(\dot{H}^\frac{1}{2})\times S(\dot{H}^\frac{1}{2})}\leq4\|(e^{it\Delta}u_0,e^{\frac{1}{2}it\Delta}v_0)\|_{S(\dot{H}^\frac{1}{2})\times S(\dot{H}^\frac{1}{2})}\right\}
\]
and a distance $d((u_1,v_1),(u_2,v_2))$ on E
\[
d((u_1,v_1),(u_2,v_2))=\|(u_1,v_1)-(u_2,v_2)\|_{S(\dot{H}^\frac{1}{2})\times S(\dot{H}^\frac{1}{2})}.
\]
Also, we define a map
\begin{align*}
\Phi_{u_0}(u,v)(t)&=e^{it\Delta}u_0+2i\int_0^te^{i(t-s)\Delta}(v\bar{u})(s)ds,\\
\Phi_{v_0}(u,v)(t)&=e^{\frac{1}{2}it\Delta}v_0+i\int_0^te^{\frac{1}{2}i(t-s)\Delta}(u^2)(s)ds
\end{align*}
for $(u,v)\in E$. Since $\left(\frac{3}{2},3\right)$ is a $\dot{H}^{-\frac{1}{2}}$\,admissible and $(6,3)$ is a $\dot{H}^\frac{1}{2}$\,admissible pair,
\begin{align*}
\|\Phi_{u_0}(u,v)\|_{S(\dot{H}^\frac{1}{2})}&\leq\|e^{it\Delta}u_0\|_{S(\dot{H}^\frac{1}{2})}+2\left\|\int_0^te^{i(t-s)\Delta}(v\bar{u})(s)ds\right\|_{S(\dot{H}^\frac{1}{2})}\\
&\leq\|e^{it\Delta}u_0\|_{S(\dot{H}^\frac{1}{2})}+2c\|vu\|_{S'(\dot{H}^{-\frac{1}{2}})}\\
&\leq\|e^{it\Delta}u_0\|_{S(\dot{H}^\frac{1}{2})}+2c\|uv\|_{L^3L^\frac{3}{2}}\\
&\leq\|e^{it\Delta}u_0\|_{S(\dot{H}^\frac{1}{2})}+2c\|u\|_{L^6L^3}\|v\|_{L^6L^3}\\
&\leq\|e^{it\Delta}u_0\|_{S(\dot{H}^\frac{1}{2})}+2c\|u\|_{S(\dot{H}^\frac{1}{2})}\|v\|_{S(\dot{H}^\frac{1}{2})}\\
&\leq\|e^{it\Delta}u_0\|_{S(\dot{H}^\frac{1}{2})}+32c\|(e^{it\Delta}u_0,e^{\frac{1}{2}it\Delta}v_0)\|_{S(\dot{H}^\frac{1}{2})\times S(\dot{H}^\frac{1}{2})}^2\\
&\leq(1+32c\delta_{sd})\|(e^{it\Delta}u_0,e^{\frac{1}{2}it\Delta}v_0)\|_{S(\dot{H}^\frac{1}{2})\times S(\dot{H}^\frac{1}{2})}.
\end{align*}
Similarly,
\[
\|\Phi_{v_0}(u,v)\|_{S(\dot{H}^\frac{1}{2})}\leq(1+32c\delta_{sd})\|(e^{it\Delta}u_0,e^{\frac{1}{2}it\Delta}v_0)\|_{S(\dot{H}^\frac{1}{2})\times S(\dot{H}^\frac{1}{2})}.
\]
Combining these inequalities,
\[
\|(\Phi_{u_0}(u,v),\Phi_{v_0}(u,v))\|_{S(\dot{H}^\frac{1}{2})\times S(\dot{H}^\frac{1}{2})}\leq2(1+32c\delta_{sd})\|(e^{it\Delta}u_0,e^{\frac{1}{2}it\Delta}v_0)\|_{S(\dot{H}^\frac{1}{2})\times S(\dot{H}^\frac{1}{2})}.
\]
Thus, if $\displaystyle\delta_{sd}\leq\frac{1}{32c}$, then
\[
\|(\Phi_{u_0}(u,v),\Phi_{v_0}(u,v))\|_{S(\dot{H}^\frac{1}{2})\times S(\dot{H}^\frac{1}{2})}\leq4\|(e^{it\Delta}u_0,e^{\frac{1}{2}it\Delta}v_0)\|_{S(\dot{H}^\frac{1}{2})\times S(\dot{H}^\frac{1}{2})}.
\]
Also, for $(u_1,v_1),\,(u_2,v_2)\in E$
\begin{align*}
&\|\Phi_{u_0}(u_1,v_1)-\Phi_{u_0}(u_2,v_2)\|_{S(\dot{H}^\frac{1}{2})}=2\left\|\int_0^te^{i(t-s)\Delta}(v_1\overline{u_1}-v_2\overline{u_2})(s)ds\right\|_{S(\dot{H}^\frac{1}{2})}\\
&\hspace{4cm}\leq2c\|v_1\overline{u_1}-v_2\overline{u_2}\|_{S'(\dot{H}^{-\frac{1}{2}})}\\
&\hspace{4cm}\leq2c\|v_1\overline{u_1}-v_2\overline{u_2}\|_{L^3L^\frac{3}{2}}\\
&\hspace{4cm}\leq2c\left(\|v_1(\overline{u_1}-\overline{u_2})\|_{L^3L^\frac{3}{2}}+\|(v_1-v_2)\overline{u_2}\|_{L^3L^\frac{3}{2}}\right)\\
&\hspace{4cm}\leq2c\left(\|v_1\|_{L^6L^3}\|u_1-u_2\|_{L^6L^3}+\|u_2\|_{L^6L^3}\|v_1-v_2\|_{L^6L^3}\right)\\
&\hspace{4cm}\leq2c\left(\|v_1\|_{S(\dot{H}^\frac{1}{2})}\|u_1-u_2\|_{S(\dot{H}^\frac{1}{2})}+\|u_2\|_{S(\dot{H}^\frac{1}{2})}\|v_1-v_2\|_{S(\dot{H}^\frac{1}{2})}\right)\\
&\hspace{4cm}\leq8c\|(e^{it\Delta}u_0,e^{\frac{1}{2}it\Delta}v_0)\|_{S(\dot{H}^\frac{1}{2})\times S(\dot{H}^\frac{1}{2})}\|(u_1,v_1)-(u_2,v_2)\|_{S(\dot{H}^\frac{1}{2})\times S(\dot{H}^\frac{1}{2})}\\
&\hspace{4cm}\leq8c\delta_{sd}\|(u_1,v_1)-(u_2,v_2)\|_{S(\dot{H}^{\frac{1}{2}})\times S(\dot{H}^{\frac{1}{2}})}.
\end{align*}
Similarly,
\[
\|\Phi_{v_0}(u_1,v_1)-\Phi_{v_0}(u_2,v_2)\|_{S(\dot{H}^\frac{1}{2})}\leq16c\delta_{sd}\|(u_1,v_1)-(u_2,v_2)\|_{S(\dot{H}^\frac{1}{2})\times S(\dot{H}^\frac{1}{2})}.
\]
Combining these inequalities,
\begin{align*}
&\|(\Phi_{u_0}(u_1,v_1),\Phi_{v_0}(u_1,v_1))-(\Phi_{u_0}(u_2,v_2),\Phi_{v_0}(u_2,v_2))\|_{S(\dot{H}^\frac{1}{2})\times S(\dot{H}^\frac{1}{2})}\\
&\ \ \ \ \ \ \ \ \ \ \ \ \ \ \ \ \ \ \ \ \ \ \ \ \ \ \ \ \ \ \ \ \ \ \ \ \ \ \ \ \ \ \ \ \ \ \ \ \ \ \ \ \ \leq24c\delta_{sd}\|(u_1,v_1)-(u_2,v_2)\|_{S(\dot{H}^\frac{1}{2})\times S(\dot{H}^\frac{1}{2})}\\
&\ \ \ \ \ \ \ \ \ \ \ \ \ \ \ \ \ \ \ \ \ \ \ \ \ \ \ \ \ \ \ \ \ \ \ \ \ \ \ \ \ \ \ \ \ \ \ \ \ \ \ \ \ \leq\frac{3}{4}\|(u_1,v_1)-(u_2,v_2)\|_{S(\dot{H}^\frac{1}{2})\times S(\dot{H}^\frac{1}{2})}.
\end{align*}
Therefore, the unique solution $(u,v)\in\dot{H}^\frac{1}{2}\!\times\!\dot{H}^\frac{1}{2}$ to (NLS) exists time-globally, and
\[
\|(u,v)\|_{S(\dot{H}^\frac{1}{2})\times S(\dot{H}^\frac{1}{2})}\leq4\|(e^{it\Delta}u_0,e^{\frac{1}{2}it\Delta}v_0)\|_{S(\dot{H}^\frac{1}{2})\times S(\dot{H}^\frac{1}{2})}.
\]
\end{proof}


\begin{lemma}\label{lemma for scattering}
Let $(u_0,v_0)\in H^1\!\times\!H^1$ and $(u,v)$ be a time-global solution with
\[
\|(u,v)\|_{S(\dot{H}^\frac{1}{2})\times S(\dot{H}^\frac{1}{2})}<\infty,\  \sup_{t\in [0,\infty)}\|(u(t),v(t))\|_{H^1\times H^1}<\infty.
\]
Then, $(u,v)$ is in $L^\rho([0,\infty):W^{1.\gamma})\times L^\rho([0,\infty):W^{1.\gamma})$ for any $L^2$ admissible pair $(\rho,\gamma)$.
\end{lemma}


\begin{proof}
We fix $T\geq0$. We consider the integral equation with data at $T$
\begin{equation}
\notag
\begin{cases}
\hspace{-0.4cm}&\displaystyle{u(t)=e^{i(t-T)\Delta}u(T)+2i\int_T^te^{i(t-s)\Delta}(v\overline{u})(s)ds,}\\[0.3cm]
\hspace{-0.4cm}&\displaystyle{v(t)=e^{\frac{1}{2}i(t-T)\Delta}v(T)+i\int_T^te^{\frac{1}{2}i(t-s)\Delta}(u^2)(s)ds.}
\end{cases}
\end{equation}
Replacing $t$ with $t+T$,
\begin{equation}
\notag
\begin{cases}
\hspace{-0.4cm}&\displaystyle{u(t+T)=e^{it\Delta}u(T)+2i\int_T^{t+T}e^{i(t+T-s)\Delta}(v\overline{u})(s)ds,}\\[0.3cm]
\hspace{-0.4cm}&\displaystyle{v(t+T)=e^{\frac{1}{2}it\Delta}v(T)+i\int_T^{t+T}e^{\frac{1}{2}i(t+T-s)\Delta}(u^2)(s)ds.}
\end{cases}
\end{equation}
Replacing $s$ with $s+T$,
\begin{equation}
\notag
\begin{cases}
\hspace{-0.4cm}&\displaystyle{u(t+T)=e^{it\Delta}u(T)+2i\int_0^te^{i(t-s)\Delta}(v\overline{u})(s+T)ds,}\\[0.3cm]
\hspace{-0.4cm}&\displaystyle{v(t+T)=e^{\frac{1}{2}it\Delta}v(T)+i\int_0^te^{\frac{1}{2}i(t-s)\Delta}(u^2)(s+T)ds.}
\end{cases}
\end{equation}
These equations combined with $\left(\frac{12}{5},3\right)$ being a $L^2$\,admissible pair, gives
\begin{align*}
\|u(\cdot+T)\|_{L_{[0,\tau]}^\rho W^{1.\gamma}}&\leq\|e^{iT\Delta}u(T)\|_{L_{[0,\tau]}^\rho W^{1.\gamma}}+2\left\|\int_0^te^{i(t-s)\Delta}(v\overline{u})(s+T)ds\right\|_{L_{[0,\tau]}^\rho W^{1.\gamma}}\\
&\leq c\|u(T)\|_{H^1}+c\|(vu)(\cdot+T)\|_{L_{[0,\tau]}^\frac{12}{7} W^{1.\frac{3}{2}}}\\
&=c\|u(T)\|_{H^1}+c\|vu\|_{L_{[T,T+\tau]}^\frac{12}{7} W^{1.\frac{3}{2}}}\\
&\leq c\|u(T)\|_{H^1}+c\|v\|_{L_{[T,T+\tau]}^\frac{12}{5} W^{1.3}}\|u\|_{L_{[T,T+\tau]}^6L^3}+c\|v\|_{L_{[T,T+\tau]}^6L^3}\|u\|_{L_{[T,T+\tau]}^\frac{12}{5} W^{1.3}}.
\end{align*}
for any $\tau>0$. Therefore,
\[
\|u\|_{L_{[T,T+\tau]}^\rho W^{1.\gamma}}\leq c\|u(T)\|_{H^1}+c\|v\|_{L_{[T,T+\tau]}^\frac{12}{5} W^{1.3}}\|u\|_{L_{[T,T+\tau]}^6L^3}+c\|v\|_{L_{[T,T+\tau]}^6L^3}\|u\|_{L_{[T,T+\tau]}^\frac{12}{5} W^{1.3}}.
\]
Similarly,
\[
\|v\|_{L_{[T,T+\tau]}^\rho W^{1.\gamma}}\leq c\|v(T)\|_{H^1}+c\|u\|_{L_{[T,T+\tau]}^\frac{12}{5} W^{1.3}}\|u\|_{L_{[T,T+\tau]}^6L^3}.
\]
Here, since $(6,3)$ is a $\dot{H}^\frac{1}{2}$ admissible pair and $\|(u,v)\|_{S(\dot{H}^\frac{1}{2})\times S(\dot{H}^\frac{1}{2})}<\infty$, we obtain \[
\max\{c\|u\|_{L_{[T,T+\tau]}^6L^3},c\|v\|_{L_{[T,T+\tau]}^6L^3}\}<\frac{1}{4}
\]
for sufficiently large $T>0$. For such $T>0$,
\begin{align}
\|u\|_{L_{[T,T+\tau]}^\rho W^{1.\gamma}}+\|v\|_{L_{[T,T+\tau]}^\rho W^{1.\gamma}}\leq c\|u(T)\|_{H^1}+c\|v(T)\|_{H^1}+\frac{1}{2}\|v\|_{L_{[T,T+\tau]}^\frac{12}{5} W^{1.3}}+\frac{1}{2}\|u\|_{L_{[T,T+\tau]}^\frac{12}{5} W^{1.3}}.\label{66}
\end{align}
We can take $(\rho,\gamma)=\left(\frac{12}{5},3\right)$, which is a $L^2$ admissible pair. The $H^1\!\times\!H^1$ time local solution is in $L^\frac{12}{5}(I:W^{1.3})$ see \cite{007}. This fact combined with $\displaystyle\sup_{t\in [0,\infty)}\|(u(t),v(t))\|_{H^1\times H^1}<\infty$ gives $u,v\in L_{\text{loc}}^\frac{12}{5}([0,\infty):W^{1.3})$.
\[
\|u\|_{L_{[T,T+\tau]}^\frac{12}{5} W^{1.3}}+\|v\|_{L_{[T,T+\tau]}^\frac{12}{5} W^{1.3}}\leq 2c\|u(T)\|_{H^1}+2c\|v(T)\|_{H^1}.
\]
Because $\tau>0$ is arbitrary, it follows that
\[
\|u\|_{L_{[T,\infty)}^\frac{12}{5} W^{1.3}}+\|v\|_{L_{[T,\infty)}^\frac{12}{5} W^{1.3}}\leq 2c\|u(T)\|_{H^1}+2c\|v(T)\|_{H^1}.
\]
This formula combined with \eqref{66} gives
\[
\|u\|_{L_{[T,\infty)}^\rho W^{1.\gamma}}+\|v\|_{L_{[T,\infty)}^\rho W^{1.\gamma}}\leq 2c\|u(T)\|_{H^1}+2c\|v(T)\|_{H^1}.
\]
Therefore, $(u,v)\in L^\rho([0,\infty):W^{1.\gamma})\!\times\!L^\rho([0,\infty):W^{1.\gamma}).$\\
\end{proof}


\begin{theorem}[$H^1\!\times\!H^1$\ Scattering]\label{Scattering}
Under the same assumption as Lemma \ref{lemma for scattering}, $(u,v)$ scatters in $H^1\!\times\!H^1$.
\end{theorem}


\begin{proof}
Let
\[
U(t)=e^{-it\Delta}u(t)=u_0+2i\int_0^te^{-is\Delta}(v\overline{u})(s)ds,
\]
\[
V(t)=e^{-\frac{1}{2}it\Delta}v(t)=v_0+i\int_0^te^{-\frac{1}{2}is\Delta}(u^2)(s)ds.
\]
If $t>\tau>0$, then
\begin{align*}
\|U(t)-U(\tau)\|_{H^1}&=2\left\|\int_\tau^te^{-is\Delta}(v\overline{u})(s)ds\right\|_{H^1}\\
&\leq c\|v\|_{L_{[\tau,t]}^\frac{12}{5} W^{1.3}}\|u\|_{L_{[\tau,t]}^6L^3}+c\|v\|_{L_{[\tau,t]}^6L^3}\|u\|_{L_{[\tau,t]}^\frac{12}{5} W^{1.3}}\\
&\longrightarrow0\ \ \text{as}\ \ \tau,t\rightarrow\infty
\end{align*}
by the proof of Lemma \ref{lemma for scattering}. Similarly,
\[
\|V(t)-V(\tau)\|_{H^1}\longrightarrow0\ \ \text{as}\ \ \tau,t\rightarrow\infty.
\]
Therefore, there exists $u^+$ and $v^+\in H^1$ such that
\[
u(t)-e^{it\Delta}u^+\longrightarrow0,\ \ v(t)-e^{\frac{1}{2}it\Delta}v^+\longrightarrow0\ \ \text{in}\ \ H^1\ \ \text{as}\ \ t\rightarrow\infty.
\]
\end{proof}

\subsection{Long time perturbation}


\begin{theorem}[Long time perturbation]\label{Long time perturbation}
For each $A>1$, there exist $\varepsilon_0(A)<1$ and $C(A)>1$ such that the following holds. Let $(u,v)=(u(x,t),v(x,t))\in H^1\!\times\!H^1$ for all $t$ and solve (NLS). Let $e_1$, $e_2$, $\widetilde{u}$, and $\widetilde{v}$ satisfy
\begin{equation}
\notag
\begin{cases}
\hspace{-0.4cm}&\displaystyle{e_1=i\partial_t\widetilde{u}+\Delta \widetilde{u}+2\widetilde{v}\overline{\widetilde{u}},}\\
\hspace{-0.4cm}&\displaystyle{e_2=i\partial_t\widetilde{v}+\frac{1}{2}\Delta \widetilde{v}+\widetilde{u}^2.}
\end{cases}
\end{equation}
\begin{align}
\|(\widetilde{u},\widetilde{v})\|_{S(\dot{H}^\frac{1}{2})\times S(\dot{H}^\frac{1}{2})}\leq A,\label{06}
\end{align}
\begin{align}
\|(e_1,e_2)\|_{S'(\dot{H}^{-\frac{1}{2}})\times S'(\dot{H}^{-\frac{1}{2}})}\leq\varepsilon_0,\label{07}
\end{align}
and
\begin{align}
\left\|\left(e^{i(t-t_0)\Delta}\left(u(t_0)-\widetilde{u}(t_0)\right),e^{\frac{1}{2}i(t-t_0)\Delta}\left(v(t_0)-\widetilde{v}(t_0)\right)\right)\right\|_{S(\dot{H}^\frac{1}{2})\times S(\dot{H}^\frac{1}{2})}\leq\varepsilon_0.\label{08}
\end{align}
Then,
\[
\|(u,v)\|_{S(\dot{H}^\frac{1}{2})\times S(\dot{H}^\frac{1}{2})}\leq C=C(A)<\infty.
\]
\end{theorem}


\begin{proof}
We define $w_1=u-\widetilde{u},\ w_2=v-\widetilde{v}$, then
\begin{align*}
0&=i\partial_tu+\Delta u+2v\bar{u}\\
&=i\partial_t(w_1+\widetilde{u})+\Delta (w_1+\widetilde{u})+2(w_2+\widetilde{v})(\overline{w_1}+\overline{\widetilde{u}})\\
&=i\partial_tw_1+i\partial_t\widetilde{u}+\Delta w_1+\Delta \widetilde{u}+2\overline{w_1}w_2+2\overline{w_1}\widetilde{v}+2w_2\overline{\widetilde{u}}+2\widetilde{v}\overline{\widetilde{u}}\\
&=i\partial_tw_1+\Delta w_1+2\overline{w_1}w_2+2\overline{w_1}\widetilde{v}+2w_2\overline{\widetilde{u}}+e_1,\\
0&=i\partial_t(w_2+\widetilde{v})+\frac{1}{2}\Delta(w_2+\widetilde{v})+(w_1+\widetilde{u})^2\\
&=i\partial_tw_2+i\partial_t\widetilde{v}+\frac{1}{2}\Delta w_2+\frac{1}{2}\Delta\widetilde{v}+w_1^2+2w_1\widetilde{u}+\widetilde{u}^2\\
&=i\partial_tw_2+\frac{1}{2}\Delta w_2+w_1^2+2w_1\widetilde{u}+e_2.
\end{align*}
Thus,
\begin{equation}
\notag
\begin{cases}
\hspace{-0.4cm}&\displaystyle{i\partial_tw_1+\Delta w_1+2\overline{w_1}w_2+2\overline{w_1}\widetilde{v}+2w_2\overline{\widetilde{u}}+e_1=0,}\\
\hspace{-0.4cm}&\displaystyle{i\partial_tw_2+\frac{1}{2}\Delta w_2+w_1^2+2w_1\widetilde{u}+e_2=0.}
\end{cases}
\end{equation}
Since $\|(\widetilde{u},\widetilde{v})\|_{S(\dot{H}^\frac{1}{2})\times S(\dot{H}^\frac{1}{2})}\leq A$, for any $\delta>0$, there exists $N\in\mathbb{N}$ and $I_j$ $(j=1,2,\cdots,N)$ with $\displaystyle[t_0,\infty)=\bigcup_{j=1}^NI_j=\bigcup_{j=1}^N[t_{j-1},t_j)$ and $I_j$ are pairwise disjoint such that $\|(\widetilde{u},\widetilde{v})\|_{S(\dot{H}^\frac{1}{2}:I_j)\times S(\dot{H}^\frac{1}{2}:I_j)}\leq\delta$.
We consider the following integral equation.
\begin{equation}
\notag
\begin{cases}
\hspace{-0.4cm}&\displaystyle{w_1(t)=e^{i(t-t_j)\Delta}w_1(t_j)+i\int_{t_j}^te^{i(t-s)\Delta}(2\overline{w_1}w_2+2\overline{w_1}\widetilde{v}+2w_2\overline{\widetilde{u}}+e_1)(s)ds,}\\[0.3cm]
\hspace{-0.4cm}&\displaystyle{w_2(t)=e^{\frac{1}{2}i(t-t_j)\Delta}w_2(t_j)+i\int_{t_j}^te^{\frac{1}{2}i(t-s)\Delta}(w_1^2+2w_1\widetilde{u}+e_2)(s)ds.}
\end{cases}
\end{equation}
We define a set
\[
E=\left\{(w_1,w_2):\|(w_1,w_2)\|_{S(\dot{H}^\frac{1}{2}:I_j)\times S(\dot{H}^\frac{1}{2}:I_j)}\leq4B\right\},
\]
where $B=\|(e^{i(t-t_j)\Delta}w_1(t_j),e^{\frac{1}{2}i(t-t_j)\Delta}w_2(t_j))\|_{S(\dot{H}^\frac{1}{2}:I_j)\times S(\dot{H}^\frac{1}{2}:I_j)}+c\varepsilon_0$.
We define a distance $d((u_1,v_1),(u_2,v_2))$ on $E$
\[
d((u_1,v_1),(u_2,v_2))=\|(u_1,v_1)-(u_2,v_2)\|_{S(\dot{H}^\frac{1}{2}:I_j)\times S(\dot{H}^\frac{1}{2}:I_j)}.
\]
We define maps
\begin{align*}
\Phi_1(w_1,w_2)(t)&=e^{i(t-t_j)\Delta}w_1(t_j)+i\int_{t_j}^te^{i(t-s)\Delta}(2\overline{w_1}w_2+2\overline{w_1}\widetilde{v}+2w_2\overline{\widetilde{u}}+e_1)(s)ds,\\
\Phi_2(w_1,w_2)(t)&=e^{\frac{1}{2}i(t-t_j)\Delta}w_2(t_j)+i\int_{t_j}^te^{\frac{1}{2}i(t-s)\Delta}(w_1^2+2w_1\widetilde{u}+e_2)(s)ds
\end{align*}
for $(w_1,w_2)\in E$. Then,
\begin{align*}
&\|\Phi_1(w_1,w_2)\|_{S(\dot{H}^\frac{1}{2}:I_j)}\\
&\ \ \ \ \ \ \ \ \leq\|e^{i(t-t_j)\Delta}w_1(t_j)\|_{S(\dot{H}^\frac{1}{2}:I_j)}+\left\|\int_{t_j}^te^{i(t-s)\Delta}(2\overline{w_1}w_2+2\overline{w_1}\widetilde{v}+2w_2\overline{\widetilde{u}}+e_1)(s)ds\right\|_{S(\dot{H}^\frac{1}{2}:I_j)}\\
&\ \ \ \ \ \ \ \ \leq\|e^{i(t-t_j)\Delta}w_1(t_j)\|_{S(\dot{H}^\frac{1}{2}:I_j)}+2c\|w_1w_2\|_{S'(\dot{H}^{-\frac{1}{2}}:I_j)}+2c\|w_1\widetilde{v}\|_{S'(\dot{H}^{-\frac{1}{2}}:I_j)}\\
&\ \ \ \ \ \ \ \ \ \ \ \ \ \ \ \ \ \ \ \ \ \ \ \ \ \ \ \ \ \ \ \ \ \ \ \ \ \ \ \ \ \ \ \ \ \ \ \ \ \ \ \ \ \ \ \ \ \ \ \ \ \ \ \ \ \ \ \ +2c\|w_2\widetilde{u}\|_{S'(\dot{H}^{-\frac{1}{2}}:I_j)}+c\|e_1\|_{S'(\dot{H}^{-\frac{1}{2}}:I_j)}\\
&\ \ \ \ \ \ \ \ \leq\|e^{i(t-t_j)\Delta}w_1(t_j)\|_{S(\dot{H}^\frac{1}{2}:I_j)}+2c\|w_1w_2\|_{L_{I_j}^3L^\frac{3}{2}}+2c\|w_1\widetilde{v}\|_{L_{I_j}^3L^\frac{3}{2}}+2c\|w_2\widetilde{u}\|_{L_{I_j}^3L^\frac{3}{2}}+c\varepsilon_0\\
&\ \ \ \ \ \ \ \ \leq\|e^{i(t-t_j)\Delta}w_1(t_j)\|_{S(\dot{H}^\frac{1}{2}:I_j)}+2c\|w_1\|_{L_{I_j}^6L^3}\|w_2\|_{L_{I_j}^6L^3}+2c\|w_1\|_{L_{I_j}^6L^3}\|\widetilde{v}\|_{L_{I_j}^6L^3}\\
&\ \ \ \ \ \ \ \ \ \ \ \ \ \ \ \ \ \ \ \ \ \ \ \ \ \ \ \ \ \ \ \ \ \ \ \ \ \ \ \ \ \ \ \ \ \ \ \ \ \ \ \ \ \ \ \ \ \ \ \ \ \ \ \ \ \ \ \ \ \ \ \ \ \ \ \ \ \ \ \ +2c\|w_2\|_{L_{I_j}^6L^3}\|\widetilde{u}\|_{L_{I_j}^6L^3}+c\varepsilon_0\\
&\ \ \ \ \ \ \ \ \leq B+32cB^2+8c\delta B+8c\delta B\\
&\ \ \ \ \ \ \ \ =(1+32cB+16c\delta)B,\\
&\|\Phi_2(w_1,w_2)\|_{S(\dot{H}^\frac{1}{2}:I_j)}\\
&\ \ \ \ \ \ \ \ \leq\|e^{\frac{1}{2}i(t-t_j)\Delta}w_2(t_j)\|_{S(\dot{H}^\frac{1}{2}:I_j)}+\left\|\int_{t_j}^te^{\frac{1}{2}i(t-s)\Delta}(w_1^2+2w_1\widetilde{u}+e_2)(s)ds\right\|_{S(\dot{H}^\frac{1}{2}:I_j)}\\
&\ \ \ \ \ \ \ \ \leq\|e^{\frac{1}{2}i(t-t_j)\Delta}w_2(t_j)\|_{S(\dot{H}^\frac{1}{2}:I_j)}+c\|w_1\|_{L_{I_j}^6L^3}^2+2c\|w_1\|_{L_{I_j}^6L^3}\|\widetilde{u}\|_{L_{I_j}^6L^3}+c\varepsilon_0\\
&\ \ \ \ \ \ \ \ \leq B+16cB^2+8c\delta B\\
&\ \ \ \ \ \ \ \ =(1+16cB+8c\delta)B.
\end{align*}
Combining these inequalities,
\[
\|(\Phi_1(w_1,w_2),\Phi_2(w_1,w_2))\|_{S(\dot{H}^\frac{1}{2}:I_j)\times S(\dot{H}^\frac{1}{2}:I_j)}\leq(2+48cB+24c\delta)B.
\]
Thus, if $48cB\leq1$ and $24c\delta\leq1$, i.e. $B\leq1/48c$ and $\delta\leq1/24c$, then
\[
\|(\Phi_1(w_1,w_2),\Phi_2(w_1,w_2))\|_{S(\dot{H}^\frac{1}{2}:I_j)\times S(\dot{H}^\frac{1}{2}:I_j)}\leq4B.
\]
This means $(\Phi_1(w_1,w_2),\Phi_2(w_1,w_2))\in E$. We estimate
\begin{align*}
&\|\Phi_1(w_1,w_2)-\Phi_1(w_1',w_2')\|_{S(\dot{H}^\frac{1}{2}:I_j)}\\
&\ \ \ \ \ \ \ \ \ \ \ \ \ \ \ \ \ \ \ \ \leq2c\|\overline{w_1}w_2-\overline{w_1'}w_2'\|_{L_{I_j}^3L^\frac{3}{2}}+2c\|\overline{w_1}\widetilde{v}-\overline{w_1'}\widetilde{v}\|_{L_{I_j}^3L^\frac{3}{2}}+2c\|w_2\overline{\widetilde{u}}-w_2'\overline{\widetilde{u}}\|_{L_{I_j}^3L^\frac{3}{2}}\\
&\ \ \ \ \ \ \ \ \ \ \ \ \ \ \ \ \ \ \ \ \leq2c\left(\|(\overline{w_1}-\overline{w_1'})w_2\|_{L_{I_j}^3L^\frac{3}{2}}+\|\overline{w_1'}(w_2-w_2')\|_{L_{I_j}^3L^\frac{3}{2}}\right)\\
&\ \ \ \ \ \ \ \ \ \ \ \ \ \ \ \ \ \ \ \ \ \ \ \ \ \ \ \ \ \ \ \ \ \ \ \ \ \ \ \ +2c\|w_1-w_1'\|_{L_{I_j}^6L^3}\|\widetilde{v}\|_{L_{I_j}^6L^3}+2c\|w_2-w_2'\|_{L_{I_j}^6L^3}\|\widetilde{u}\|_{L_{I_j}^6L^3}\\
&\ \ \ \ \ \ \ \ \ \ \ \ \ \ \ \ \ \ \ \ \leq2c\left(\|w_1-w_1'\|_{L_{I_j}^6L^3}\|w_2\|_{L_{I_j}^6L^3}+\|w_1'\|_{L_{I_j}^6L^3}\|w_2-w_2'\|_{L_{I_j}^6L^3}\right)\\
&\ \ \ \ \ \ \ \ \ \ \ \ \ \ \ \ \ \ \ \ \ \ \ \ \ \ \ \ \ \ \ \ \ \ \ \ \ \ \ \ +2c\|w_1-w_1'\|_{L_{I_j}^6L^3}\|\widetilde{v}\|_{L_{I_j}^6L^3}+2c\|w_2-w_2'\|_{L_{I_j}^6L^3}\|\widetilde{u}\|_{L_{I_j}^6L^3}\\
&\ \ \ \ \ \ \ \ \ \ \ \ \ \ \ \ \ \ \ \ \leq2c\left(4B\|w_1-w_1'\|_{L_{I_j}^6L^3}+4B\|w_2-w_2'\|_{L_{I_j}^6L^3}\right)\\
&\ \ \ \ \ \ \ \ \ \ \ \ \ \ \ \ \ \ \ \ \ \ \ \ \ \ \ \ \ \ \ \ \ \ \ \ \ \ \ \ \ \ \ \ \ \ \ \ \ \ \ \ \ \ \ \ \ +2c\delta\|w_1-w_1'\|_{L_{I_j}^6L^3}+2c\delta\|w_2-w_2'\|_{L_{I_j}^6L^3}\\
&\ \ \ \ \ \ \ \ \ \ \ \ \ \ \ \ \ \ \ \ =(8cB+2c\delta)\left(\|w_1-w_1'\|_{S(\dot{H}^\frac{1}{2}:I_j)}+\|w_2-w_2'\|_{S(\dot{H}^\frac{1}{2}:I_j)}\right)\\
&\ \ \ \ \ \ \ \ \ \ \ \ \ \ \ \ \ \ \ \ \leq\frac{1}{4}\|(w_1,w_2)-(w_1',w_2')\|_{S(\dot{H}^\frac{1}{2}:I_j)\times S(\dot{H}^\frac{1}{2}:I_j)},\\
&\|\Phi_2(w_1,w_2)-\Phi_2(w_1',w_2')\|_{S(\dot{H}^\frac{1}{2}:I_j)}\\
&\ \ \ \ \ \ \ \ \ \ \ \ \ \ \ \ \ \ \ \ \leq c\|w_1^2-w_1'^2\|_{L_{I_j}^3L^\frac{3}{2}}+2c\|w_1\widetilde{u}-w_1'\widetilde{u}\|_{L_{I_j}^3L^\frac{3}{2}}\\
&\ \ \ \ \ \ \ \ \ \ \ \ \ \ \ \ \ \ \ \ \leq c\left(\|w_1\|_{L_{I_j}^6L^3}+\|w_1'\|_{L_{I_j}^6L^3}+2\|\widetilde{u}\|_{L_{I_j}^6L^3}\right)\|w_1-w_1'\|_{L_{I_j}^6L^3}\\
&\ \ \ \ \ \ \ \ \ \ \ \ \ \ \ \ \ \ \ \ \leq(8cB+2c\delta)\|w_1-w_1'\|_{L_{I_j}^6L^3}\\
&\ \ \ \ \ \ \ \ \ \ \ \ \ \ \ \ \ \ \ \ \leq\frac{1}{4}\|(w_1,w_2)-(w_1',w_2')\|_{S(\dot{H}^\frac{1}{2}:I_j)\times S(\dot{H}^\frac{1}{2}:I_j)}.
\end{align*}
Combining these inequalities,
\begin{align*}
&\|(\Phi_1(w_1,w_2),\Phi_2(w_1,w_2))-(\Phi_1(w_1',w_2'),\Phi_2(w_1',w_2'))\|_{S(\dot{H}^\frac{1}{2}:I_j)\times S(\dot{H}^\frac{1}{2}:I_j)}\\
&~~~~~~~~~~~~~~~~~~~~~~~~~~~~~~~~~~~~~~~~~~~~~~~~~~~~~~~~~~~~~~~~~~~~~~~~~~~~\leq\frac{1}{2}\|(w_1,w_2)-(w_1',w_2')\|_{S(\dot{H}^\frac{1}{2}:I_j)\times S(\dot{H}^\frac{1}{2}:I_j)}.
\end{align*}
Therefore, the unique solution $(w_1,w_2)$ exists on $E$.\\
Substituting $t=t_{j+1}$ into the integral equation,
\begin{equation}
\notag
\begin{cases}
\hspace{-0.4cm}&\displaystyle{w_1(t_{j+1})=e^{i(t_{j+1}-t_j)\Delta}w_1(t_j)+i\int_{t_j}^{t_{j+1}}e^{i(t_{j+1}-s)\Delta}(2\overline{w_1}w_2+2\overline{w_1}\widetilde{v}+2w_2\overline{\widetilde{u}}+e_1)(s)ds,}\\[0.3cm]
\hspace{-0.4cm}&\displaystyle{w_2(t_{j+1})=e^{\frac{1}{2}i(t_{j+1}-t_j)\Delta}w_2(t_j)+i\int_{t_j}^{t_{j+1}}e^{\frac{1}{2}i(t_{j+1}-s)\Delta}(w_1^2+2w_1\widetilde{u}+e_2)(s)ds,}
\end{cases}
\end{equation}
and so
\begin{equation}
\notag
\begin{cases}
\hspace{-0.4cm}&\displaystyle{e^{i(t-t_{j+1})\Delta}w_1(t_{j+1})=e^{i(t-t_j)\Delta}w_1(t_j)+i\int_{t_j}^{t_{j+1}}e^{i(t-s)\Delta}(2\overline{w_1}w_2+2\overline{w_1}\widetilde{v}+2w_2\overline{\widetilde{u}}+e_1)(s)ds,}\\[0.3cm]
\hspace{-0.4cm}&\displaystyle{e^{\frac{1}{2}i(t-t_{j+1})\Delta}w_2(t_{j+1})=e^{\frac{1}{2}i(t-t_j)\Delta}w_2(t_j)+i\int_{t_j}^{t_{j+1}}e^{\frac{1}{2}i(t-s)\Delta}(w_1^2+2w_1\widetilde{u}+e_2)(s)ds.}
\end{cases}
\end{equation}
By the same argument as the proof of uniqueness for $(w_1,w_2)$,
\begin{align*}
&\left\|\left(e^{i(t-t_{j+1})\Delta}w_1(t_{j+1}),e^{\frac{1}{2}i(t-t_{j+1})\Delta}w_2(t_{j+1})\right)\right\|_{S(\dot{H}^\frac{1}{2})\times S(\dot{H}^\frac{1}{2})}\\
&\ \ \ \ \ \ \ \ \ \ \ \ \ \ \ \ \ \ \ \ \ \ \ \ \ \ \ \ \ \ \ \ \ \ \ \ \ \ \leq4\left\|\left(e^{i(t-t_j)\Delta}w_1(t_j),e^{\frac{1}{2}i(t-t_j)\Delta}w_2(t_j)\right)\right\|_{S(\dot{H}^\frac{1}{2})\times S(\dot{H}^\frac{1}{2})}+4c\varepsilon_0.
\end{align*}
Iterating by beginning with $j=0$, we obtain
\begin{align*}
&\left\|\left(e^{i(t-t_j)\Delta}w_1(t_j),e^{\frac{1}{2}i(t-t_j)\Delta}w_2(t_j)\right)\right\|_{S(\dot{H}^\frac{1}{2})\times S(\dot{H}^\frac{1}{2})}+c\varepsilon_0\\
&\hspace{1.5cm}\leq4^j\left\|\left(e^{i(t-t_0)\Delta}w_1(t_0),e^{\frac{1}{2}i(t-t_0)\Delta}w_2(t_0)\right)\right\|_{S(\dot{H}^\frac{1}{2})\times S(\dot{H}^\frac{1}{2})}+(4^j+4^{j-1}+\cdots+4+1)c\varepsilon_0\\
&\hspace{1.5cm}\leq4^j\varepsilon_0+\frac{4^{j+1}-1}{4-1}c\varepsilon_0\\
&\hspace{1.5cm}=4^j\varepsilon_0+\frac{1}{3}(4^{j+1}-1)c\varepsilon_0.
\end{align*}
We remark that $N$ is fixed. We take $\varepsilon_0>0$ with $\displaystyle4^N\varepsilon_0+\frac{1}{3}(4^{N+1}-1)c\varepsilon_0\leq\frac{1}{48c}$. Then,
\begin{align*}
\|(w_1,w_2)\|_{S(\dot{H}^\frac{1}{2})\times S(\dot{H}^\frac{1}{2})}&\leq\sum_{j=0}^{N-1}\|(w_1,w_2)\|_{S(\dot{H}^\frac{1}{2}:I_j)\times S(\dot{H}^\frac{1}{2}:I_j)}\\
&\leq\sum_{j=0}^{N-1}4\left(\|(e^{i(t-t_j)\Delta}w_1(t_j),e^{\frac{1}{2}i(t-t_j)\Delta}w_2(t_j))\|_{S(\dot{H}^\frac{1}{2}:I_j)\times S(\dot{H}^\frac{1}{2}:I_j)}+c\varepsilon_0\right)\\
&\leq\sum_{j=0}^{N-1}4\left(4^j\varepsilon_0+\frac{1}{3}(4^{j+1}-1)c\varepsilon_0\right)\\
&\leq\sum_{j=0}^{N-1}4^{j+2}c\varepsilon_0\\
&=\frac{16(4^N-1)}{3}c\varepsilon_0.
\end{align*}
Therefore,
\begin{align*}
\|(u,v)\|_{S(\dot{H}^\frac{1}{2})\times S(\dot{H}^\frac{1}{2})}&\leq\|(\widetilde{u},\widetilde{v})\|_{S(\dot{H}^\frac{1}{2})\times S(\dot{H}^\frac{1}{2})}+\|(w_1,w_2)\|_{S(\dot{H}^\frac{1}{2})\times S(\dot{H}^\frac{1}{2})}\\
&\leq A+\frac{16(4^N-1)}{3}c\varepsilon_0=\vcentcolon C(A)<\infty.
\end{align*}
\end{proof}

\subsection{Localized virial identity}


\begin{lemma}[Radial Sobolev inequality]\label{Radial Sobolev inequality}
There exists $c>0$ such that for any $R>0$ and $u\in H^1$,
\[
\|u\|_{L^3(|x|>R)}\leq\frac{c}{R^\frac{2}{3}}\|u\|_{L^2(|x|>R)}^\frac{5}{6}\|\nabla u\|_{L^2(|x|>R)}^\frac{1}{6}.
\]
\end{lemma}


\begin{proof}
By Strauss' theorem\,:\,$\|u\|_{L^\infty(|x|>R)}\leq\frac{c}{R^2}\|u\|_{L^2(|x|>R)}^\frac{1}{2}\|\nabla u\|_{L^2(|x|>R)}^\frac{1}{2}$, we obtain
\[
\|u\|_{L^3(|x|>R)}^3\leq\|u\|_{L^\infty(|x|>R)}\|u\|_{L^2(|x|>R)}^2\leq\frac{c}{R^2}\|u\|_{L^2(|x|>R)}^\frac{5}{2}\|\nabla u\|_{L^2(|x|>R)}^\frac{1}{2}.
\]
Therefore,
\[
\|u\|_{L^3(|x|>R)}\leq\frac{c}{R^\frac{2}{3}}\|u\|_{L^2(|x|>R)}^\frac{5}{6}\|\nabla u\|_{L^2(|x|>R)}^\frac{1}{6}.
\]
\end{proof}


\begin{lemma}[Localized virial identity]\label{Localized virial identity}
Let $\chi\in C^4(\mathbb{R}^5)$ and $(u,v)$ be the solution to (NLS).\\
Let $\displaystyle I(t)=\int_{\mathbb{R}^5}\chi(x)\left(|u(t,x)|^2+2|v(t,x)|^2\right)dx$.\ Then, we have
\begin{align}
I'(t)=2\text{Im}\int_{\mathbb{R}^5}\left(\nabla\chi\cdot\nabla u\overline{u}+\nabla\chi\cdot\nabla v\overline{v}\right)dx,\label{09}
\end{align}
\begin{align}
I''(t)=\text{Re}\sum_{j=1}^5\sum_{k=1}^5\int_{\mathbb{R}^5}\chi_{jk}\left(4u_j\overline{u}_k+2v_j\overline{v}_k\right)dx-\int_{\mathbb{R}^5}\Delta^2\chi\left(|u|^2+\frac{1}{2}|v|^2\right)dx-2\text{Re}\int_{\mathbb{R}^5}\Delta\chi v\overline{u}^2dx.\label{10}
\end{align}
If $\chi$ is radial, then we can write
\begin{align}
I'(t)=2\text{Im}\int_{\mathbb{R}^5}\chi'\left(\frac{x\cdot\nabla u}{r}\overline{u}+\frac{x\cdot\nabla v}{r}\overline{v}\right)dx,\label{54}
\end{align}
\begin{align}
I''(t)&=\int_{\mathbb{R}^5}\left(\frac{\chi''}{r^2}-\frac{\chi'}{r^3}\right)\left(4|x\cdot\nabla u|^2+2|x\cdot\nabla v|^2\right)dx+\int_{\mathbb{R}^5}\frac{\chi'}{r}\left(4|\nabla u|^2+2|\nabla v|^2\right)dx \notag \\
&~~~~-\int_{\mathbb{R}^5}\left(\chi^{(4)}+\frac{8}{r}\chi^{(3)}+\frac{8}{r^2}\chi''-\frac{8}{r^3}\chi'\right)\left(|u|^2+\frac{1}{2}|v|^2\right)dx-2\text{Re}\int_{\mathbb{R}^5}\left(\chi''+\frac{4}{r}\chi'\right)v\overline{u}^2dx\label{55}
\end{align}
for $r=|x|$.
\end{lemma}


\begin{proof}
Since $(u,v)$ is the solution to (NLS), $(u,v)$ satisfies
\begin{equation}
\notag
\begin{cases}
\hspace{-0.4cm}&\displaystyle{\partial_tu=i\Delta u+2iv\overline{u},}\\
\hspace{-0.4cm}&\displaystyle{\partial_tv=\frac{1}{2}i\Delta v+iu^2.}
\end{cases}
\end{equation}
Thus, we obtain
\begin{align*}
I'(t)&=\int_{\mathbb{R}^5}\chi\left(u_t\overline{u}+u\overline{u}_t+2v_t\overline{v}+2v\overline{v}_t\right)dx\\
&=2\text{Re}\int_{\mathbb{R}^5}\chi\left(u_t\overline{u}+2v_t\overline{v}\right)dx\\
&=2\text{Re}\int_{\mathbb{R}^5}\chi\left(i\Delta u\overline{u}+2iv\overline{u}^2+i\Delta v\overline{v}+2iu^2\overline{v}\right)dx\\
&=2\text{Re}\int_{\mathbb{R}^5}\chi\left(i\Delta u\overline{u}+i\Delta v\overline{v}\right)dx\\
&=-2\text{Im}\int_{\mathbb{R}^5}\chi\left(\Delta u\overline{u}+\Delta v\overline{v}\right)dx\\
&=2\text{Im}\int_{\mathbb{R}^5}\left(\nabla\chi\cdot\nabla u\overline{u}+\chi\nabla u\cdot\nabla\overline{u}+\nabla\chi\cdot\nabla v\overline{v}+\chi\nabla v\cdot\nabla\overline{v}\right)dx\\
&=2\text{Im}\int_{\mathbb{R}^5}\left(\nabla\chi\cdot\nabla u\overline{u}+\nabla\chi\cdot\nabla v\overline{v}\right)dx,
\end{align*}
\begin{align}
I''(t)&=2\text{Im}\int_{\mathbb{R}^5}\left(\nabla\chi\cdot\nabla u_t\overline{u}+\nabla\chi\cdot\nabla u\overline{u}_t+\nabla\chi\cdot\nabla v_t\overline{v}+\nabla\chi\cdot\nabla v\overline{v}_t\right)dx \notag \\
&=2\text{Im}\int_{\mathbb{R}^5}\left\{\nabla\chi\cdot\nabla\left(i\Delta u+2iv\overline{u}\right)\overline{u}+\nabla\chi\cdot\nabla u\left(-i\Delta\overline{u}-2i\overline{v}u\right)\frac{ }{ }\right. \notag \\
&\left.\ \ \ \ \ \ \ \ \ \ \ \ \ \ \ \ \ \ \ \ \ \ \ \ \ \ \ \ \ \ \ \ \ \ +\nabla\chi\cdot\nabla\left(\frac{1}{2}i\Delta v+iu^2\right)\overline{v}+\nabla\chi\cdot\nabla v\left(-\frac{1}{2}i\Delta\overline{v}-i\overline{u}^2\right)\right\}dx \notag \\
&=2\text{Im}\int_{\mathbb{R}^5}\left\{i\nabla\chi\cdot\nabla\left(\Delta u\right)\overline{u}+2i\nabla\chi\cdot\nabla(v\overline{u})\overline{u}+\nabla\chi\cdot\nabla u\left(-i\Delta\overline{u}-2i\overline{v}u\right)\frac{ }{ }\right. \notag \\
&\left.\ \ \ \ \ \ \ \ \ \ \ \ \ \ \ \ \ \ \ \ \ \ \ +\frac{1}{2}i\nabla\chi\cdot\nabla\left(\Delta v\right)\overline{v}+i\nabla\chi\cdot\nabla(u^2)\overline{v}+\nabla\chi\cdot\nabla v\left(-\frac{1}{2}i\Delta\overline{v}-i\overline{u}^2\right)\right\}dx \notag \\
&=2\text{Re}\int_{\mathbb{R}^5}\left\{\nabla\chi\cdot\nabla\left(\Delta u\right)\overline{u}+2\nabla\chi\cdot\nabla(v\overline{u})\overline{u}+\nabla\chi\cdot\nabla u\left(-\Delta\overline{u}-2\overline{v}u\right)\frac{ }{ }\right. \notag \\
&\left.\ \ \ \ \ \ \ \ \ \ \ \ \ \ \ \ \ \ \ \ \ \ \ \ \ \ \ +\frac{1}{2}\nabla\chi\cdot\nabla\left(\Delta v\right)\overline{v}+\nabla\chi\cdot\nabla(u^2)\overline{v}+\nabla\chi\cdot\nabla v\left(-\frac{1}{2}\Delta\overline{v}-\overline{u}^2\right)\right\}dx \notag \\
&=2\text{Re}\int_{\mathbb{R}^5}\left\{\nabla\chi\cdot\nabla\left(\Delta u\right)\overline{u}+2\nabla\chi\cdot\nabla v\overline{u}^2+2\nabla\chi\cdot\nabla\overline{u}v\overline{u}-\nabla\chi\cdot\nabla u\Delta\overline{u}-2\nabla\chi\cdot\nabla u\overline{v}u\frac{ }{ }\right. \notag \\
&\left.\ \ \ \ \ \ \ \ \ \ \ \ \ \ \ \ \ \ \ \ \ \ \ \ +\frac{1}{2}\nabla\chi\cdot\nabla\left(\Delta v\right)\overline{v}+2\nabla\chi\cdot\nabla uu\overline{v}-\frac{1}{2}\nabla\chi\cdot\nabla v\Delta\overline{v}-\nabla\chi\cdot\nabla v\overline{u}^2\right\}dx \notag \\
&=\text{Re}\int_{\mathbb{R}^5}\left\{2\nabla\chi\cdot\nabla\left(\Delta u\right)\overline{u}+\nabla\chi\cdot\nabla\left(\Delta v\right)\overline{v}+4\nabla\chi\cdot\nabla u\overline{v}u\right. \notag \\
&\left.\ \ \ \ \ \ \ \ \ \ \ \ \ \ \ \ \ \ \ \ \ \ \ \ \ \ \ \ \ \ \ \ \ \ \ \ \ \ \ \ \ \ \ \ \ \ \ \ +2\nabla\chi\cdot\nabla v\overline{u}^2-2\nabla\chi\cdot\nabla u\Delta\overline{u}-\nabla\chi\cdot\nabla v\Delta\overline{v}\right\}dx.\label{11}
\end{align}
On the other hand, we calculate the right side of \eqref{10}.
\begin{align*}
&\sum_{j=1}^5\sum_{k=1}^5\int_{\mathbb{R}^5}\chi_{jk}u_j\overline{u}_kdx\\
&\hspace{1cm}=\sum_{j=1}^5\sum_{k=1}^5\int_{\mathbb{R}^4}\left\{\left[\chi_ju_j\overline{u}_k\right]_{x_k=-\infty}^{x_k=\infty}-\int_{\mathbb{R}}\chi_j\left(u_{kj}\overline{u}_k+u_j\overline{u}_{kk}\right)dx_k\right\}d\overline{x_k}\\
&\hspace{1cm}=-\sum_{j=1}^5\sum_{k=1}^5\int_{\mathbb{R}^5}\chi_ju_{kj}\overline{u}_kdx-\int_{\mathbb{R}^5}\nabla\chi\cdot\nabla u\Delta\overline{u}dx \notag \\
&\hspace{1cm}=-\sum_{j=1}^5\sum_{k=1}^5\int_{\mathbb{R}^4}\left\{\left[\chi_ju_k\overline{u}_k\right]_{x_j=-\infty}^{x_j=\infty}-\int_{\mathbb{R}}\left(\chi_{jj}\overline{u}_k+\chi_j\overline{u}_{jk}\right)u_kdx_j\right\}d\overline{x_j}-\int_{\mathbb{R}^5}\nabla\chi\cdot\nabla u\Delta\overline{u}dx\\
&\hspace{1cm}=\sum_{j=1}^5\sum_{k=1}^5\int_{\mathbb{R}^5}\chi_j\overline{u}_{jk}u_kdx+\int_{\mathbb{R}^5}\Delta\chi|\nabla u|^2dx-\int_{\mathbb{R}^5}\nabla\chi\cdot\nabla u\Delta\overline{u}dx\\
&\hspace{1cm}=\sum_{j=1}^5\sum_{k=1}^5\int_{\mathbb{R}^4}\left\{\left[\chi_j\overline{u}_ju_k\right]_{x_k=-\infty}^{x_k=\infty}-\int_{\mathbb{R}}\left(\chi_{kj}u_k+\chi_ju_{kk}\right)\overline{u}_jdx_k\right\}d\overline{x_k}\\
&\hspace{9cm}+\int_{\mathbb{R}^5}\Delta\chi|\nabla u|^2dx-\int_{\mathbb{R}^5}\nabla\chi\cdot\nabla u\Delta\overline{u}dx\\
&\hspace{1cm}=-\sum_{j=1}^5\sum_{k=1}^5\int_{\mathbb{R}^5}\chi_{kj}u_k\overline{u}_j
dx-\int_{\mathbb{R}^5}\nabla\chi\cdot\nabla\overline{u}\Delta udx+\int_{\mathbb{R}^5}\Delta\chi|\nabla u|^2dx-\int_{\mathbb{R}^5}\nabla\chi\cdot\nabla u\Delta\overline{u}dx.
\end{align*}
This gives the following for the first term of \eqref{10}:
\begin{align}
2\text{Re}\sum_{j=1}^5\sum_{k=1}^5\int_{\mathbb{R}^5}\chi_{jk}u_j\overline{u}_kdx&=\sum_{j=1}^5\sum_{k=1}^5\int_{\mathbb{R}^5}\left(\chi_{jk}u_j\overline{u}_k+\chi_{jk}\overline{u}_ju_k\right)dx \notag \\
&=\int_{\mathbb{R}^5}\Delta\chi|\nabla u|^2dx-2\text{Re}\int_{\mathbb{R}^5}\nabla\chi\cdot\nabla u\Delta\overline{u}dx \notag \\
&=-\int_{\mathbb{R}^5}\nabla\chi\cdot\nabla\left(|\nabla u|^2\right)dx-2\text{Re}\int_{\mathbb{R}^5}\nabla\chi\cdot\nabla u\Delta\overline{u}dx.\label{12}
\end{align}
We obtain
\begin{align}
\int_{\mathbb{R}^5}\Delta^2\chi|u|^2dx&=-\int_{\mathbb{R}^5}\nabla\left(\Delta\chi\right)\cdot\left(\nabla u\overline{u}+u\nabla\overline{u}\right)dx \notag \\
&=-2\text{Re}\int_{\mathbb{R}^5}\nabla\left(\Delta\chi\right)\cdot\nabla u\overline{u}dx \notag \\
&=2\text{Re}\int_{\mathbb{R}^5}\Delta\chi\left(\Delta u\overline{u}+|\nabla u|^2\right)dx \notag \\
&=-2\text{Re}\int_{\mathbb{R}^5}\nabla\chi\cdot\left(\nabla\left(\Delta u\right)\overline{u}+\Delta u\nabla\overline{u}+\nabla\left(|\nabla u|^2\right)\right)dx,\label{13}
\end{align}
for the second term of \eqref{10} and
\begin{align}
\int_{\mathbb{R}^5}\Delta\chi v\overline{u}^2dx=-\int_{\mathbb{R}^5}\nabla\chi\cdot\left(\nabla v\overline{u}^2+2v\nabla\overline{u}\,\overline{u}\right)dx.\label{14}
\end{align}
for the third term of \eqref{10}. Combining \eqref{11},\,\eqref{12},\,\eqref{13} and \eqref{14},
\begin{align*}
&\text{Re}\sum_{j=1}^5\sum_{k=1}^5\int_{\mathbb{R}^5}\chi_{jk}\left(4u_j\overline{u}_k+2v_j\overline{v}_k\right)dx-\int_{\mathbb{R}^5}\Delta^2\chi\left(|u|^2+\frac{1}{2}|v|^2\right)dx-2\text{Re}\int_{\mathbb{R}^5}\Delta\chi v\overline{u}^2dx\\
&=-2\int_{\mathbb{R}^5}\nabla\chi\cdot\nabla\left(|\nabla u|^2\right)dx-4\text{Re}\int_{\mathbb{R}^5}\nabla\chi\cdot\nabla u\Delta\overline{u}dx-\int_{\mathbb{R}^5}\nabla\chi\cdot\nabla\left(|\nabla v|^2\right)dx\\
&\ \ \ \ \ \ \ \ -2\text{Re}\int_{\mathbb{R}^5}\nabla\chi\cdot\nabla v\Delta\overline{v}dx+2\text{Re}\int_{\mathbb{R}^5}\nabla\chi\cdot\left(\nabla\left(\Delta u\right)\overline{u}+\Delta u\nabla\overline{u}+\nabla\left(|\nabla u|^2\right)\right)dx\\
&\ \ \ \ \ \ \ \ +\text{Re}\int_{\mathbb{R}^5}\nabla\chi\cdot\left(\nabla\left(\Delta v\right)\overline{v}+\Delta v\nabla\overline{v}+\nabla\left(|\nabla v|^2\right)\right)dx+2\text{Re}\int_{\mathbb{R}^5}\nabla\chi\cdot\left(\nabla v\overline{u}^2+2v\nabla\overline{u}\overline{u}\right)dx\\
&=\text{Re}\int_{\mathbb{R}^5}\left\{2\nabla\chi\cdot\nabla\left(\Delta u\right)\overline{u}+\nabla\chi\cdot\nabla\left(\Delta v\right)\overline{v}+4\nabla\chi\cdot\nabla u\overline{v}u\right.\\
&\left.\ \ \ \ \ \ \ \ +2\nabla\chi\cdot\nabla v\overline{u}^2-2\nabla\chi\cdot\nabla u\Delta\overline{u}-\nabla\chi\cdot\nabla v\Delta\overline{v}\right\}dx\\
&=I''(t).
\end{align*}
When $\chi$ is radial,
\begin{align}
\int_{\mathbb{R}^5}\nabla\chi\cdot\nabla u\overline{u}dx&=\int_{\mathbb{R}^5}\sum_{j=1}^5\chi_ju_j\overline{u}dx=\int_{\mathbb{R}^5}\sum_{j=1}^5\chi'\frac{x_j}{r}u_j\overline{u}dx=\int_{\mathbb{R}^5}\chi'\frac{x\cdot\nabla u}{r}\overline{u}dx,\label{15}
\end{align}
\begin{align}
\sum_{j=1}^5\sum_{k=1}^5\int_{\mathbb{R}^5}\chi_{jk}u_j\overline{u}_kdx&=\sum_{j=1}^5\sum_{k=1}^5\int_{\mathbb{R}^5}\partial_j\left(\chi'\frac{x_k}{r}\right)u_j\overline{u}_kdx \notag \\
&=\sum_{j=1}^5\sum_{k=1}^5\int_{\mathbb{R}^5}\left(\chi''\frac{x_jx_k}{r^2}+\chi'\frac{\delta_{jk}}{r}-\chi'\frac{x_jx_k}{r^3}\right)u_j\overline{u}_kdx \notag \\
&=\int_{\mathbb{R}^5}\left(\frac{\chi''}{r^2}-\frac{\chi'}{r^3}\right)|x\cdot\nabla u|^2dx+\int_{\mathbb{R}^5}\frac{\chi'}{r}|\nabla u|^2dx,\label{16}
\end{align}
where $\delta_{jk}$ denotes the Kronecker delta.
\begin{align}
\int_{\mathbb{R}^5}\Delta\chi v\overline{u}^2dx&=\int_{\mathbb{R}^5}\sum_{j=1}^5\partial_j\left(\chi'\frac{x_j}{r}\right)v\overline{u}^2dx \notag \\
&=\int_{\mathbb{R}^5}\sum_{j=1}^5\left(\chi''\frac{x_j^2}{r^2}+\chi'\frac{1}{r}-\chi'\frac{x_j^2}{r^3}\right)v\overline{u}^2dx \notag \\
&=\int_{\mathbb{R}^5}\left(\chi''+\frac{4}{r}\chi'\right)v\overline{u}^2dx,\label{17}
\end{align}
\begin{align}
\Delta^2\chi&=\Delta\left(\frac{4}{r}\chi'+\chi''\right)=\sum_{j=1}^5\partial_j\left(-\frac{4x_j}{r^3}\chi'+\frac{4x_j}{r^2}\chi''+\frac{x_j}{r}\chi^{(3)}\right) \notag \\
&=\sum_{j=1}^5\left(-\frac{4}{r^3}\chi'+\frac{12x_j^2}{r^5}\chi'-\frac{4x_j^2}{r^4}\chi''+\frac{4}{r^2}\chi''-\frac{8x_j^2}{r^4}\chi''+\frac{4x_j^2}{r^3}\chi^{(3)}+\frac{1}{r}\chi^{(3)}-\frac{x_j^2}{r^3}\chi^{(3)}+\frac{x_j^2}{r^2}\chi^{(4)}\right) \notag \\
&=\chi^{(4)}+\frac{8}{r}\chi^{(3)}+\frac{8}{r^2}\chi''-\frac{8}{r^3}\chi'.\label{18}
\end{align}
Combining \eqref{09} and \eqref{15},
\[
I'(t)=2\text{Im}\int_{\mathbb{R}^5}\chi'\left(\frac{x\cdot\nabla u}{r}\overline{u}+\frac{x\cdot\nabla v}{r}\overline{v}\right)dx,
\]
Combining \eqref{10},\,\eqref{16},\,\eqref{17} and \eqref{18}, we have
\begin{align*}
I''(t)&=\int_{\mathbb{R}^5}\left(\frac{\chi''}{r^2}-\frac{\chi'}{r^3}\right)\left(4|x\cdot\nabla u|^2+2|x\cdot\nabla v|^2\right)dx+\int_{\mathbb{R}^5}\frac{\chi'}{r}\left(4|\nabla u|^2+2|\nabla v|^2\right)dx\\
&~~~~~~~~~~~~~~~-\int_{\mathbb{R}^5}\left(\chi^{(4)}+\frac{8}{r}\chi^{(3)}+\frac{8}{r^2}\chi''-\frac{8}{r^3}\chi'\right)\left(|u|^2+\frac{1}{2}|v|^2\right)dx-2\text{Re}\int_{\mathbb{R}^5}\left(\chi''+\frac{4}{r}\chi'\right)v\overline{u}^2dx.
\end{align*}
\end{proof}

\section{Global versus blowing-up dichotomy}

\subsection{Global versus blowing-up dichotomy}

　\\
We recall that $(\phi_\omega,\psi_\omega)$ attains the infimum of $\mu_\omega^{20,8}$ and solves (gNLS).

\begin{lemma}[Estimates for $K_\omega^{20,8}(u,v)$]\label{estimates for K}
Let $(u_0,v_0)\in H^1\!\times\!H^1$ and $(u,v)$ be the corresponding solution to (NLS). Let $[0,T^\ast)$ be the maximal forward lifespan of $(u,v)$. If $I_\omega(u_0,v_0)<I_\omega(\phi_\omega,\psi_\omega)$, then the following holds.\\
If $K_\omega^{20,8}(u_0,v_0)>0$, then $K_\omega^{20,8}(u(t),v(t))\geq\min\left\{I_\omega(\phi_\omega,\psi_\omega)-I_\omega(u,v),\,K(u,v)\right\}>0$ for any $t\in[0,T^\ast)$,\\
If $K_\omega^{20,8}(u_0,v_0)<0$, then $K_\omega^{20,8}(u(t),v(t))\leq16\left(I_\omega(u,v)-I_\omega(\phi_\omega,\psi_\omega)\right)<0$ for any $t\in[0,T^\ast)$.
\end{lemma}


\begin{proof}
We define that
\begin{align*}
&J_{\alpha,\beta}(\lambda)=I_\omega\left(e^{\alpha\lambda}u(e^{\beta\lambda}\cdot),e^{\alpha\lambda}v(e^{\beta\lambda}\cdot)\right)\\
&\ \ \ \ \ \ \ \ \ \ =\frac{\omega}{2}\|e^{\alpha\lambda}u(e^{\beta\lambda}\cdot)\|_{L^2}^2+\omega\|{e^{\alpha\lambda}v(e^{\beta\lambda}\cdot)}\|_{L^2}^2+\frac{1}{2}\|\nabla e^{\alpha\lambda}u(e^{\beta\lambda}\cdot)\|_{L^2}^2\\
&\hspace{6cm}+\frac{1}{4}\|\nabla e^{\alpha\lambda}v(e^{\beta\lambda}\cdot)\|_{L^2}^2-\text{Re}\left(e^{\alpha\lambda}v(e^{\beta\lambda}\cdot),(e^{\alpha\lambda}u(e^{\beta\lambda}\cdot))^2\right)_{L^2}\\
&\ \ \ \ \ \ \ \ \ \ =\frac{\omega}{2}e^{(2\alpha-5\beta)\lambda}\|u\|_{L^2}^2+\omega e^{(2\alpha-5\beta)\lambda}\|v\|_{L^2}^2+\frac{1}{2}e^{(2\alpha-3\beta)\lambda}\|\nabla u\|_{L^2}^2\\
&\ \ \ \ \ \ \ \ \ \ \ \ \ \ \ \ \ \ \ \ \ \ \ \ \ \ \ \ \ \ \ \ \ \ \ \ \ \ \ \ \ \ \ \ \ \ \ \ \ \ \ \ \ \ \ \ \ \ \ \ +\frac{1}{4}e^{(2\alpha-3\beta)\lambda}\|\nabla v\|_{L^2}^2-e^{(3\alpha-5\beta)\lambda}\text{Re}(v,u^2)_{L^2}.
\end{align*}
If $\alpha=20,\beta=8$, then
\[
J_{20,8}(\lambda)=\frac{\omega}{2}\|u\|_{L^2}^2+\omega\|v\|_{L^2}^2+\frac{1}{2}e^{16\lambda}\|\nabla u\|_{L^2}^2+\frac{1}{4}e^{16\lambda}\|\nabla v\|_{L^2}^2-e^{20\lambda}\text{Re}(v,u^2)_{L^2}.
\]
Thus,
\[
J_{20,8}(0)=\frac{\omega}{2}M(u,v)+\frac{1}{2}E(u,v)=I_\omega(u,v).
\]
Also,
\begin{align*}
J'_{20,8}(\lambda)&=8e^{16\lambda}\|\nabla u\|_{L^2}^2+4e^{16\lambda}\|\nabla v\|_{L^2}^2-20e^{20\lambda}\text{Re}(v,u^2)_{L^2}\\
&=4e^{16\lambda}(2\|\nabla u\|_{L^2}^2+\|\nabla v\|_{L^2}^2-5e^{4\lambda}\text{Re}(v,u^2)_{L^2}).
\end{align*}
Thus, $J'_{20,8}(0)=K_\omega^{20,8}(u,v)$. Moreover,
\begin{align*}
J''_{20,8}(\lambda)&=8\cdot16e^{16\lambda}\|\nabla u\|_{L^2}^2+4\cdot16e^{16\lambda}\|\nabla v\|_{L^2}^2-20\cdot20e^{20\lambda}\text{Re}(v,u^2)_{L^2}\\
&=16(8e^{16\lambda}\|\nabla u\|_{L^2}^2+4e^{16\lambda}\|\nabla v\|_{L^2}^2-20e^{20\lambda}\text{Re}(v,u^2)_{L^2})-80e^{20\lambda}\text{Re}(v,u^2)_{L^2}\\
&=16J'_{20,8}(\lambda)-80e^{20\lambda}\text{Re}(v,u^2)_{L^2}.
\end{align*}
In the case $K_\omega^{20,8}(u_0,v_0)>0$, we first prove that $K_\omega^{20,8}(u(t),v(t))>0$ for any $t\in[0,T^\ast)$. If not, then there exists $ t\in[0,T^\ast)$ such that $K_\omega^{20,8}(u(t),v(t))=0.$ By the definition of $\mu_\omega^{20,8}(=I_\omega(\phi_\omega,\psi_\omega))$, we have $I_\omega(\phi_\omega,\psi_\omega)\leq I_\omega(u(t),v(t))$ for such $t\in[0,T^\ast)$. Thus, $I_\omega(u(t),v(t))<I_\omega(\phi_\omega,\psi_\omega)\leq I_\omega(u(t),v(t))$, which is contradiction. Therefore, $K_\omega^{20,8}(u,v)>0$ for any $t\in[0,T^\ast)$.\\
We will prove that $K_\omega^{20,8}(u(t),v(t))\geq\min\{I_\omega(\phi_\omega,\psi_\omega)-I_\omega(u,v),K(u,v)\}$ from here.\\
If $P(u,v)\leq0$, then $K_\omega^{20,8}(u,v)\geq8K(u,v)$.\\
Let $P(u,v)>0$. We solve $J'_{20,8}(\lambda)=0$. Since
\[
4e^{16\lambda}(2K(u,v)-5e^{4\lambda}P(u,v))=0,
\]
it follows that
\[
e^{4\lambda}=\frac{2K(u,v)}{5P(u,v)}>1,
\]
\[
\lambda=\frac{1}{4}\log\frac{2K(u,v)}{5P(u,v)}>0.
\]
Let $\lambda_0$ be this $\lambda$, i.e.
\[
\lambda_0=\frac{1}{4}\log\frac{2K(u,v)}{5P(u,v)}>0.
\]
Next, we solve $J''_{20,8}(\lambda)+J'_{20,8}(\lambda)=0$.
\[
\left(8\cdot16e^{16\lambda}K(u,v)-20\cdot20e^{20\lambda}P(u,v)\right)+\left(8e^{16\lambda}K(u,v)-20e^{20\lambda}P(u,v)\right)=0,
\]
\[
e^{16\lambda}\left(8\cdot17K(u,v)-20\cdot21e^{4\lambda}P(u,v)\right)=0,
\]
\[
e^{4\lambda}=\frac{8\cdot17K(u,v)}{20\cdot21P(u,v)},
\]
\[
\lambda=\frac{1}{4}\log\frac{8\cdot17K(u,v)}{20\cdot21P(u,v)}.
\]
Let $\lambda_1$ be this $\lambda$, i.e.
\[
\lambda_1=\frac{1}{4}\log\frac{8\cdot17K(u,v)}{20\cdot21P(u,v)}.
\]
In the case $\lambda_1<0$, because $J''_{20,8}(\lambda)+J'_{20,8}(\lambda)<0$\ \ in\ \ $[0,\lambda_0]$,
\begin{align}
0&>\int_0^{\lambda_0}(J''_{20,8}(\lambda)+J'_{20,8}(\lambda))d\lambda \notag \\
&=J'_{20,8}(\lambda_0)-J'_{20,8}(0)+J_{20,8}(\lambda_0)-J_{20,8}(0) \notag \\
&=-K_\omega^{20,8}(u,v)+J_{20,8}(\lambda_0)-I_\omega(u,v).\label{76}
\end{align}
Since
\[
J'_{20,8}(\lambda_0)=8K\left(e^{20\lambda_0}u(e^{8\lambda_0}\cdot),e^{20\lambda_0}v(e^{8\lambda_0}\cdot)\right)-20P\left(e^{20\lambda_0}u(e^{8\lambda_0}\cdot),e^{20\lambda_0}v(e^{8\lambda_0}\cdot)\right)=0,
\]
it follows that
\[
K_\omega^{20,8}\left(e^{20\lambda_0}u(e^{8\lambda_0}\cdot),e^{20\lambda_0}v(e^{8\lambda_0}\cdot)\right)=0.
\]
Hence, $J_{20,8}(\lambda_0)=I_\omega\left(e^{20\lambda_0}u(e^{8\lambda_0}\cdot),e^{20\lambda_0}v(e^{8\lambda_0}\cdot)\right)\geq I_\omega(\phi_\omega,\psi_\omega)$.\\
This inequality combined with \eqref{76} gives
\[
0>-K_\omega^{20,8}(u,v)+I_\omega(\phi_\omega,\psi_\omega)-I_\omega(u,v),
\]
i.e.
\[
K_\omega^{20,8}(u,v)>I_\omega(\phi_\omega,\psi_\omega)-I_\omega(u,v).
\]
In the case $\lambda_1\geq0$, because $\frac{8\cdot17}{21}K(u,v)\geq20P(u,v)$,
\[
K_\omega^{20,8}(u,v)=8K(u,v)-20P(u,v)\geq8K(u,v)-\frac{8\cdot17}{21}K(u,v)=\frac{32}{21}K(u,v)\geq K(u,v).
\]
Therefore,
\[
K_\omega^{20,8}(u,v)\geq\min\left\{I_\omega(\phi_\omega,\psi_\omega)-I_\omega(u,v),\,K(u,v)\right\}.
\]
In the case $K_\omega^{20,8}(u_0,v_0)<0$, it follows that $K_\omega^{20,8}(u,v)<0$ for any $t\in[0,T^\ast)$ by the same argument as $K_\omega^{20,8}(u_0,v_0)>0$. Thus, $P(u,v)>0$. 
For $\lambda_0$ taked above, we have $J'_{20,8}(\lambda)<0$ for $\lambda\in (\lambda_0,0)$\ \ and\ \ $J'_{20,8}(\lambda_0)=0$. Also, it follows that
\[
J''_{20,8}(\lambda)=16J'_{20,8}(\lambda)-80e^{20\lambda}\text{Re}(v,u^2)_{L^2}\leq16J'_{20,8}(\lambda).
\]
Integrating the most left side and the most right side in $[\lambda_0,0]$,
\[
\int_{\lambda_0}^0J''_{20,8}(\lambda)d\lambda\leq16\int_{\lambda_0}^0J'_{20,8}(\lambda)d\lambda,
\]
\[
J'_{20,8}(0)-J'_{20,8}(\lambda_0)\leq16(J_{20,8}(0)-J_{20,8}(\lambda_0)),
\]
\[
K_\omega^{20,8}(u,v)\leq16(J_{20,8}(0)-J_{20,8}(\lambda_0))\leq16(I_\omega(u,v)-I_\omega(\phi_\omega,\psi_\omega))<0.
\]
\end{proof}


\begin{theorem}[Global versus blowing-up dichotomy]\label{Global versus blow-up dichotomy}
Let $(u_0,v_0)\in H^1\!\times\!H^1$ and $(u,v)$ be the solution to (NLS) with initial data $(u_0,v_0)$. We assume that
\[
I_\omega(u_0,v_0)<I_\omega(\phi_\omega,\psi_\omega).
\]
\begin{itemize}
\item[(1)]In the case $K_\omega^{20,8}(u_0,v_0)\geq0$, $(u(t),v(t))$ is time-global.
\item[(2)]In the case $K_\omega^{20,8}(u_0,v_0)<0$, if $(xu_0,xv_0)\in L^2\times L^2$ or $(u_0,v_0)$ is radial, then $(u(t),v(t))$ blows up.
\end{itemize}
\end{theorem}


\begin{proof}
(1)\ \ $K_\omega^{20,8}(u,v)=8\|\nabla u(t)\|_{L^2}^2+4\|\nabla v(t)\|_{L^2}^2-20\text{Re}(v,u^2)_{L^2}>0.$\\
By the energy conservation \eqref{02}, $K_\omega^{20,8}(u,v)=8E(u_0,v_0)-4\text{Re}(v,u^2)_{L^2}>0$, i.e.
\begin{align}
-\text{Re}(v,u^2)_{L^2}>-2E(u_0,v_0).\label{36}
\end{align}
Thus,
\begin{align*}
I_\omega(\phi_\omega,\psi_\omega)&>I_\omega(u,v)\\
&=\frac{\omega}{2}\left(\|u(t)\|_{L^2}^2+2\|v(t)\|_{L^2}^2\right)+\frac{1}{2}\left(\|\nabla u(t)\|_{L^2}^2+\frac{1}{2}\|\nabla v(t)\|_{L^2}^2-2\text{Re}(v,u^2)_{L^2}\right)\\
&>\frac{\omega}{2}\left(\|u(t)\|_{L^2}^2+2\|v(t)\|_{L^2}^2\right)+\frac{1}{2}\left(\|\nabla u(t)\|_{L^2}^2+\frac{1}{2}\|\nabla v(t)\|_{L^2}^2\right)-2E(u_0,v_0).
\end{align*}
Therefore,
\begin{align}
\displaystyle I_\omega(\phi_\omega,\psi_\omega)+2E(u_0,v_0)>\frac{\omega}{2}\left(\|u(t)\|_{L^2}^2+2\|v(t)\|_{L^2}^2\right)+\frac{1}{2}\left(\|\nabla u(t)\|_{L^2}^2+\frac{1}{2}\|\nabla v(t)\|_{L^2}^2\right)\label{56}
\end{align}
and hence, the solution exists time-globally.\\
(2)\ \ Let $(xu_0,xv_0)\in L^2\!\times\!L^2$. Combining Proposition \ref{K and virial} and Lemma \ref{estimates for K},
\[
\frac{d^2}{dt^2}(\|xu(t)\|_{L^2}^2+2\|xv(t)\|_{L^2}^2)=K_\omega^{20,8}(u,v)\leq16(I_\omega(u,v)-I_\omega(\phi_\omega,\psi_\omega))<0.
\]
Therefore, the solution blows up.\\
Let $(u_0,v_0)$ be radial.\\
We take $\chi\in C_0^\infty(\mathbb{R}^5)$, which is radial and satisfies
\begin{equation}
\notag \chi(r)=
\begin{cases}
\hspace{-0.4cm}&\displaystyle{\ \ \ \,r^2\ \ \ \ \ \,(0\leq r\leq1),}\\
\hspace{-0.4cm}&\displaystyle{smooth\ \ (1\leq r\leq 3),}\\
\hspace{-0.4cm}&\displaystyle{\ \ \ \ 0\ \ \ \ \ \ \,(3\leq r),}
\end{cases}
\end{equation}
where $r=|x|$. Also, $\chi$ satisfies $\chi''(r)\leq2$ $(r\geq0)$. Moreover, we define $\chi_R(r)=R^2\chi(\frac{r}{R})$. Then,\\
$(\chi_R)'(r)=R\chi'(\frac{r}{R})\,,\ (\chi_R)''(r)=\chi''(\frac{r}{R})\,,\ (\chi_R)^{(3)}(r)=\frac{1}{R}\chi^{(3)}(\frac{r}{R})\,,$ and $(\chi_R)^{(4)}(r)=\frac{1}{R^2}\chi^{(4)}(\frac{r}{R}).$\\
By Lemma \ref{Localized virial identity}, for $I(t)=\int_{\mathbb{R}^5}\chi_R\left(|u|^2+2|v|^2\right)dx$,
\begin{align*}
I''(t)&=K_\omega^{20,8}(u,v)+\int_{\mathbb{R}^5}\left\{\frac{1}{r^2}\chi''\left(\frac{r}{R}\right)-\frac{R}{r^3}\chi'\left(\frac{r}{R}\right)\right\}\left(4|x\cdot\nabla u|^2+2|x\cdot\nabla v|^2\right)dx\\
&~~~~~~~~~~~~~~~~~~+\int_{\mathbb{R}^5}\left\{\frac{R}{r}\chi'\left(\frac{r}{R}\right)-2\right\}\left(4|\nabla u|^2+2|\nabla v|^2\right)dx\\
&~~~~~~~~~~~~~~~~~~-\int_{\mathbb{R}^5}\left\{\frac{1}{R^2}\chi^{(4)}\left(\frac{r}{R}\right)+\frac{8}{Rr}\chi^{(3)}\left(\frac{r}{R}\right)+\frac{8}{r^2}\chi''\left(\frac{r}{R}\right)-\frac{8R}{r^3}\chi'\left(\frac{r}{R}\right)\right\}\left(|u|^2+\frac{1}{2}|v|^2\right)dx\\
&~~~~~~~~~~~~~~~~~~-2\text{Re}\int_{\mathbb{R}^5}\left\{\chi''\left(\frac{r}{R}\right)+\frac{4R}{r}\chi'\left(\frac{r}{R}\right)-10\right\}v\overline{u}^2dx.
\end{align*}
Let
\begin{align*}
&R_1=\int_{\mathbb{R}^5}\left\{\frac{1}{r^2}\chi''\left(\frac{r}{R}\right)-\frac{R}{r^3}\chi'\left(\frac{r}{R}\right)\right\}\left(4|x\cdot\nabla u|^2+2|x\cdot\nabla v|^2\right)dx\\
&\ \ \ \ \ \ \ \ \ \ \ \ \ \ \ \ \ \ \ \ \ \ \ \ \ \ \ \ \ \ \ \ \ \ \ \ \ \ \ \ \ \ \ \ \ \ \ \ \ \ \ +\int_{\mathbb{R}^5}\left\{\frac{R}{r}\chi'\left(\frac{r}{R}\right)-2\right\}\left(4|\nabla u|^2+2|\nabla v|^2\right)dx,
\end{align*}
\[
R_2=-\int_{\mathbb{R}^5}\left\{\frac{1}{R^2}\chi^{(4)}\left(\frac{r}{R}\right)+\frac{8}{Rr}\chi^{(3)}\left(\frac{r}{R}\right)+\frac{8}{r^2}\chi''\left(\frac{r}{R}\right)-\frac{8R}{r^3}\chi'\left(\frac{r}{R}\right)\right\}\left(|u|^2+\frac{1}{2}|v|^2\right)dx,
\]
and
\[
R_3=-2\text{Re}\int_{\mathbb{R}^5}\left\{\chi''\left(\frac{r}{R}\right)+\frac{4R}{r}\chi'\left(\frac{r}{R}\right)-10\right\}v\overline{u}^2dx.
\]
By Lemma \ref{estimates for K}, $K_\omega^{20,8}(u,v)\leq16(I_\omega(u,v)-I_\omega(\phi_\omega,\psi_\omega))<0$.\\
For $R_1$, in the case $\frac{1}{r^2}\chi''\left(\frac{r}{R}\right)-\frac{R}{r^3}\chi'\left(\frac{r}{R}\right)\leq0$, because $\chi''(\frac{r}{R})\leq2$, we have $\chi'(\frac{r}{R})\leq\frac{2r}{R}$. Thus,
\begin{align}
R_1&\leq\int_{\mathbb{R}^5}\left\{\frac{R}{r}\chi'\left(\frac{r}{R}\right)-2\right\}\left(4|\nabla u|^2+2|\nabla v|^2\right)dx\leq0.\label{19}
\end{align}
In the case $\frac{1}{r^2}\chi''\left(\frac{r}{R}\right)-\frac{R}{r^3}\chi'\left(\frac{r}{R}\right)>0$, we have
\begin{align}
R_1&\leq\int_{\mathbb{R}^5}\left\{\frac{1}{r^2}\chi''\left(\frac{r}{R}\right)-\frac{R}{r^3}\chi'\left(\frac{r}{R}\right)\right\}\left(4r^2|\nabla u|^2+2r^2|\nabla v|^2\right)dx \notag \\
&\ \ \ \ \ \ \ \ \ \ \ \ \ \ \ \ \ \ \ \ \ \ \ \ \ \ \ \ \ \ \ \ \ \ \ \ \ \ \ \ \ +\int_{\mathbb{R}^5}\left\{\frac{R}{r}\chi'\left(\frac{r}{R}\right)-2\right\}\left(4|\nabla u|^2+2|\nabla v|^2\right)dx \notag \\
&=\int_{\mathbb{R}^5}\left\{\chi''\left(\frac{r}{R}\right)-2\right\}\left(4|\nabla u|^2+2|\nabla v|^2\right)dx\leq0.\label{20}
\end{align}
For $R_2$,
\begin{align}
R_2&=-\int_{\mathbb{R}^5}\left\{\frac{1}{R^2}\chi^{(4)}\left(\frac{r}{R}\right)+\frac{8}{Rr}\chi^{(3)}\left(\frac{r}{R}\right)+\frac{8}{r^2}\chi''\left(\frac{r}{R}\right)-\frac{8R}{r^3}\chi'\left(\frac{r}{R}\right)\right\}\left(|u|^2+\frac{1}{2}|v|^2\right)dx \notag \\
&=-\int_{R\leq|x|\leq3R}\left\{\frac{1}{R^2}\chi^{(4)}\left(\frac{r}{R}\right)+\frac{8}{Rr}\chi^{(3)}\left(\frac{r}{R}\right)+\frac{8}{r^2}\chi''\left(\frac{r}{R}\right)-\frac{8R}{r^3}\chi'\left(\frac{r}{R}\right)\right\}\left(|u^2|+\frac{1}{2}|v|^2\right)dx \notag \\
&\leq\frac{c}{R^2}\int_{R\leq|x|\leq3R}\left(|u|^2+\frac{1}{2}|v|^2\right)dx\label{21}\\
&\leq\frac{c}{R^2}\int_{\mathbb{R}^5}\left(|u|^2+2|v|^2\right)dx \notag \\
&=\frac{c}{R^2}M(u,v). \notag
\end{align}
For $R_3$, by Lemma \ref{Radial Sobolev inequality},
\begin{align*}
R_3&=-2\text{Re}\int_{\mathbb{R}^5}\left\{\chi''\left(\frac{r}{R}\right)+\frac{4R}{r}\chi'\left(\frac{r}{R}\right)-10\right\}v\overline{u}^2dx\\
&=-2\text{Re}\int_{R\leq|x|\leq3R}\left\{\chi''\left(\frac{r}{R}\right)+\frac{4R}{r}\chi'\left(\frac{r}{R}\right)-10\right\}v\overline{u}^2dx\\
&\leq c\int_{R\leq|x|}|vu^2|dx\\
&\leq c\|v\|_{L^3(R\leq|x|)}\|u\|_{L^3(R\leq|x|)}^2\\
&\leq\frac{c}{R^2}\|v\|_{L^2(R\leq|x|)}^\frac{5}{6}\|\nabla v\|_{L^2(R\leq|x|)}^\frac{1}{6}\|u\|_{L^2(R\leq|x|)}^\frac{5}{3}\|\nabla u\|_{L^2(R\leq|x|)}^\frac{1}{3}\\
&\leq\frac{c}{R^2}M(u,v)^\frac{5}{4}\|\nabla v\|_{L^2(R\leq|x|)}^\frac{1}{6}\|\nabla u\|_{L^2(R\leq|x|)}^\frac{1}{3}\\
&=c\frac{1}{R^2\varepsilon^2}M(u,v)^\frac{5}{4}\cdot\varepsilon\|\nabla v\|_{L^2(R\leq|x|)}^\frac{1}{6}\cdot\varepsilon\|\nabla u\|_{L^2(R\leq|x|)}^\frac{1}{3}\\
&\leq c\left(\frac{3}{4}\cdot\frac{1}{R^\frac{8}{3}\varepsilon^\frac{8}{3}}M(u,v)^\frac{5}{3}+\frac{1}{12}\varepsilon^{12}\|\nabla v\|_{L^2(R\leq|x|)}^2+\frac{1}{6}\varepsilon^6\|\nabla u\|_{L^2(R\leq|x|)}^2\right)\\
&\leq\frac{c}{R^\frac{8}{3}\varepsilon^\frac{8}{3}}M(u,v)^\frac{5}{3}+\frac{1}{2}\varepsilon\|\nabla v\|_{L^2}^2+\varepsilon\|\nabla u\|_{L^2}^2.
\end{align*}
Since $I_\omega(u,v)<I_\omega(\phi_\omega,\psi_\omega)$, we may take $\delta>0$ with $I_\omega(u,v)<(1-\delta)I_\omega(\phi_\omega,\psi_\omega)$. Then, using Lemma \ref{lemma of inferior},
\begin{align*}
I''(t)&\leq K_\omega^{20,8}(u,v)+\varepsilon K(u,v)+\frac{c}{R^2}M(u,v)+\frac{c}{R^\frac{8}{3}\varepsilon^\frac{8}{3}}M(u,v)^\frac{5}{3}\\
&=20I_\omega(u,v)-10\omega M(u,v)-(2-\varepsilon)K(u,v)+\frac{c}{R^2}M(u,v)+\frac{c}{R^\frac{8}{3}\varepsilon^\frac{8}{3}}M(u,v)^\frac{5}{3}\\
&<20(1-\delta)I_\omega(\phi_\omega,\psi_\omega)-(2-\varepsilon)\left\{K(u,v)+5\omega M(u,v)\right\}\\
&\ \ \ \ \ \ \ \ \ \ \ \ \ \ \ \ \ \ \ \ \ \ \ \ \ \ \ \ \ \ \ \ \ \ \ \ \ \ \ -5\varepsilon\omega M(u,v)+\frac{c}{R^2}M(u,v)+\frac{c}{R^\frac{8}{3}\varepsilon^\frac{8}{3}}M(u,v)^\frac{5}{3}\\
&\leq\left\{20(1-\delta)-10(2-\varepsilon)\right\}I_\omega(\phi_\omega,\psi_\omega)+\frac{c}{R^2}M(u,v)+\frac{c}{R^\frac{8}{3}\varepsilon^\frac{8}{3}}M(u,v)^\frac{5}{3}\\
&=10(\varepsilon-2\delta)\mu_\omega^{20,8}+\frac{c}{R^2}M(u,v)+\frac{c}{R^\frac{8}{3}\varepsilon^\frac{8}{3}}M(u,v)^\frac{5}{3}.
\end{align*}
Thus, if we take sufficiently small $\varepsilon>0$, and sufficiently large $R>0$, then $I''(t)<0$. Therefore, the solution blows up.
\end{proof}

\subsection{Blowing-up or growing-up}


\begin{lemma}\label{lemma for blow-up or grow-up}
Let $(u_0,v_0)\in H^1\!\times\!H^1$ and $(u,v)$ be the time-global solution to (NLS). We fix $\eta_0>0$.\\
If we define $\displaystyle C_0=\max\left\{\sup_{t\in\mathbb{R}}\|\nabla u(t)\|_{L^2},\,\sup_{t\in\mathbb{R}}\|\nabla v(t)\|_{L^2}\right\}$, then for any $0\leq t\leq\frac{\eta_0R}{16C_0M(u,v)^\frac{1}{2}}$, we have
\[
\int_{|x|\geq R}\left(|u(x,t)|^2+2|v(x,t)|^2\right)dx\leq\eta_0+o_R(1).
\]
\end{lemma}


\begin{proof}
We take $\chi\in C^\infty(\mathbb{R}^5)$, which is radial and satisfies
\begin{equation}
\notag \chi(r)=
\begin{cases}
\hspace{-0.4cm}&\displaystyle{\ \ \ \ 0\ \ \ \ \ \,\,(0\leq r\leq R/2),}\\
\hspace{-0.4cm}&\displaystyle{smooth\ \ (R/2\leq r\leq R),}\\
\hspace{-0.4cm}&\displaystyle{\ \ \ \ \,1\ \ \ \ \ \ (R\leq r),}
\end{cases}
\end{equation}
where $r=|x|$ and $R>0$. $\chi$ satisfies $\chi'(r)\leq\frac{4}{R}$ $(r\geq0)$.\\
We define $I(t)=\int_{\mathbb{R}^5}\chi(r)\left(|u(t,x)|^2+2|v(t,x)|^2\right)dx$. Then, by Lemma \ref{Localized virial identity},
\begin{align}
I(t)&=I(0)+\int_0^tI'(s)ds \notag \\
&\leq I(0)+\int_0^t|I'(s)|ds \notag \\
&=I(0)+2\int_0^t\left|\text{Im}\int_{\mathbb{R}^5}\chi'\left(\frac{x\cdot\nabla u}{r}\overline{u}+\frac{x\cdot\nabla v}{r}\overline{v}\right)dx\right|ds \notag \\
&\leq I(0)+2\int_0^t\int_{\mathbb{R}^5}|\chi'|\left(|\nabla u||u|+|\nabla v||v|\right)dxds \notag \\
&\leq I(0)+2\int_0^t\|\chi'\|_{L^\infty}\left(\|\nabla u(s)\|_{L^2}\|u(s)\|_{L^2}+\|\nabla v(s)\|_{L^2}\|v(s)\|_{L^2}\right)ds \notag \\
&\leq I(0)+\frac{16tC_0M(u,v)^\frac{1}{2}}{R}.\label{22}
\end{align}
Moreover,
\begin{align}
I(0)=\int_{|x|\geq\frac{R}{2}}\chi(r)\left(|u_0(x)|^2+2|v_0(x)|^2\right)dx\leq\int_{|x|\geq\frac{R}{2}}\left(|u_0(x)|^2+2|v_0(x)|^2\right)dx=o_R(1),\label{23}
\end{align}
\begin{align}
\int_{|x|\geq R}\left(|u(t,x)|^2+2|v(t,x)|^2\right)dx\leq I(t),\label{24}
\end{align}
\begin{align}
0\leq t\leq\frac{\eta_0R}{16C_0M(u,v)^\frac{1}{2}}.\label{25}
\end{align}
Combining \eqref{22},\,\eqref{23},\,\eqref{24} and \eqref{25}, it follows that
\[
\int_{|x|\geq R}\left(|u(t,x)|^2+2|v(t,x)|^2\right)dx\leq\eta_0+o_R(1).
\]
\end{proof}

We prove the part of growing-up or blowing-up for Theorem \ref{Main theorem 1} (2) in the next theorem.


\begin{theorem}[Blowing-up or growing-up]
Let $(u_0,v_0)\in H^1\!\times\! H^1$ and $I=[0,T^\ast)$ be the maximal forward lifespan of the solution $(u,v)$ to (NLS). We assume for $\omega>0$
\[
I_\omega(\phi_\omega,\psi_\omega)>I_\omega(u_0,v_0)\ \ \ \text{and}\ \ \ K_\omega^{20,8}(u_0,v_0)<0.
\]
Then, the solution grows up or blows up.
\end{theorem}


\begin{proof}
We prove that there  exists no time-global solutions $(u,v)$ such that
\[
(u,v)\in C(\mathbb{R}^+\!\!:H^1)\times C(\mathbb{R}^+\!\!:H^1)\ with\ \sup_{t\in\mathbb{R}^+}\|u(t)\|_{L^q}+\sup_{t\in\mathbb{R}^+}\|v(t)\|_{L^q}<\infty\ \ \text{for some}\ q>3
\]
by contradiction. Let $\displaystyle C_0=\sup_{t\in\mathbb{R}^+}\|u(t)\|_{L^q}+\sup_{t\in\mathbb{R}^+}\|v(t)\|_{L^q}<\infty$. Then,
\begin{align*}
\|\nabla u(t)\|_{L^2}^2+\frac{1}{2}\|\nabla v(t)\|_{L^2}^2&=E(u,v)+2\text{Re}(v,u^2)_{L^2}\\
&\leq E(u,v)+2\|v(t)\|_{L^3}\|u(t)\|_{L^3}^2\\
&\leq E(u,v)+2\|v(t)\|_{L^q}^{1-\theta}\|v(t)\|_{L^2}^\theta\|v(t)\|_{L^q}^{2(1-\theta)}\|v(t)\|_{L^2}^{2\theta}\\
&\leq E(u,v)+2C_0^{3(1-\theta)}M(u,v)^{\frac{3}{2}\theta}<\infty
\end{align*}
for $\theta\in(0,1)$ satisfying $\frac{1}{3}=\frac{1-\theta}{q}+\frac{\theta}{2}$. Thus, we obtain
\[
\displaystyle\overline{C_0}\vcentcolon=\max\left\{\sup_{t\in\mathbb{R}}\|\nabla u(t)\|_{L^2},\,\sup_{t\in\mathbb{R}}\|\nabla v(t)\|_{L^2}\right\}<\infty.
\]
We take $\chi\in C_0^\infty(\mathbb{R}^5)$, which is radial and satisfies
\begin{equation}
\notag \chi(r)=
\begin{cases}
\hspace{-0.4cm}&\displaystyle{\ \ \ \,r^2\ \ \ \,\ \ (0\leq r\leq1),}\\
\hspace{-0.4cm}&\displaystyle{smooth\ \ (1\leq r\leq 3),}\\
\hspace{-0.4cm}&\displaystyle{\ \ \ \ 0\ \ \ \ \ \ (3\leq r),}
\end{cases}
\end{equation}
where $r=|x|$. Also, $\chi$ satisfies $\chi''(r)\leq2$ $(r\geq0)$. We define $\chi_R(r)=R^2\chi(\frac{r}{R})$.\\
By Lemma \ref{Localized virial identity}, for $I(t)=\int_{\mathbb{R}^5}\chi_R\left(|u|^2+2|v|^2\right)dx$
\begin{align*}
I''(t)&=K_\omega^{20,8}(u,v)+\int_{\mathbb{R}^5}\left\{\frac{1}{r^2}\chi''\left(\frac{r}{R}\right)-\frac{R}{r^3}\chi'\left(\frac{r}{R}\right)\right\}\left(4|x\cdot\nabla u|^2+2|x\cdot\nabla v|^2\right)dx\\
&~~~~~~~~~~~~~~~~~~+\int_{\mathbb{R}^5}\left\{\frac{R}{r}\chi'\left(\frac{r}{R}\right)-2\right\}\left(4|\nabla u|^2+2|\nabla v|^2\right)dx\\
&~~~~~~~~~~~~~~~~~~-\int_{\mathbb{R}^5}\left\{\frac{1}{R^2}\chi^{(4)}\left(\frac{r}{R}\right)+\frac{8}{Rr}\chi^{(3)}\left(\frac{r}{R}\right)+\frac{8}{r^2}\chi''\left(\frac{r}{R}\right)-\frac{8R}{r^3}\chi'\left(\frac{r}{R}\right)\right\}\left(|u|^2+\frac{1}{2}|v|^2\right)dx\\
&~~~~~~~~~~~~~~~~~~-2\text{Re}\int_{\mathbb{R}^5}\left\{\chi''\left(\frac{r}{R}\right)+\frac{4R}{r}\chi'\left(\frac{r}{R}\right)-10\right\}v\overline{u}^2dx.
\end{align*}
Let
\begin{align*}
R_1=\int_{\mathbb{R}^5}\left\{\frac{1}{r^2}\chi''\left(\frac{r}{R}\right)-\frac{R}{r^3}\chi'\left(\frac{r}{R}\right)\right\}&\left(4|x\cdot\nabla u|^2+2|x\cdot\nabla v|^2\right)dx\\
&+\int_{\mathbb{R}^5}\left\{\frac{R}{r}\chi'\left(\frac{r}{R}\right)-2\right\}\left(4|\nabla u|^2+2|\nabla v|^2\right)dx,
\end{align*}
\[
R_2=-\int_{\mathbb{R}^5}\left\{\frac{1}{R^2}\chi^{(4)}\left(\frac{r}{R}\right)+\frac{8}{Rr}\chi^{(3)}\left(\frac{r}{R}\right)+\frac{8}{r^2}\chi''\left(\frac{r}{R}\right)-\frac{8R}{r^3}\chi'\left(\frac{r}{R}\right)\right\}\left(|u|^2+\frac{1}{2}|v|^2\right)dx,
\]
and
\[
R_3=-2\text{Re}\int_{\mathbb{R}^5}\left\{\chi''\left(\frac{r}{R}\right)+\frac{4R}{r}\chi'\left(\frac{r}{R}\right)-10\right\}v\overline{u}^2dx.
\]
By Lemma \ref{estimates for K}, $K_\omega^{20,8}(u,v)<16\left(I_\omega(u,v)-I_\omega(\phi_\omega,\psi_\omega)\right)=\vcentcolon16\widetilde{C_0}<0$.\\
Combining \eqref{19} and \eqref{20}, we have $R_1\leq0.$\\
By \eqref{21},
\begin{align*}
R_2&\leq\frac{c}{R^2}\int_{R\leq|x|}\left(|u|^2+\frac{1}{2}|v|^2\right)dx\\
&=\frac{c}{R^2}\left(\|u\|_{L^2(R\leq|x|)}^{2\theta}\|u\|_{L^2(R\leq|x|)}^{2(1-\theta)}+\frac{1}{2}\|v\|_{L^2(R\leq|x|)}^{2\theta}\|v\|_{L^2(R\leq|x|)}^{2(1-\theta)}\right)\\
&\leq\frac{c}{R^2}\left(\|u\|_{L^2(R\leq|x|)}^{2(1-\theta)}+\|v\|_{L^2(R\leq|x|)}^{2(1-\theta)}\right)\left(\|u\|_{L^2(R\leq|x|)}^{2\theta}+\frac{1}{2}\|v\|_{L^2(R\leq|x|)}^{2\theta}\right)\\
&\leq\frac{c}{R^2}\left(\|u\|_{L^2(R\leq|x|)}^{2\theta}+\frac{1}{2}\|v\|_{L^2(R\leq|x|)}^{2\theta}\right),
\end{align*}
\begin{align*}
R_3&=-2\text{Re}\int_{R\leq|x|\leq3R}\left\{\chi''\left(\frac{r}{R}\right)+\frac{4R}{r}\chi'\left(\frac{r}{R}\right)-10\right\}v\overline{u}^2dx\\
&\leq c\|v\|_{L^3(R\leq|x|)}\|u\|_{L^3(R\leq|x|)}^2\\
&\leq c\|v\|_{L^q(R\leq|x|)}^{1-\theta}\|v\|_{L^2(R\leq|x|)}^{\theta}\|u\|_{L^q(R\leq|x|)}^{2(1-\theta)}\|u\|_{L^2(R\leq|x|)}^{2\theta}\\
&\leq c\,C_0^{3(1-\theta)}\|v\|_{L^2(R\leq|x|)}^{\theta}\|u\|_{L^2(R\leq|x|)}^{2\theta}\\
&\leq c\,C_0^{3(1-\theta)}\|u\|_{L^2(R\leq|x|)}^\theta\left(\frac{1}{2}\|u\|_{L^2(R\leq|x|)}^{2\theta}+\frac{1}{2}\|v\|_{L^2(R\leq|x|)}^{2\theta}\right)\\
&\leq c\,C_0^{3(1-\theta)}\|u\|_{L^2(R\leq|x|)}^\theta\left(\|u\|_{L^2(R\leq|x|)}^{2\theta}+\frac{1}{2}\|v\|_{L^2(R\leq|x|)}^{2\theta}\right).
\end{align*}
Thus, if $R>1$, then
\begin{align}
I''(t)&\leq16\widetilde{C_0}+\frac{c}{R^2}\left(\|u\|_{L^2(R\leq|x|)}^{2\theta}+\frac{1}{2}\|v\|_{L^2(R\leq|x|)}^{2\theta}\right) \notag \\
&\ \ \ \ \ \ \ \ \ \ \ \ \ \ \ \ \ \ \ \ \ \ \ \ \ \ \ \ \ \ \ \ \ \ \ \ \ \ \ +c\,C_0^{3(1-\theta)}\|u\|_{L^2(R\leq|x|)}^\theta\left(\|u\|_{L^2(R\leq|x|)}^{2\theta}+\frac{1}{2}\|v\|_{L^2(R\leq|x|)}^{2\theta}\right) \notag \\
&\leq16\widetilde{C_0}+\widetilde{C}\left(\|u\|_{L^2(R\leq|x|)}^{2\theta}+\frac{1}{2}\|v\|_{L^2(R\leq|x|)}^{2\theta}\right),\label{75}
\end{align}
where $\widetilde{C}=\widetilde{C}\left(M(u,v),q,C_0\right)$.\\
By Lemma \ref{lemma for blow-up or grow-up},
\[
\int_{R\leq|x|}\left(|u(x,t)|^2+2|v(x,t)|^2\right)dx\leq\eta_0+o_R(1)
\]
for any $0\leq t\leq\frac{R\eta_0}{16\overline{C_0}M(u,v)^\frac{1}{2}}=\vcentcolon T$. Hence, $\|u(t)\|_{L^2(R\leq|x|)}^2\leq\eta_0+o_R(1)$ and $\|v(t)\|_{L^2(R\leq|x|)}^2\leq\eta_0+o_R(1).$\\
This inequality combined with \eqref{75} gives
\[
I''(t)\leq16\widetilde{C_0}+\widetilde{C}\left(\eta_0^\theta+o_R(1)\right).
\]
Integrating this inequality in $[0,t]$,
\[
I'(t)\leq I'(0)+\left\{16\widetilde{C_0}+\widetilde{C}\left(\eta_0^\theta+o_R(1)\right)\right\}t.
\]
Also, integrating this inequality in $[0,T]$,
\[
I(T)\leq I(0)+I'(0)T+\left\{16\widetilde{C_0}+\widetilde{C}\left(\eta_0^\theta+o_R(1)\right)\right\}\cdot\frac{1}{2}T^2.
\]
Here, we take $\eta_0>0$ with $\widetilde{C}\eta_0^\theta=-8\widetilde{C_0}>0$. Then,
\[
I(T)\leq I(0)+I'(0)\cdot\frac{R\eta_0}{16\overline{C_0}M(u,v)^\frac{1}{2}}+\left(8\widetilde{C_0}+o_R(1)\right)\cdot\frac{1}{2}\left(\frac{\eta_0}{16\overline{C_0}M(u,v)^\frac{1}{2}}\right)^2R^2.
\]
We take $R>0$ with $8\widetilde{C_0}+o_R(1)<4\widetilde{C_0}$. Then,
\begin{align*}
I(T)&<I(0)+I'(0)\cdot\frac{R\eta_0}{16\overline{C_0}M(u,v)^\frac{1}{2}}+2\widetilde{C_0}\left(\frac{\eta_0}{16\overline{C_0}M(u,v)^\frac{1}{2}}\right)^2R^2\\
&=\vcentcolon I(0)+I'(0)\cdot\frac{R\eta_0}{16\overline{C_0}M(u,v)^\frac{1}{2}}+2\alpha_0R^2.
\end{align*}
In addition,
\begin{align*}
I(0)&=\int_{\mathbb{R}^5}\chi_R(r)\left(|u_0(x)|^2+2|v_0(x)|^2\right)dx\\
&=\int_{|x|\leq\sqrt{R}}r^2\left(|u_0(x)|^2+2|v_0(x)|^2\right)dx+\int_{\sqrt{R}\leq|x|\leq3R}R^2\chi\left(\frac{r}{R}\right)\left(|u_0(x)|^2+2|v_0(x)|^2\right)dx\\
&\leq RM(u,v)+cR^2\int_{\sqrt{R}\leq|x|\leq 3R}\left(|u_0(x)|^2+2|v_0(x)|^2\right)dx\\
&=o_R(1)R^2,\\
I'(0)&=2\text{Im}\int_{\mathbb{R}^5}\chi_R'\left(\frac{x\cdot\nabla u_0}{r}\overline{u_0}+\frac{x\cdot\nabla v_0}{r}\overline{v_0}\right)dx\\
&=4\text{Im}\int_{|x|\leq\sqrt{R}}r\left(\frac{x\cdot\nabla u_0}{r}\overline{u_0}+\frac{x\cdot\nabla v_0}{r}\overline{v_0}\right)dx\\
&\ \ \ \ \ \ \ \ \ \ \ \ \ \ \ \ \ \ \ \ \ \ \ \ \ \ \ \ \ \ \ \ \ \ \ \ \ \ +2\text{Im}\int_{\sqrt{R}\leq|x|\leq3R}R\chi'\left(\frac{r}{R}\right)\left(\frac{x\cdot\nabla u_0}{r}\overline{u_0}+\frac{x\cdot\nabla v_0}{r}\overline{v_0}\right)dx\\
&\leq4\int_{|x|\leq\sqrt{R}}|x|\left(|\nabla u_0||u_0|+|\nabla v_0||v_0|\right)dx\\
&\ \ \ \ \ \ \ \ \ \ \ \ \ \ \ \ \ \ \ \ \ \ \ \ \ \ \ \ \ \ \ \ \ \ \ \ \ \ +cR\int_{\sqrt{R}\leq|x|\leq3R}\left(|\nabla u_0||u_0|+|\nabla v_0||v_0|\right)dx\\
&\leq4\sqrt{R}\left(\|\nabla u_0\|_{L^2}\|u_0\|_{L^2}+\|\nabla v_0\|_{L^2}\|v_0\|_{L^2}\right)\\
&\ \ \ \ \ \ \ \ \ \ \ \ \ \ \ \ \ \ \ \ \ \ \ \ \ \ \ \ \ \ \ \ \ \ \ \ \ \ +cR\left(\|\nabla u_0\|_{L^2}\|u_0\|_{L^2(\sqrt{R}\leq|x|)}+\|\nabla v_0\|_{L^2}\|v_0\|_{L^2(\sqrt{R}\leq|x|)}\right)\\
&=o_R(1)R.
\end{align*}
Thus,
\[
I(T)<o_R(1)R^2+2\alpha_0R^2.
\]
If we take $R>0$ with $o_R(1)+2\alpha_0<\alpha_0$, then
\[
I(T)<\alpha_0R^2<0.
\]
However, this is a contradiction with $I(T)=\int_{\mathbb{R}^5}\chi_R(r)\left(|u(x,T)|^2+2|v(x,T)|^2\right)dx\geq0$. Therefore, there exists no time-global solutions $(u,v)\in C(\mathbb{R}^+\!\!:H^1)\times C(\mathbb{R}^+\!\!:H^1)$ such that
\[
\sup_{t\in\mathbb{R}^+}\|u(t)\|_{L^q}+\sup_{t\in\mathbb{R}^+}\|v(t)\|_{L^q}<\infty\ \text{ for some }q>3.
\]
Applying Sobolev embedding $H^1(\mathbb{R}^5)\hookrightarrow L^\frac{10}{3}(\mathbb{R}^5)$, we completely prove.
\end{proof}

\section{Profile decomposition}
\subsection{Linear profile decomposition}


\begin{lemma}[Comparability of $K_\omega$ and $I_\omega$]\label{Comparability of K and I}
Let $(u,v)\in H^1\!\times\! H^1$ satisfy
\[
I_\omega(u,v)<I_\omega(\phi_\omega,\psi_\omega)\ \ \ \text{and}\ \ \ K_\omega^{20,8}(u,v)>0.
\]
Then, we have
\[
\frac{1}{10}K_\omega(u,v)<I_\omega(u,v).
\]
\end{lemma}

\begin{proof}
By \eqref{36}, for any $t\in\mathbb{R}$, $-\text{Re}(v,u^2)_{L^2}>-2E(u_0,v_0)=-2\|\nabla u\|_{L^2}^2-\|\nabla v\|_{L^2}^2+4\text{Re}(v,u^2)_{L^2}$, i.e. $-5\text{Re}(v,u^2)_{L^2}>-2\|\nabla u\|_{L^2}^2-\|\nabla v\|_{L^2}^2.$\\
Thus,
\begin{align*}
I_\omega(u,v)&=\frac{\omega}{2}M(u,v)+\frac{1}{2}E(u,v)\\
&=\frac{\omega}{2}M(u,v)+\frac{1}{2}\left(\|\nabla u\|_{L^2}^2+\frac{1}{2}\|\nabla v\|_{L^2}^2-2\text{Re}(v,u^2)_{L^2}\right)\\
&>\frac{\omega}{2}M(u,v)+\frac{1}{2}\left(\|\nabla u\|_{L^2}^2+\frac{1}{2}\|\nabla v\|_{L^2}^2\right)-\frac{1}{5}\left(2\|\nabla u\|_{L^2}^2+\|\nabla v\|_{L^2}^2\right)\\
&=\frac{\omega}{2}M(u,v)+\frac{1}{10}K(u,v)\\
&\geq\frac{1}{10}K_\omega(u,v).
\end{align*}
Therefore,
\[
\frac{1}{10}K_\omega(u,v)<I_\omega(u,v).
\]
\end{proof}

\begin{lemma}\label{time shift lemma}
Let $\{t_n\}\subset\mathbb{R}$ and $\{x_n\}\subset\mathbb{R}^5$ be two sequences.\\
(a)\ \ If $\displaystyle\lim_{n\rightarrow\infty}\left(|t_n|+|x_n|\right)=\infty$, then $e^{it_n\Delta}\psi(\cdot+x_n)\rightharpoonup0\ \ \text{in}\ \ H^1$ for any $\psi\in H^1$.\\
(b)\ \ If there exists $\{z_n\}\subset H^1$ and $\psi\in H^1\setminus\{0\}$ such that
\[
z_n\rightharpoonup0\ \ \ \text{and}\ \ \ e^{it_n\Delta}z_n(\cdot+x_n)\rightharpoonup\psi\ \ \text{in}\ \ H^1,
\]
then $\displaystyle\lim_{n\rightarrow\infty}\left(|t_n|+|x_n|\right)=\infty$.
\end{lemma}

\begin{proof}
Since $C_0^\infty\subset H^1$ is dense, for any $\varepsilon>0$ and $\phi$, $\psi\in H^1$ there exists $f$, $g\in C_0^\infty$ such that $\|\phi-f\|_{H^1}<\varepsilon\,,\ \|\psi-g\|_{H^1}<\varepsilon$.\\
We will prove Lemma \ref{time shift lemma} by contradiction.\\
If not $(a)$, then there exists $\varepsilon>0$, $\psi$, $\phi\in H^1$ and $\{t_{n_k}\}$, $\{x_{n_k}\}$ such that $|(e^{it_{n_k}\Delta}\psi(\cdot+x_{n_k}),\phi)_{H^1}|\geq\varepsilon$.
In the case $|t_{n_k}|\rightarrow\infty$, using Theorem \ref{Linear estimate}, it follows that
\begin{align*}
&|(e^{it_{n_k}\Delta}\psi(\cdot+x_{n_k}),\phi)_{H^1}|\\
&\ \ \ \ \ \ \ \ \ \leq|(e^{it_{n_k}\Delta}(\psi-g)(\cdot+x_{n_k}),\phi)_{H^1}|+|(e^{it_n\Delta}g(\cdot+x_{n_k}),\phi-f)_{H^1}|+|(e^{it_{n_k}\Delta}g(\cdot+x_{n_k}),f)_{H^1}|\\
&\ \ \ \ \ \ \ \ \ \leq\|\psi-g\|_{H^1}\|\phi\|_{H^1}+\|g\|_{H^1}\|\phi-f\|_{H^1}+\|e^{it_{n_k}\Delta}g\|_{W^{1,\infty}}\|f\|_{W^{1,1}}\\
&\ \ \ \ \ \ \ \ \ \leq\|\psi-g\|_{H^1}\|\phi\|_{H^1}+\|g\|_{H^1}\|\phi-f\|_{H^1}+c|t_{n_k}|^{-\frac{5}{2}}\|g\|_{W^{1,1}}\|f\|_{W^{1,1}}<\varepsilon.
\end{align*}
This is contradiction.\\
When the case $\{t_{n_k}\}$ is bounded, there exists a subsequence such that $t_{n_k}\longrightarrow t^\ast$. Also, by $|x_{n_k}|\rightarrow\infty$,
\begin{align*}
|(e^{it_{n_k}\Delta}\psi(\cdot+x_{n_k}),\phi)_{H^1}|&\leq|((e^{it_{n_k}\Delta}-e^{it^\ast\Delta})\psi,\phi(\cdot-x_{n_k}))_{H^1}|+|(e^{it^\ast\Delta}\psi,\phi(\cdot-x_{n_k}))_{H^1}|\\
&\leq\|(e^{it_{n_k}\Delta}-e^{it^\ast\Delta})\psi\|_{H^1}\|\phi\|_{H^1}+|(e^{it^\ast\Delta}(\psi-g),\phi(\cdot-x_{n_k}))_{H^1}|\\
&\ \ \ \ \ \ \ \ \ \ \ \ \ \ \ +|(e^{it^\ast\Delta}g,(\phi-f)(\cdot-x_{n_k}))_{H^1}|+|(e^{it^\ast\Delta}g,f(\cdot-x_{n_k}))_{H^1}|\\
&\leq\|(e^{it_{n_k}\Delta}-e^{it^\ast\Delta})\psi\|_{H^1}\|\phi\|_{H^1}+\|\psi-g\|_{H^1}\|\phi\|_{H^1}\\
&\ \ \ \ \ \ \ \ \ \ \ \ \ \ \ +\|g\|_{H^1}\|\phi-f\|_{H^1}+|(e^{it^\ast\Delta}g,f(\cdot-x_{n_k}))_{H^1}|
<\varepsilon.
\end{align*}
This is contradiction. Therefore, $e^{it_n\Delta}\psi(\cdot+x_n)\rightharpoonup0\ \ \text{in}\ \ H^1$ holds.\\
Next, we prove (b). If not (b), then there exists a subsequence $\{t_{n_k}\}$, $\{x_{n_k}\}$ such that $(t_{n_k},x_{n_k})\longrightarrow(t^\ast,x^\ast)$. Since $z_n\rightharpoonup0$, for any $\phi\in H^1$,
\begin{align*}
&|(e^{it_{n_k}\Delta}z_{n_k}(\cdot+x_{n_k}),\phi)_{H^1}|\\
&\ \ \ \ \ \ \ \ \ \ \ =|(z_{n_k},e^{-it_{n_k}\Delta}\phi(\cdot-x_{n_k}))_{H^1}|\\
&\ \ \ \ \ \ \ \ \ \ \ \leq|(z_{n_k},e^{-it_{n_k}\Delta}(\phi-f)(\cdot-x_{n_k}))_{H^1}|+|(z_{n_k},(e^{-it_{n_k}\Delta}-e^{-it^\ast\Delta})f(\cdot-x_{n_k}))_{H^1}|\\
&\ \ \ \ \ \ \ \ \ \ \ \ \ \ \ \ \ \ \ \ \ \ \ \ \ \ \ \ \ \ \ \ \ +|(z_{n_k},e^{-it^\ast\Delta}(f(\cdot-x_{n_k})-f(\cdot-x^\ast)))_{H^1}|+|(z_{n_k},e^{-it^\ast\Delta}f(\cdot-x^\ast))_{H^1}|\\
&\ \ \ \ \ \ \ \ \ \ \ \leq\|z_{n_k}\|_{H^1}\|\phi-f\|_{H^1}+\|z_{n_k}\|_{H^1}\|(e^{-it_{n_k}\Delta}-e^{-it^\ast\Delta})f\|_{H^1}\\
&\ \ \ \ \ \ \ \ \ \ \ \ \ \ \ \ \ \ \ \ \ \ \ \ \ \ \ \ \ \ \ \ \ +\|z_{n_k}\|_{H^1}\|f(\cdot-x_{n_k})-f(\cdot-x^\ast)\|_{H^1}+|(z_{n_k},e^{-it^\ast\Delta}f(\cdot-x^\ast))_{H^1}|\\
&\ \ \ \ \ \ \ \ \ \ \ \longrightarrow0\ \ \text{as}\ \ k\rightarrow\infty.
\end{align*}
Thus, $e^{it_{n_k}\Delta}z_{n_k}(\cdot+x_{n_k})\rightharpoonup0$ holds. However, this is contradiction. Therefore, $|t_n|+|x_n|\longrightarrow\infty$ holds.
\end{proof}


\begin{theorem}[Linear profile decomposition]\label{Linear profile decomposition}
Let $(\phi_n,\psi_n)$ be a bounded sequence in $H^1\!\times\! H^1$. Then, after passing to a subsequence of $(\phi_n,\psi_n)$ necessary, also denoted $(\phi_n,\psi_n)$, and\\
(1)\ \ for each $1\leq j\leq M$, there exists a profile $(\phi^j,\psi^j)$ (fixed in $n$) in $H^1\times H^1$,\\
(2)\ \ for each $1\leq j\leq M$, there exists a sequence (in $n$) of time shifts $\{t_n^j\}$,\\
(3)\ \ for each $1\leq j\leq M$, there exists a sequence (in $n$) of space shifts $\{x_n^j\}$,\\
(4)\ \ there exists a sequence (in $n$) of remainders $(\Phi_n^M,\Psi_n^M)$ in $H^1\!\times\! H^1$,\\
such that
\[
(\phi_n(x),\psi_n(x))=\sum_{j=1}^M(e^{-it_n^j\Delta}\phi^j(x-x_n^j),e^{-\frac{1}{2}it_n^j\Delta}\psi^j(x-x_n^j))+(\Phi_n^M(x),\Psi_n^M(x)).
\]
for any $M\in\mathbb{R}$. For fixed $j$, we have
\begin{align}
either\ \ \ t_n^j=0\ \ \ ^\forall n\in\mathbb{N}\ \ \ or\ \ \ t_n^j\longrightarrow-\infty\ \ as\ \ n\rightarrow\infty,\label{77}
\end{align}
\begin{align}
either\ \ \ x_n^j=0\ \ \ ^\forall n\in\mathbb{N}\ \ \ or\ \ \ |x_n^j|\longrightarrow\infty\ \ as\ \ n\rightarrow\infty.\label{78}
\end{align}
Pairwise divergence property:
\begin{align}
1\leq^\forall i\neq^\forall j\leq M,\ \ \lim_{n\rightarrow\infty}(|t_n^i-t_n^j|+|x_n^i-x_n^j|)=\infty.\label{40}
\end{align}
Asymptotic smallness property:
\begin{align}
\lim_{M\rightarrow\infty}\left[\lim_{n\rightarrow\infty}\|(e^{it\Delta}\Phi_n^M,e^{\frac{1}{2}it\Delta}\Psi_n^M)\|_{S(\dot{H}^\frac{1}{2})\times S(\dot{H}^\frac{1}{2})}\right]=0.\label{26}
\end{align}
For fixed $M$ and any $0\leq s\leq 1$, we have the asymptotic Pythagorean expansion:
\begin{align}
\|\phi_n\|_{\dot{H}^s}^2=\sum_{j=1}^M\|\phi^j\|_{\dot{H}^s}^2+\|\Phi_n^M\|_{\dot{H}^s}^2+o_n(1),\label{27}
\end{align}
\begin{align}
\|\psi_n\|_{\dot{H}^s}^2=\sum_{j=1}^M\|\psi^j\|_{\dot{H}^s}^2+\|\Psi_n^M\|_{\dot{H}^s}^2+o_n(1).\label{28}
\end{align}
\end{theorem}


\begin{proof}
Since $(\phi_n,\psi_n)$ is bounded in $H^1\!\times\! H^1$, there exists $c_1>0$ such that $\|(\phi_n,\psi_n)\|_{H^1\times H^1}\leq c_1$.\\
Let $(q,r)$ be a $\dot{H}^\frac{1}{2}$\,admissible pair.
\begin{align*}
\|e^{it\Delta}\Phi_n^M\|_{L^qL^r}&\leq\left\|\|e^{it\Delta}\Phi_n^M\|_{L^{r_1}}^\theta\|e^{it\Delta}\Phi_n^M\|_{L^\frac{5}{2}}^{1-\theta}\right\|_{L^q}\\
&\leq\|e^{it\Delta}\Phi_n^M\|_{L^{q\theta}L^{r_1}}^\theta\|e^{it\Delta}\Phi_n^M\|_{L^\infty L^\frac{5}{2}}^{1-\theta}.
\end{align*}
for $(r_1,\theta)$ satisfying $\frac{1}{r}=\frac{\theta}{r_1}+\frac{2(1-\theta)}{5}$. Here, we assume that $(q\theta,r_1)$ is $\dot{H}^\frac{1}{2}$ admissible and $(q,r)$ satisfies $4<q<\infty$, $\frac{5}{2}<r<\frac{10}{3}$. If $r_1=\frac{3}{2}r$, then $\frac{15}{4}<r_1<5$. Since $\displaystyle\frac{1}{r}=\frac{2\theta}{3r}+\frac{2(1-\theta)}{5}$, i.e. $\displaystyle0<\theta=\frac{6r-15}{6r-10}<1$,\\
we have $\displaystyle\frac{1}{3}<\frac{1}{q\theta}=\left(1-\frac{5}{2r}\right)\cdot\frac{6r-10}{6r-15}=1-\frac{5}{3r}<\frac{1}{2}$, i.e. $\displaystyle2<q\theta<3$. Also, we obtain $\displaystyle\frac{2}{q\theta}+\frac{5}{r_1}=\frac{6r-10}{3r}+\frac{10}{3r}=2$. Thus
\[
\|e^{it\Delta}\Phi_n^M\|_{L^qL^r}\leq\|e^{it\Delta}\Phi_n^M\|_{L^{q\theta}L^{r_1}}^\theta\|e^{it\Delta}\Phi_n^M\|_{L^\infty L^\frac{5}{2}}^{1-\theta}\leq c\|\Phi_n^M\|_{\dot{H}^\frac{1}{2}}^\theta\|e^{it\Delta}\Phi_n^M\|_{L^\infty L^\frac{5}{2}}^{1-\theta}.
\]
This inequality also holds when $q=\infty$, $r=\frac{5}{2}$. Hence, it suffices to establish
\[
\limsup_{n\rightarrow\infty}\|(\Phi_n^M,\Psi_n^M)\|_{\dot{H}^\frac{1}{2}\times \dot{H}^\frac{1}{2}}\leq c<\infty\ \ \ \text{and}\ \ \ \lim_{M\rightarrow\infty}\left[\limsup_{n\rightarrow\infty}\|(e^{it\Delta}\Phi_n^M,e^{\frac{1}{2}it\Delta}\Psi_n^M)\|_{L^\infty L^\frac{5}{2}\times L^\infty L^\frac{5}{2}}\right]=0
\]
to obtain \eqref{26}.\\
(i)\ \ Let $M=1$.\\
We set $\displaystyle A_1=\limsup_{n\rightarrow\infty}\|(e^{it\Delta}\phi_n,e^{\frac{1}{2}it\Delta}\psi_n)\|_{L^\infty L^\frac{5}{2}\times L^\infty L^\frac{5}{2}}$. If $A_1=0$, then Theorem \ref{Linear profile decomposition} holds by taking $\phi^j=0,\psi^j=0\ \ (j=1,2,\cdots,M)$. Hence, we assume $A_1>0$. We take a subsequence $\{(\phi_n,\psi_n)\}$ so that $\displaystyle A_1=\lim_{n\rightarrow\infty}\|(e^{it\Delta}\phi_n,e^{\frac{1}{2}it\Delta}\psi_n)\|_{L^\infty L^\frac{5}{2}\times L^\infty L^\frac{5}{2}}$. We take $\chi\in S(\mathbb{R}^5)$ with $\widehat{\chi}(\xi)=1$ $(\frac{1}{r}\leq|\xi|\leq r)$ and $\text{supp}\widehat{\chi}\subset[\frac{1}{2r},2r]$ for $r>0$. Applying Sobolev embedding,
\begin{align*}
\|e^{it\Delta}\phi_n-\chi\ast e^{it\Delta}\phi_n\|_{L^\infty L^\frac{5}{2}}^2&\leq c\|e^{it\Delta}\phi_n-\chi\ast e^{it\Delta}\phi_n\|_{L^\infty \dot{H}^\frac{1}{2}}^2\\
&=c\|\,|\xi|^\frac{1}{2}e^{-4\pi^2it|\xi|^2}\widehat{\phi}_n-|\xi|^\frac{1}{2}\widehat{\chi}e^{-4\pi^2it|\xi|^2}\widehat{\phi}_n\|_{L^\infty L^2}^2\\
&=c\int_{\mathbb{R}^5}|\xi||1-\widehat{\chi}(\xi)|^2|\widehat{\phi}_n|^2d\xi\\
&\leq c\int_{|\xi|\leq\frac{1}{r}}|\xi||\widehat{\phi}_n|^2
d\xi+c\int_{|\xi|\geq r}|\xi||\widehat{\phi}_n|^2d\xi\\
&\leq\frac{c}{r}\|\widehat{\phi}_n\|_{L^2}^2+\frac{c}{r}\|\,|\xi|\widehat{\phi}_n\|_{L^2}^2\\
&=\frac{c}{r}\|\phi_n\|_{H^1}^2\leq\frac{c\cdot c_1^2}{r}.
\end{align*}
Thus,
\[
\|e^{it\Delta}\phi_n-\chi\ast e^{it\Delta}\phi_n\|_{L^\infty L^\frac{5}{2}}\leq\frac{c\cdot c_1}{\sqrt{r}}.
\]
Similarly,
\[
\|e^{\frac{1}{2}it\Delta}\psi_n-\chi\ast e^{\frac{1}{2}it\Delta}\psi_n\|_{L^\infty L^\frac{5}{2}}\leq\frac{c\cdot c_1}{\sqrt{r}}.
\]
We deduce
\begin{align*}
&\|e^{it\Delta}\phi_n\|_{L^\infty L^\frac{5}{2}}-\|\chi\ast e^{it\Delta}\phi_n\|_{L^\infty L^\frac{5}{2}}\leq\frac{c\cdot c_1}{\sqrt{r}},\\
&\|e^{\frac{1}{2}it\Delta}\psi_n\|_{L^\infty L^\frac{5}{2}}-\|\chi\ast e^{\frac{1}{2}it\Delta}\psi_n\|_{L^\infty L^\frac{5}{2}}\leq\frac{c\cdot c_1}{\sqrt{r}}
\end{align*}
from these inequalities, and hence
\begin{align*}
\|\chi\ast e^{it\Delta}\phi_n\|_{L^\infty L^\frac{5}{2}}+\|\chi\ast e^{\frac{1}{2}it\Delta}\psi_n\|_{L^\infty L^\frac{5}{2}}&\geq\|e^{it\Delta}\phi_n\|_{L^\infty L^\frac{5}{2}}+\|e^{\frac{1}{2}it\Delta}\psi_n\|_{L^\infty L^\frac{5}{2}}-\frac{2c\cdot c_1}{\sqrt{r}}\\
&\longrightarrow A_1-\frac{2c\cdot c_1}{\sqrt{r}}\ \ \text{as}\ \ n\rightarrow\infty.
\end{align*}
Here, if we take $r>0$ with $\displaystyle\frac{2c\cdot c_1}{\sqrt{r}}=\frac{A_1}{4}\Longleftrightarrow r=\frac{2^6c^2\,c_1^2}{A_1^2}$, then there exists $n_0\in \mathbb{N}$ such that for any $n\geq n_0$,\\
\[
\|\chi\ast e^{it\Delta}\phi_n\|_{L^\infty L^\frac{5}{2}}+\|\chi\ast e^{\frac{1}{2}it\Delta}\psi_n\|_{L^\infty L^\frac{5}{2}}\geq\frac{A_1}{2}.
\]
So, if we take a subsequence $\{(\phi_n,\psi_n)\}_{n\geq n_0}$ of $\{(\phi_n,\psi_n)\}_{n\in\mathbb{N}}$, then for any $n\in\mathbb{N}$,
\begin{align}
\|\chi\ast e^{it\Delta}\phi_n\|_{L^\infty L^\frac{5}{2}}+\|\chi\ast e^{\frac{1}{2}it\Delta}\psi_n\|_{L^\infty L^\frac{5}{2}}\geq\frac{A_1}{2}.\label{53}
\end{align}
Also, since
\begin{align*}
\|\chi\ast e^{it\Delta}\phi_n\|_{L^\infty L^\frac{5}{2}}^\frac{5}{2}&\leq\|\chi\ast e^{it\Delta}\phi_n\|_{L^\infty L^\infty}^\frac{1}{2}\|\chi\ast e^{it\Delta}\phi_n\|_{L^\infty L^2}^2\\
&=\|\chi\ast e^{it\Delta}\phi_n\|_{L^\infty L^\infty}^\frac{1}{2}\|\widehat{\chi}e^{-4\pi^2it|\cdot|^2}\widehat{\phi}_n\|_{L^\infty L^2}^2\\
&\leq\|\chi\ast e^{it\Delta}\phi_n\|_{L^\infty L^\infty}^\frac{1}{2}\|\phi_n\|_{L^2}^2\\
&\leq c_1^2\|\chi\ast e^{it\Delta}\phi_n\|_{L^\infty L^\infty}^\frac{1}{2},
\end{align*}
it follows that
\begin{align*}
\left(\|\chi\ast e^{it\Delta}\phi_n\|_{L^\infty L^\frac{5}{2}}+\|\chi\ast e^{\frac{1}{2}it\Delta}\psi_n\|_{L^\infty L^\frac{5}{2}}\right)^5&\leq c\left(\|\chi\ast e^{it\Delta}\phi_n\|_{L^\infty L^\frac{5}{2}}^5+\|\chi\ast e^{\frac{1}{2}it\Delta}\psi_n\|_{L^\infty L^\frac{5}{2}}^5\right)\\
&\leq c\cdot c_1^4\left(\|\chi\ast e^{it\Delta}\phi_n\|_{L^\infty L^\infty}+\|\chi\ast e^{\frac{1}{2}it\Delta}\psi_n\|_{L^\infty L^\infty}\right).
\end{align*}
This inequality combined with \eqref{53} gives
\[
c\cdot c_1^4\left(\|\chi\ast e^{it\Delta}\phi_n\|_{L^\infty L^\infty}+\|\chi\ast e^{\frac{1}{2}it\Delta}\psi_n\|_{L^\infty L^\infty}\right)\geq\left(\frac{A_1}{2}\right)^5,
\]
\[
\|\chi\ast e^{it\Delta}\phi_n\|_{L^\infty L^\infty}+\|\chi\ast e^{\frac{1}{2}it\Delta}\psi_n\|_{L^\infty L^\infty}\geq\frac{A_1^5}{2^5c\,c_1^4},
\]
\[
\max\left\{\|\chi\ast e^{it\Delta}\phi_n\|_{L^\infty L^\infty},\ \|\chi\ast e^{\frac{1}{2}it\Delta}\psi_n\|_{L^\infty L^\infty}\right\}\geq\frac{A_1^5}{2^6c\,c_1^4}.
\]
Therefore, we can take time shifts $\{t_n^1\}\subset(-\infty,0]$ and space shifts $\{x_n^1\}\subset\mathbb{R}^5$ so that
\[
\max\left\{|\chi\ast e^{it_n^1\Delta}\phi_n(x_n^1)|,\ |\chi\ast e^{\frac{1}{2}it_n^1\Delta}\psi_n(x_n^1)|\right\}\geq\frac{A_1^5}{2^7c\,c_1^4}.
\]
for any $n\in\mathbb{N}$. Thus, we obtain
\begin{align}
|\chi\ast e^{it_n^1\Delta}\phi_n(x_n^1)|+|\chi\ast e^{\frac{1}{2}it_n^1\Delta}\psi_n(x_n^1)|\geq\frac{A_1^5}{2^7c\,c_1^4}.\label{79}
\end{align}
Since $\{e^{it_n^1\Delta}\phi_n(\cdot+x_n^1)\}$ and $\{e^{\frac{1}{2}it_n^1\Delta}\psi_n(\cdot+x_n^1)\}$ are bounded sequences in $H^1$, we can set that
\begin{align}
e^{it_n^1\Delta}\phi_n(\cdot+x_n^1)\rightharpoonup\widetilde{\phi^1}\,,\ \ e^{\frac{1}{2}it_n^1\Delta}\psi_n(\cdot+x_n^1)\rightharpoonup\widetilde{\psi^1}\ \ \text{in}\ \ H^1\label{80}
\end{align}
by passing to subsequences.
\begin{align*}
\left|\int_{\mathbb{R}^5}\chi(y)\widetilde{\phi^1}(-y)dy\right|&=\left|\int_{\mathbb{R}^5}\chi(y)e^{it_n^1\Delta}\phi_n(x_n^1-y)dy+\int_{\mathbb{R}^5}\chi(y)(\widetilde{\phi^1}(-y)-e^{it_n^1\Delta}\phi_n(x_n^1-y))dy\right|\\
&\geq\left|\int_{\mathbb{R}^5}\chi(y)e^{it_n^1\Delta}\phi_n(x_n^1-y)dy\right|-\left|\int_{\mathbb{R}^5}\chi(y)(\widetilde{\phi^1}(-y)-e^{it_n^1\Delta}\phi_n(x_n^1-y))dy\right|\\
&=|\chi\ast e^{it_n^1\Delta}\phi_n(x_n^1)|-|(\widetilde{\phi^1}-e^{it_n^1\Delta}\phi_n(x_n^1+\cdot),\chi(-\ \cdot\ ))_{L^2}|.
\end{align*}
Similarly,
\[
\left|\int_{\mathbb{R}^5}\chi(y)\widetilde{\psi^1}(-y)dy\right|\geq|\chi\ast e^{\frac{1}{2}it_n^1\Delta}\psi_n(x_n^1)|-|(\widetilde{\psi^1}-e^{\frac{1}{2}it_n^1\Delta}\psi_n(x_n^1+\cdot),\chi(-\ \cdot\ ))_{L^2}|.
\]
Thus, using \eqref{79} and \eqref{80},
\begin{align*}
&\left|\int_{\mathbb{R}^5}\chi(y)\widetilde{\phi^1}(-y)dy\right|+\left|\int_{\mathbb{R}^5}\chi(y)\widetilde{\psi^1}(-y)dy\right|\\
&\ \ \ \ \ \ \ \ \ \ \ \ \ \geq|\chi\ast e^{it_n^1\Delta}\phi_n(x_n^1)|+|\chi\ast e^{\frac{1}{2}it_n^1\Delta}\psi_n(x_n^1)|\\
&\ \ \ \ \ \ \ \ \ \ \ \ \ \ \ \ \ \ \ \ \ \ \ \ \ \ -|(\widetilde{\phi^1}-e^{it_n^1\Delta}\phi_n(x_n^1+\cdot),\chi(-\ \cdot\ ))_{L^2}|-|(\widetilde{\psi^1}-e^{\frac{1}{2}it_n^1\Delta}\psi_n(x_n^1+\cdot),\chi(-\ \cdot\ ))_{L^2}|\\
&\ \ \ \ \ \ \ \ \ \ \ \ \ \geq\frac{A_1^5}{2^7c\,c_1^4}-|(\widetilde{\phi^1}-e^{it_n^1\Delta}\phi_n(x_n^1+\cdot),\chi(-\ \cdot\ ))_{L^2}|-|(\widetilde{\psi^1}-e^{\frac{1}{2}it_n^1\Delta}\psi_n(x_n^1+\cdot),\chi(-\ \cdot\ ))_{L^2}|\\
&\ \ \ \ \ \ \ \ \ \ \ \ \ \longrightarrow\frac{A_1^5}{2^7c\,c_1^4}\ \ \text{as}\ \ n\rightarrow\infty.
\end{align*}
Here,
\begin{align*}
\left|\int_{\mathbb{R}^5}\chi(y)\widetilde{\phi^1}(-y)dy\right|&=\left|\int_{\mathbb{R}^5}\widehat{\chi}(\xi)\widehat{\widetilde{\phi^1}}(\xi)d\xi\right|\\
&\leq\int_{\mathbb{R}^5}\frac{1}{|\xi|^\frac{1}{2}}|\widehat{\chi}(\xi)|\cdot|\xi|^\frac{1}{2}|\widehat{\widetilde{\phi^1}}(\xi)|d\xi\\
&\leq\|\chi\|_{\dot{H}^{-\frac{1}{2}}}\|\widetilde{\phi^1}\|_{\dot{H}^\frac{1}{2}},\\
\|\chi\|_{\dot{H}^{-\frac{1}{2}}}^2&=\int_{\frac{1}{2r}\leq|\xi|\leq2r}\frac{1}{|\xi|}|\widehat{\chi}(\xi)|^2d\xi\\
&\leq\int_{\frac{1}{2r}\leq|\xi|\leq2r}\frac{1}{|\xi|}d\xi.
\end{align*}
We transform $\xi$ into a polar coordinate, i.e. $\xi_1=\eta\cos\theta_1,\,\xi_2=\eta\cos\theta_2\sin\theta_1,\,\xi_3=\eta\cos\theta_3\sin\theta_2\sin\theta_1$,\\
$\xi_4=\eta\cos\theta_4\sin\theta_3\sin\theta_2\sin\theta_1$, $\xi_5=\eta\sin\theta_4\sin\theta_3\sin\theta_2\sin\theta_1$.\\
Then,  $|\xi|=\eta$ and $d\xi=\eta^4\sin^3\theta_1\sin^2\theta_2\sin\theta_3d\eta d\theta_1d\theta_2d\theta_3d\theta_4$. Thus,
\begin{align*}
\|\chi\|_{\dot{H}^{-\frac{1}{2}}}^2&\leq\int_0^{2\pi}\int_0^\pi\int_0^\pi\int_0^\pi\int_\frac{1}{2r}^{2r}\frac{1}{\eta}\eta^4\sin^3\theta_1\sin^2\theta_2\sin\theta_3\,d\eta d\theta_1d\theta_2d\theta_3d\theta_4\\
&=\int_\frac{1}{2r}^{2r}\eta^3d\eta\int_0^\pi\sin^3\theta_1d\theta_1\int_0^\pi\sin^2\theta_2d\theta_2\int_0^\pi\sin\theta_3d\theta_3\int_0^{2\pi}d\theta_4\\
&=\left[\frac{1}{4}\eta^4\right]_\frac{1}{2r}^{2r}\left(\frac{2}{3}\int_0^\pi\sin\theta_1d\theta_1\right)\left(\frac{1}{2}\pi\right)\left[-\cos\theta_3\frac{}{}\right]_0^\pi\cdot2\pi\\
&=\frac{1}{4}\left(16r^4-\frac{1}{16r^4}\right)\cdot\frac{4}{3}\cdot\frac{\pi}{2}\cdot2\cdot2\pi\\
&\leq\frac{32\pi^2r^4}{3}.
\end{align*}
Therefore,
\begin{align*}
\left|\int_{\mathbb{R}^5}\chi(y)\widetilde{\phi^1}(-y)dy\right|&\leq\|\chi\|_{\dot{H}^{-\frac{1}{2}}}\|\widetilde{\phi^1}\|_{\dot{H}^\frac{1}{2}}\\
&\leq\frac{4\sqrt{2}\,\pi r^2}{\sqrt{3}}\|\widetilde{\phi^1}\|_{\dot{H}^\frac{1}{2}}=\frac{4\sqrt{2}\,\pi}{\sqrt{3}}\cdot\frac{2^{12}c^4\,c_1^4}{A_1^4}\|\widetilde{\phi^1}\|_{\dot{H}^\frac{1}{2}}=\frac{2^{14}\sqrt{2}\,\pi c^4\,c_1^4}{\sqrt{3}A_1^4}\|\widetilde{\phi^1}\|_{\dot{H}^\frac{1}{2}}.
\end{align*}
Similarly,
\[
\left|\int_{\mathbb{R}^5}\chi(y)\widetilde{\psi^1}(-y)dy\right|\leq\frac{2^{14}\sqrt{2}\,\pi c^4\,c_1^4}{\sqrt{3}A_1^4}\|\widetilde{\psi^1}\|_{\dot{H}^\frac{1}{2}}.
\]
Thus, we have
\[
\frac{2^{14}\sqrt{2}\,\pi c^4\,c_1^4}{\sqrt{3}A_1^4}\left(\|\widetilde{\phi^1}\|_{\dot{H}^\frac{1}{2}}+\|\widetilde{\psi^1}\|_{\dot{H}^\frac{1}{2}}\right)\geq\frac{A_1^5}{2^7c\,c_1^4},
\]
\[
\|\widetilde{\phi^1}\|_{\dot{H}^\frac{1}{2}}+\|\widetilde{\psi^1}\|_{\dot{H}^\frac{1}{2}}\geq\frac{\sqrt{3}A_1^9}{2^{21}\sqrt{2}\,\pi c^5\,c_1^8}>0.
\]
We set
\[
\widetilde{\Phi_n^1}(x)=\phi_n(x)-e^{-it_n^1\Delta}\widetilde{\phi^1}(x-x_n^1)\,,\ \ \widetilde{\Psi_n^1}(x)=\psi_n(x)-e^{-\frac{1}{2}it_n^1\Delta}\widetilde{\psi^1}(x-x_n^1).
\]
Then,
\begin{align*}
\|\widetilde{\Phi_n^1}\|_{\dot{H}^s}^2-\|\phi_n\|_{\dot{H}^s}^2&=\|\phi_n-e^{-it_n^1\Delta}\widetilde{\phi^1}(\cdot-x_n^1)\|_{\dot{H}^s}^2-\|\phi_n\|_{\dot{H}^s}^2\\
&=\|e^{it_n^1\Delta}\phi_n-\widetilde{\phi^1}(\cdot-x_n^1)\|_{\dot{H}^s}^2-\|\phi_n\|_{\dot{H}^s}^2\\
&=\|e^{it_n^1\Delta}\phi_n(\cdot+x_n^1)-\widetilde{\phi^1}\|_{\dot{H}^s}^2-\|\phi_n\|_{\dot{H}^s}^2\\
&=-2\text{Re}(e^{it_n^1\Delta}\phi_n(\cdot+x_n^1),\widetilde{\phi^1})_{\dot{H}^s}+\|\widetilde{\phi^1}\|_{\dot{H}^s}^2\\
&\longrightarrow-\|\widetilde{\phi^1}\|_{\dot{H}^s}^2\ \ \text{as}\ \ n\rightarrow\infty.
\end{align*}
Hence,
\[
\|\phi_n\|_{\dot{H}^s}^2=\|\widetilde{\phi^1}\|_{\dot{H}^s}^2+\|\widetilde{\Phi_n^1}\|_{\dot{H}^s}^2+o_n(1).
\]
Similarly,
\[
\|\psi_n\|_{\dot{H}^s}^2=\|\widetilde{\psi^1}\|_{\dot{H}^s}^2+\|\widetilde{\Psi_n^1}\|_{\dot{H}^s}^2+o_n(1).
\]
Here, we will prove that we can take $\{t_n\}\in(-\infty,0]$ and $\{x_n\}\in\mathbb{R}^5$ with
\[
either\ \ \ t_n^1=0\ \ \ ^\forall n\in\mathbb{N}\ \ \ or\ \ \ t_n^1\longrightarrow-\infty\ \ as\ \ n\rightarrow\infty,
\]
\[
either\ \ \ x_n^1=0\ \ \ ^\forall n\in\mathbb{N}\ \ \ or\ \ \ |x_n^1|\longrightarrow\infty\ \ as\ \ n\rightarrow\infty.
\]
In the case $t_n^1\longrightarrow-\infty\ \ $as$\ \ n\rightarrow\infty$\ \ and\ \ $|x_n^1|\longrightarrow\infty\ \ $as$\ \ n\rightarrow\infty$, it holds.\\
In the case $t_n^1\longrightarrow t_1<\infty\ \ $as$\ \ n\rightarrow\infty$\ \ and\ \ $|x_n^1|\longrightarrow\infty\ \ $as$\ \ n\rightarrow\infty$, we set\\
$\phi_n(x)=e^{-it_n^1\Delta}\widetilde{\phi^1}(x-x_n^1)+\widetilde{\Phi_n^1}(x)=e^{-it^1\Delta}\widetilde{\phi^1}(x-x_n^1)+\Phi_n^1(x)$. Since
\begin{align*}
&\|\Phi_n^1\|_{\dot{H}^s}^2-\|\phi_n\|_{\dot{H}^s}^2=\|\phi_n-e^{-it^1\Delta}\widetilde{\phi^1}(\cdot-x_n^1)\|_{\dot{H}^s}^2-\|\phi_n\|_{\dot{H}^s}^2\\
&\ \ \ \ \ \ =-2\text{Re}(\phi_n,e^{-it^1\Delta}\widetilde{\phi^1}(\cdot-x_n^1))_{\dot{H}^s}+\|\widetilde{\phi^1}\|_{\dot{H}^s}^2\\
&\ \ \ \ \ \ =-2\text{Re}(\phi_n,e^{-it_n^1\Delta}\widetilde{\phi^1}(\cdot-x_n^1))_{\dot{H}^s}-2\text{Re}(\phi_n,e^{-it^1\Delta}\widetilde{\phi^1}(\cdot-x_n^1)-e^{-it_n^1\Delta}\widetilde{\phi^1}(\cdot-x_n^1))_{\dot{H}^s}+\|\widetilde{\phi^1}\|_{\dot{H}^s}^2\\
&\ \ \ \ \ \ =-2\text{Re}(e^{it_n^1\Delta}\phi_n(\cdot+x_n^1),\widetilde{\phi^1})_{\dot{H}^s}-2\text{Re}(\phi_n,e^{-it^1\Delta}\widetilde{\phi^1}(\cdot-x_n^1)-e^{-it_n^1\Delta}\widetilde{\phi^1}(\cdot-x_n^1))_{\dot{H}^s}+\|\widetilde{\phi^1}\|_{\dot{H}^s}^2\\
&\ \ \ \ \ \ \longrightarrow-\|\widetilde{\phi^1}\|_{\dot{H}^s}^2\ \ \text{as}\ \ n\rightarrow\infty,
\end{align*}
we obtain
\[
\|\phi_n\|_{\dot{H}^s}^2=\|\widetilde{\phi^1}\|_{\dot{H}^s}^2+\|\Phi_n^1\|_{\dot{H}^s}^2+o_n(1).
\]
If we set $\phi^1=e^{-it^1\Delta}\widetilde{\phi^1}$, then since $\phi_n(x)=\phi^1(x-x_n^1)+\Phi_n^1(x)$ and $\|\phi^1\|_{\dot{H}^s}=\|\widetilde{\phi^1}\|_{\dot{H}^s}$, we have
\[
\|\phi_n\|_{\dot{H}^s}^2=\|\phi^1\|_{\dot{H}^s}^2+\|\Phi_n^1\|_{\dot{H}^s}^2+o_n(1).
\]
Also, for any $\varphi\in H^1$,
\begin{align*}
&\left|\left(\phi_n(\cdot+x_n^1)-\phi^1,\varphi\right)_{\dot{H}^s}\right|=\left|\left(\phi_n(\cdot+x_n^1)-e^{-it^1\Delta}\widetilde{\phi^1},\varphi\right)_{\dot{H}^s}\right|\\
&\hspace{1.5cm}\leq\left|\left(\phi_n(\cdot+x_n^1)-e^{-it_n^1\Delta}\widetilde{\phi^1},\varphi\right)_{\dot{H}^s}\right|+\left|\left(e^{-it_n^1\Delta}\widetilde{\phi^1}-e^{-it^1\Delta}\widetilde{\phi^1},\varphi\right)_{\dot{H}^s}\right|\\
&\hspace{1.5cm}\leq\left|\left(e^{it_n^1\Delta}\phi_n(\cdot+x_n^1)-\widetilde{\phi^1},e^{it_n^1\Delta}\varphi\right)_{\dot{H}^s}\right|+\left\|e^{-it_n^1\Delta}\widetilde{\phi^1}-e^{-it^1\Delta}\widetilde{\phi^1}\right\|_{\dot{H}^s}\|\varphi\|_{\dot{H}^s}\\
&\hspace{1.5cm}\leq\left|\left(e^{it_n^1\Delta}\phi_n(\cdot+x_n^1)-\widetilde{\phi^1},e^{it^1\Delta}\varphi\right)_{\dot{H}^s}\right|+\left|\left(e^{it_n^1\Delta}\phi_n(\cdot+x_n^1)-\widetilde{\phi^1},e^{it_n^1\Delta}\varphi-e^{it^1\Delta}\varphi\right)_{\dot{H}^s}\right|\\
&\hspace{10cm}+\left\|e^{-it_n^1\Delta}\widetilde{\phi^1}-e^{-it^1\Delta}\widetilde{\phi^1}\right\|_{\dot{H}^s}\|\varphi\|_{\dot{H}^s}\\
&\hspace{1.5cm}\leq\left|\left(e^{it_n^1\Delta}\phi_n(\cdot+x_n^1)-\widetilde{\phi^1},e^{it^1\Delta}\varphi\right)_{\dot{H}^s}\right|+\left(c_1+\|\widetilde{\phi^1}\|_{\dot{H}^s}\right)\|e^{it_n^1\Delta}\varphi-e^{it^1\Delta}\varphi\|_{\dot{H}^s}\\
&\hspace{10cm}+\left\|e^{-it_n^1\Delta}\widetilde{\phi^1}-e^{-it^1\Delta}\widetilde{\phi^1}\right\|_{\dot{H}^s}\|\varphi\|_{\dot{H}^s}\\
&\hspace{1.5cm}\longrightarrow0\ \ \text{as}\ \ n\rightarrow\infty.
\end{align*}
Thus, $\phi_n(\cdot+x_n^1)\rightharpoonup\phi^1$\ \ in\ \ $H^1$\ \ as\ \ $n\rightarrow\infty.$\\
In the case $t_n^1\longrightarrow-\infty\ \ $as$\ \ n\rightarrow\infty$\ \ and\ \  $x_n^1\longrightarrow x^1\in\mathbb{R}^5\ \ $as$\ \ n\rightarrow\infty$, we set\\
$\phi_n(x)=e^{-it_n^1\Delta}\widetilde{\phi^1}(x-x_n^1)+\widetilde{\Phi_n^1}(x)=e^{-it_n^1\Delta}\widetilde{\phi^1}(x-x^1)+\Phi_n^1(x)$. Since
\begin{align*}
&\|\Phi_n^1\|_{\dot{H}^s}^2-\|\phi_n\|_{\dot{H}^s}^2=\|\phi_n-e^{-it_n^1\Delta}\widetilde{\phi^1}(\cdot-x^1)\|_{\dot{H}^s}^2-\|\phi_n\|_{\dot{H}^s}^2\\
&\hspace{1cm}=-2\text{Re}(\phi_n,e^{-it_n^1\Delta}\widetilde{\phi^1}(\cdot-x^1))_{\dot{H}^s}+\|\widetilde{\phi^1}\|_{\dot{H}^s}^2\\
&\hspace{1cm}=-2\text{Re}(\phi_n,e^{-it_n^1\Delta}\widetilde{\phi^1}(\cdot-x_n^1))_{\dot{H}^s}-2\text{Re}(\phi_n,e^{-it_n^1\Delta}\widetilde{\phi^1}(\cdot-x^1)-e^{-it_n^1\Delta}\widetilde{\phi^1}(\cdot-x_n^1))_{\dot{H}^s}+\|\widetilde{\phi^1}\|_{\dot{H}^s}^2\\
&\hspace{1cm}=-2\text{Re}(e^{it_n^1\Delta}\phi_n(\cdot+x_n^1),\widetilde{\phi^1})_{\dot{H}^s}-2\text{Re}(\phi_n,e^{-it_n^1\Delta}\widetilde{\phi^1}(\cdot-x^1)-e^{-it_n^1\Delta}\widetilde{\phi^1}(\cdot-x_n^1))_{\dot{H}^s}+\|\widetilde{\phi^1}\|_{\dot{H}^s}^2\\
&\hspace{1cm}\longrightarrow-\|\widetilde{\phi^1}\|_{\dot{H}^s}^2\ \ \text{as}\ \ n\rightarrow\infty,
\end{align*}
we obtain
\[
\|\phi_n\|_{\dot{H}^s}^2=\|\widetilde{\phi^1}\|_{\dot{H}^s}^2+\|\Phi_n^1\|_{\dot{H}^s}^2+o_n(1).
\]
If we set $\phi^1(x)=\widetilde{\phi^1}(x-x^1)$, then since $\phi_n(x)=e^{-it_n^1\Delta}\phi^1(x)+\Phi_n^1(x)$ and $\|\phi^1\|_{\dot{H}^s}=\|\widetilde{\phi^1}\|_{\dot{H}^s}$, we obtain
\[
\|\phi_n\|_{\dot{H}^s}^2=\|\phi^1\|_{\dot{H}^s}^2+\|\Phi_n^1\|_{\dot{H}^s}^2+o_n(1).
\]
Also, for any $\varphi\in H^1$,
\begin{align*}
&\left|\left(e^{it_n^j\Delta}\phi_n-\phi^1,\varphi\right)_{\dot{H}^s}\right|=\left|\left(e^{it_n^j\Delta}\phi_n-\widetilde{\phi^1}(\cdot-x^1),\varphi\right)_{\dot{H}^s}\right|\\
&\hspace{1cm}\leq\left|\left(e^{it_n^j\Delta}\phi_n-\widetilde{\phi^1}(\cdot-x_n^1),\varphi\right)_{\dot{H}^s}\right|+\left|\left(\widetilde{\phi^1}(\cdot-x_n^1)-\widetilde{\phi^1}(\cdot-x^1),\varphi\right)_{\dot{H}^s}\right|\\
&\hspace{1cm}\leq\left|\left(e^{it_n^j\Delta}\phi_n(\cdot+x_n^1)-\widetilde{\phi^1},\varphi(\cdot+x_n^1)\right)_{\dot{H}^s}\right|+\|\widetilde{\phi^1}(\cdot-x_n^1)-\widetilde{\phi^1}(\cdot-x^1)\|_{\dot{H}^s}\|\varphi\|_{\dot{H}^s}\\
&\hspace{1cm}\leq\left|\left(e^{it_n^j\Delta}\phi_n(\cdot+x_n^1)-\widetilde{\phi^1},\varphi(\cdot+x^1)\right)_{\dot{H}^s}\right|+\left|\left(e^{it_n^j\Delta}\phi_n(\cdot+x_n^1)-\widetilde{\phi^1},\varphi(\cdot+x_n^1)-\varphi(\cdot+x^1)\right)_{\dot{H}^s}\right|\\
&\hspace{10cm}+\|\widetilde{\phi^1}(\cdot-x_n^1)-\widetilde{\phi^1}(\cdot-x^1)\|_{\dot{H}^s}\|\varphi\|_{\dot{H}^s}\\
&\hspace{1cm}\leq\left|\left(e^{it_n^j\Delta}\phi_n(\cdot+x_n^1)-\widetilde{\phi^1},\varphi(\cdot+x^1)\right)_{\dot{H}^s}\right|+\left(c_1+\|\widetilde{\phi^1}\|_{\dot{H^s}}\right)\|\varphi(\cdot+x_n^1)-\varphi(\cdot+x^1)\|_{\dot{H}^s}\\
&\hspace{10cm}+\|\widetilde{\phi^1}(\cdot-x_n^1)-\widetilde{\phi^1}(\cdot-x^1)\|_{\dot{H}^s}\|\varphi\|_{\dot{H}^s}.
\end{align*}
Here, since $C_0^\infty$ is dense in $\dot{H}^s$, for any $\varepsilon>0$, there exists $f\in C_0^\infty$ such that $\|\varphi-f\|_{\dot{H}^s}<\varepsilon$. Thus,
\begin{align*}
&\|\varphi(\cdot+x_n^1)-\varphi(\cdot+x^1)\|_{\dot{H}^s}\\
&\hspace{2cm}\leq\|\varphi(\cdot+x_n^1)-f(\cdot+x_n^1)\|_{\dot{H}^s}+\|f(\cdot+x_n^1)-f(\cdot+x^1)\|_{\dot{H}^s}+\|f(\cdot+x^1)-\varphi(\cdot+x^1)\|_{\dot{H}^s}\\
&\hspace{2cm}=\|\varphi-f\|_{\dot{H}^s}+\|f(\cdot+x_n^1)-f(\cdot+x^1)\|_{\dot{H}^s}+\|f-\varphi\|_{\dot{H}^s}<\varepsilon.
\end{align*}
Similarly,
\[
\|\widetilde{\phi^1}(\cdot-x_n^1)-\widetilde{\phi^1}(\cdot-x^1)\|_{\dot{H}^s}<\varepsilon.
\]
Therefore,
\[
\left|\left(e^{it_n^j\Delta}\phi_n-\phi^1,\varphi\right)_{\dot{H}^s}\right|\longrightarrow0\ \ \text{as}\ \ n\rightarrow\infty,
\]
i.e. $e^{it_n^j\Delta}\phi_n\rightharpoonup\phi^1$\ \ in\ \ $H^1$\ \ as\ \ $n\rightarrow\infty$.\\
In the case $t_n^1\longrightarrow t^1<\infty\ \ $as$\ \ n\rightarrow\infty$\ \ and\ \ $x_n^1\longrightarrow x^1\in\mathbb{R}^5\ \ $as$\ \ n\rightarrow\infty$, we set\\
$\phi_n(x)=e^{-it_n^1\Delta}\widetilde{\phi^1}(x-x_n^1)+\widetilde{\Phi_n^1}(x)=e^{-it^1\Delta}\widetilde{\phi^1}(x-x^1)+\Phi_n^1(x)$. Since
\begin{align*}
\|\Phi_n^1\|_{\dot{H}^s}^2-\|\phi_n\|_{\dot{H}^s}^2&=\|\phi_n-e^{-it^1\Delta}\widetilde{\phi^1}(\cdot-x^1)\|_{\dot{H}^s}^2-\|\phi_n\|_{\dot{H}^s}^2\\
&=-2\text{Re}(\phi_n,e^{-it^1\Delta}\widetilde{\phi^1}(\cdot-x^1))_{\dot{H}^s}+\|\widetilde{\phi^1}\|_{\dot{H}^s}^2\\
&=-2\text{Re}(\phi_n,e^{-it_n^1\Delta}\widetilde{\phi^1}(\cdot-x_n^1))_{\dot{H}^s}-2\text{Re}(\phi_n,e^{-it_n^1\Delta}\widetilde{\phi^1}(\cdot-x^1)-e^{-it_n^1\Delta}\widetilde{\phi^1}(\cdot-x_n^1))_{\dot{H}^s}\\
&\hspace{4cm}-2\text{Re}(\phi_n,e^{-it^1\Delta}\widetilde{\phi^1}(\cdot-x^1)-e^{-it_n^1\Delta}\widetilde{\phi^1}(\cdot-x^1))_{\dot{H}^s}+\|\widetilde{\phi^1}\|_{\dot{H}^s}^2\\
&\longrightarrow-\|\widetilde{\phi^1}\|_{\dot{H}^s}^2\ \ \text{as}\ \ n\rightarrow\infty,
\end{align*}
we obtain
\[
\|\phi_n\|_{\dot{H}^s}^2=\|\widetilde{\phi^1}\|_{\dot{H}^s}^2+\|\Phi_n^1\|_{\dot{H}^s}^2+o_n(1).
\]
If we set $\phi^1(x)=e^{-it^1\Delta}\widetilde{\phi^1}(x-x^1)$, then since $\phi_n(x)=\phi^1(x)+\Phi_n^1(x)$ and $\|\phi^1\|_{\dot{H}^s}=\|\widetilde{\phi^1}\|_{\dot{H}^s}$, we have
\[
\|\phi_n\|_{\dot{H}^s}^2=\|\phi^1\|_{\dot{H}^s}^2+\|\Phi_n^1\|_{\dot{H}^s}^2+o_n(1).
\]
Also, for any $\varphi\in H^1$,
\begin{align*}
\left|\left(\phi_n-\phi^1,\varphi\right)_{\dot{H}^s}\right|&=\left|\left(\phi_n-e^{-it^1\Delta}\widetilde{\phi^1}(\cdot-x^1),\varphi\right)_{\dot{H}^s}\right|\\
&\leq\left|\left(\phi_n-e^{-it_n^1\Delta}\widetilde{\phi^1}(\cdot-x_n^1),\varphi\right)_{\dot{H}^s}\right|+\left|\left(e^{-it_n^1\Delta}\widetilde{\phi^1}(\cdot-x_n^1)-e^{-it^1\Delta}\widetilde{\phi^1}(\cdot-x_n^1),\varphi\right)_{\dot{H}^s}\right|\\
&\hspace{1cm}+\left|\left(e^{-it^1\Delta}\widetilde{\phi^1}(\cdot-x_n^1)-e^{-it^1\Delta}\widetilde{\phi^1}(\cdot-x^1),\varphi\right)_{\dot{H}^s}\right|\\
&\leq\left|\left(e^{it_n^1\Delta}\phi_n(\cdot+x_n^1)-\widetilde{\phi^1},e^{it_n^1\Delta}\varphi(\cdot+x_n^1)\right)_{\dot{H}^s}\right|+\|e^{-it_n^1\Delta}\widetilde{\phi^1}-e^{-it^1\Delta}\widetilde{\phi^1}\|_{\dot{H}^s}\|\varphi\|_{\dot{H}^s}\\
&\hspace{1cm}+\|\widetilde{\phi^1}(\cdot-x_n^1)-\widetilde{\phi^1}(\cdot-x^1)\|_{\dot{H}^s}\|\varphi\|_{\dot{H}^s}\\
&\leq\left|\left(e^{it_n^1\Delta}\phi_n(\cdot+x_n^1)-\widetilde{\phi^1},e^{it^1\Delta}\varphi(\cdot+x^1)\right)_{\dot{H}^s}\right|\\
&\hspace{1cm}+\left|\left(e^{it_n^1\Delta}\phi_n(\cdot+x_n^1)-\widetilde{\phi^1},e^{it_n^1\Delta}\varphi(\cdot+x^1)-e^{it^1\Delta}\varphi(\cdot+x^1)\right)_{\dot{H}^s}\right|\\
&\hspace{1cm}+\left|\left(e^{it_n^1\Delta}\phi_n(\cdot+x_n^1)-\widetilde{\phi^1},e^{it_n^1\Delta}\varphi(\cdot+x_n^1)-e^{it_n^1\Delta}\varphi(\cdot+x^1)\right)_{\dot{H}^s}\right|\\
&\hspace{1cm}+\|e^{-it_n^1\Delta}\widetilde{\phi^1}-e^{-it^1\Delta}\widetilde{\phi^1}\|_{\dot{H}^s}\|\varphi\|_{\dot{H}^s}+\|\widetilde{\phi^1}(\cdot-x_n^1)-\widetilde{\phi^1}(\cdot-x^1)\|_{\dot{H}^s}\|\varphi\|_{\dot{H}^s}\\
&\leq\left|\left(e^{it_n^1\Delta}\phi_n(\cdot+x_n^1)-\widetilde{\phi^1},e^{it^1\Delta}\varphi(\cdot+x^1)\right)_{\dot{H}^s}\right|+\left(c_1+\|\widetilde{\phi^1}\|_{\dot{H}^s}\right)\|e^{it_n^1\Delta}\varphi-e^{it^1\Delta}\varphi\|_{\dot{H}^s}\\
&\hspace{1cm}+\left(c_1+\|\widetilde{\phi^1}\|_{\dot{H}^s}\right)\|\varphi(\cdot+x_n^1)-\varphi(\cdot+x^1)\|_{\dot{H}^s}+\|e^{-it_n^1\Delta}\widetilde{\phi^1}-e^{-it^1\Delta}\widetilde{\phi^1}\|_{\dot{H}^s}\|\varphi\|_{\dot{H}^s}\\
&\hspace{1cm}+\|\widetilde{\phi^1}(\cdot-x_n^1)-\widetilde{\phi^1}(\cdot-x^1)\|_{\dot{H}^s}\|\varphi\|_{\dot{H}^s}\longrightarrow0\ \ \text{as}\ \ n\rightarrow\infty,
\end{align*}
i.e. we have $\phi_n\rightharpoonup\phi^1$\ \ in\ \ $H^1$\ \ as\ \ $n\rightarrow\infty.$\\
We construct $\psi^1$ and $\Psi_n^1$ with $\widetilde{\psi^1}$ and $\widetilde{\Psi_n^1}$ respectively. Then, Theorem \ref{Linear profile decomposition} for $M=1$ holds. Moreover, $\phi^1$, $\psi^1$, $\Phi_n^1$, $\Psi_n^1$ satisfy the following properties, which $\widetilde{\phi^1}$, $\widetilde{\psi^1}$, $\widetilde{\Phi_n^1}$, $\widetilde{\Psi_n^1}$ satisfy.
\[
\phi_n(x)=e^{-it_n^1\Delta}\phi^1(x-x_n^1)+\Phi_n^1(x)\,,\ \ \ \psi_n(x)=e^{-\frac{1}{2}it_n^1\Delta}\psi^1(x-x_n^1)+\Psi_n^1(x),
\]
\[
e^{it_n^1\Delta}\phi_n(\cdot+x_n^1)\rightharpoonup\phi^1\,,\ \ e^{\frac{1}{2}it_n^1\Delta}\psi_n(\cdot+x_n^1)\rightharpoonup\psi^1\ \ \text{in}\ \ H^1\ \ \text{as}\ \ n\rightarrow\infty,
\]
\[
\|\phi^1\|_{\dot{H}^\frac{1}{2}}+\|\psi^1\|_{\dot{H}^\frac{1}{2}}\geq\frac{\sqrt{3}A_1^9}{2^{21}\sqrt{2}\,\pi c^5\,c_1^8}>0,
\]
\[
\|\phi_n\|_{\dot{H}^s}^2=\|\phi^1\|_{\dot{H}^s}^2+\|\Phi_n^1\|_{\dot{H}^s}^2+o_n(1)\,,\ \ \ \|\psi_n\|_{\dot{H}^s}^2=\|\psi^1\|_{\dot{H}^s}^2+\|\Psi_n^1\|_{\dot{H}^s}^2+o_n(1).
\]
(ii)\ \ Let $M=2$.\\
We define $\displaystyle A_2=\limsup_{n\rightarrow\infty}\left(\|e^{it\Delta}\Phi_n^1\|_{L^\infty L^\frac{5}{2}}+\|e^{it\Delta}\Psi_n^1\|_{L^\infty L^\frac{5}{2}}\right)$.\\
In the case $A_2=0$, if we take $\phi^j=0,\,\psi^j=0\ (j\geq2)$, then Theorem \ref{Linear profile decomposition} holds.\\
So, we assume that $A_2>0$.\\
Since $\displaystyle \limsup_{n\rightarrow\infty}\left(\|\Phi_n^1\|_{\dot{H}^s}+\|\Psi_n^1\|_{\dot{H}^s}\right)\leq\limsup_{n\rightarrow\infty}\left(\|\phi_n\|_{\dot{H}^s}+\|\psi_n\|_{\dot{H}^s}\right)\leq c_1$, we can take $\{t_n^2\}\subset(-\infty,0]\,,\ \{x_n^2\}\subset\mathbb{R}^5$ by applying the argument for $M=1$.
\[
either\ \ \ t_n^2=0\ \ \ ^\forall n\in\mathbb{N}\ \ \ or\ \ \ t_n^2\longrightarrow-\infty\ \ as\ \ n\rightarrow\infty,
\]
\[
either\ \ \ x_n^2=0\ \ \ ^\forall n\in\mathbb{N}\ \ \ or\ \ \ |x_n^2|\longrightarrow\infty\ \ as\ \ n\rightarrow\infty,
\]
\[
e^{it_n^2\Delta}\Phi_n^1(\cdot+x_n^2)\rightharpoonup\phi^2\,,\ \ e^{\frac{1}{2}it_n^2\Delta}\Psi_n^1(\cdot+x_n^2)\rightharpoonup\psi^2\ \ \text{in}\ \ H^1,
\]
\[
\|\phi^2\|_{\dot{H}^\frac{1}{2}}+\|\psi^2\|_{\dot{H}^\frac{1}{2}}\geq\frac{\sqrt{3}A_2^9}{2^{21}\sqrt{2}\,\pi c^5\,c_1^8}>0.
\]
Here, we define
\[
\Phi_n^2(x)=\Phi_n^1(x)-e^{-it_n^2\Delta}\phi^2(x-x_n^2)\,,\ \ \Psi_n^2(x)=\Psi_n^1(x)-e^{-\frac{1}{2}it_n^2\Delta}\psi^2(x-x_n^2).
\]
Then, we have
\begin{align*}
\|\Phi_n^2\|_{\dot{H}^s}^2-\|\Phi_n^1\|_{\dot{H}^s}^2&=\|\Phi_n^1-e^{-it_n^2\Delta}\phi^2(\cdot-x_n^2)\|_{\dot{H}^s}^2-\|\Phi_n^1\|_{\dot{H}^s}^2\\
&=\|e^{it_n^2\Delta}\Phi_n^1(\cdot+x_n^2)-\phi^2\|_{\dot{H}^s}^2-\|\Phi_n^1\|_{\dot{H}^s}^2\\
&=\|\phi^2\|_{\dot{H}^s}^2-2\text{Re}\left(e^{it_n^2\Delta}\Phi_n^1(\cdot+x_n^2),\,\phi^2\right)_{\dot{H}^s}\\
&\longrightarrow-\|\phi^2\|_{\dot{H}^s}^2\ \ \text{as}\ \ n\rightarrow\infty.
\end{align*}
Thus,
\[
\|\Phi_n^2\|_{\dot{H}^s}^2=\|\Phi_n^1\|_{\dot{H}^s}^2-\|\phi^2\|_{\dot{H}^s}^2+o_n(1)=\|\phi_n\|_{\dot{H}^s}^2-\|\phi^1\|_{\dot{H}^s}^2-\|\phi^2\|_{\dot{H}^s}^2+o_n(1),
\]
i.e.
\[
\|\phi_n\|_{\dot{H}^s}^2=\|\phi^1\|_{\dot{H}^s}^2+\|\phi^2\|_{\dot{H}^s}^2+\|\Phi_n^2\|_{\dot{H}^s}^2+o_n(1).
\]
Similarly,
\[
\|\psi_n\|_{\dot{H}^s}^2=\|\psi^1\|_{\dot{H}^s}^2+\|\psi^2\|_{\dot{H}^s}^2+\|\Psi_n^2\|_{\dot{H}^s}^2+o_n(1).
\]
Here, we will prove that $\displaystyle\lim_{n\rightarrow\infty}\left(|t_n^2-t_n^1|+|x_n^2-x_n^1|\right)=\infty$.\\
From $\displaystyle\|\phi^2\|_{\dot{H}^\frac{1}{2}}+\|\psi^2\|_{\dot{H}^\frac{1}{2}}>0$, we assume that $\phi^2\neq0$ without loss of generality. Then,
\[
z_n\vcentcolon=e^{it_n^1\Delta}\Phi_n^1(\cdot+x_n^1)=e^{it_n^1\Delta}\phi_n(\cdot+x_n^1)-\phi^1\rightharpoonup0\ \ \text{in}\ \ H^1\ \ \text{as}\ \ n\rightarrow\infty
\]
and
\[
e^{i(t_n^2-t_n^1)\Delta}z_n(\cdot+x_n^2-x_n^1)=e^{it_n^2\Delta}\Phi_n^1(\cdot+x_n^2)\rightharpoonup\phi^2\ \ \text{in}\ \ H^1\ \ \text{as}\ \ n\rightarrow\infty.
\]
Since $\phi^2\neq0$, $\displaystyle\lim_{n\rightarrow\infty}\left(|t_n^2-t_n^1|+|x_n^2-x_n^1|\right)=\infty$ holds by applying Lemma \ref{time shift lemma}.\\
Therefore, Theorem \ref{Linear profile decomposition} holds for $M=2$.\\
(iii)\ \ Let $M\geq3$.\\
Since $\displaystyle \limsup_{n\rightarrow\infty}\left(\|\Phi_n^{j-1}\|_{\dot{H}^s}+\|\Psi_n^{j-1}\|_{\dot{H}^s}\right)\leq\limsup_{n\rightarrow\infty}\left(\|\phi_n\|_{\dot{H}^s}+\|\psi_n\|_{\dot{H}^s}\right)\leq c_1$, we can construct $\{t_n^j\}\subset(-\infty,0],\ \{x_n^j\}\subset\mathbb{R}^5,\ \phi^j,\ \psi^j\ (1\leq j\leq M)$ inductively.\\
We define $\displaystyle A_j=\limsup_{n\rightarrow\infty}\left(\|e^{it\Delta}\Phi_n^{j-1}\|_{L^\infty L^\frac{5}{2}}+\|e^{\frac{1}{2}it\Delta}\Psi_n^{j-1}\|_{L^\infty L^\frac{5}{2}}\right)$.\\
When there exists $1\leq j\leq M$ such that $A_j=0$, if we take $\phi^i=0,\ \psi^i=0\ \ (j\leq i\leq M)$, then Theorem \ref{Linear profile decomposition} holds.\\
Thus, we assume that $A_j>0$ for any $1\leq j\leq M$. Then, we have
\[
either\ \ \ t_n^j=0\ \ \ ^\forall n\in\mathbb{N}\ \ \ or\ \ \ t_n^j\longrightarrow-\infty\ \ as\ \ n\rightarrow\infty,
\]
\[
either\ \ \ x_n^j=0\ \ \ ^\forall n\in\mathbb{N}\ \ \ or\ \ \ |x_n^j|\longrightarrow\infty\ \ as\ \ n\rightarrow\infty,
\]
\[
e^{it_n^j\Delta}\Phi_n^{j-1}(\cdot+x_n^j)\rightharpoonup\phi^j\,,\ \ e^{\frac{1}{2}it_n^j\Delta}\Psi_n^{j-1}(\cdot+x_n^j)\rightharpoonup\psi^j\ \ \text{in}\ \ H^1,
\]
\[
\|\phi^j\|_{\dot{H}^\frac{1}{2}}+\|\psi^j\|_{\dot{H}^\frac{1}{2}}\geq\frac{\sqrt{3}A_j^9}{2^{21}\sqrt{2}\,\pi c^5\,c_1^8}>0,
\]
\[
\Phi_n^j(x)=\Phi_n^{j-1}(x)-e^{-it_n^j\Delta}\phi^j(x-x_n^j)\,,\ \ \Psi_n^j(x)=\Psi_n^{j-1}(x)-e^{-\frac{1}{2}it_n^j\Delta}\psi^j(x-x_n^j).
\]
We will prove asymptotic Pythagorean expansion by induction. We assume that
\[
\|\phi_n\|_{\dot{H}^s}^2=\sum_{j=1}^{M-1}\|\phi^j\|_{\dot{H}^s}^2+\|\Phi_n^{M-1}\|_{\dot{H}^s}^2+o_n(1)
\]
holds. Then,
\begin{align*}
\|\Phi_n^M\|_{\dot{H}^s}^2-\|\phi_n\|_{\dot{H}^s}^2&=\|\Phi_n^{M-1}-e^{-it_n^M\Delta}\phi^M(\cdot-x_n^M)\|_{\dot{H}^s}^2-\|\phi_n\|_{\dot{H}^s}^2\\
&=\|e^{it_n^M\Delta}\Phi_n^{M-1}(\cdot+x_n^M)-\phi^M\|_{\dot{H}^s}^2-\|\phi_n\|_{\dot{H}^s}^2\\
&=\|\Phi_n^{M-1}\|_{\dot{H}^s}^2+\|\phi^M\|_{\dot{H}^s}^2-2\text{Re}\left(e^{it_n^M\Delta}\Phi_n^{M-1}(\cdot+x_n^M),\,\phi^M\right)_{\dot{H}^s}-\|\phi_n\|_{\dot{H}^s}^2\\
&=-\sum_{j=1}^{M-1}\|\phi^j\|_{\dot{H}^s}^2+\|\phi^M\|_{\dot{H}^s}^2-2\text{Re}\left(e^{it_n^M\Delta}\Phi_n^{M-1}(\cdot+x_n^M),\,\phi^M\right)_{\dot{H}^s}+o_n(1)\\
&\longrightarrow-\sum_{j=1}^M\|\phi^j\|_{\dot{H}^s}^2\ \ \text{as}\ \ n\rightarrow\infty.
\end{align*}
Thus,
\[
\|\phi_n\|_{\dot{H}^s}^2=\sum_{j=1}^M\|\phi^j\|_{\dot{H}^s}^2+\|\Phi_n^M\|_{\dot{H}^s}^2+o_n(1),
\]
i.e. asymptotic Pythagorean expansion holds for any $M\in\mathbb{N}$.\\
Next, we will prove pairwise divergence property by induction. We assume that for any $j,\,k\in\{1,2,\cdots,M-1\}$ with $j \neq k$,
\[
\lim_{n\rightarrow\infty}\left(|t_n^k-t_n^j|+|x_n^k-x_n^j|\right)=\infty.
\]
Since $\|\phi^M\|_{\dot{H}^\frac{1}{2}}+\|\psi^M\|_{\dot{H}^\frac{1}{2}}>0$, we set $\phi^M\neq0$ without loss of generality.
\begin{align*}
&e^{it_n^j\Delta}\Phi_n^{j-1}(x+x_n^j)-e^{it_n^j\Delta}\Phi_n^{M-1}(x+x_n^j)-\phi^j\\
&\hspace{2cm}=\left(e^{it_n^j\Delta}\phi_n(x+x_n^j)-\sum_{k=1}^{j-1}e^{i(t_n^j-t_n^k)\Delta}\phi^k(x+x_n^j-x_n^k)\right)\\
&\hspace{4cm}-\left(e^{it_n^j\Delta}\phi_n(x+x_n^j)-\sum_{k=1}^{M-1}e^{i(t_n^j-t_n^k)\Delta}\phi^k(x+x_n^j-x_n^k)\right)-\phi^j\\
&\hspace{2cm}=\sum_{k=j}^{M-1}e^{i(t_n^j-t_n^k)\Delta}\phi^k(x+x_n^j-x_n^k)-\phi^j\\
&\hspace{2cm}=\sum_{k=j+1}^{M-1}e^{i(t_n^j-t_n^k)\Delta}\phi^k(x+x_n^j-x_n^k).
\end{align*}
Since $e^{it_n^j\Delta}\Phi_n^{j-1}(\cdot+x_n^j)\rightharpoonup\phi^j$ holds and the right side converges weakly on 0 by Lemma \ref{time shift lemma}, we have\\
$z_n=e^{it_n^j\Delta}\Phi_n^{M-1}(x+x_n^j)\rightharpoonup0$.\\
Also, since $e^{i(t_n^M-t_n^j)\Delta}z_n(\cdot+x_n^M-x_n^j)=e^{it_n^M\Delta}\Phi_n^{M-1}(\cdot+x_n^M)\rightharpoonup\phi^M$, it follows that
\[
\lim_{n\rightarrow\infty}\left(|t_n^M-t_n^j|+|x_n^M-x_n^j|\right)=\infty.
\]
by applying Lemma \ref{time shift lemma}.\\
Finally, we prove asymptotic smallness property. We already established
\[
\limsup_{n\rightarrow\infty}\|(\Phi_n^M,\Psi_n^M)\|_{\dot{H}^\frac{1}{2}\times \dot{H}^\frac{1}{2}}\leq c_1.
\]
Since
\[
\sum_{M=1}^\infty\left(\frac{\sqrt{3}A_M^9}{2^{21}\sqrt{2}\pi c^5c_1^8}\right)^2\leq2\sum_{M=1}^\infty\left(\|\phi^M\|_{\dot{H}^\frac{1}{2}}^2+\|\psi^M\|_{\dot{H}^\frac{1}{2}}^2\right)\leq2\limsup_{n\rightarrow\infty}\left(\|\phi_n\|_{\dot{H}^\frac{1}{2}}^2+\|\psi_n\|_{\dot{H}^\frac{1}{2}}^2\right)\leq 4c_1^2<\infty,
\]
it follows that $\displaystyle\lim_{M\rightarrow\infty}A_M=0$.
\end{proof}


\begin{corollary}\label{I decomposition}
Under the same assumption as Theorem \ref{Linear profile decomposition}, we have
\[
I_\omega(\phi_n,\psi_n)=\sum_{j=1}^MI_\omega(e^{-it_n^j\Delta}\phi^j,e^{-\frac{1}{2}it_n^j\Delta}\psi^j)+I_\omega(\Phi_n^M,\Psi_n^M)+o_{n}(1).
\]
\end{corollary}

\begin{proof}
Combining
\begin{align*}
I_\omega(u,v)&=\frac{\omega}{2}M(u,v)+\frac{1}{2}E(u,v)\\
&=\frac{\omega}{2}(\|u\|_{L^2}^2+2\|v\|_{L^2}^2)+\frac{1}{2}\left(\|\nabla u\|_{L^2}^2+\frac{1}{2}\|\nabla v\|_{L^2}^2-2\text{Re}(v,u^2)_{L^2}\right),
\end{align*}
\eqref{27}, and \eqref{28}, we find that it is sufficient to prove
\[
\text{Re}(\psi_n,\phi_n^2)_{L^2}=\sum_{j=1}^M\text{Re}\left(e^{-\frac{1}{2}it_n^j\Delta}\psi^j,\left(e^{-it_n^j\Delta}\phi^j\right)^2\right)_{L^2}+\text{Re}(\Psi_n^M,({\Phi_n^M})^2)_{L^2}+o_n(1).
\]
We prove the following three equations to prove this equation.
\[
\int_{\mathbb{R}^5}\Psi_n^{M_1}\overline{\Phi_n^{M_1}}^2dx=o_{M_1,n}(1),
\]
\[
\int_{\mathbb{R}^5}\left(\psi_n\overline{\phi_n}^2-\Psi_n^{M_1}\overline{\Phi_n^{M_1}}^2\right)dx=\sum_{j=1}^{M_1}\int_{\mathbb{R}^5}e^{-\frac{1}{2}it_n^j\Delta}\psi^j\left(\overline{e^{-it_n^j\Delta}\phi^j}\right)^2dx+o_{M_1,n}(1),
\]
\[
\sum_{j=M+1}^{M_1}\int_{\mathbb{R}^5}e^{-\frac{1}{2}it_n^j\Delta}\psi^j\left(\overline{e^{-it_n^j\Delta}\phi^j}\right)^2dx=\int_{\mathbb{R}^5}\left(\Psi_n^M\overline{\Phi_n^M}^2-\Psi_n^{M_1}\overline{\Phi_n^{M_1}}^2\right)dx+o_{M_1,n}(1).
\]
For first equation,
\begin{align*}
\left|\int_{\mathbb{R}^5}\Psi_n^{M_1}\overline{\Phi_n^{M_1}}^2dx\right|&\leq\left\|\Psi_n^{M_1}\right\|_{L^\frac{5}{2}}\left\|\Phi_n^{M_1}\right\|_{L^\frac{10}{3}}^2\leq\left\|e^{\frac{1}{2}it\Delta}\Psi_n^{M_1}\right\|_{L^\infty L^\frac{5}{2}}\left\|\Phi_n^{M_1}\right\|_{\dot{H}^1}^2\\
&\leq\left\|e^{\frac{1}{2}it\Delta}\Psi_n^{M_1}\right\|_{S(\dot{H}^\frac{1}{2})}\left\|\Phi_n^{M_1}\right\|_{\dot{H}^1}^2\longrightarrow0\ \ \text{as}\ \ M_1,n\rightarrow\infty.
\end{align*}
For third equation,
\begin{align*}
&\int_{\mathbb{R}^5}\left(\Psi_n^M\overline{\Phi_n^M}^2-\Psi_n^{M_1}\overline{\Phi_n^{M_1}}^2\right)dx\\
&\hspace{0.3cm}=\int_{\mathbb{R}^5}\left\{\left(\Psi_n^{M_1}+\sum_{j=M+1}^{M_1}e^{-\frac{1}{2}it_n^j\Delta}\psi^j(\cdot-x_n^j)\right)\left(\overline{\Phi_n^{M_1}}+\sum_{j=M+1}^{M_1}\overline{e^{-it_n^j\Delta}\phi^j(\cdot-x_n^j)}\right)^2-\Psi_n^{M_1}\overline{\Phi_n^{M_1}}^2\right\}dx\\
&\hspace{0.3cm}=\int_{\mathbb{R}^5}\left\{\left(\Psi_n^{M_1}+\sum_{j=M+1}^{M_1}e^{-\frac{1}{2}it_n^j\Delta}\psi^j(\cdot-x_n^j)\right)\left(\overline{\Phi_n^{M_1}}^2+2\overline{\Phi_n^{M_1}}\sum_{j=M+1}^{M_1}\overline{e^{-it_n^j\Delta}\phi^j(\cdot-x_n^j)}\right.\right.\\
&\hspace{0.6cm}\left.\left.+\sum_{j=M+1}^{M_1}\overline{e^{-it_n^j\Delta}\phi^j(\cdot-x_n^j)}^2+\sum_{\substack{j,k=M+1\\j\neq k}}^{M_1}\overline{e^{-it_n^j\Delta}\phi^j(\cdot-x_n^j)}\cdot\overline{e^{-it_n^k\Delta}\phi^k(\cdot-x_n^k)}\right)-\Psi_n^{M_1}\overline{\Phi_n^{M_1}}^2\right\}dx\\
&\hspace{0.3cm}=\int_{\mathbb{R}^5}\left(2\Psi_n^{M_1}\overline{\Phi_n^{M_1}}\sum_{j=M+1}^{M_1}\overline{e^{-it_n^j\Delta}\phi^j(\cdot-x_n^j)}+\Psi_n^{M_1}\sum_{j=M+1}^{M_1}\overline{e^{-it_n^j\Delta}\phi^j(\cdot-x_n^j)}^2\right.\\
&\hspace{0.6cm}+\Psi_n^{M_1}\sum_{\substack{j,k=M+1\\j\neq k}}^{M_1}\overline{e^{-it_n^j\Delta}\phi^j(\cdot-x_n^j)}\cdot\overline{e^{-it_n^k\Delta}\phi^k(\cdot-x_n^k)}+\overline{\Phi_n^{M_1}}^2\sum_{j=M+1}^{M_1}e^{-\frac{1}{2}it_n^j\Delta}\psi^j(\cdot-x_n^j)\\
&\hspace{0.6cm}+2\overline{\Phi_n^{M_1}}\sum_{j=M+1}^{M_1}e^{-\frac{1}{2}it_n^j\Delta}\psi^j(\cdot-x_n^j)\sum_{j=M+1}^{M_1}\overline{e^{-it_n^j\Delta}\phi^j(\cdot-x_n^j)}\\
&\hspace{0.6cm}+\sum_{j=M+1}^{M_1}e^{-\frac{1}{2}it_n^j\Delta}\psi^j(\cdot-x_n^j)\sum_{j=M+1}^{M_1}\overline{e^{-it_n^j\Delta}\phi^j(\cdot-x_n^j)}^2\\
&\hspace{0.6cm}\left.+\sum_{j=M+1}^{M_1}e^{-\frac{1}{2}it_n^j\Delta}\psi^j(\cdot-x_n^j)\sum_{\substack{j,k=M+1\\j\neq k}}^{M_1}\overline{e^{-it_n^j\Delta}\phi^j(\cdot-x_n^j)}\cdot\overline{e^{-it_n^k\Delta}\phi^k(\cdot-x_n^k)}\right)dx.
\end{align*}
Here, we consider 
\[
\int_{\mathbb{R}^5}e^{-\frac{1}{2}it_n^j\Delta}\psi^j(x-x_n^j)\cdot\overline{e^{-it_n^k\Delta}\phi^k(x-x_n^k)}\cdot\overline{e^{-it_n^l\Delta}\phi^l(x-x_n^l)}dx\ \ (j\neq k).
\]
Since $C_0^\infty$ is dense in $H^1$, for any $\varepsilon>0$, there exists $f,\,g\in C_0^\infty$ such that
\[
\|\phi^k-f\|_{H^1}<\varepsilon,\,\|\psi^j-g\|_{H^1}<\varepsilon.
\]
In the case $\displaystyle\lim_{n\rightarrow\infty}t_n^j=-\infty$\ \ or\ \ $\displaystyle\lim_{n\rightarrow\infty}t_n^k=-\infty$, we set that $\displaystyle\lim_{n\rightarrow\infty}t_n^j=-\infty$ without loss of generality.
\begin{align*}
&\left|\int_{\mathbb{R}^5}e^{-\frac{1}{2}it_n^j\Delta}\psi^j(x-x_n^j)\cdot\overline{e^{-it_n^k\Delta}\phi^k(x-x_n^k)}\cdot\overline{e^{-it_n^l\Delta}\phi^l(x-x_n^l)}dx\right|\\
&\hspace{1.5cm}\leq\left\|e^{-\frac{1}{2}it_n^j\Delta}\psi^j(\cdot-x_n^j)\cdot e^{-it_n^k\Delta}\phi^k(\cdot-x_n^k)\cdot e^{-it_n^l\Delta}\phi^l(\cdot-x_n^l)\right\|_{L^1}\\
&\hspace{1.5cm}\leq\left\|e^{-\frac{1}{2}it_n^j\Delta}(\psi^j-g)(\cdot-x_n^j)\cdot e^{-it_n^k\Delta}\phi^k(\cdot-x_n^k)\cdot e^{-it_n^l\Delta}\phi^l(\cdot-x_n^l)\right\|_{L^1}\\
&\hspace{3cm}+\left\|e^{-\frac{1}{2}it_n^j\Delta}g(\cdot-x_n^j)\cdot e^{-it_n^k\Delta}\phi^k(\cdot-x_n^k)\cdot e^{-it_n^l\Delta}\phi^l(\cdot-x_n^l)\right\|_{L^1}\\
&\hspace{1.5cm}\leq\left\|e^{-\frac{1}{2}it_n^j\Delta}(\psi^j-g)\right\|_{L^\frac{5}{2}}\left\|e^{-it_n^k\Delta}\phi^k\right\|_{L^\frac{10}{3}}\left\|e^{-it_n^l\Delta}\phi^l\right\|_{L^\frac{10}{3}}+\left\|e^{-\frac{1}{2}it_n^j\Delta}g\right\|_{L^\infty}\left\|\phi^k\right\|_{L^2}\left\|\phi^l\right\|_{L^2}\\
&\hspace{1.5cm}\leq\left\|\psi^j-g\right\|_{\dot{H}^\frac{1}{2}}\left\|\phi^k\right\|_{\dot{H}^1}\left\|\phi^l\right\|_{\dot{H}^1}+c|t_n^j|^{-\frac{5}{2}}\left\|g\right\|_{L^1}\left\|\phi^k\right\|_{L^2}\left\|\phi^l\right\|_{L^2}\\
&\hspace{1.5cm}<\varepsilon.
\end{align*}
In the case $\displaystyle t_n^j=t_n^k=0$ for any $n\in\mathbb{N}$, we have $\displaystyle\lim_{n\rightarrow\infty}|x_n^j-x_n^k|=\infty$ by Theorem \ref{Linear profile decomposition}.
\begin{align*}
&\left|\int_{\mathbb{R}^5}e^{-\frac{1}{2}it_n^j\Delta}\psi^j(x-x_n^j)\cdot\overline{e^{-it_n^k\Delta}\phi^k(x-x_n^k)}\cdot\overline{e^{-it_n^l\Delta}\phi^l(x-x_n^l)}dx\right|\\
&\hspace{1cm}\leq\left\|\psi^j(\cdot-x_n^j)\cdot\phi^k(\cdot-x_n^k)\cdot e^{-it_n^l\Delta}\phi^l(\cdot-x_n^l)\right\|_{L^1}\\
&\hspace{1cm}\leq\left\|(\psi^j-g)(\cdot-x_n^j)\cdot\phi^k(\cdot-x_n^k)\cdot e^{-it_n^l\Delta}\phi^l(\cdot-x_n^l)\right\|_{L^1}\\
&\hspace{7cm}+\left\|g(\cdot-x_n^j)\cdot\left(\phi^k-f\right)\left(\cdot-x_n^k\right)\cdot e^{-it_n^l\Delta}\phi^l(\cdot-x_n^l)\right\|_{L^1}\\
&\hspace{7cm}+\left\|g(\cdot-x_n^j)\cdot f(\cdot-x_n^k)\cdot e^{-it_n^l\Delta}\phi^l(\cdot-x_n^l)\right\|_{L^1}\\
&\hspace{1cm}\leq\left\|\psi^j-g\right\|_{L^\frac{5}{2}}\left\|\phi^k\right\|_{L^\frac{10}{3}}\left\|e^{-it_n^l\Delta}\phi^l\right\|_{L^\frac{10}{3}}+\left\|g\right\|_{L^\frac{5}{2}}\left\|\phi^k-f\right\|_{L^\frac{10}{3}}\left\|e^{-it_n^l\Delta}\phi^l\right\|_{L^\frac{10}{3}}\\
&\hspace{7cm}+\left\|g\cdot f(\cdot+x_n^j-x_n^k)\right\|_{L^2}\left\|\phi^l\right\|_{L^2}\\
&\hspace{1cm}\leq\left\|\psi^j-g\right\|_{\dot{H}^\frac{1}{2}}\left\|\phi^k\right\|_{\dot{H}^1}\left\|\phi^l\right\|_{\dot{H}^1}+\left\|g\right\|_{\dot{H}^\frac{1}{2}}\left\|\phi^k-f\right\|_{\dot{H}^1}\left\|\phi^l\right\|_{\dot{H}^1}+\left\|g\cdot f(\cdot+x_n^j-x_n^k)\right\|_{L^2}\left\|\phi^l\right\|_{\dot{H}^1}\\
&\hspace{1cm}<\varepsilon.
\end{align*}
The term including $e^{-it_n^j\Delta},\,e^{-it_n^k\Delta}\ (j\neq k)$ converges on 0 by the same argument. Also, since
\begin{align*}
\left|\int_{\mathbb{R}^5}\Psi_n^{M_1}\overline{\Phi_n^{M_1}}\sum_{j=M+1}^{M_1}\overline{e^{-it_n^j\Delta}\phi^j(\cdot-x_n^j)}dx\right|&\leq\left\|\Psi_n^{M_1}\right\|_{L^\frac{5}{2}}\left\|\Phi_n^{M_1}\right\|_{L^\frac{10}{3}}\left\|\sum_{j=M+1}^{M_1}e^{-it_n^j\Delta}\phi^j\right\|_{L^\frac{10}{3}}\\
&\leq c\left\|e^{\frac{1}{2}it\Delta}\Psi_n^{M_1}\right\|_{L^\infty L^\frac{5}{2}}\left\|\Phi_n^{M_1}\right\|_{\dot{H}^1}\left\|\Phi_n^M-\Phi_n^{M_1}\right\|_{\dot{H}^1}\\
&\leq c\left\|e^{\frac{1}{2}it\Delta}\Psi_n^{M_1}\right\|_{S(\dot{H}^\frac{1}{2})}\left\|\Phi_n^{M_1}\right\|_{\dot{H}^1}\left(\|\Phi_n^M\|_{\dot{H}^1}+\|\Phi_n^{M_1}\|_{\dot{H}^1}\right)\\
&\longrightarrow0\ \ \text{as}\ \ M_1,n\rightarrow\infty
\end{align*}
and
\[
\left|\int_{\mathbb{R}^5}\Psi_n^{M_1}\sum_{j=M+1}^{M_1}\overline{e^{-it_n^j\Delta}\phi^j(\cdot-x_n^j)}^2dx\right|\leq c\left\|e^{-\frac{1}{2}it\Delta}\Psi_n^{M_1}\right\|_{S(\dot{H}^\frac{1}{2})}\sum_{j=1}^{\infty}\left\|\phi^j\right\|_{\dot{H}^1}^2\longrightarrow0\ \ \text{as}\ \ M_1,n\rightarrow\infty,
\]
it follows that
\[
\sum_{j=M+1}^{M_1}\int_{\mathbb{R}^5}e^{-\frac{1}{2}it_n^j\Delta}\psi^j\left(\overline{e^{-it_n^j\Delta}\phi^j}\right)^2dx=\int_{\mathbb{R}^5}\left(\Psi_n^M\overline{\Phi_n^M}^2-\Psi_n^{M_1}\overline{\Phi_n^{M_1}}^2\right)dx+o_{M_1,n}(1).
\]
For second identity,
\begin{align*}
&\int_{\mathbb{R}^5}\left(\psi_n\overline{\phi_n}^2-\Psi_n^{M_1}\overline{\Phi_n^{M_1}}^2\right)dx\\
&\hspace{0.2cm}=\int_{\mathbb{R}^5}\left\{\left(\Psi_n^{M_1}+\sum_{j=1}^{M_1}e^{-\frac{1}{2}it_n^j\Delta}\psi^j(\cdot-x_n^j)\right)\left(\overline{\Phi_n^{M_1}}+\sum_{j=1}^{M_1}\overline{e^{-it_n^j\Delta}\phi^j(\cdot-x_n^j)}\right)^2-\Psi_n^{M_1}\overline{\Phi_n^{M_1}}^2\right\}dx\\
&\hspace{0.2cm}=\int_{\mathbb{R}^5}\left\{\left(\Psi_n^{M_1}+\sum_{j=1}^{M_1}e^{-\frac{1}{2}it_n^j\Delta}\psi^j(\cdot-x_n^j)\right)\left(\overline{\Phi_n^{M_1}}^2+2\overline{\Phi_n^{M_1}}\sum_{j=1}^{M_1}\overline{e^{-it_n^j\Delta}\phi^j(\cdot-x_n^j)}\right.\right.\\
&\hspace{0.4cm}\left.\left.+\sum_{j=1}^{M_1}\overline{e^{-it_n^j\Delta}\phi^j(\cdot-x_n^j)}^2+\sum_{\substack{j,k=1\\j\neq k}}^{M_1}\overline{e^{-it_n^j\Delta}\phi^j(\cdot-x_n^j)}\cdot\overline{e^{-it_n^k\Delta}\phi^k(\cdot-x_n^k)}\right)-\Psi_n^{M_1}\overline{\Phi_n^{M_1}}^2\right\}dx\\
&\hspace{0.2cm}=\int_{\mathbb{R}^5}\left(2\Psi_n^{M_1}\overline{\Phi_n^{M_1}}\sum_{j=1}^{M_1}\overline{e^{-it_n^j\Delta}\phi^j(\cdot-x_n^j)}+\Psi_n^{M_1}\sum_{j=1}^{M_1}\overline{e^{-it_n^j\Delta}\phi^j(\cdot-x_n^j)}^2\right.\\
&\hspace{0.4cm}+\Psi_n^{M_1}\sum_{\substack{j,k=1\\j\neq k}}^{M_1}\overline{e^{-it_n^j\Delta}\phi^j(\cdot-x_n^j)}\cdot\overline{e^{-it_n^k\Delta}\phi^k(\cdot-x_n^k)}+\overline{\Phi_n^{M_1}}^2\sum_{j=1}^{M_1}e^{-\frac{1}{2}it_n^j\Delta}\psi^j(\cdot-x_n^j)\\
&\hspace{0.4cm}+2\overline{\Phi_n^{M_1}}\sum_{j=1}^{M_1}e^{-\frac{1}{2}it_n^j\Delta}\psi^j(\cdot-x_n^j)\sum_{j=1}^{M_1}\overline{e^{-it_n^j\Delta}\phi^j(\cdot-x_n^j)}+\sum_{j=1}^{M_1}e^{-\frac{1}{2}it_n^j\Delta}\psi^j(\cdot-x_n^j)\sum_{j=1}^{M_1}\overline{e^{-it_n^j\Delta}\phi^j(\cdot-x_n^j)}^2\\
&\hspace{0.4cm}\left.+\sum_{j=1}^{M_1}e^{-\frac{1}{2}it_n^j\Delta}\psi^j(\cdot-x_n^j)\sum_{\substack{j,k=1\\j\neq k}}^{M_1}\overline{e^{-it_n^j\Delta}\phi^j(\cdot-x_n^j)}\cdot\overline{e^{-it_n^k\Delta}\phi^k(\cdot-x_n^k)}\right)dx.
\end{align*}
From the same argument as the proof of the third equation, we have
\[
\int_{\mathbb{R}^5}\left(\psi_n\overline{\phi_n}^2-\Psi_n^{M_1}\overline{\Phi_n^{M_1}}^2\right)dx=\sum_{j=1}^{M_1}\int_{\mathbb{R}^5}e^{-\frac{1}{2}it_n^j\Delta}\psi^j\left(\overline{e^{-it_n^j\Delta}\phi^j}\right)^2dx+o_{M_1,n}(1).
\]
Applying these equations,
\begin{align*}
&\text{Re}(\psi_n,\phi_n^2)_{L^2}=\text{Re}\int_{\mathbb{R}^5}\psi_n\overline{\phi_n}^2dx\\
&\hspace{1cm}=\text{Re}\int_{\mathbb{R}^5}\left(\psi_n\overline{\phi_n}^2-\Psi_n^{M_1}\overline{\Phi_n^{M_1}}^2\right)dx+o_{M_1,n}(1)\\
&\hspace{1cm}=\sum_{j=1}^{M_1}\text{Re}\int_{\mathbb{R}^5}e^{-\frac{1}{2}it_n^j\Delta}\psi^j\left(\overline{e^{-it_n^j\Delta}\phi^j}\right)^2dx+o_{M_1,n}(1)\\
&\hspace{1cm}=\sum_{j=1}^{M}\text{Re}\int_{\mathbb{R}^5}e^{-\frac{1}{2}it_n^j\Delta}\psi^j\left(\overline{e^{-it_n^j\Delta}\phi^j}\right)^2dx+\sum_{j=M+1}^{M_1}\text{Re}\int_{\mathbb{R}^5}e^{-\frac{1}{2}it_n^j\Delta}\psi^j\left(\overline{e^{-it_n^j\Delta}\phi^j}\right)^2dx+o_{M_1,n}(1)\\
&\hspace{1cm}=\sum_{j=1}^{M}\text{Re}\int_{\mathbb{R}^5}e^{-\frac{1}{2}it_n^j\Delta}\psi^j\left(\overline{e^{-it_n^j\Delta}\phi^j}\right)^2dx+\text{Re}\int_{\mathbb{R}^5}\left(\Psi_n^M\overline{\Phi_n^M}^2-\Psi_n^{M_1}\overline{\Phi_n^{M_1}}^2\right)dx+o_{M_1,n}(1)\\
&\hspace{1cm}=\sum_{j=1}^{M}\text{Re}\int_{\mathbb{R}^5}e^{-\frac{1}{2}it_n^j\Delta}\psi^j\left(\overline{e^{-it_n^j\Delta}\phi^j}\right)^2dx+\text{Re}\int_{\mathbb{R}^5}\Psi_n^M\overline{\Phi_n^M}^2dx+o_{M_1,n}(1)\\
&\hspace{1cm}=\sum_{j=1}^{M}\text{Re}\left(e^{-\frac{1}{2}it_n^j\Delta}\psi^j,\left(e^{-it_n^j\Delta}\phi^j\right)^2\right)_{L^2}+\text{Re}\left(\Psi_n^M,\left(\Phi_n^M\right)^2\right)_{L^2}+o_{M_1,n}(1).
\end{align*}
Therefore, we obtain
\[
\text{Re}(\psi_n,\phi_n^2)_{L^2}=\sum_{j=1}^{M}\text{Re}\left(e^{-\frac{1}{2}it_n^j\Delta}\psi^j,\left(e^{-it_n^j\Delta}\phi^j\right)^2\right)_{L^2}+\text{Re}\left(\Psi_n^M,\left(\Phi_n^M\right)^2\right)_{L^2}+o_{n}(1).
\]
\end{proof}


\begin{corollary}\label{K decomposition}
Under the same assumption as Theorem \ref{Linear profile decomposition}, we have
\[
K_\omega^{20,8}(\phi_n,\psi_n)=\sum_{j=1}^MK_\omega^{20,8}(e^{-it_n^j\Delta}\phi^j,e^{-\frac{1}{2}it_n^j\Delta}\psi^j)+K_\omega^{20,8}(\Phi_n^M,\Psi_n^M)+o_{n}(1).
\]
\end{corollary}

\begin{proof}
This follows from the proof of Corollary \ref{I decomposition}.
\end{proof}


\begin{lemma}\label{lemma for critical solution}
Let $M\in\mathbb{N}$. We assume that $(\phi^j,\psi^j)\in H^1\!\times\!H^1{\setminus}\{(0,0)\}$ ($j\in\{1,\cdots,M\}$) satisfy
\[
I_\omega\left(\sum_{j=1}^M(\phi^j,\psi^j)\right)\leq I_\omega(\phi_\omega,\psi_\omega)-\delta\,,\ \ I_\omega\left(\sum_{j=1}^M(\phi^j,\psi^j)\right)\geq\sum_{j=1}^MI_\omega(\phi^j,\psi^j)-\varepsilon,
\]
\[
K_\omega^{20,8}\left(\sum_{j=1}^M(\phi^j,\psi^j)\right)\geq-\varepsilon\,,\ \ K_\omega^{20,8}\left(\sum_{j=1}^M(\phi^j,\psi^j)\right)\leq\sum_{j=1}^MK_\omega^{20,8}(\phi^j,\psi^j)+\varepsilon.
\]
for $\delta,\,\varepsilon>0$ with $2\varepsilon<\delta$. Then, we have
\[
0<I_\omega(\phi^j,\psi^j)<I_\omega(\phi_\omega,\psi_\omega)\,,\ \ K_\omega^{20,8}(\phi^j,\psi^j)>0
\]
for any $j\in\{1,\cdots,M\}$.
\end{lemma}


\begin{proof}
We assume that there exists $i\in\{1,\cdots,M\}$ such that $K_\omega^{20,8}(\phi^i,\psi^i)\leq0$ and deduce contradiction. From Lemma \ref{lemma of inferior},
\begin{align*}
I_\omega(\phi_\omega,\psi_\omega)=\mu_\omega^{20,8}&\leq\frac{\omega}{2}M(\phi^i,\psi^i)+\frac{1}{10}K(\phi^i,\psi^i)\\
&\leq\sum_{j=1}^M\left(\frac{\omega}{2}M(\phi^j,\psi^j)+\frac{1}{10}K(\phi^j,\psi^j)\right)\\
&=\sum_{j=1}^M\left(I_\omega(\phi^j,\psi^j)-\frac{1}{20}K_\omega^{20,8}(\phi^j,\psi^j)\right)\\
&=\sum_{j=1}^MI_\omega(\phi^j,\psi^j)-\frac{1}{20}\sum_{j=1}^MK_\omega^{20,8}(\phi^j,\psi^j)\\
&\leq I_\omega\left(\sum_{j=1}^M(\phi^j,\psi^j)\right)+\varepsilon-\frac{1}{20}\left(K_\omega^{20,8}\left(\sum_{j=1}^M(\phi^j,\psi^j)\right)-\varepsilon\right)\\
&\leq I_\omega(\phi_\omega,\psi_\omega)-\delta+\varepsilon+\frac{1}{10}\varepsilon<I_\omega(\phi_\omega,\psi_\omega).
\end{align*}
This is contradiction. Thus, $K_\omega^{20,8}(\phi^j,\psi^j)>0$ for any $j\in\{1,\cdots,M\}$. Also,
\begin{align*}
I_\omega(\phi^j,\psi^j)&=\frac{\omega}{2}M(\phi^j,\psi^j)+\frac{1}{2}K(\phi^j,\psi^j)-P(\phi^j,\psi^j)\\
&>\frac{\omega}{2}M(\phi^j,\psi^j)+\frac{1}{2}K(\phi^j,\psi^j)-\frac{2}{5}K(\phi^j,\psi^j)\\
&=\frac{\omega}{2}M(\phi^j,\psi^j)+\frac{1}{10}K(\phi^j,\psi^j)\geq0,
\end{align*}
and
\[
I_\omega(\phi^j,\psi^j)\leq\sum_{j=1}^MI_\omega(\phi^j,\psi^j)\leq I_\omega\left(\sum_{j=1}^M(\phi^j,\psi^j)\right)+\varepsilon\leq I_\omega(\phi_\omega,\psi_\omega)-\delta+\varepsilon<I_\omega(\phi_\omega,\psi_\omega).
\]
\end{proof}

\section{Scattering}
\subsection{Existence of a critical solution}

\begin{definition}\label{I omega c}
Let $(u_0,v_0)\in H^1\!\times\!H^1$ and $(u,v)$ be the $H^1\!\times\! H^1$ solution to (NLS). We say that $SC(u_0,v_0)$ holds if $T^\ast=\infty$ and $\|(u,v)\|_{S(\dot{H}^\frac{1}{2})\times S(\dot{H}^\frac{1}{2})}<\infty$.\\
We define $I_\omega^c=\sup\left\{\delta>0:\text{If }K_\omega^{20,8}(u,v)>0\text{ and }I_\omega(u,v)<\delta,\text{ then }SC(u_0,v_0)\text{ holds}.\right\}$.
\end{definition}

If $I_\omega^c\geq I_\omega(\phi_\omega,\psi_\omega)$, then Theorem \ref{Main theorem 1} (1) holds. Therefore, we proceed with the proof of Theorem \ref{Main theorem 1} (1) by assuming that $I_\omega^c<I_\omega(\phi_\omega,\psi_\omega)$ and ultimately deduce contradiction.

\begin{remark}
A set $\left\{\delta>0:\text{If }K_\omega^{20,8}(u,v)>0\text{ and }I_\omega(u,v)<\delta,\text{ then }SC(u_0,v_0)\text{ holds}.\right\}$ is not empty. In fact, If $(u_0,v_0)$ satisfies $K_\omega^{20,8}(u_0,v_0)>0$ and $I_\omega(u_0,v_0)<\delta<I_\omega(\phi_\omega,\psi_\omega)$ for sufficiently small $\delta>0$, then we obtain $T^\ast=\infty$ by Theorem \ref{Global versus blow-up dichotomy}. Also, it follows that $K_\omega(u_0,v_0)<10I_\omega(u_0,v_0)<10\delta$ by Theorem \ref{Comparability of K and I}. i.e. $\|(u_0,v_0)\|_{H^1\times H^1}<c\delta$ holds. Applying Theorem \ref{Small data globally existence}, we have $\|(u,v)\|_{S(\dot{H}^\frac{1}{2})\times S(\dot{H^\frac{1}{2}})}<\infty$ and hence, $SC(u_0,v_0)$ holds.
\end{remark}


\begin{lemma}[Existence of wave operators]\label{Existence of wave operators}
Suppose $(\phi,\psi)\in H^1\!\times\!H^1\setminus\{(0,0)\}$ and
\[
\frac{\omega}{2}M(\phi,\psi)+\frac{1}{2}K(\phi,\psi)<I_\omega(\phi_\omega,\psi_\omega).
\]
Then there exists $(u_0,v_0)\in H^1\!\times\!H^1$ such that $(u,v)$ solving (NLS) with initial data $(u_0,v_0)$ is time-global in $H^1\!\times\!H^1$ with
\[
I_\omega(u_0,v_0)<I_\omega(\phi_\omega,\psi_\omega),\ \ \ K_\omega^{20,8}(u_0,v_0)>0\,,\ \ \ \|(u,v)\|_{S(\dot{H}^\frac{1}{2})\times S(\dot{H}^\frac{1}{2})}\leq4\|(e^{it\Delta}\phi,e^{\frac{1}{2}it\Delta}\psi)\|_{S(\dot{H}^\frac{1}{2})\times S(\dot{H}^\frac{1}{2})},
\]
and
\[
\lim_{t\rightarrow\infty}\|(u,v)(t)-(e^{it\Delta}\phi,e^{\frac{1}{2}it\Delta}\psi)\|_{H^1\times H^1}=0.
\]
\end{lemma}


\begin{proof}
By $(\phi,\psi)\in H^1\!\times\!H^1$, there exists sufficiently large $T>0$ such that
\[
\|(e^{it\Delta}\phi,e^{\frac{1}{2}it\Delta}\psi)\|_{S(\dot{H}^\frac{1}{2}:[T,\infty))\times S(\dot{H}^\frac{1}{2}:[T,\infty))}\leq\delta_{sd}.
\]
We consider a integral equation:
\begin{equation}
\notag
\begin{cases}
\hspace{-0.4cm}&\displaystyle{u(t)=e^{it\Delta}\phi+2i\int_t^\infty e^{i(t-s)\Delta}(v\overline{u})(s)ds,}\\[0.3cm]
\hspace{-0.4cm}&\displaystyle{v(t)=e^{\frac{1}{2}it\Delta}\psi+i\int_t^\infty e^{\frac{1}{2}i(t-s)\Delta}(u^2)(s)ds.}
\end{cases}
\end{equation}
We define a set
\[
E=\left\{(u,v):\|(u,v)\|_{S(\dot{H}^\frac{1}{2}:[T,\infty))\times S(\dot{H}^\frac{1}{2}:[T,\infty))}\leq4\|(e^{it\Delta}\phi,e^{\frac{1}{2}it\Delta}\psi)\|_{S(\dot{H}^\frac{1}{2}:[T,\infty))\times S(\dot{H}^\frac{1}{2}:[T,\infty))}\right\}
\]
and a distance $d((u_1,v_1),\,(u_2,v_2))$ on $E$
\[
d(u,v)=\|(u_1,v_1)-(u_2,v_2)\|_{S(\dot{H}^\frac{1}{2}:[T,\infty))\times S(\dot{H}^\frac{1}{2}:[T,\infty))}
\]
for $(u_1,v_1),\,(u_2,v_2)\in E$. By the same argument as the proof for Theorem \ref{Small data globally existence}, there exists the unique solution on $E$. Also,
\begin{align}
\|u(t)-e^{it\Delta}\phi\|_{H^1}&=2\left\|\int_t^\infty e^{i(t-s)\Delta}(v\overline{u})(s)ds\right\|_{H^1} \notag \\
&\leq2c\|vu\|_{{L_{[T,\infty)}^\frac{12}{7}}W^{1,\frac{3}{2}}} \notag \\
&\leq2c\|v\|_{L_{[T,\infty)}^\frac{12}{5}W^{1,3}}\|u\|_{L_{[T,\infty)}^6L^3}+2c\|v\|_{L_{[T,\infty)}^6L^3}\|u\|_{L_{[T,\infty)}^\frac{12}{5}W^{1,3}}.\label{38}
\end{align}
Here, we observe a boundness of $\|v\|_{L_{[T,\infty)}^\frac{12}{5}W^{1,3}},\,\|u\|_{L_{[T,\infty)}^\frac{12}{5}W^{1,3}}$. Since $\|(u,v)\|_{S(\dot{H}^\frac{1}{2}:[T,\infty))\times S(\dot{H}^\frac{1}{2}:[T,\infty))}\leq\delta_{sd}$, we may take pairwise disjoint sets $I_j=[t_j,t_{j+1})$ $(j=0,\cdots,J<\infty)$ with $t_0=T$, $t_{J+1}=\infty$ and 
\[
[\,T,\infty)=\bigcup_{j=0}^JI_j\,,\ \ \|(u,v)\|_{S(\dot{H}^\frac{1}{2}:I_j)\times S(\dot{H}^\frac{1}{2}:I_j)}<\frac{1}{8c}.
\]
As in \eqref{38}, calculating the integral equation in the interval $I_j$,
\begin{align*}
\|u\|_{L_{I_j}^\frac{12}{5}W^{1,3}}&\leq c\|u(t_j)\|_{H^1}+2c\|v\|_{L_{I_j}^\frac{12}{5}W^{1,3}}\|u\|_{L_{I_j}^6L^3}+2c\|v\|_{L_{I_j}^6L^3}\|u\|_{L_{I_j}^\frac{12}{5}W^{1,3}}\\
&\leq c\|u(t_j)\|_{H^1}+\frac{1}{4}\|v\|_{L_{I_j}^\frac{12}{5}W^{1,3}}+\frac{1}{4}\|u\|_{L_{I_j}^\frac{12}{5}W^{1,3}}.
\end{align*}
Similarly,
\[
\|v\|_{L_{I_j}^\frac{12}{5}W^{1,3}}\leq c\|v(t_j)\|_{H^1}+\frac{1}{4}\|u\|_{L_{I_j}^\frac{12}{5}W^{1,3}}.
\]
We add these equalities, then
\[
\|(u,v)\|_{L_{I_j}^\frac{12}{5}W^{1,3}\times L_{I_j}^\frac{12}{5}W^{1,3}}\leq c\|(u(t_j),v(t_j))\|_{H^1\times H^1}+\frac{1}{2}\|(u,v)\|_{L_{I_j}^\frac{12}{5}W^{1,3}\times L_{I_j}^\frac{12}{5}W^{1,3}},
\]
i.e.
\[
\|(u,v)\|_{L_{I_j}^\frac{12}{5}W^{1,3}\times L_{I_j}^\frac{12}{5}W^{1,3}}\leq 2c\|(u(t_j),v(t_j))\|_{H^1\times H^1}.
\]
Also,
\[
\|u\|_{L_{I_j}^\infty H^1}\leq c\|u(t_j)\|_{H^1}+\frac{1}{4}\|(u,v)\|_{L_{I_j}^\frac{12}{5}W^{1,3}\times L_{I_j}^\frac{12}{5}W^{1,3}}.
\]
Similarly,
\[
\|v\|_{L_{I_j}^\infty H^1}\leq c\|v(t_j)\|_{H^1}+\frac{1}{4}\|(u,v)\|_{L_{I_j}^\frac{12}{5}W^{1,3}\times L_{I_j}^\frac{12}{5}W^{1,3}}.
\]
Thus,
\[
\|(u,v)\|_{L_{I_j}^\infty H^1\times L_{I_j}^\infty H^1}\leq c\|(u(t_j),v(t_j))\|_{H^1\times H^1}+\frac{1}{2}\|(u,v)\|_{L_{I_j}^\frac{12}{5}W^{1,3}\times L_{I_j}^\frac{12}{5}W^{1,3}}\leq2c\|(u(t_j),v(t_j))\|_{H^1\times H^1}.
\]
Hence,
\[
\|(u(t_j),v(t_j))\|_{H^1\times H^1}\leq2c\|(u(t_{j-1}),v(t_{j-1}))\|_{H^1\times H^1}\leq\cdots\leq(2c)^j\|u(t_0),v(t_0)\|_{H^1\times H^1}.
\]
Therefore,
\begin{align*}
\|(u,v)\|_{L_{[T,\infty)}^\frac{12}{5}W^{1,3}\times L_{[T,\infty)}^\frac{12}{5}W^{1,3}}&\leq\sum_{j=0}^J\|(u,v)\|_{L_{I_j}^\frac{12}{5}W^{1,3}\times L_{I_j}^\frac{12}{5}W^{1,3}}\\
&\leq\sum_{j=0}^J2c\|(u(t_j),v(t_j))\|_{H^1\times H^1}\\
&\leq\sum_{j=0}^J(2c)^j\|(u(T),v(T))\|_{H^1\times H^1}<\infty.
\end{align*}
Also, since $\|(u,v)\|_{L_{[T,\infty)}^6L^3\times L_{[T,\infty)}^6L^3}\leq\|(u,v)\|_{S(\dot{H}^\frac{1}{2}:[T,\infty))\times S(\dot{H}^\frac{1}{2}:[T,\infty))}<\infty$, we obtain
\[
\|u(t)-e^{it\Delta}\phi\|_{H^1}\longrightarrow0\ \ \text{as}\ \ t\rightarrow\infty.
\]
Similarly,
\[
\|v(t)-e^{\frac{1}{2}it\Delta}\psi\|_{H^1}\longrightarrow0\ \ \text{as}\ \ t\rightarrow\infty.
\]
Combining these formulas,
\[
\lim_{t\rightarrow\infty}\|(u(t),v(t))\|_{H^1\times H^1}=\|(\phi,\psi)\|_{H^1\times H^1}.
\]
Also,
\begin{align*}
|P(u,v)|&\leq|P(u,v)-P(e^{it\Delta}\phi,e^{\frac{1}{2}it\Delta}\psi)|+|P(e^{it\Delta}\phi,e^{\frac{1}{2}it\Delta}\psi)|\\
&=\left|\int_{\mathbb{R}^5}v\overline{u}^2-e^{\frac{1}{2}it\Delta}\psi(\overline{e^{it\Delta}\phi})^2dx\right|+\left|\int_{\mathbb{R}^5}e^{\frac{1}{2}it\Delta}\psi(\overline{e^{it\Delta}\phi})^2dx\right|\\
&\leq\|v\overline{u}^2-e^{\frac{1}{2}it\Delta}\psi(\overline{e^{it\Delta}\phi})^2\|_{L^1}+\|e^{\frac{1}{2}it\Delta}\psi(\overline{e^{it\Delta}\phi})^2\|_{L^1}\\
&\leq\|v(\overline{u}^2-(\overline{e^{it\Delta}\phi})^2)\|_{L^1}+\|(v-e^{\frac{1}{2}it\Delta}\psi)(\overline{e^{it\Delta}\phi})^2\|_{L^1}+\|e^{\frac{1}{2}it\Delta}\psi(\overline{e^{it\Delta}\phi})^2\|_{L^1}\\
&\leq\|v\|_{L^3}\|u+e^{it\Delta}\phi\|_{L^3}\|u-e^{it\Delta}\phi\|_{L^3}+\|v-e^{\frac{1}{2}it\Delta}\psi\|_{L^3}\|e^{it\Delta}\phi\|_{L^3}^2+\|e^{\frac{1}{2}it\Delta}\psi\|_{L^3}\|e^{it\Delta}\phi\|_{L^3}^2\\
&\leq c\|v\|_{H^1}(\|u\|_{H^1}+\|\phi\|_{H^1})\|u-e^{it\Delta}\phi\|_{H^1}+c\|v-e^{\frac{1}{2}it\Delta}\psi\|_{H^1}\|\phi\|_{H^1}^2+\|e^{\frac{1}{2}it\Delta}\psi\|_{L^3}\|e^{it\Delta}\phi\|_{L^3}^2.
\end{align*}
Here, since $C_0^\infty$ is dense in $H^1$, for any $\varepsilon>0$, there exists $f,\,g\in C_0^\infty$ such that
\[
\|\phi-f\|_{H^1}<\varepsilon\,,\ \ \ \|\psi-g\|_{H^1}<\varepsilon
\]
and hence,
\begin{align*}
\|e^{\frac{1}{2}it\Delta}\psi\|_{L^3}\|e^{it\Delta}\phi\|_{L^3}^2&\leq(\|e^{\frac{1}{2}it\Delta}(\psi-g)\|_{L^3}+\|e^{\frac{1}{2}it\Delta}g\|_{L^3})(\|e^{it\Delta}(\phi-f)\|_{L^3}+\|e^{it\Delta}f\|_{L^3})^2\\
&\leq(c\|\psi-g\|_{H^1}+c|t|^{-\frac{5}{6}}\|g\|_{L^\frac{3}{2}})(c\|\phi-f\|_{H^1}+c|t|^{-\frac{5}{6}}\|f\|_{L^\frac{3}{2}})^2.
\end{align*}
Thus,
\[
P(u,v)\longrightarrow0\ \ \text{as}\ \ t\rightarrow\infty,
\]
then
\[
\lim_{t\rightarrow\infty}I_\omega(u(t),v(t))=\frac{\omega}{2}M(\phi,\psi)+\frac{1}{2}K(\phi,\psi)<I_\omega(\phi_\omega,\psi_\omega)\,,\ \ \lim_{t\rightarrow\infty}K_\omega^{20,8}(u(t),v(t))=8K(\phi,\psi)>0.
\]
For sufficiently large $t>0$,
\[
I_\omega(u(t),v(t))<I_\omega(\phi_\omega,\psi_\omega)\,,\ \ K_\omega^{20,8}(u(t),v(t))>0.
\]
If we solve (NLS) with a initial data at this time, then the corresponding solution to (NLS) exists time-globally by Theorem \ref{Global versus blow-up dichotomy}. Also, from \eqref{01},\,\eqref{02}, and Lemma \ref{estimates for K},
\[
I_\omega(u_0,v_0)<I_\omega(\phi_\omega,\psi_\omega)\,,\ \ K_\omega^{20,8}(u_0,v_0)>0.
\]
Furthermore, we have
\[
\|(u,v)\|_{S(\dot{H}^\frac{1}{2})\times S(\dot{H}^\frac{1}{2})}\leq4\|(e^{it\Delta}\phi,e^{\frac{1}{2}it\Delta}\psi)\|_{S(\dot{H}^\frac{1}{2})\times S(\dot{H}^\frac{1}{2})}.
\]
\end{proof}

\begin{lemma}\label{lemma for critical solution 1}
For any $P>1$ and $l\geq2$, there exists $C_{P,l}>0$ such that for any $\{z_j\}_{1\leq j\leq l}\subset\mathbb{C}$, we have
\begin{align}
\left|\left|\sum_{j=1}^lz_j\right|^P-\sum_{j=1}^l|z_j|^P\right|\leq C_{P,l}\sum_{1\leq j\neq k\leq l}|z_j||z_k|^{P-1}.\label{39}
\end{align}
\end{lemma}

For the convenient of the reader, we give a proof of this lemma.

\begin{proof}
We prove by induction with respect to $l\geq2$.\\
In the case $l=2$, we assume that $|z_1|\geq|z_2|$ without loss of generality.
\begin{align*}
\left||z_1+z_2|^P-|z_1|^P-|z_2|^P\right|&\leq\left||z_1+z_2|^P-|z_1|^P\right|+|z_2|^P\\
&\leq C_P\left(|z_1+z_2|^{P-1}+|z_1|^{P-1}\right)|z_2|+|z_2|^P\\
&\leq C_P(|z_1|+|z_2|)^{P-1}|z_2|+2|z_1|^{P-1}|z_2|\\
&\leq C_P(2^{P-1}+2)|z_1|^{P-1}|z_2|\\
&\leq C_P(|z_1|^{P-1}|z_2|+|z_1||z_2|^{P-1})
\end{align*}
Thus, Theorem \ref{lemma for critical solution 1} holds in $l=2$.\\
We assume that Theorem \ref {lemma for critical solution 1} holds in $l-1$, i.e.
\[
\left|\left|\sum_{j=1}^{l-1}z_j\right|^P-\sum_{j=1}^{l-1}|z_j|^P\right|\leq C_{P,l}\sum_{1\leq j\neq k\leq l-1}|z_j||z_k|^{P-1}.
\]
for $l\geq3$. Also, we assume that $\displaystyle|z_1|=\max_{1\leq j\leq l}|z_j|$ without loss of generality.
\begin{align*}
\left|\left|\sum_{j=1}^lz_j\right|^P-\sum_{j=1}^l|z_j|^P\right|&=\left|\left|\sum_{j=1}^lz_j\right|^P-\left|\sum_{j=1}^{l-1}z_j\right|^P-|z_l|^P+\left|\sum_{j=1}^{l-1}z_j\right|^P-\sum_{j=1}^{l-1}|z_j|^P\right|\\
&\leq\left|\left|\sum_{j=1}^{l-1}z_j+z_l\right|^P-\left|\sum_{j=1}^{l-1}z_j\right|^P-|z_l|^P\right|+\left|\left|\sum_{j=1}^{l-1}z_j\right|^P-\sum_{j=1}^{l-1}|z_j|^P\right|\\
&\leq C_P\left(\left|\sum_{j=1}^{l-1}z_j\right|^{P-1}|z_l|+\left|\sum_{j=1}^{l-1}z_j\right||z_l|^{P-1}\right)+C_{P,l}\sum_{1\leq j\neq k\leq l-1}|z_j||z_k|^{P-1}\\
&\leq C_P\left\{(l-1)^{P-1}|z_1|^{P-1}|z_l|+\sum_{j=1}^{l-1}|z_j||z_l|^{P-1}\right\}+C_{P,l}\sum_{1\leq j\neq k\leq l-1}|z_j||z_k|^{P-1}\\
&=C_{P,l}\sum_{1\leq j\neq k\leq l}|z_j||z_k|^{P-1}
\end{align*}
Therefore, Theorem \ref{lemma for critical solution 1} also holds in $l$.
\end{proof}

\begin{lemma}\label{lemma for critical solution 2}
There exists $0<\delta_{sd}\leq1$ such that if $\|(u_0,v_0)\|_{H^1\times H^1}\leq\delta_{sd}$, then the unique solution $(u(t),v(t))$ to (NLS) exists time-globally and
\begin{align*}
&\|(u,v)\|_{L^\frac{14}{5}L^\frac{14}{5}\times L^\frac{14}{5}L^\frac{14}{5}}+\|(u,v)\|_{L^6L^3\times L^6L^3}+\|(u,v)\|_{L^\infty H^1\times L^\infty H^1}+\|(u,v)\|_{L^2W^{1,\frac{10}{3}}\times L^2W^{1,\frac{10}{3}}}\\
&\hspace{12cm}\leq 8c\|(u_0,v_0)\|_{H^1\times H^1}.
\end{align*}
\end{lemma}

\begin{remark}
This theorem is different from Theorem \ref{Small data globally existence} in the point to be able to estimate Strichartz norms for $L^2$ admissible $(\frac{14}{5}, \frac{14}{5})$ and $(\infty,2)$ by $H^1$ norm. We will use this theorem in the next theorem.
\end{remark}

\begin{proof}
We define a notation
\begin{align*}
\|(u,v)\|_{X\times X}&=\|(u,v)\|_{L^\frac{14}{5}L^\frac{14}{5}\times L^\frac{14}{5}L^\frac{14}{5}\cap L^6L^3\times L^6L^3\cap L^\infty H^1\times L^\infty H^1\cap L^2W^{1,\frac{10}{3}}\times L^2W^{1,\frac{10}{3}}}\\
&=\|(u,v)\|_{L^\frac{14}{5}L^\frac{14}{5}\times L^\frac{14}{5}L^\frac{14}{5}}+\|(u,v)\|_{L^6L^3\times L^6L^3}\\
&\hspace{4cm}+\|(u,v)\|_{L^\infty H^1\times L^\infty H^1}+\|(u,v)\|_{L^2W^{1,\frac{10}{3}}\times L^2W^{1,\frac{10}{3}}},
\end{align*}
a set
\[
E=\{(u,v):\|(u,v)\|_{X\times X}\leq 8c\|(u_0,v_0)\|_{H^1\times H^1}\},
\]
and a distance on $E$
\[
d((u_1,v_1),(u_2,v_2))=\|(u_1,v_1)-(u_2,v_2)\|_{X\times X}
\]
for $(u_1,v_1)$, $(u_2,v_2)$ and a map on $E$
\[
\Phi_{u_0}(u,v)(t)=e^{it\Delta}u_0+2i\int_0^t e^{i(t-s)\Delta}(v\overline{u})(s)ds,
\]
\[
\Phi_{v_0}(u,v)(t)=e^{\frac{1}{2}it\Delta}v_0+i\int_0^t e^{\frac{1}{2}i(t-s)\Delta}(u^2)(s)ds
\]
for $(u,v)\in E$. Since $(\frac{21}{10},\frac{42}{13})$ is a $L^2$\,admissible pair,
\begin{align*}
\|\Phi_{u_0}(u,v)\|_{L^\frac{14}{5}L^\frac{14}{5}}&\leq\|e^{it\Delta}u_0\|_{L^\frac{14}{5}L^\frac{14}{5}}+2\left\|\int_0^t e^{i(t-s)\Delta}(v\overline{u})(s)ds\right\|_{L^\frac{14}{5}L^\frac{14}{5}}\\
&\leq c\|u_0\|_{L^2}+2c\|vu\|_{S'(L^2)}\\
&\leq c\|u_0\|_{H^1}+2c\|vu\|_{L^\frac{21}{11}L^\frac{42}{29}}\\
&\leq c\|u_0\|_{H^1}+2c\|v\|_{L^\frac{14}{5}L^\frac{14}{5}}\|u\|_{L^6L^3}\\
&\leq c\|u_0\|_{H^1}+16c^2\delta_{sd}\|v\|_{L^\frac{14}{5}L^\frac{14}{5}}.
\end{align*}
Similarly,
\[
\|\Phi_{v_0}(u,v)\|_{L^\frac{14}{5}L^\frac{14}{5}}\leq c\|v_0\|_{H^1}+8c^2\delta_{sd}\|u\|_{L^\frac{14}{5}L^\frac{14}{5}}.
\]
Combining these inequalities,
\begin{align}
\|(\Phi_{u_0}(u,v),\Phi_{v_0}(u,v))\|_{L^\frac{14}{5}L^\frac{14}{5}\times L^\frac{14}{5}L^\frac{14}{5}}\leq c\|(u_0,v_0)\|_{H^1\times H^1}+16c^2\delta_{sd}\|(u,v)\|_{L^\frac{14}{5}L^\frac{14}{5}\times L^\frac{14}{5}L^\frac{14}{5}}.\label{67}
\end{align}
Also,
\begin{align*}
\|\Phi_{u_0}(u,v)\|_{L^6L^3}&\leq \|e^{it\Delta}u_0\|_{L^6L^3}+2\left\|\int_0^t e^{i(t-s)\Delta}(v\overline{u})(s)ds\right\|_{L^6L^3}\\
&\leq c\|u_0\|_{\dot{H}^\frac{1}{2}}+2c\|vu\|_{L^3L^\frac{3}{2}}\\
&\leq c\|u_0\|_{H^1}+2c\|v\|_{L^6L^3}\|u\|_{L^6L^3}\\
&\leq c\|u_0\|_{H^1}+16c^2\delta_{sd}\|v\|_{L^6L^3}.
\end{align*}
Similarly,
\[
\|\Phi_{v_0}(u,v)\|_{L^6L^3}\leq c\|v_0\|_{H^1}+8c^2\delta_{sd}\|u\|_{L^6L^3}.
\]
Combining these inequalities,
\begin{align}
\|(\Phi_{u_0}(u,v),\Phi_{v_0}(u,v))\|_{L^6L^3\times L^6L^3}\leq c\|(u_0,v_0)\|_{H^1\times H^1}+16c^2\delta_{sd}\|(u,v)\|_{L^6L^3\times L^6L^3}.\label{68}
\end{align}
Since $(3,\frac{30}{11})$ is a $L^2$\,admissible pair,
\begin{align*}
\|\Phi_{u_0}(u,v)\|_{L^\infty H^1}&\leq \|e^{it\Delta}u_0\|_{L^\infty H^1}+2\left\|\int_0^t e^{i(t-s)\Delta}(v\overline{u})(s)ds\right\|_{L^\infty H^1}\\
&\leq c\|u_0\|_{H^1}+2c\|vu\|_{L^\frac{3}{2}L^\frac{30}{19}}+2c\|u\nabla v\|_{L^\frac{3}{2}L^\frac{30}{19}}+2c\|v\nabla u\|_{L^\frac{3}{2}L^\frac{30}{19}}\\
&\leq c\|u_0\|_{H^1}+2c\|v\|_{L^2L^\frac{10}{3}}\|u\|_{L^6L^3}+2c\|u\|_{L^6L^3}\|\nabla v\|_{L^2L^\frac{10}{3}}+2c\|v\|_{L^6L^3}\|\nabla u\|_{L^2L^\frac{10}{3}}\\
&\leq c\|u_0\|_{H^1}+2c\|v\|_{L^2W^{1,\frac{10}{3}}}\|u\|_{L^6L^3}+2c\|v\|_{L^6L^3}\|u\|_{L^2W^{1,\frac{10}{3}}}\\
&\leq c\|u_0\|_{H^1}+16c^2\delta_{sd}\|(u,v)\|_{L^2W^{1,\frac{10}{3}}\times L^2W^{1,\frac{10}{3}}}.
\end{align*}
Similarly,
\[
\|\Phi_{v_0}(u,v)\|_{L^\infty H^1}\leq c\|v_0\|_{H^1}+16c^2\delta_{sd}\|(u,v)\|_{L^2W^{1,\frac{10}{3}}\times L^2W^{1,\frac{10}{3}}}.
\]
Combining these inequalities,
\begin{align}
\|(\Phi_{u_0}(u,v),\Phi_{v_0}(u,v))\|_{L^\infty H^1\times L^\infty H^1}\leq c\|(u_0,v_0)\|_{H^1\times H^1}+32c^2\delta_{sd}\|(u,v)\|_{L^2W^{1,\frac{10}{3}}\times L^2W^{1,\frac{10}{3}}}.\label{69}
\end{align}
Here, since both $(\infty,2)$ and $\left(2,\frac{10}{3}\right)$ are $L^2$\,admissible pairs,
\begin{align}
\|(\Phi_{u_0}(u,v),\Phi_{v_0}(u,v))\|_{L^2 W^{1,\frac{10}{3}}\times L^2 W^{1,\frac{10}{3}}}\leq c\|(u_0,v_0)\|_{H^1\times H^1}+32c^2\delta_{sd}\|(u,v)\|_{L^2W^{1,\frac{10}{3}}\times L^2W^{1,\frac{10}{3}}}.\label{70}
\end{align}
Thus, combining \eqref{67},\,\eqref{68},\,\eqref{69}, and \eqref{70},
\begin{align*}
&\|(\Phi_{u_0}(u,v),\Phi_{v_0}(u,v))\|_{X\times X}\\
&~~~~~~~~\leq4c\|(u_0,v_0)\|_{H^1\times H^1}+16c^2\delta_{sd}\|(u,v)\|_{L^\frac{14}{5}L^\frac{14}{5}\times L^\frac{14}{5}L^\frac{14}{5}\cap L^6L^3\times L^6L^3}+64c^2\delta_{sd}\|(u,v)\|_{L^2W^{1,\frac{10}{3}}\times L^2W^{1,\frac{10}{3}}}\\
&~~~~~~~~\leq4c\|(u_0,v_0)\|_{H^1\times H^1}+64c^2\delta_{sd}\|(u,v)\|_{X\times X}\\
&~~~~~~~~\leq(4c+512c^3\delta_{sd})\|(u_0,v_0)\|_{H^1\times H^1}.
\end{align*}
If we take $\delta_{sd}$ with $512c^3\delta_{sd}\leq4c$, i.e. $\delta_{sd}\leq\frac{1}{128c^2}$, then
\[
\|(\Phi_{u_0}(u,v),\Phi_{v_0}(u,v))\|_{X\times X}\leq8c\|(u_0,v_0)\|_{H^1\times H^1}.
\]
Also, for $(u_1,v_1),\,(u_2,v_2)\in E$,
\begin{align*}
\|\Phi_{u_0}(u_1,v_1)-\Phi_{u_0}(u_2,v_2)\|_{L^\frac{14}{5}L^\frac{14}{5}}&=2\left\|\int_0^te^{it\Delta}(v_1\overline{u_1}-v_2\overline{u_2})(s)ds\right\|_{L^\frac{14}{5}L^\frac{14}{5}}\\
&\leq2c\|v_1\overline{u_1}-v_2\overline{u_2}\|_{L^\frac{21}{11}L^\frac{42}{29}}\\
&\leq2c\|v_1(\overline{u_1}-\overline{u_2})\|_{L^\frac{21}{11}L^\frac{42}{29}}+2c\|(v_1-v_2)\overline{u_2}\|_{L^\frac{21}{11}L^\frac{42}{29}}\\
&\leq2c\|v_1\|_{L^6L^3}\|u_1-u_2\|_{L^\frac{14}{5}L^\frac{14}{5}}+2c\|v_1-v_2\|_{L^\frac{14}{5}L^\frac{14}{5}}\|u_2\|_{L^6L^3}\\
&\leq16c^2\delta_{sd}\|u_1-u_2\|_{L^\frac{14}{5}L^\frac{14}{5}}+16c^2\delta_{sd}\|v_1-v_2\|_{L^\frac{14}{5}L^\frac{14}{5}}\\
&=16c^2\delta_{sd}\|(u_1,v_1)-(u_2,v_2)\|_{L^\frac{14}{5}L^\frac{14}{5}\times L^\frac{14}{5}L^\frac{14}{5}}.
\end{align*}
Similarly,
\[
\|\Phi_{v_0}(u_1,v_1)-\Phi_{v_0}(u_2,v_2)\|_{L^\frac{14}{5}L^\frac{14}{5}}\leq16c^2\delta_{sd}\|(u_1,v_1)-(u_2,v_2)\|_{L^\frac{14}{5}L^\frac{14}{5}\times L^\frac{14}{5}L^\frac{14}{5}}.
\]
Combining these inequalities,
\begin{align}
&\|(\Phi_{u_0}(u_1,v_1),\Phi_{v_0}(u_1,v_1))-(\Phi_{u_0}(u_2,v_2),\Phi_{v_0}(u_2,v_2))\|_{L^\frac{14}{5}L^\frac{14}{5}} \notag \\
&\hspace{6cm}\leq32c^2\delta_{sd}\|(u_1,v_1)-(u_2,v_2)\|_{L^\frac{14}{5}L^\frac{14}{5}\times L^\frac{14}{5}L^\frac{14}{5}}.\label{71}
\end{align}
Also,
\begin{align*}
\|\Phi_{u_0}(u_1,v_1)-\Phi_{u_0}(u_2,v_2)\|_{L^6L^3}&=2\left\|\int_0^te^{it\Delta}(v_1\overline{u_1}-v_2\overline{u_2})(s)ds\right\|_{L^6L^3}\\
&\leq2c\|v_1\overline{u_1}-v_2\overline{u_2}\|_{L^3L^\frac{3}{2}}\\
&\leq2c\|v_1(\overline{u_1}-\overline{u_2})\|_{L^3L^\frac{3}{2}}+2c\|(v_1-v_2)\overline{u_2}\|_{L^3L^\frac{3}{2}}\\
&\leq2c\|v_1\|_{L^6L^3}\|u_1-u_2\|_{L^6L^3}+2c\|v_1-v_2\|_{L^6L^3}\|u_2\|_{L^6L^3}\\
&\leq16c^2\delta_{sd}\|(u_1,v_1)-(u_2,v_2)\|_{L^6L^3\times L^6L^3}.
\end{align*}
Similarly,
\[
\|\Phi_{v_0}(u_1,v_1)-\Phi_{v_0}(u_2,v_2)\|_{L^6L^3}\leq16c^2\delta_{sd}\|(u_1,v_1)-(u_2,v_2)\|_{L^6L^3\times L^6L^3}.
\]
Combining these inequalities,
\begin{align}
&\|(\Phi_{u_0}(u_1,v_1),\Phi_{v_0}(u_1,v_1))-(\Phi_{u_0}(u_2,v_2),\Phi_{v_0}(u_2,v_2))\|_{L^6L^3\times L^6L^3} \notag \\
&\hspace{6cm}\leq32c^2\delta_{sd}\|(u_1,v_1)-(u_2,v_2)\|_{L^6L^3\times L^6L^3}.\label{72}
\end{align}
Also,
\begin{align*}
&\|\Phi_{u_0}(u_1,v_1)-\Phi_{u_0}(u_2,v_2)\|_{L^\infty H^1}\leq2\left\|\int_0^te^{it\Delta}(v_1\overline{u_1}-v_2\overline{u_2})(s)ds\right\|_{L^\infty H^1}\\
&\hspace{2cm}\leq2c\|v_1\overline{u_1}-v_2\overline{u_2}\|_{L^\frac{3}{2}L^\frac{30}{19}}+2c\|\nabla(v_1\overline{u_1}-v_2\overline{u_2})\|_{L^\frac{3}{2}L^\frac{30}{19}}\\
&\hspace{2cm}\leq2c\|v_1(\overline{u_1}-\overline{u_2})\|_{L^\frac{3}{2}L^\frac{30}{19}}+2c\|(v_1-v_2)\overline{u_2}\|_{L^\frac{3}{2}L^\frac{30}{19}}\\
&\hspace{4cm}+2c\|\nabla v_1(\overline{u_1}-\overline{u_2})\|_{L^\frac{3}{2}L^\frac{30}{19}}+2c\|\nabla(v_1-v_2)\overline{u_2}\|_{L^\frac{3}{2}L^\frac{30}{19}}\\
&\hspace{4cm}+2c\|v_1\nabla(\overline{u_1}-\overline{u_2})\|_{L^\frac{3}{2}L^\frac{30}{19}}+2c\|\nabla\overline{u_2}(v_1-v_2)\|_{L^\frac{3}{2}L^\frac{30}{19}}\\
&\hspace{2cm}\leq2c\|v_1\|_{L^6L^3}\|u_1-u_2\|_{L^2L^\frac{10}{3}}+2c\|v_1-v_2\|_{L^2L^\frac{10}{3}}\|u_2\|_{L^6L^3}\\
&\hspace{4cm}+2c\|\nabla v_1\|_{L^2L^\frac{10}{3}}\|u_1-u_2\|_{L^6L^3}+2c\|\nabla(v_1-v_2)\|_{L^2L^\frac{10}{3}}\|u_2\|_{L^6L^3}\\
&\hspace{4cm}+2c\|v_1\|_{L^6L^3}\|\nabla(u_1-u_2)\|_{L^2L^\frac{10}{3}}+2c\|\nabla u_2\|_{L^2L^\frac{10}{3}}\|v_1-v_2\|_{L^6L^3}\\
&\hspace{2cm}\leq16c^2\delta_{sd}\|u_1-u_2\|_{L^2L^\frac{10}{3}}+16c^2\delta_{sd}\|v_1-v_2\|_{L^2L^\frac{10}{3}}\\
&\hspace{4cm}+16c^2\delta_{sd}\|u_1-u_2\|_{L^6L^3}+16c^2\delta_{sd}\|\nabla(v_1-v_2)\|_{L^2L^\frac{10}{3}}\\
&\hspace{4cm}+16c^2\delta_{sd}\|\nabla(u_1-u_2)\|_{L^2L^\frac{10}{3}}+16c^2\delta_{sd}\|v_1-v_2\|_{L^6L^3}\\
&\hspace{2cm}\leq16c^2\delta_{sd}\|(u_1,v_1)-(u_2,v_2)\|_{L^6L^3\times L^6L^3\cap L^2W^{1,\frac{10}{3}}\times L^2W^{1,\frac{10}{3}}}.
\end{align*}
Similarly,
\begin{align*}
&\|\Phi_{v_0}(u_1,v_1)-\Phi_{v_0}(u_2,v_2)\|_{L^\infty H^1}\\
&\hspace{3cm}\leq16c^2\delta_{sd}\|(u_1,v_1)-(u_2,v_2)\|_{L^6L^3\times L^6L^3\cap L^2W^{1,\frac{10}{3}}\times L^2W^{1,\frac{10}{3}}}.
\end{align*}
Combining these inequalities,
\begin{align}
&\|(\Phi_{u_0}(u_1,v_1),\Phi_{v_0}(u_1,v_1))-(\Phi_{u_0}(u_2,v_2),\Phi_{v_0}(u_2,v_2))\|_{L^\infty H^1\times L^\infty H^1} \notag \\
&\hspace{5cm}\leq32c^2\delta_{sd}\|(u_1,v_1)-(u_2,v_2)\|_{L^6L^3\times L^6L^3\cap L^2W^{1,\frac{10}{3}}\times L^2W^{1,\frac{10}{3}}}.\label{73}
\end{align}
Also,
\begin{align}
&\|(\Phi_{u_0}(u_1,v_1),\Phi_{v_0}(u_1,v_1))-(\Phi_{u_0}(u_2,v_2),\Phi_{v_0}(u_2,v_2))\|_{L^2W^{1,\frac{10}{3}}\times L^2W^{1,\frac{10}{3}}} \notag \\
&\hspace{4cm}\leq32c^2\delta_{sd}\|(u_1,v_1)-(u_2,v_2)\|_{L^6L^3\times L^6L^3\cap L^2W^{1,\frac{10}{3}}\times L^2W^{1,\frac{10}{3}}}.\label{74}
\end{align}
Thus, combining \eqref{71},\,\eqref{72},\,\eqref{73}, and \eqref{74},
\begin{align*}
&\|(\Phi_{u_0}(u_1,v_1),\Phi_{v_0}(u_1,v_1))-(\Phi_{u_0}(u_2,v_2),\Phi_{v_0}(u_2,v_2))\|_{X\times X} \notag \\
&\hspace{6cm}\leq96c^2\delta_{sd}\|(u_1,v_1)-(u_2,v_2)\|_{X\times X}.
\end{align*}
From $\delta_{sd}\leq\frac{1}{128c^2}$, it follows that $96c^2\delta_{sd}\leq\frac{3}{4}<1$. Therefore, there exists a unique solution $(u,v)$ to (NLS) on $E$.
\end{proof}


\begin{proposition}[Existence of a critical solution]\label{Existence of a critical solution}
Let $(u_c,v_c)$ be the time-global solution to (NLS) with initial data $(u_{c,0},v_{c,0})$. We can construct $(u_{c,0},v_{c,0})$ and $(u_c,v_c)$ so that
\[
I_\omega(u_{c.0},v_{c.0})=I_\omega^c<I_\omega(\phi_\omega,\psi_\omega)\ ,\ \ \ K_\omega^{20,8}(u_c,v_c)>0\ \ \text{for\ all}\ \ 0\leq t<\infty,
\]
\[
\|(u_c,v_c)\|_{S(\dot{H}^\frac{1}{2})\times S(\dot{H}^\frac{1}{2})}=\infty.
\]
\end{proposition}


\begin{proof}
By the definition of $I_\omega^c$, we may take a sequence of data $\{(u_{n,0},v_{n,0})\}\subset H^1\times H^1$ with
\[
K_\omega^{20,8}(u_{n,0},v_{n,0})>0,\ \ \ \ I_\omega(\phi_\omega,\psi_\omega)>I_\omega(u_{n,0},v_{n,0})\searrow I_\omega^c
\]
and $SC(u_{n,0},v_{n,0})$ does not holds. By Theorem \ref{Global versus blow-up dichotomy}, $(u_n,v_n)$ exists time-globally. By the definition of $SC(u_{n,0},v_{n,0})$, $\|(u_n,v_n)\|_{S(\dot{H}^\frac{1}{2})\times S(\dot{H}^\frac{1}{2})}=\infty$ for any $n\in\mathbb{N}$. By Lemma \ref{Comparability of K and I}, we have $\frac{1}{10}K_\omega(u_{n,0},v_{n,0})<I_\omega(u_{n,0},v_{n,0})$. Thus,
\begin{align}
\|u_{n,0}\|_{H^1}^2\leq10\left(\frac{1}{\omega}+1\right)I_\omega(\phi_\omega,\psi_\omega)\,,\ \ \|v_{n,0}\|_{H^1}^2\leq5\left(\frac{1}{\omega}+4\right)I_\omega(\phi_\omega,\psi_\omega).\label{57}
\end{align}
i.e. $\|(u_{n,0},v_{n,0})\|_{H^1\times H^1}$ is bounded. Appling Theorem \ref{Linear profile decomposition},
\begin{align}
(u_{n,0}(x),v_{n,0}(x))=\sum_{j=1}^M(e^{-it_n^j\Delta}u^j(x-x_n^j),e^{-\frac{1}{2}it_n^j\Delta}v^j(x-x_n^j))+(U_n^M(x),V_n^M(x)).\label{81}
\end{align}
First, we will prove that there is only one $j$ with $(u^j,v^j)\neq(0,0)$.\\
We assume that $(u^j,v^j)=(0,0)$ for any $j$. Then,
\[
\displaystyle\lim_{n\rightarrow\infty}\|(e^{it\Delta}U_n^M,e^{\frac{1}{2}it\Delta}V_n^M)\|_{S(\dot{H}^\frac{1}{2})\times S(\dot{H}^\frac{1}{2})}=\lim_{n\rightarrow\infty}\|(e^{it\Delta}u_{n,0},e^{\frac{1}{2}it\Delta}v_{n,0})\|_{S(\dot{H}^\frac{1}{2})\times S(\dot{H}^\frac{1}{2})}=0
\]
by \eqref{26}. Hence, 
\[
\|(e^{it\Delta}u_{n,0},e^{\frac{1}{2}it\Delta}v_{n,0})\|_{S(\dot{H}^\frac{1}{2})\times S(\dot{H}^\frac{1}{2})}\leq\delta_{sd}
\]
for sufficiently large $n$. From Theorem \ref{Small data globally existence},
\begin{align*}
\|(u_n,v_n)\|_{S(\dot{H}^\frac{1}{2})\times S(\dot{H}^\frac{1}{2})}\leq 4\|(e^{it\Delta}u_{n,0},e^{\frac{1}{2}it\Delta}v_{n,0})\|_{S(\dot{H}^\frac{1}{2})\times S(\dot{H}^\frac{1}{2})}\leq4\delta_{sd}<\infty
\end{align*}
for such $n$. This is in contradiction to $\|(u_n,v_n)\|_{S(\dot{H}^\frac{1}{2})\times S(\dot{H}^\frac{1}{2})}=\infty$. Thus, there exists $j$ such that $(u^j,v^j)\neq(0,0)$. We set $(u^j,v^j)\neq(0,0)$ for any $j$ by removing $j$ with $(u^j,v^j)=(0,0)$.\\
From Corollary \ref{I decomposition},
\[
I_\omega(u_{n,0},v_{n,0})=\sum_{j=1}^MI_\omega(e^{-it_n^j\Delta}u^j(\cdot-x_n^j),e^{-\frac{1}{2}it_n^j\Delta}v^j(\cdot-x_n^j))+I_\omega(U_n^M,V_n^M)+o_n(1).
\]
From Corollary \ref{K decomposition},
\[
K_\omega^{20,8}(u_{n,0},v_{n,0})=\sum_{j=1}^MK_\omega^{20,8}(e^{-it_n^j\Delta}u^j(\cdot-x_n^j),e^{-\frac{1}{2}it_n^j\Delta}v^j(\cdot-x_n^j))+K_\omega^{20,8}(U_n^M,V_n^M)+o_n(1).
\]
For sufficiently large $n$, there exists $\delta>0$ and $\varepsilon>0$ with $2\varepsilon<\delta$ such that
\[
I_\omega(u_{n,0},v_{n,0})\leq I_\omega(\phi_\omega,\psi_\omega)-\delta,
\]
\[
I_\omega(u_{n,0},v_{n,0})\geq\sum_{j=1}^MI_\omega(e^{-it_n^j\Delta}u^j(\cdot-x_n^j),e^{-\frac{1}{2}it_n^j\Delta}v^j(\cdot-x_n^j))+I_\omega(U_n^M,V_n^M)-\varepsilon,
\]
\[
K_\omega^{20,8}(u_{n,0},v_{n,0})\geq-\varepsilon,
\]
\[
K_\omega^{20,8}(u_{n,0},v_{n,0})\leq\sum_{j=1}^MK_\omega^{20,8}(e^{-it_n^j\Delta}u^j(\cdot-x_n^j),e^{-\frac{1}{2}it_n^j\Delta}v^j(\cdot-x_n^j))+K_\omega^{20,8}(U_n^M,V_n^M)+\varepsilon.
\]
From Lemma \ref{lemma for critical solution},
\begin{align}
0<I_\omega(e^{-it_n^j\Delta}u^j(\cdot-x_n^j),e^{-\frac{1}{2}it_n^j\Delta}v^j(\cdot-x_n^j))<I_\omega(\phi_\omega,\psi_\omega)\,,\ \ 0\leq I_\omega(U_n^M,V_n^M)<I_\omega(\phi_\omega,\psi_\omega),\label{58}
\end{align}
\begin{align*}
K_\omega^{20,8}(e^{-it_n^j\Delta}u^j(\cdot-x_n^j),e^{-\frac{1}{2}it_n^j\Delta}v^j(\cdot-x_n^j))>0\,,\ \ K_\omega^{20,8}(U_n^M,V_n^M)\geq0.
\end{align*}
for sufficiently large $n$. (The equality holds if and only if $(U_n^M,V_n^M)=(0,0)$.)\\
Thus,
\begin{align}
0\leq\lim_{n\rightarrow\infty}I_\omega(e^{-it_n^j\Delta}u^j,e^{-\frac{1}{2}it_n^j\Delta}v^j)\leq\lim_{n\rightarrow\infty}I_\omega(u_{n,0},v_{n,0})=I_\omega^c.\label{29}
\end{align}
We assume $(u^j,v^j)\neq(0,0)$ for more than one $j$.\\
In the case $t_n^j\longrightarrow-\infty$, using Theorem \ref{Linear estimate}
\begin{align*}
|P(e^{-it_n^j\Delta}u^j,e^{-\frac{1}{2}it_n^j\Delta}v^j)|&\leq\|e^{-it_n^j\Delta}u^j\|_{L^3}^2\|e^{-\frac{1}{2}it_n^j\Delta}v^j\|_{L^3}\leq c|t_n^j|^{-\frac{5}{2}}\|u^j\|_{L^\frac{3}{2}}^2\|v^j\|_{L^\frac{3}{2}}\\
&\leq c|t_n^j|^{-\frac{5}{2}}\|u^j\|_{H^1}^2\|v^j\|_{H^1}\longrightarrow0\ \ \text{as}\ \ n\rightarrow\infty
\end{align*}
and hence
\[
\frac{\omega}{2}M(u^j,v^j)+\frac{1}{2}K(u^j,v^j)=I_\omega(e^{-it_n^j\Delta}u^j,e^{-\frac{1}{2}it_n^j\Delta}v^j)+P(e^{-it_n^j\Delta}u^j,e^{-\frac{1}{2}it_n^j\Delta}v^j)<I_\omega(\phi_\omega,\psi_\omega)
\]
for sufficiently large $n\in\mathbb{N}$. From Lemma \ref{Existence of wave operators}, there exists a data $(\widetilde{u}_0^j,\widetilde{v}_0^j)$ and a corresponding solution $(\widetilde{u}^j,\widetilde{v}^j)$ to (NLS) such that
\[
0<I_\omega(\widetilde{u}_0^j,\widetilde{v}_0^j)<I_\omega(\phi_\omega,\psi_\omega)\,,\ \ K_\omega^{20,8}(\widetilde{u}_0^j,\widetilde{v}_0^j)>0\,,
\]
\[
\|(\widetilde{u}^j,\widetilde{v}^j)(-t_n^j)-(e^{-it_n^j\Delta}u^j,e^{-\frac{1}{2}it_n^j\Delta}v^j)\|_{H^1\times H^1}\longrightarrow0\ \ \text{as}\ \ n\rightarrow\infty.
\]
In the case $t_n^j=0$, we set $(\widetilde{u}_0^j,\widetilde{v}_0^j)=(u^j,v^j)$ and a solution $(\widetilde{u}^j,\widetilde{v}^j)$ to (NLS) with a data $(\widetilde{u}_0^j,\widetilde{v}_0^j)$. Then,
\[
0<I_\omega(\widetilde{u}_0^j,\widetilde{v}_0^j)<I_\omega(\phi_\omega,\psi_\omega)\,,\ \ K_\omega^{20,8}(\widetilde{u}_0^j,\widetilde{v}_0^j)>0\,,
\]
\[
\|(\widetilde{u}^j,\widetilde{v}^j)(-t_n^j)-(e^{-it_n^j\Delta}u^j,e^{-\frac{1}{2}it_n^j\Delta}v^j)\|_{H^1\times H^1}\longrightarrow0\ \ \text{as}\ \ n\rightarrow\infty.
\]
Therefore, in both cases, we can take new profile\,$(\widetilde{u}_0^j,\widetilde{v}_0^j)$ (which is called nonlinear profile) with
\[
0<I_\omega(\widetilde{u}_0^j,\widetilde{v}_0^j)<I_\omega(\phi_\omega,\psi_\omega)\,,\ \ K_\omega^{20,8}(\widetilde{u}_0^j,\widetilde{v}_0^j)>0\,,
\]
\begin{align}
\|(\widetilde{u}^j,\widetilde{v}^j)(-t_n^j)-(e^{-it_n^j\Delta}u^j,e^{-\frac{1}{2}it_n^j\Delta}v^j)\|_{H^1\times H^1}\longrightarrow0\ \ \text{as}\ \ n\rightarrow\infty.\label{41}
\end{align}
for each linear profile\,$(u^j,v^j)$. Also, we take $(\widetilde{U}_n^M,\widetilde{V}_n^M)$ with
\[
(u_{n,0}(x),v_{n,0}(x))=\sum_{j=1}^M(\widetilde{u}^j(x-x_n^j,-t_n^j),\widetilde{v}^j(x-x_n^j,-t_n^j))+(\widetilde{U}_n^M(x),\widetilde{V}_n^M(x))
\]
(which is called nonlinear profile decomposition). Then, by \eqref{81}, Theorem \ref{Strichartz estimates},\,\eqref{41}, and \eqref{26}, we have
\begin{align}
&\|(e^{it\Delta}\widetilde{U}_n^M,e^{\frac{1}{2}it\Delta}\widetilde{V}_n^M)\|_{S(\dot{H}^\frac{1}{2})\times S(\dot{H}^\frac{1}{2})} \notag \\
&=\left\|(e^{it\Delta}u_{n,0},e^{\frac{1}{2}it\Delta}v_{n,0})-\sum_{j=1}^M(e^{it\Delta}\widetilde{u}^j(x-x_n^j,-t_n^j),e^{\frac{1}{2}it\Delta}\widetilde{v}^j(x-x_n^j,-t_n^j))\right\|_{S(\dot{H}^\frac{1}{2})\times S(\dot{H}^\frac{1}{2})} \notag \\
&\leq\sum_{j=1}^M\left\|\left(e^{it\Delta}(e^{-it_n^j}u^j-\widetilde{u}^j(-t_n^j)),e^{\frac{1}{2}it\Delta}(e^{-\frac{1}{2}it_n^j\Delta}v^j-\widetilde{v}^j(-t_n^j))\right)\right\|_{S(\dot{H}^\frac{1}{2})\times S(\dot{H}^\frac{1}{2})}\\
&\hspace{9.5cm}+\left\|(e^{it\Delta}U_n^M,e^{\frac{1}{2}it\Delta}V_n^M)\right\|_{S(\dot{H}^\frac{1}{2})\times S(\dot{H}^\frac{1}{2})} \notag \\
&\leq c\sum_{j=1}^M\left\|\left(e^{-it_n^j}u^j-\widetilde{u}^j(-t_n^j),e^{-\frac{1}{2}it_n^j\Delta}v^j-\widetilde{v}^j(-t_n^j)\right)\right\|_{H^1\times H^1}+\left\|(e^{it\Delta}U_n^M,e^{\frac{1}{2}it\Delta}V_n^M)\right\|_{S(\dot{H}^\frac{1}{2})\times S(\dot{H}^\frac{1}{2})} \notag \\
&\longrightarrow0\ \ \text{as}\ \ M,n\rightarrow\infty.\label{30}
\end{align}
Here, we set $(\widetilde{u}_n^j,\widetilde{v}_n^j)$ and $(\widetilde{u}_n^{\leq M},\widetilde{v}_n^{\leq M})$ as follows.
\[
(\widetilde{u}_n^j(x,t),\widetilde{v}_n^j(x,t))=(\widetilde{u}^j(x-x_n^j,t-t_n^j),\widetilde{v}^j(x-x_n^j,t-t_n^j))\,,
\]
\[
(\widetilde{u}_n^{\leq M}(x,t),\widetilde{v}_n^{\leq M}(x,t))=\sum_{j=1}^M(\widetilde{u}_n^j(x,t),\widetilde{v}_n^j(x,t)).
\]
Then,
\begin{align*}
i\partial_t\widetilde{u}_n^{\leq M}+\Delta\widetilde{u}_n^{\leq M}+2\widetilde{v}_n^{\leq M}\overline{\widetilde{u}_n^{\leq M}}&=i\partial_t\sum_{j=1}^M\widetilde{u}_n^j+\Delta\sum_{j=1}^M\widetilde{u}_n^j+2\left(\sum_{j=1}^M\widetilde{v}_n^j\right)\left(\sum_{j=1}^M\overline{\widetilde{u}_n^j}\right)\\
&=\sum_{j=1}^M(i\partial_t\widetilde{u}_n^j+\Delta\widetilde{u}_n^j+2\widetilde{v}_n^j\overline{\widetilde{u}_n^j})+2\sum_{1\leq j\neq k\leq M}\widetilde{v}_n^j\overline{\widetilde{u}_n^k}\\
&=2\sum_{1\leq j\neq k\leq M}\widetilde{v}_n^j\overline{\widetilde{u}_n^k}=\vcentcolon e_{1,n}.
\end{align*}
Similarly,
\[
i\partial_t\widetilde{v}_n^{\leq M}+\frac{1}{2}\Delta\widetilde{v}_n^{\leq M}+(\widetilde{u}_n^{\leq M})^2=\sum_{1\leq j\neq k\leq M}\widetilde{u}_n^j\widetilde{u}_n^k=\vcentcolon e_{2,n}.
\]
We will prove that the assumption of Theorem \ref{Long time perturbation} holds.\\
First, we will establish \eqref{08}. Since
\begin{align*}
\|(e^{it\Delta}(\widetilde{u}_n^{\leq M}(0)-u_{n,0}),e^{\frac{1}{2}it\Delta}(\widetilde{u}_n^{\leq M}(0)-v_{n,0}))\|_{S(\dot{H}^\frac{1}{2})\times S(\dot{H}^\frac{1}{2})}&=\|(e^{it\Delta}\widetilde{U}_n^M,e^{\frac{1}{2}it\Delta}\widetilde{V}_n^M)\|_{S(\dot{H}^\frac{1}{2})\times S(\dot{H}^\frac{1}{2})}\\
&\longrightarrow0\ \ \text{as}\ \ M,n\rightarrow\infty,
\end{align*}
for any $\varepsilon>0$, there exists $M_1(\varepsilon)=M_1>0$ such that for any $M>M_1$, there exists $n_1=n_1(M)$ such that for any $n>n_1$,
\begin{align}
\|(e^{it\Delta}(\widetilde{u}_n^{\leq M}(0)-u_{n,0}),e^{\frac{1}{2}it\Delta}(\widetilde{u}_n^{\leq M}(0)-v_{n,0}))\|_{S(\dot{H}^\frac{1}{2})\times S(\dot{H}^\frac{1}{2})}<\varepsilon.\label{31}
\end{align}
Next, we will prove \eqref{06} and \eqref{07}.\\
From Theorem \ref{Linear profile decomposition}\,\eqref{27} and \eqref{28},
\[
\|u_{n,0}\|_{H^1}^2=\sum_{j=1}^M\|u^j\|_{H^1}^2+\|U_n^M\|_{H^1}^2+o_n(1)\,,\ \ \|v_{n,0}\|_{H^1}^2=\sum_{j=1}^M\|v^j\|_{H^1}^2+\|V_n^M\|_{H^1}^2+o_n(1)
\]
and hence
\begin{align}
C_1\vcentcolon=\sum_{j=1}^\infty\|u^j\|_{H^1}^2<\infty\,,\ \ C_2\vcentcolon=\sum_{j=1}^\infty\|v^j\|_{H^1}^2<\infty.\label{42}
\end{align}
In particular, there exists $M_2=M_2(\delta_{sd})>0$ such that $\|(u^j,v^j)\|_{H^1\times H^1}<\frac{\delta_{sd}}{c}$ for any $j>M_2$. Then, from \eqref{41}, for any $j>M_2$, there exists $n_2=n_2(j)$ such that
\[
\|(\widetilde{u}^j(-t_n^j),\widetilde{v}^j(-t_n^j))\|_{H^1}\leq\delta_{sd}\leq1\,,\ \ \ M_2\leq j\leq M,\,n\geq n_2.
\]
Moreover, from \eqref{41} and \eqref{42}, for any $M\geq1$, there exists $n_3=n_3(M)\geq1$ such that
\begin{align}
\sum_{j=1}^M\|\widetilde{u}^j(-t_n^j)\|_{H^1}^2\leq 2C_1\,,\ \sum_{j=1}^M\|\widetilde{v}^j(-t_n^j)\|_{H^1}^2\leq2C_2\,,\ \ \ n\geq n_3.\label{59}
\end{align}
Here, we consider $\|(\widetilde{u}_n^{\leq M},\widetilde{v}_n^{\leq M})\|_{S(\dot{H}^\frac{1}{2})\times S(\dot{H}^\frac{1}{2})}$.
\begin{align*}
\|\widetilde{u}_n^{\leq M}\|_{S(\dot{H}^\frac{1}{2})}&=\left\|e^{it\Delta}\widetilde{u}_n^{\leq M}(0)+i\int_0^te^{i(t-s)\Delta}(2\widetilde{v}_n^{\leq M}\overline{\widetilde{u}_n^{\leq M}}-e_{1,n})ds\right\|_{S(\dot{H}^\frac{1}{2})}\\
&\leq\left\|e^{it\Delta}\widetilde{u}_n^{\leq M}(0)\right\|_{S(\dot{H}^\frac{1}{2})}+\left\|\int_0^te^{i(t-s)\Delta}(2\widetilde{v}_n^{\leq M}\overline{\widetilde{u}_n^{\leq M}}-e_{1,n})ds\right\|_{S(\dot{H}^\frac{1}{2})}\\
&\leq\left\|e^{it\Delta}(u_{n,0}+\widetilde{U}_n^M)\right\|_{S(\dot{H}^\frac{1}{2})}+c\left\|2\widetilde{v}_n^{\leq M}\overline{\widetilde{u}_n^{\leq M}}-e_{1,n}\right\|_{S'(\dot{H}^{-\frac{1}{2}})}\\
&\leq c\left\|u_{n,0}\right\|_{\dot{H}^\frac{1}{2}}+\left\|e^{it\Delta}\widetilde{U}_n^M\right\|_{S(\dot{H}^\frac{1}{2})}+2c\left\|\widetilde{v}_n^{\leq M}\widetilde{u}_n^{\leq M}\right\|_{L^3L^\frac{3}{2}}+c\left\|e_{1,n}\right\|_{L^3L^\frac{3}{2}}\\
&\leq c\left\|u_{n,0}\right\|_{\dot{H}^\frac{1}{2}}+\left\|e^{it\Delta}\widetilde{U}_n^M\right\|_{S(\dot{H}^\frac{1}{2})}+2c\left\|\widetilde{v}_n^{\leq M}\right\|_{L^6L^3}\left\|\widetilde{u}_n^{\leq M}\right\|_{L^6L^3}+c\left\|e_{1,n}\right\|_{L^3L^\frac{3}{2}}.
\end{align*}
Since $C_0^\infty$ is dense in $L^6(\mathbb{R};L^3)$, for any $\varepsilon>0$, there exists $f,\,g\in C_0^\infty$ such that 
\[
\|f-\widetilde{u}^j\|_{L^6L^3}<\varepsilon\,,\ \|g-\widetilde{v}^j\|_{L^6L^3}<\varepsilon.
\]
Using these inequalities and \eqref{40},
\begin{align}
&\left\|e_{1,n}\right\|_{L^3L^\frac{3}{2}}=2\left\|\sum_{1\leq j\neq k\leq M}\widetilde{v}^j(\cdot-x_n^j,\cdot-t_n^j)\overline{\widetilde{u}^k}(\cdot-x_n^k,\cdot-t_n^k)\right\|_{L^3L^\frac{3}{2}} \notag \\
&\hspace{0.5cm}\leq2\sum_{1\leq j\neq k\leq M}\left\|\widetilde{v}^j(\cdot-x_n^j,\cdot-t_n^j)\widetilde{u}^k(\cdot-x_n^k,\cdot-t_n^k)\right\|_{L^3L^\frac{3}{2}} \notag \\
&\hspace{0.5cm}\leq2\sum_{1\leq j\neq k\leq M}\left(\left\|(\widetilde{v}^j-g)(\cdot-x_n^j,\cdot-t_n^j)\widetilde{u}^k(\cdot-x_n^k,\cdot-t_n^k)\right\|_{L^3L^\frac{3}{2}}\right. \notag \\
&\hspace{1cm}\left.+\left\|g(\cdot-x_n^j,\cdot-t_n^j)(\widetilde{u}^k-f)(\cdot-x_n^k,\cdot-t_n^k)\right\|_{L^3L^\frac{3}{2}}+\left\|g(\cdot-x_n^j,\cdot-t_n^j)f(\cdot-x_n^k,\cdot-t_n^k)\right\|_{L^3L^\frac{3}{2}}\right) \notag \\
&\hspace{0.5cm}\leq2\sum_{1\leq j\neq k\leq M}\left(\|\widetilde{v}^j-g\|_{L^6L^3}\|\widetilde{u}^k\|_{L^6L^3}+\|g\|_{L^6L^3}\|\widetilde{u}^k-f\|_{L^6L^3}\right. \notag \\
&\hspace{9cm}\left.+\left\|g(\cdot-x_n^j,\cdot-t_n^j)f(\cdot-x_n^k,\cdot-t_n^k)\right\|_{L^3L^\frac{3}{2}}\right) \notag \\
&\hspace{0.5cm}\longrightarrow0\ \ \text{as}\ \ n\rightarrow\infty.\label{32}
\end{align}
Similarly, we have
\begin{align}
\|e_{2,n}\|_{L^3L^\frac{3}{2}}\longrightarrow0\ \ \text{as}\ \ n\rightarrow\infty\label{82}
\end{align}
Applying Lemma \ref{Existence of wave operators}, $\|(\widetilde{u}^j,\widetilde{v}^j)\|_{S(\dot{H}^\frac{1}{2})\times S(\dot{H}^\frac{1}{2})}\leq4\|(e^{it\Delta}u^j,e^{\frac{1}{2}\Delta}v^j)\|_{S(\dot{H}^\frac{1}{2})\times S(\dot{H}^\frac{1}{2})}$. We set $M_0=\max\{M_1,M_2\}$, then it follows that
\begin{align*}
\|\widetilde{v}_n^{\leq M}\|_{L^6L^3}^6&=\left\|\sum_{j=1}^M\widetilde{v}^j(\cdot-x_n^j,\cdot-t_n^j)\right\|_{L^6L^3}^6\\
&\leq c\left\|\sum_{j=1}^{M_0}\widetilde{v}^j(\cdot-x_n^j,\cdot-t_n^j)\right\|_{L^6L^3}^6+c\left\|\sum_{j=M_0+1}^M\widetilde{v}^j(\cdot-x_n^j,\cdot-t_n^j)\right\|_{L^6L^3}^6,
\end{align*}
\begin{align*}
&\left\|\sum_{j=1}^{M_0}\widetilde{v}^j(\cdot-x_n^j,\cdot-t_n^j)\right\|_{L^6L^3}\leq\sum_{j=1}^{M_0}\left\|\widetilde{v}^j\right\|_{L^6L^3}\leq\sum_{j=1}^{M_0}\left\|\widetilde{v}^j\right\|_{S(\dot{H}^\frac{1}{2})}\\
&\hspace{3cm}\leq4\sum_{j=1}^{M_0}\|(e^{-it_n^j\Delta}u^j,e^{-\frac{1}{2}it_n^j\Delta}v^j)\|_{S(\dot{H}^\frac{1}{2})\times S(\dot{H}^\frac{1}{2})}\leq c\sum_{j=1}^{M_0}\|(u^j,v^j)\|_{\dot{H}^\frac{1}{2}\times\dot{H}^\frac{1}{2}}<\infty,
\end{align*}
and
\begin{align*}
&\left\|\sum_{j=M_0+1}^M\widetilde{v}^j(\cdot-x_n^j,\cdot-t_n^j)\right\|_{L^6L^3}^6\\
&\hspace{3cm}\leq\left\|\left\|\sum_{j=M_0+1}^M\widetilde{v}^j(\cdot-x_n^j,\cdot-t_n^j)\right\|_{L^\frac{16}{5}}^\frac{8}{15}\left\|\sum_{j=M_0+1}^M\widetilde{v}^j(\cdot-x_n^j,\cdot-t_n^j)\right\|_{L^\frac{14}{5}}^\frac{7}{15}\right\|_{L^6}^6\\
&\hspace{3cm}\leq\left\|\left\|\sum_{j=M_0+1}^M\widetilde{v}^j(\cdot-x_n^j,\cdot-t_n^j)\right\|_{L^\frac{16}{5}}^\frac{8}{15}\right\|_{L^\infty}^6\left\|\left\|\sum_{j=M_0+1}^M\widetilde{v}^j(\cdot-x_n^j,\cdot-t_n^j)\right\|_{L^\frac{14}{5}}^\frac{7}{15}\right\|_{L^6}^6\\
&\hspace{3cm}\leq\left\|\sum_{j=M_0+1}^M\widetilde{v}^j(\cdot-x_n^j,\cdot-t_n^j)\right\|_{L^\infty L^\frac{16}{5}}^\frac{16}{5}\left\|\sum_{j=M_0+1}^M\widetilde{v}^j(\cdot-x_n^j,\cdot-t_n^j)\right\|_{L^\frac{14}{5}L^\frac{14}{5}}^\frac{14}{5}\\
&\hspace{3cm}\leq c\left\|\sum_{j=M_0+1}^M\widetilde{v}^j(\cdot-x_n^j,\cdot-t_n^j)\right\|_{L^\infty H^1}^\frac{16}{5}\left\|\sum_{j=M_0+1}^M\widetilde{v}^j(\cdot-x_n^j,\cdot-t_n^j)\right\|_{L^\frac{14}{5}L^\frac{14}{5}}^\frac{14}{5}.
\end{align*}
Here, we estimate
\[
\left\|\sum_{j=M_0+1}^M\widetilde{v}^j(\cdot-x_n^j,\cdot-t_n^j)\right\|_{L^\infty H^1}^2\,\text{and}\ \ \ \left\|\sum_{j=M_0+1}^M\widetilde{v}^j(\cdot-x_n^j,\cdot-t_n^j)\right\|_{L^\frac{14}{5}L^\frac{14}{5}}^\frac{14}{5}.
\]
\begin{align}
&\left\|\sum_{j=M_0+1}^M\widetilde{v}^j(\cdot-x_n^j,t-t_n^j)\right\|_{H^1}^2 \notag \\
&\hspace{1cm}\leq\sum_{j=M_0+1}^M\|\widetilde{v}^j(t-t_n^j)\|_{H^1}^2+2\sum_{M_0+1\leq j\neq k\leq M}\text{Re}(\widetilde{v}^j(\cdot-x_n^j,t-t_n^j),\widetilde{v}^k(\cdot-x_n^k,t-t_n^k))_{H^1}.\label{43}
\end{align}
By Lemma \ref{lemma for critical solution 2},
\begin{align}
\sum_{j=M_0+1}^M\|\widetilde{v}^j(t-t_n^j)\|_{H^1}^2\leq 64c^2\sum_{j=M_0+1}^M\|\widetilde{v}^j(-t_n^j)\|_{H^1}^2\leq128c^2C_1\,,\ \ \ n\geq\max\{n_2,n_3\}.\label{44}
\end{align}
On the other hand, by \eqref{40},
\begin{align}
|(\widetilde{v}^j(\cdot-x_n^j,t-t_n^j),\widetilde{v}^k(\cdot-x_n^k,t-t_n^k))_{H^1}|\longrightarrow0\ \ \text{as}\ \ n\rightarrow\infty.\label{45}
\end{align}
From \eqref{43},\,\eqref{44}, and \eqref{45}, for any $M\geq M_0+1$, there exists $n_4=n_4(M)\geq1$ such that
\begin{align}
\left\|\sum_{j=M_0+1}^M\widetilde{v}^j(\cdot-x_n^j,t-t_n^j)\right\|_{L^\infty H^1}^2\leq256c^2C_1\,,\ \ \ M\geq M_0+1,\,n\geq n_4.\label{46}
\end{align}
Using Lemma \ref{lemma for critical solution 1},
\begin{align*}
&\left\|\sum_{j=M_0+1}^M\widetilde{v}^j(\cdot-x_n^j,\cdot-t_n^j)\right\|_{L^\frac{14}{5}L^\frac{14}{5}}^\frac{14}{5}\\
&\hspace{1cm}\leq\sum_{j=M_0+1}^M\|\widetilde{v}^j\|_{L^\frac{14}{5}L^\frac{14}{5}}^\frac{14}{5}+c\sum_{M_0+1\leq j\neq k\leq M}\int_{\mathbb{R}}\int_{\mathbb{R}^5}|\widetilde{v}^j(x-x_n^j,t-t_n^j)|^\frac{9}{5}|\widetilde{v}^k(x-x_n^k,t-t_n^k)|dxdt.
\end{align*}
By Lemma \ref{lemma for critical solution 2} and \eqref{59},
\begin{align*}
\sum_{j=M_0+1}^M\|\widetilde{v}^j(\cdot-t_n^j)\|_{L^\frac{14}{5}L^\frac{14}{5}}^\frac{14}{5}&\leq (8c)^\frac{14}{5}\sum_{j=M_0+1}^M\|\widetilde{v}^j(-t_n^j)\|_{H^1}^\frac{14}{5}\leq (8c)^\frac{14}{5}\sum_{j=M_0+1}^M\|\widetilde{v}^j(-t_n^j)\|_{H^1}^2\leq 2C_1(8c)^\frac{14}{5}
\end{align*}
for $n\geq\max\{n_2,n_3\}$. On the other hand,
\begin{align*}
&\int_{\mathbb{R}}\int_{\mathbb{R}^5}|\widetilde{v}^j(x-x_n^j,t-t_n^j)|^\frac{9}{5}|\widetilde{v}^k(x-x_n^k,t-t_n^k)|dxdt\\
&\hspace{2cm}=\||\widetilde{v}^j(\cdot-x_n^j,\cdot-t_n^j)|^\frac{9}{5}|\widetilde{v}^k(\cdot-x_n^k,\cdot-t_n^k)|\|_{L^1L^1}\\
&\hspace{2cm}\leq\left\|\|\widetilde{v}^j(\cdot-x_n^j,\cdot-t_n^j)\widetilde{v}^k(\cdot-x_n^k,\cdot-t_n^k)\|_{L^\frac{7}{5}}\||\widetilde{v}^j(\cdot-x_n^j,\cdot-t_n^j)|^\frac{4}{5}\|_{L^\frac{7}{2}}\right\|_{L^1}\\
&\hspace{2cm}=\left\|\|\widetilde{v}^j(\cdot-x_n^j,\cdot-t_n^j)\widetilde{v}^k(\cdot-x_n^k,\cdot-t_n^k)\|_{L^\frac{7}{5}}\|\widetilde{v}^j(\cdot-x_n^j,\cdot-t_n^j)\|_{L^\frac{14}{5}}^\frac{4}{5}\right\|_{L^1}\\
&\hspace{2cm}\leq\left\|\|\widetilde{v}^j(\cdot-x_n^j,\cdot-t_n^j)\widetilde{v}^k(\cdot-x_n^k,\cdot-t_n^k)\|_{L^\frac{7}{5}}\right\|_{L^\frac{7}{5}}\left\|\|\widetilde{v}^j(\cdot-x_n^j,\cdot-t_n^j)\|_{L^\frac{14}{5}}^\frac{4}{5}\right\|_{L^\frac{7}{2}}\\
&\hspace{2cm}=\left\|\widetilde{v}^j(\cdot-x_n^j,\cdot-t_n^j)\widetilde{v}^k(\cdot-x_n^k,\cdot-t_n^k)\right\|_{L^\frac{7}{5}L^\frac{7}{5}}\left\|\widetilde{v}^j(\cdot-x_n^j,\cdot-t_n^j)\right\|_{L^\frac{14}{5}L^\frac{14}{5}}^\frac{4}{5}\\
&\hspace{2cm}\leq (8c)^\frac{4}{5}\|\widetilde{v}^j(-t_n^j)\|_{H^1}^\frac{4}{5}\left(\int_{\mathbb{R}}\int_{\mathbb{R}^5}|\widetilde{v}^j(x-x_n^j,t-t_n^j)|^\frac{7}{5}|\widetilde{v}^k(x-x_n^k,t-t_n^k)|^\frac{7}{5}dxdt\right)^\frac{5}{7}
\end{align*}
for $M_0+1\leq j\neq k\leq M$ by Lemma \ref{lemma for critical solution 2}.
Using \eqref{40},
\[
\int_{\mathbb{R}}\int_{\mathbb{R}^5}|\widetilde{v}^j(x-x_n^j,t-t_n^j)|^\frac{9}{5}|\widetilde{v}^k(x-x_n^k,t-t_n^k)|dxdt\longrightarrow0\ \ \text{as}\ \ n\rightarrow\infty.
\]
Thus,
\begin{align}
\left\|\widetilde{u}_n^{\leq M}\right\|_{S(\dot{H}^\frac{1}{2})}\leq A<\infty.\label{33}
\end{align}
Similarly,
\begin{align}
\left\|\widetilde{v}_n^{\leq M}\right\|_{S(\dot{H}^\frac{1}{2})}\leq A<\infty.\label{34}
\end{align}
Combining \eqref{31},\,\eqref{32},\,\eqref{82},\,\eqref{33},\,\eqref{34}, and Theorem \ref{Long time perturbation},
\[
\left\|(u_n,v_n)\right\|_{S(\dot{H}^\frac{1}{2})\times S(\dot{H}^\frac{1}{2})}\leq C(A)<\infty.
\]
However, this is contradiction. Therefore, there is only one $j$ with $(u^j,v^j)\neq(0,0)$ and we set $j=1$ for such $j$ by rearranging. We consider a profile\,$(u_{c,0},v_{c,0})$ given in the above argument and the corresponding solution $(u_c,v_c)$ to (NLS). Also, we define $(\widetilde{U}_n^1,\widetilde{V}_n^1)$ as follows.
\[
(u_{n,0}(x),v_{n,0}(x))=(u_c(x-x_n^1,-t_n^1),u_c(x-x_n^1,-t_n^1))+(\widetilde{U}_n^1(x),\widetilde{V}_n^1(x)).
\]
From the calculation of \eqref{30}, we have
\[
\|(e^{it\Delta}\widetilde{U}_n^1,e^{\frac{1}{2}it\Delta}\widetilde{V}_n^1)\|_{S(\dot{H}^\frac{1}{2})\times S(\dot{H}^\frac{1}{2})}\longrightarrow0\ \ \text{as}\ \ n\rightarrow\infty.
\]
Here, we set that $(u_{c,n}(x,t),v_{c,n}(x,t))=(u_c(x-x_n^1,t-t_n^1),v_c(x-x_n^1,t-t_n^1))$ and assume that
\[
A\vcentcolon=\|(u_{c,n},v_{c,n})\|_{S(\dot{H}^\frac{1}{2})\times S(\dot{H}^\frac{1}{2})}=\|(u_c,v_c)\|_{S(\dot{H}^\frac{1}{2})\times S(\dot{H}^\frac{1}{2})}<\infty.
\]
We will prove that this assumption is wrong by using Theorem \ref{Long time perturbation}.
\[
i\partial_tu_{c,n}+\Delta u_{c,n}+2v_{c,n}\overline{u_{c,n}}=0\,,\ \ i\partial_tv_{c,n}+\frac{1}{2}\Delta v_{c,n}+u_{c,n}^2=0,
\]
\begin{align*}
\|(e^{it\Delta}(u_{c,n}(0)-u_{n,0}),e^{\frac{1}{2}it\Delta}(v_{c,n}(0)-v_{n,0}))\|_{S(\dot{H}^\frac{1}{2})\times S(\dot{H}^\frac{1}{2})}&=\|(e^{it\Delta}\widetilde{U}_n^1,e^{\frac{1}{2}it\Delta}\widetilde{V}_n^1)\|_{S(\dot{H}^\frac{1}{2})\times S(\dot{H}^\frac{1}{2})}\\
&\longrightarrow0\ \ \text{as}\ \ n\rightarrow\infty.
\end{align*}
Hence, from Theorem \ref{Long time perturbation},
\[
\|(u_n,v_n)\|_{S(\dot{H}^\frac{1}{2})\times S(\dot{H}^\frac{1}{2})}\leq C(A)<\infty.
\]
However, this is in contradiction to the way of taking $(u_n,v_n)$. Thus, we have $\|(u_c,v_c)\|_{S(\dot{H}^\frac{1}{2})\times S(\dot{H}^\frac{1}{2})}=\infty$. Also, from the way of taking a profile,
\[
0<I_\omega(u_{c,0},v_{c,0})<I_\omega(\phi_\omega,\psi_\omega)\,,\ \ K_\omega^{20,8}(u_{c,0},v_{c,0})>0,
\]
\[
\lim_{n\rightarrow\infty}\|(u_c(-t_n^1),v_c(-t_n^1))-(e^{-it_n^1\Delta}u^1,e^{-\frac{1}{2}it_n^1\Delta}v^1)\|_{H^1\times H^1}=0,
\]
i.e.
\[
\lim_{n\rightarrow\infty}\|(u_c(-t_n^1),v_c(-t_n^1))\|_{H^1\times H^1}=\|(u^1,v^1)\|_{H^1\times H^1}.
\]
Therefore, it follows that
\[
|K_\omega(u_c(-t_n^1),v_c(-t_n^1))-K_\omega(e^{-it_n^1\Delta}u^1,e^{-\frac{1}{2}it_n^1\Delta}v^1)|\longrightarrow0\ \ \text{as}\ \ n\rightarrow\infty,
\]
\begin{align*}
&|P(u_c(-t_n^1),v_c(-t_n^1))-P(e^{-it_n^1\Delta}u^1,e^{-\frac{1}{2}it_n^1\Delta}v^1)|\\
&\hspace{1cm}=\left|\int_{\mathbb{R}^5}v_c(x,-t_n^1)\overline{u_c}(x,-t_n^1)^2-e^{-\frac{1}{2}it_n^1\Delta}v^1(x)\overline{e^{-it_n^1\Delta}u^1}(x)^2dx\right|\\
&\hspace{1cm}=\left|\int_{\mathbb{R}^5}v_c(x,-t_n^1)(\overline{u_c}(x,-t_n^1)^2-\overline{e^{-it_n^1\Delta}u^1}(x)^2)+(v_c(x,-t_n^1)-e^{-\frac{1}{2}it_n^1\Delta}v^1(x))\overline{e^{-it_n^1\Delta}u^1}(x)^2dx\right|\\
&\hspace{1cm}\leq\|v_c(-t_n^1)\|_{L^3}\|u_c(-t_n^1)+e^{-it_n^1\Delta}u^1\|_{L^3}\|u_c(-t_n^1)-e^{-it_n^1\Delta}u^1\|_{L^3}\\
&\hspace{9.5cm}+\|v_c(-t_n^1)-e^{-\frac{1}{2}it_n^1\Delta}v^1\|_{L^3}\|e^{-it_n^1\Delta}u^1\|_{L^3}^2\\
&\hspace{1cm}\leq\|v_c(-t_n^1)\|_{H^1}(\|u_c(-t_n^1)\|_{H^1}+\|u^1\|_{H^1})\|u_c(-t_n^1)-e^{-it_n^1\Delta}u^1\|_{H^1}\\
&\hspace{9.5cm}+\|v_c(-t_n^1)-e^{-\frac{1}{2}it_n^1\Delta}v^1\|_{H^1}\|u^1\|_{H^1}^2\\
&\hspace{1cm}\longrightarrow0\ \ \text{as}\ \ n\rightarrow\infty.
\end{align*}
Thus, we have
\[
\lim_{n\rightarrow\infty}|I_\omega(u_c(-t_n^1),v_c(-t_n^1))-I_\omega(e^{-it_n^1\Delta}u^1,e^{-\frac{1}{2}it_n^1\Delta}v^1)|=0,
\]
\[
\lim_{n\rightarrow\infty}|K_\omega^{20,8}(u_c(-t_n^1),v_c(-t_n^1))-K_\omega^{20,8}(e^{-it_n^1\Delta}u^1,e^{-\frac{1}{2}it_n^1\Delta}v^1)|=0.
\]
Since $I_\omega$ conserves with respect to time, it follows that $I_\omega(u_c(-t_n^1),v_c(-t_n^1))=I_\omega(u_{c,0},v_{c,0})$. Also, we obtain $0\leq I_\omega(u_{c,0},v_{c,0})\leq I_\omega^c$ by $\displaystyle0\leq\lim_{n\rightarrow\infty}I_\omega(e^{-it_n^1\Delta}u^1,e^{-\frac{1}{2}it_n^1\Delta}v^1)\leq I_\omega^c$. Here, we assume that $I_\omega(u_{c,0},v_{c,0})<I_\omega^c$, then it follows that $\|(u_c,v_c)\|_{S(\dot{H}^\frac{1}{2})\times S(\dot{H}^\frac{1}{2})}<\infty$ by definition of $I_\omega^c$. However, this is contradiction. Therefore, we obtain $I_\omega(u_{c,0},v_{c,0})=I_\omega^c$.
\end{proof}

\subsection{Compactness of the critical solution}

　\\
We define a equivalence relation $\sim$ on $H^1\!\times\!H^1$ as follows:
\[
^\exists x_0\in\mathbb{R}^5\ \ \text{s.t.}\ \ (\phi_1,\phi_2)=(\psi_1(\cdot-x_0),\psi_2(\cdot-x_0))\ \Longrightarrow\ (\phi_1,\phi_2)\sim(\psi_1,\psi_2).
\]
$H^1\!\times\!H^1/{\sim}$ denotes the quotient space, which is constructed by the whole equivalence class with respect to $\sim$. We represent an element of $H^1\!\times\!H^1/{\sim}$ by $[(\phi_1,\phi_2)]$, and let $\pi:H^1\!\times\!H^1\longrightarrow H^1\!\times\!H^1/{\sim}$ be the natural projection.\\ 
Here, we check that $\sim$ defines the equivalence relation on $H^1\!\times\!H^1$.
\[
\text{Since }(\phi_1,\phi_2)=(\phi_1(\cdot-0),\phi_2(\cdot-0)),\ (\phi_1,\phi_2)\sim(\phi_1,\phi_2).
\]
Since
\[
(\phi_1,\phi_2)=(\psi_1(\cdot-x_0),\psi_2(\cdot-x_0))\Longrightarrow(\psi_1,\psi_2)=(\phi_1(\cdot+x_0),\phi_2(\cdot+x_0)),
\]
\[
(\phi_1,\phi_2)\sim(\psi_1,\psi_2)\Longrightarrow(\psi_1,\psi_2)\sim(\phi_1,\phi_2).
\]
Since
\begin{align*}
&(\phi_1,\phi_2)=(\psi_1(\cdot-x_0),\psi_2(\cdot-x_0))\,,\ (\psi_1,\psi_2)=(\varphi_1(\cdot-x_1),\varphi_2(\cdot-x_1))\\
&\hspace{7cm}\Longrightarrow(\phi_1,\phi_2)=(\varphi_1(\cdot-x_0-x_1),\varphi_2(\cdot-x_0-x_1)),
\end{align*}
\[
(\phi_1,\phi_2)\sim(\psi_1,\psi_2)\,,\ (\psi_1,\psi_2)\sim(\varphi_1,\varphi_2)\Longrightarrow(\phi_1,\phi_2)\sim(\varphi_1,\varphi_2).
\]
Therefore, $\sim$ defines the equivalence relation on $H^1\!\times\!H^1$.


\begin{lemma}\label{Precompactness of the flow of the critical solution 1}
$H^1\!\times\!H^1/{\sim}$ is metrizable with a distance
\[
d([(\phi_1,\phi_2)],[(\psi_1,\psi_2)])=\inf_{x_0\in\mathbb{R}^5}\|(\phi_1(\cdot-x_0),\phi_2(\cdot-x_0))-(\psi_1,\psi_2)\|_{H^1\times H^1}.
\]
Moreover, $H^1\!\times\!H^1/{\sim}$ is complete with respect to this distance.
\end{lemma}

\begin{remark}
$H^1\!\times\!H^1/{\sim}$ is not a vector space.
\end{remark}

\begin{proof}
First, we establish that the orbits of $(\phi_1(\cdot-x_0),\phi_2(\cdot-x_0))$ are closed in $H^1\!\times\!H^1$. The orbit of $(\phi_1,\phi_2)=(0,0)$ is $(0,0)$. Suppose $(\phi_1,\phi_2)\neq(0,0)$, $\{x_n\}\subset\mathbb{R}^5$, and $(\phi_1(\cdot-x_n),\phi_2(\cdot-x_n))$ converges on $(\psi_1,\psi_2)$ in $H^1\!\times\!H^1$. Then, we claim that $x_n$ converges. If not, then either $x_n$ is unbounded and there exists a subsequence $x_n$ such that $|x_n|\longrightarrow\infty$, or $x_n$ is bounded and two subsequences $x_n\longrightarrow x_0$ and $x_{n'}\longrightarrow x_0'$. In the first case, we obtain $(\psi_1,\psi_2)=(0,0)$. We consider convergence of $(\phi_1,\phi_2)$ in $B(0,R)$ for any $R>0$. Then, $(\psi_1(x),\psi_2(x))=(0,0)$ for $x\in B(0,R)$. Because $R>0$ is arbitrary, $(\psi_1,\psi_2)=(0,0)$. This implies $(\phi_1,\phi_2)=(0,0)$. In the second case, we obtain $(\phi_1(\cdot-x_0),\phi_2(\cdot-x_0))=(\phi_1(\cdot-x_0'),\phi_2(\cdot-x_0'))$. Because $C_0^\infty$ is dense in $H^1$, for any $\varepsilon>0$, there exists $f,\,g\in C_0^\infty$ such that
\[
\|\phi_1-f\|_{H^1}<\varepsilon\ \ \text{ and }\ \ \|\phi_2-g\|_{H^1}<\varepsilon.
\]
Thus,
\begin{align*}
&\|(\phi_1(\cdot-x_0),\phi_2(\cdot-x_0))-(\psi_1,\psi_2)\|_{H^1\times H^1}\\
&\hspace{1cm}\leq\|(\phi_1,\phi_2)(\cdot-x_0)-(f,g)(\cdot-x_0)\|_{H^1\times H^1}+\|(f,g)(\cdot-x_0)-(f,g)(\cdot-x_n)\|_{H^1\times H^1}\\
&\hspace{2cm}+\|(f,g)(\cdot-x_n)-(\phi_1,\phi_2)(\cdot-x_n)\|_{H^1\times H^1}+\|(\phi_1,\phi_2)(\cdot-x_n)-(\psi_1,\psi_2)\|_{H^1\times H^1}<\varepsilon
\end{align*}
and we obtain $(\phi_1(\cdot-x_0),\phi_2(\cdot-x_0))=(\psi_1,\psi_2)$. Similarly, $(\phi_1(\cdot-x_0'),\phi_2(\cdot-x_0'))=(\psi_1,\psi_2)$, i.e. $(\phi_1(\cdot-x_0),\phi_2(\cdot-x_0))=(\phi_1(\cdot-x_0'),\phi_2(\cdot-x_0'))$. This equality holds only at $(\phi_1,\phi_2)=(0,0)$. Therefore, the both cases cause contradiction.\\
Next, we verify that the distance $d$ defines a distance on $H^1\times H^1/{\sim}$. If $[(\phi_1,\phi_2)]=[(\psi_1,\psi_2)]$, then it is clear that $d([(\phi_1,\phi_2)],[(\psi_1,\psi_2)])=0$. Conversely, if $d([(\phi_1,\phi_2)],[(\psi_1,\psi_2)])=0$, then $\displaystyle\inf_{x_0\in\mathbb{R}^5}\|(\phi_1(\cdot-x_0),\phi_2(\cdot-x_0))-(\psi_1,\psi_2)\|_{H^1\times H^1}=0$ by definition of the distance $d$. Since the orbits of $(\phi_1(\cdot-x_0),\phi_2(\cdot-x_0))$ are closed in $H^1\!\times\!H^1$, there exists $x_0^\ast\in\mathbb{R}^5$ such that $\|(\phi_1(\cdot-x_0^\ast),\phi_2(\cdot-x_0^\ast))-(\psi_1,\psi_2)\|_{H^1\times H^1}=0$. Thus, $(\phi_1(\cdot-x_0^\ast),\phi_2(\cdot-x_0^\ast))=(\psi_1,\psi_2)$, i.e. $[(\phi_1,\phi_2)]=[(\psi_1,\psi_2)]$. It follows clearly that $d([(\phi_1,\phi_2)],[(\psi_1,\psi_2)])=d([(\psi_1,\psi_2)],[(\phi_1,\phi_2)])$. If $(\phi_1,\phi_2),\ (\psi_1,\psi_2),\ (\eta_1,\eta_2)\in H^1\!\times\!H^1$, then there exists $x_1^\ast\in\mathbb{R}^5$ and $x_2^\ast\in\mathbb{R}^5$ such that 
\[
d([(\phi_1,\phi_2)],[(\eta_1,\eta_2)])=\|(\phi_1(\cdot-x_1^\ast),\phi_2(\cdot-x_1^\ast))-(\eta_1,\eta_2)\|_{H^1\times H^1},
\]
and
\[
d([(\eta_1,\eta_2)],[(\psi_1,\psi_2)])=\|(\eta_1(\cdot-x_2^\ast),\eta_2(\cdot-x_2^\ast))-(\psi_1,\psi_2)\|_{H^1\times H^1}.
\]
Hence,
\begin{align*}
d([(\phi_1,\phi_2)],[(\psi_1,\psi_2)])&=\inf_{x_0\in\mathbb{R}^5}\|(\phi_1(\cdot-x_0),\phi_2(\cdot-x_0))-(\psi_1,\psi_2)\|_{H^1\times H^1}\\
&\leq\|(\phi_1(\cdot-x_1^\ast-x_2^\ast),\phi_2(\cdot-x_1^\ast-x_2^\ast))-(\psi_1,\psi_2)\|_{H^1\times H^1}\\
&\leq\|(\phi_1(\cdot-x_1^\ast-x_2^\ast),\phi_2(\cdot-x_1^\ast-x_2^\ast))-(\eta_1(\cdot-x_2^\ast),\eta_2(\cdot-x_2^\ast))\|_{H^1\times H^1}\\
&\hspace{5cm}+\|(\eta_1(\cdot-x_2^\ast),\eta_2(\cdot-x_2^\ast))-(\psi_1,\psi_2)\|_{H^1\times H^1}\\
&\leq\|(\phi_1(\cdot-x_1^\ast),\phi_2(\cdot-x_1^\ast))-(\eta_1,\eta_2)\|_{H^1\times H^1}\\
&\hspace{5cm}+\|(\eta_1(\cdot-x_2^\ast),\eta_2(\cdot-x_2^\ast))-(\psi_1,\psi_2)\|_{H^1\times H^1}\\
&=d([(\phi_1,\phi_2)],[(\eta_1,\eta_2)])+d([(\eta_1,\eta_2)],[(\psi_1,\psi_2)]).
\end{align*}
Thus, the triangle inequality also holds. Next, we prove completeness. Suppose that $[(\phi_{1,n},\phi_{2,n})]$ is a Cauchy sequence. It suffices to show that a subsequence converges. We can pass to a subsequence $[(\phi_{1,n},\phi_{2,n})],[(\phi_{1,n+1},\phi_{2,n+1})]$ so that $d([(\phi_{1,n},\phi_{2,n})],[(\phi_{1,n+1},\phi_{2,n+1})])\leq2^{-n}$. We take $x_1=0$. We construct a sequence $x_n$ inductively as follows: given $x_{n-1}$, select $x_n$ so that $\|(\phi_{1,n-1}(\cdot-x_{n-1}),\phi_{2,n-1}(\cdot-x_{n-1}))-(\phi_{1,n}(\cdot-x_n),\phi_{2,n}(\cdot-x_n))\|_{H^1\times H^1}\leq2^{-n+1}$. Then, $(\phi_{1,n}(\cdot-x_{n}),\phi_{2,n}(\cdot-x_{n}))$ is a Cauchy sequence in $H^1\!\times\!H^1$, and hence, there exists $(\phi_1,\phi_2)$ such that $(\phi_{1,n}(\cdot-x_n),\phi_{2,n}(\cdot-x_n))\longrightarrow(\phi_1,\phi_2)$ in $H^1\!\times\!H^1$. Then,
\[
d([\phi_{1,n},\phi_{2,n}],[\phi_1,\phi_2])\leq\|(\phi_{1,n}(\cdot-x_n),\phi_{2,n}(\cdot-x_n))-(\phi_1,\phi_2)\|_{H^1\times H^1}\longrightarrow0.
\]
Thus, $[(\phi_{1,n},\phi_{2,n})]\longrightarrow[(\phi_1,\phi_2)]\ \ \text{in}\ \ H^1\!\times\!H^1/{\sim}.$\\
Finally, we prove that $\pi(B((\phi_1,\phi_2),r))=B([(\phi_1,\phi_2)],r)$ for any $(\phi_1,\phi_2)\in H^1\!\times\!H^1$ and $r>0$. We take any $(\psi_1,\psi_2)\in B((\phi_1,\phi_2),r)$. Then, $\|(\phi_1,\phi_2)-(\psi_1,\psi_2)\|_{H^1\times H^1}<r$, and hence,
\begin{align*}
&d([(\phi_1,\phi_2)],[(\psi_1,\psi_2)])\\
&\hspace{1.5cm}=\inf_{x_0\in\mathbb{R}^5}\|(\phi_1(\cdot-x_0),\phi_2(\cdot-x_0))-(\psi_1,\psi_2)\|_{H^1\times H^1}\leq\|(\phi_1,\phi_2)-(\psi_1,\psi_2)\|_{H^1\times H^1}<r.
\end{align*}
Thus, $[(\psi_1,\psi_2)]\in B([(\phi_1,\phi_2)],r)$. We take any $[(\psi_1,\psi_2)]\in B([(\phi_1,\phi_2)],r)$. Then,
\[
d([(\phi_1,\phi_2)],[(\psi_1,\psi_2)])<r.
\]
Since the orbits of $(\phi_1(\cdot-x_0),\phi_2(\cdot-x_0))$ is closed in $H^1\!\times\!H^1$, there exists $x_0^\ast$ such that
\begin{align*}
d([(\phi_1,\phi_2)],[(\psi_1,\psi_2)])&=\|(\phi_1(\cdot-x_0^\ast),\phi_2(\cdot-x_0^\ast))-(\psi_1,\psi_2)\|_{H^1\times H^1}\\
&=\|(\phi_1,\phi_2)-(\psi_1(\cdot+x_0^\ast),\psi_2(\cdot+x_0^\ast))\|_{H^1\times H^1}<r.
\end{align*}
Since $(\psi_1(\cdot+x_0^\ast),\psi_2(\cdot+x_0^\ast))\in B((\phi_1,\phi_2),r)$, we obtain $[(\psi_1,\psi_2)]\in\pi(B((\phi_1,\phi_2),r))$. Therefore, we have $\pi(B((\phi_1,\phi_2),r))=B([(\phi_1,\phi_2)],r)$ and the topology deduced from the distance $d$ on $H^1\!\times\!H^1/{\sim}$ is quotient topology.
\end{proof}


\begin{lemma}\label{Precompactness of the flow of the critical solution 2}
Let $K$ be a precompact subset of $H^1\!\times\! H^1/{\sim}$. Assume
\begin{align*}
^\exists \eta>0\text{ such that }^\forall(\phi_1,\phi_2)\in\pi^{-1}(K),\ \ \ \|(\phi_1,\phi_2)\|_{H^1\times H^1}\geq\eta.
\end{align*}
Then there exists $\widetilde{K}\subset H^1\!\times\! H^1$ such that $\pi(\widetilde{K})=K$ and $\widetilde{K}$ is precompact in $H^1\!\times\!H^1$.
\end{lemma}

\begin{proof}
We first show that there exists $\varepsilon>0$ such that for any $(p,q)\in K$, there exists $(\psi_1(p),\psi_2(q))=(\psi_1,\psi_2)\in\pi^{-1}(p,q)$ such that
\begin{align}
\|(\psi_1,\psi_2)\|_{H^1(B(0,1))\times H^1(B(0,1))}\geq\varepsilon \label{83}
\end{align}
by contradiction. If not, then for any $n\in\mathbb{N}$, there exists $[(\phi_{1,n},\phi_{2,n})]\in K$
\begin{align}
\sup_{x_0\in\mathbb{R}^5}\|(\phi_{1,n}(\cdot-x_0),\phi_{2,n}(\cdot-x_0))\|_{H^1(B(0,1))\times H^1(B(0,1))}\leq\frac{1}{n}.\label{60}
\end{align}
Since $K$ is precompact, we take a subsequence from $\{[(\phi_{1,n},\phi_{2,n})]\}$ if necessary, then there exists $(\phi_1,\phi_2)\in H^1\!\times\! H^1$ such that $[(\phi_{1,n},\phi_{2,n})]\longrightarrow(p,q)$ in  $H^1\!\times\! H^1/{\sim}$, where we set $\pi(\phi_1,\phi_2)=(p,q)$. In other wards, if $(\phi_1,\phi_2)$ is fixed in $\pi^{-1}(p,q)$, then
\[
\inf_{x_0\in\mathbb{R}^5}\|(\phi_{1,n}(\cdot-x_0),\phi_{2,n}(\cdot-x_0))-(\phi_1,\phi_2)\|_{H^1\times H^1}\longrightarrow0\ \ \text{as}\ \ n\rightarrow\infty.
\]
Thus, we may find a sequence $\{x_n\}\subset\mathbb{R}^5$ such that
\[
\|(\phi_{1,n}(\cdot-x_n),\phi_{2,n}(\cdot-x_n))-(\phi_1,\phi_2)\|_{H^1\times H^1}\longrightarrow0\ \ \text{as}\ \ n\rightarrow\infty.
\]
From \eqref{60}, for any $x_0\in\mathbb{R}^5$,
\[
\|(\phi_{1,n}(\cdot+x_0-x_n),\phi_{2,n}(\cdot+x_0-x_n))\|_{H^1(B(0,1))\times H^1(B(0,1))}\leq\frac{1}{n}.
\]
Thus,
\begin{align*}
&\|(\phi_1,\phi_2)\|_{H^1(B(x_0,1))\times H^1(B(x_0,1))}\\
&\hspace{2cm}\leq\|(\phi_1,\phi_2)-(\phi_{1,n}(\cdot-x_n),\phi_{2,n}(\cdot-x_n))\|_{H^1\times H^1}\\
&\hspace{4cm}+\|(\phi_{1,n}(\cdot-x_n),\phi_{2,n}(\cdot-x_n))\|_{H^1(B(x_0,1))\times H^1(B(x_0,1))}\\
&\hspace{2cm}=\|(\phi_1,\phi_2)-(\phi_{1,n}(\cdot-x_n),\phi_{2,n}(\cdot-x_n))\|_{H^1\times H^1}\\
&\hspace{4cm}+\|(\phi_{1,n}(\cdot+x_0-x_n),\phi_{2,n}(\cdot+x_0-x_n))\|_{H^1(B(0,1))\times H^1(B(0,1))}\\
&\hspace{2cm}\leq\|(\phi_1,\phi_2)-(\phi_{1,n}(\cdot-x_n),\phi_{2,n}(\cdot-x_n))\|_{H^1\times H^1}+\frac{1}{n}\longrightarrow0\ \ \text{as}\ \ n\rightarrow\infty.
\end{align*}
Hence, $(\phi_1,\phi_2)$ vanishes on $B(x_0,1)\times B(x_0,1)$. Because $x_0$ is arbitrary, $(\phi_1,\phi_2)=(0,0)$. However, this is contradiction. Next, let $\widetilde{K}=\{(\psi_1(p),\psi_2(q)):(p,q)\in K\text{ and }(\psi_1,\psi_2)\text{ satisfies }\eqref{83}\}$. Then, $(\psi_1,\psi_2)\in\widetilde{K}$ satisfies $\|(\psi_1,\psi_2)\|_{H^1(B(0,1))\times H^1(B(0,1))}\geq\varepsilon$. Also, we have $\pi(\widetilde{K})=K$.
We prove that $\widetilde{K}$ is precompact. Let $(\phi_{1,n},\phi_{2,n})$ be a sequence in $\widetilde{K}$. Because $K$ is precompact, if we pass to a subsequence, then there exists $(\phi_1,\phi_2)\in H^1\!\times\!H^1$ and a sequence $\{x_n\}\subset\mathbb{R}^5$ such that
\begin{align}
\lim_{n\rightarrow\infty}\|(\phi_{1,n}(\cdot-x_n),\phi_{2,n}(\cdot-x_n))-(\phi_1,\phi_2)\|_{H^1\times H^1}=0.\label{61}
\end{align}
Also, because $(\phi_{1,n},\phi_{2,n})$ is bounded in $H^1\!\times\!H^1$, we can pass to a subsequence so that
\begin{align}
\lim_{n\rightarrow\infty}\|(\phi_{1,n},\phi_{2,n})\|_{H^1\times H^1}=l\in(0,\infty).\label{62}
\end{align}
Here, we prove that $\{x_n\}$ is bounded. If not, then we can pass to a subsequence so that $|x_n|\longrightarrow\infty$.
\begin{align*}
&\limsup_{n\rightarrow\infty}\|(\phi_{1,n}(\cdot-x_n),\phi_{2,n}(\cdot-x_n))\|_{H^1(B(0,|x_n|-1))\times H^1(B(0,|x_n|-1))}\\
&\hspace{2cm}=\limsup_{n\rightarrow\infty}\|(\phi_{1,n},\phi_{2,n})\|_{H^1(B(-x_n,|x_n|-1))\times H^1(B(-x_n,|x_n|-1))}\\
&\hspace{2cm}\leq\limsup_{n\rightarrow\infty}\|(\phi_{1,n},\phi_{2,n})\|_{H^1\times H^1}-\limsup_{n\rightarrow\infty}\|(\phi_{1,n},\phi_{2,n})\|_{H^1(B(0,1))\times H^1(B(0,1))}\\
&\hspace{2cm}\leq l-\varepsilon.
\end{align*}
We obtain $\|(\phi_1,\phi_2)\|_{H^1\times H^1}\leq l-\varepsilon$. However, this is in contradiction to \eqref{61} and \eqref{62}. Thus, $|x_n|$ is bounded and we can pass to a subsequence so that $x_n\longrightarrow\widetilde{x}\ \ \text{as}\ \ n\rightarrow\infty$. Because $C_0^\infty$ is dense in $H^1$, for any $\varepsilon>0$, there exists $f,\,g\in C_0^\infty$ such that
\[
\|\phi_1-f\|_{H^1}<\varepsilon,\,\|\phi_2-g\|_{H^1}<\varepsilon.
\]
Thus,
\begin{align*}
&\|(\phi_{1,n},\phi_{2,n})-(\phi_1(\cdot+\widetilde{x}),\phi_2(\cdot+\widetilde{x}))\|_{H^1\times H^1}\\
&\hspace{1cm}\leq\|(\phi_{1,n},\phi_{2,n})-(\phi_1,\phi_2)(\cdot+x_n)\|_{H^1\times H^1}+\|(\phi_1,\phi_2)(\cdot+x_n)-(f,g)(\cdot+x_n)\|_{H^1\times H^1}\\
&\hspace{2cm}+\|(f,g)(\cdot+x_n)-(f,g)(\cdot+\widetilde{x})\|_{H^1\times H^1}+\|(f,g)(\cdot+\widetilde{x})-(\phi_1,\phi_2)(\cdot+\widetilde{x})\|_{H^1\times H^1}<\varepsilon.
\end{align*}
Hence, $(\phi_{1,n},\phi_{2,n})$ converge in $H^1\!\times\!H^1$. Therefore, $\widetilde{K}$ is precompact.
\end{proof}


\begin{lemma}\label{Precompactness of the flow of the critical solution 3}
Let $(u,v)$ be the time-global $H^1\!\times\!H^1$ solution to (NLS). Suppose that
\[
K=\pi(\{(u(\cdot,t),v(\cdot,t)):t\in[0,\infty)\})
\]
is precompact in $H^1\!\times\! H^1/{\sim}$. Then there exists $x(t)$, a continuous path in $\mathbb{R}^5$, such that
\[
\{(u(\cdot-x(t),t),v(\cdot-x(t),t)):t\in[0,\infty)\}
\]
is precompact in $H^1\!\times\!H^1$.
\end{lemma}

\begin{proof}
In the case $(u,v)=(0,0)$, the claim holds clearly. Let $(u,v)\neq(0,0)$. Then, it follows that
\[
\|(u(x-x_0,t),v(x-x_0,t))\|_{H^1\times H^1}\geq M(u_0,v_0)>0
\]
for any $(u(x-x_0,t),v(x-x_0,t))\in\pi^{-1}(K)$. We can take $\widetilde{K}$ satisfying $\pi(\widetilde{K})=K$ and being precompact in $H^1\!\times\!H^1$ by Lemma \ref{Precompactness of the flow of the critical solution 2}. Since a map $(u,v):[N,N+1]\ni t\longmapsto\|(u(t),v(t))\|_{H^1\times H^1}\in\mathbb{R}$ is uniformly continuous for each $N\in\mathbb{N}$, there exists $\delta_N>0$
\[
\|(u(\cdot,t),v(\cdot,t))-(u(\cdot,t'),v(\cdot,t'))\|_{H^1\times H^1}\leq\frac{1}{N}
\]
for any $t,t'\in[N,N+1]$ with $|t-t'|\leq\delta_N$. Let a sequence $\{t_n\}$ increase, satisfy $t_n\longrightarrow\infty$, and devide a interval $[N,N+1]$ into less than width $\delta_N$. Since $\pi(\widetilde{K})=K$, we can select $x(t_n)\in\mathbb{R}^5$ with $(u(\cdot-x(t_n),t_n),v(\cdot-x(t_n),t_n))\in\widetilde{K}$ for any $n$. We set that $x(t)$ is a continuous function connecting $x(t_n)$ and $x(t_{n+1})$ with a segment. We will prove that $\{(u(\cdot-x(t),t),v(\cdot-x(t),t)):t\in[0,\infty)\}$ is precompact in $H^1\!\times\!H^1$. We set that a sequence $s_k$ in $[0,\infty)$ satisfies $s_k\longrightarrow s_0$ or $s_k\longrightarrow\infty$ by passing to a subsequence $s_k$. In the case $s_k\longrightarrow s_0$, it follows that $(u(\cdot-x(s_k),s_k),v(\cdot-x(s_k),s_k))\longrightarrow(u(\cdot-x(s_0),s_0),v(\cdot-x(s_0),s_0))$ by the continuity of $(u(t),v(t))$ and $x(t)$. In the case $s_k\longrightarrow\infty$, we can take a unique index $n(k)$ with $t_{n(k)-1}\leq s_k<t_{n(k)}$ for any $k$. Since $\widetilde{K}$ is precompact, $(u(\cdot-x(t_{n(k)-1}),t_{n(k)-1}),v(\cdot-x(t_{n(k)-1}),t_{n(k)-1}))$ and $(u(\cdot-x(t_{n(k)}),t_{n(k)}),v(\cdot-x(t_{n(k)}),t_{n(k)}))$ converge in $H^1\!\times\!H^1$ by passing to a subsequence (with respect to $k$). $(u(\cdot x(t_{n(k)-1}),t_{n(k)}),v(\cdot-x(t_{n(k)-1}),t_{n(k)}))$ converge by the way of taking $t_n$ and the uniform continuity of $(u,v)$. It suffices to prove that $(u(\cdot-x(s_k),t_{n(k)}),v(\cdot-x(s_k),t_{n(k)}))$ has a converging subsequence. Since $(u(\cdot-x(t_{n(k)-1}),t_{n(k)}),v(\cdot-x(t_{n(k)-1}),t_{n(k)}))$ and $(u(\cdot-x(t_{n(k)}),t_{n(k)}),v(\cdot-x(t_{n(k)}),t_{n(k)}))$ converge, $x(t_{n(k)-1})-x(t_{n(k)})$ converges. $x(s_k)-x(t_{n(k)-1})$ converges by passing to a subsequence. Therefore, $(u(\cdot-x(s_k),t_{n(k)}),v(\cdot-x(s_k),t_{n(k)}))$ converges in $H^1\!\times\!H^1$.
\end{proof}


\begin{proposition}[Precompactness of the flow of the critical solution]\label{Precompactness of the flow of the critical solution}
With the same $(u_c,v_c)$ as Proposition \ref{Existence of a critical solution}, there exists a continuous path $x(t)$ in $\mathbb{R}^5$ such that
\[
K=\{(u_c(\cdot-x(t),t),v_c(\cdot-x(t),t)):t\in [0,\infty)\}\subset H^1\times H^1
\]
is precompact in $H^1\!\times\!H^1$.
\end{proposition}


\begin{proof}
We will deduce contradiction by assuming that Proposition \ref{Precompactness of the flow of the critical solution} does not hold. We assume that Propositon \ref{Precompactness of the flow of the critical solution} does not hold. Then, $\pi(\{(u_c(t),v_c(t)):t\in[0,\infty)\})$ is not precompact in $H^1\!\times\!H^1/{\sim}$ by Lemma \ref{Precompactness of the flow of the critical solution 3}, so there exists $\{[(u_c(t_n),v_c(t_n))]\}$ such that any subsequence does not converges, i.e. for any $[(\phi^1,\psi^1)]\in H^1\!\times\!H^1/{\sim}$, there exists $\varepsilon>0$ such that for any $N\in\mathbb{N}$, there exists $n\geq N$ such that
\[
\inf_{x_0\in\mathbb{R}^5}\|(u_c(\cdot-x_0,t_n),v_c(\cdot-x_0,t_n))-(\phi^1,\psi^1)\|_{H^1\times H^1}\geq\varepsilon.
\]
Thus, it suffices to prove that for any sequence $\{t_n\}\subset[0,\infty)$, there exists a subsequence $\{t_n\}$ such that $\{(u_c(\cdot-x(t_n),t_n),v_c(\cdot-x(t_n),t_n))\}$ converges in $H^1\!\times\!H^1$.\\
In the case $\{t_n\}$ is a covergence sequence, it holds because the solution to (NLS) is continuous with respect to $t$. Thus, we consider a $\{t_n\}$ with $t_n\rightarrow\infty$\ \ as\ \ $n\rightarrow\infty$. We set that $(\phi_n,\psi_n)=(u_c(t_n),v_c(t_n))$. Applying Theorem \ref{Existence of a critical solution}, we have
\[
I_\omega(u_{c},v_{c})=I_\omega^c<I_\omega(\phi_\omega,\psi_\omega)\ ,\ \ \ K_\omega^{20,8}(u_c,v_c)>0
\]
for any $0\leq t<\infty$. From \eqref{57}, $\|(\phi_n,\psi_n)\|_{H^1\times H^1}$ is bounded (with respect to $n$). Using Theorem \ref{Linear profile decomposition}, it follows that
\[
(\phi_n,\psi_n)=\sum_{j=1}^M(e^{-it_n^j\Delta}\phi^j(\cdot-x_n^j),e^{-\frac{1}{2}it_n^j\Delta}\psi^j(\cdot-x_n^j))+(\Phi_n^M,\Psi_n^M).
\]
From the same argument as \eqref{58}, we have
\[
0<I_\omega(e^{-it_n^j\Delta}\phi^j,e^{-\frac{1}{2}it_n^j\Delta}\psi^j)\leq I_\omega^c\,,\ \ 0\leq I_\omega(\Phi_n^M,\Psi_n^M)<I_\omega^c,
\]
\[
K_\omega^{20,8}(e^{-it_n^j\Delta}\phi^j,e^{-\frac{1}{2}it_n^j\Delta}\psi^j)>0\,,\ \ \ K_\omega^{20,8}(\Phi_n^M,\Psi_n^M)\geq0
\]
for sufficiently large $n$. (The equal sign holds if and only if $(\Phi_n^M,\Psi_n^M)=(0,0)$.)\\
There exists only one $j$ with $(\phi^j,\psi^j)\neq(0,0)$ by the same argument as the proof of Proposition \ref{Existence of a critical solution}. Thus, we have
\[
(\phi_n,\psi_n)=(e^{-it_n^1\Delta}\phi^1(\cdot-x_n^1),e^{-\frac{1}{2}it_n^1\Delta}\psi^1(\cdot-x_n^1))+(\Phi_n^1,\Psi_n^1).
\]
Also, applying Corollary \ref{I decomposition},
\[
I_\omega^c=\lim_{n\rightarrow\infty}I_\omega(e^{-it_n^1\Delta}\phi^1,e^{-\frac{1}{2}it_n^1\Delta}\psi^1)+\lim_{n\rightarrow\infty}I_\omega(\Phi_n^1,\Psi_n^1).
\]
Here, we assume that $\displaystyle\lim_{n\rightarrow\infty}I_\omega(e^{-it_n^1\Delta}\phi^1,e^{-\frac{1}{2}it_n^1\Delta}\psi^1)<I_\omega^c$. Also, we express the corresponding solution to (NLS) with initial data $(u_0,v_0)$ in $\text{NLS}[(u_0,v_0)](t)=(\text{NLS}[u_0](t),\text{NLS}[v_0](t))$. Then, we have $\|\text{NLS}[(e^{-it_n^1\Delta}\phi^1,e^{-\frac{1}{2}it_n^1\Delta}\psi^1)](\cdot-t_n)\|_{S(\dot{H}^\frac{1}{2})\times S(\dot{H}^\frac{1}{2})}<\infty$ by the definition of $I_\omega^c$. Futhermore, it follows that
\begin{align*}
&\|e^{i(t-t_n)\Delta}(u_c(t_n)-\text{NLS}[e^{-it_n^1\Delta}\phi^1(\cdot-x_n^1)](t_n-t_n))\|_{S(\dot{H}^\frac{1}{2})}\\
&\hspace{4cm}=\|e^{i(t-t_n)\Delta}(e^{-it_n^1\Delta}\phi^1(\cdot-x_n^1)+\Phi_n^1-e^{-it_n^1\Delta}\phi^1(\cdot-x_n^1))\|_{S(\dot{H}^\frac{1}{2})}\\
&\hspace{4cm}=\|e^{i(t-t_n)\Delta}\Phi_n^1\|_{S(\dot{H}^\frac{1}{2})}=\|e^{it\Delta}\Phi_n^1\|_{S(\dot{H}^\frac{1}{2})}\longrightarrow0\ \ \text{as}\ \ n\rightarrow\infty.
\end{align*}
Similarly,
\[
\|e^{\frac{1}{2}i(t-t_n)\Delta}(v_c(t_n)-\text{NLS}[e^{-\frac{1}{2}it_n^1\Delta}\psi^1(\cdot-x_n^1)](t_n-t_n))\|_{S(\dot{H}^\frac{1}{2})}=\|e^{\frac{1}{2}it\Delta}\Psi_n^1\|_{S(\dot{H}^\frac{1}{2})}\longrightarrow0\ \ \text{as}\ \ n\rightarrow\infty.
\]
Thus, we obtain $\|(u_c,v_c)\|_{S(\dot{H}^\frac{1}{2})\times S(\dot{H}^\frac{1}{2})}<\infty$ by Theorem \ref{Long time perturbation}, but this is in contradiction to Proposition \ref{Existence of a critical solution}. Therefore,
\[
\lim_{n\rightarrow\infty}I_\omega(e^{-it_n^1\Delta}\phi^1,e^{-\frac{1}{2}it_n^1\Delta}\psi^1)=I_\omega^c\,,\ \ \ \lim_{n\rightarrow\infty}I_\omega(\Phi_n^1,\Psi_n^1)=0.
\]
Combining this fact and Lemma \ref{Comparability of K and I},
\[
\lim_{n\rightarrow\infty}K_\omega(\Phi_n^1,\Psi_n^1)\leq10I_\omega(\Phi_n^1,\Psi_n^1)=0,
\]
i.e.
\[
\lim_{n\rightarrow\infty}\|(\Phi_n^1,\Psi_n^1)\|_{H^1\times H^1}=0.
\]
Next, we will prove $t_n^1=0$.\\
We assume that $t_n^1\rightarrow-\infty$. Then,
\begin{align*}
&\left\|\left(e^{it\Delta}\phi_n,e^{\frac{1}{2}it\Delta}\psi_n\right)\right\|_{S(\dot{H}^\frac{1}{2}:[0,\infty))\times S(\dot{H}^\frac{1}{2}:[0,\infty))}\\
&\hspace{0.5cm}=\left\|\left(e^{i(t-t_n^1)\Delta}\phi^1(\cdot-x_n^1),e^{\frac{1}{2}i(t-t_n^1)\Delta}\psi^1(\cdot-x_n^1)\right)+\left(e^{it\Delta}\Phi_n^1,e^{\frac{1}{2}it\Delta}\Psi_n^1\right)\right\|_{S(\dot{H}^\frac{1}{2}:[0,\infty))\times S(\dot{H}^\frac{1}{2}:[0,\infty))}\\
&\hspace{0.5cm}\leq\left\|\left(e^{i(t-t_n^1)\Delta}\phi^1,e^{\frac{1}{2}i(t-t_n^1)\Delta}\psi^1\right)\right\|_{S(\dot{H}^\frac{1}{2}:[0,\infty))\times S(\dot{H}^\frac{1}{2}:[0,\infty))}\\
&\hspace{8cm}+\left\|\left(e^{it\Delta}\Phi_n^1,e^{\frac{1}{2}it\Delta}\Psi_n^1\right)\right\|_{S(\dot{H}^\frac{1}{2}:[0,\infty))\times S(\dot{H}^\frac{1}{2}:[0,\infty))}.
\end{align*}
Here,
\begin{align*}
&\lim_{n\rightarrow\infty}\left\|\left(e^{i(t-t_n^1)\Delta}\phi^1,e^{\frac{1}{2}i(t-t_n^1)\Delta}\psi^1\right)\right\|_{S(\dot{H}^\frac{1}{2}:[0,\infty))\times S(\dot{H}^\frac{1}{2}:[0,\infty))}\\
&\hspace{5cm}=\lim_{n\rightarrow\infty}\left\|\left(e^{it\Delta}\phi^1,e^{\frac{1}{2}it\Delta}\psi^1\right)\right\|_{S(\dot{H}^\frac{1}{2}:[-t_n^1,\infty))\times S(\dot{H}^\frac{1}{2}:[-t_n^1,\infty))}=0
\end{align*}
and $\left\|\left(e^{it\Delta}\Phi_n^1,e^{\frac{1}{2}it\Delta}\Psi_n^1\right)\right\|_{S(\dot{H}^\frac{1}{2}:[0,\infty))\times S(\dot{H}^\frac{1}{2}:[0,\infty))}\leq\frac{1}{4}\delta_{sd}$ for sufficiently large $n\in\mathbb{N}$. These formulas are in contradiction to Theorem \ref{Small data globally existence}. Hence, we have $t_n^1=0$.
\begin{align*}
\left\|(u_c(\cdot+x_n^1,t_n),v_c(\cdot+x_n^1,t_n))-(\phi^1,\psi^1)\right\|_{H^1\times H^1}&=\left\|(\phi_n,\psi_n)-(\phi^1(\cdot-x_n^1),\psi^1(\cdot-x_n^1))\right\|_{H^1\times H^1}\\
&=\left\|(\Phi_n^1,\Psi_n^1)\right\|_{H^1\times H^1}\\
&\longrightarrow0\ \ \text{as}\ \ n\rightarrow\infty.
\end{align*}
Thus, $\{(u_c(\cdot-x(t_n),t_n),v_c(\cdot-x(t_n),t_n))\}$ converges in $H^1\!\times\!H^1$ and we set that $(\phi^1,\psi^1)\in H^1\!\times\!H^1$ denotes the limit. Then,
\begin{align*}
&\inf_{x_0\in\mathbb{R}^5}\|(u_c(\cdot-x_0,t_n),v_c(\cdot-x_0,t_n))-(\phi^1,\phi^2)\|_{H^1\times H^1}\\
&\hspace{2cm}\leq\|(u_c(\cdot-x(t_n),t_n),v_c(\cdot-x(t_n),t_n))-(\phi^1,\phi^2)\|_{H^1\times H^1}\longrightarrow0\ \ \text{as}\ \ n\rightarrow\infty.
\end{align*}
Therefore, $\{(u_c(\cdot-x(t),t),v_c(\cdot-x(t),t))\}$ is precompact.
\end{proof}


\begin{lemma}[Precompactness of the flow implies uniform localization]\label{Precompactness of the flow implies uniform localization}
Let $(u,v)$ be the solution to (NLS) such that
\[
K=\{(u(\cdot-x(t),t),v(\cdot-x(t),t)):t\in [0,\infty)\}
\]
is precompact in $H^1\!\times\!H^1$. Then for each $\varepsilon>0$, there exists $R>0$ so that
\[
\int_{|x+x(t)|>R}\left(|u(x,t)|^2+|v(x,t)|^2+|\nabla u(x,t)|^2+|\nabla v(x,t)|^2+|v(x,t)u(x,t)^2|\right)dx\leq\varepsilon
\]
for all $0\leq t<\infty$.
\end{lemma}


\begin{proof}
We will prove by contradiction.\\
We assume that there exists $\varepsilon>0$ such that for any $n\in\mathbb{N}$, there exists $t_n$ such that
\begin{align}
\int_{|x+x(t_n)|>n}\left(|u(x,t_n)|^2+|v(x,t_n)|^2+|\nabla u(x,t_n)|^2+|\nabla v(x,t_n)|^2+|v(x,t_n)u(x,t_n)^2|\right)dx>\varepsilon.\label{35}
\end{align}
By changing  variables,
\begin{align*}
&\int_{|x|>n}\left(|u(x-x(t_n),t_n)|^2+|v(x-x(t_n),t_n)|^2+|\nabla u(x-x(t_n),t_n)|^2\right.\\
&\left.\hspace{3cm}+|\nabla v(x-x(t_n),t_n)|^2+|v(x-x(t_n),t_n)u(x-x(t_n),t_n)^2|\right)dx>\varepsilon.
\end{align*}
Since $K$ is precompact, there exists $(U,V)\in H^1\!\times\!H^1$ such that
\[
\left(u(\cdot-x(t_n),t_n),v(\cdot-x(t_n),t_n)\right)\longrightarrow(U,V)\ \ \text{as}\ \ n\rightarrow\infty\ \ \text{in}\ \ H^1\!\times\!H^1.
\]
 by passing to a subsequence of $\{t_n\}$. From $(U,V)\in H^1\!\times\!H^1\hookrightarrow L^3\!\times\!L^3$, we have 
\[
\|U\|_{L^2(|x|>n)}<\frac{\sqrt{\varepsilon}}{2\sqrt{5}}\,,\ \ \|V\|_{L^2(|x|>n)}<\frac{\sqrt{\varepsilon}}{2\sqrt{5}}\,,\ \ \|\nabla U\|_{L^2(|x|>n)}<\frac{\sqrt{\varepsilon}}{2\sqrt{5}},
\]
\[
\|\nabla V\|_{L^2(|x|>n)}<\frac{\sqrt{\varepsilon}}{2\sqrt{5}}\,,\ \ \|U\|_{L^3(|x|>n)}<\frac{\sqrt[3]{\varepsilon}}{2\sqrt[3]{5}}\,,\ \ \|V\|_{L^3(|x|>n)}<\frac{\sqrt[3]{\varepsilon}}{2\sqrt[3]{5}}.
\]
by taking sufficiently large $n$. Also, since $(u(\cdot-x(t_n),t_n),v(\cdot-x(t_n),t_n))\longrightarrow(U,V)\ \ \text{as}\ \ n\rightarrow\infty\ \ \text{in}\ \ H^1\!\times\!H^1$, if we take sufficiently large $n$, then
\[
\|u(\cdot-x(t_n),t_n)-U\|_{L^2}<\frac{\sqrt{\varepsilon}}{2\sqrt{5}}\,,\ \ \|v(\cdot-x(t_n),t_n)-V\|_{L^2}<\frac{\sqrt{\varepsilon}}{2\sqrt{5}},
\]
\[
\|\nabla(u(\cdot-x(t_n),t_n)-U)\|_{L^2}<\frac{\sqrt{\varepsilon}}{2\sqrt{5}}\,,\ \ \|\nabla(v(\cdot-x(t_n),t_n)-V)\|_{L^2}<\frac{\sqrt{\varepsilon}}{2\sqrt{5}},
\]
\[
\|u(\cdot-x(t_n),t_n)-U\|_{L^3}<\frac{\sqrt[3]{\varepsilon}}{2\sqrt[3]{5}}\,,\ \ \|v(\cdot-x(t_n),t_n)-V\|_{L^3}<\frac{\sqrt[3]{\varepsilon}}{2\sqrt[3]{5}}.
\]
Thus,
\begin{align*}
\|u(\cdot,t_n)\|_{L^2(|x+x(t_n)|>n)}&=\|u(\cdot-x(t_n),t_n)\|_{L^2(|x|>n)}\\
&\leq\|u(\cdot-x(t_n),t_n)-U\|_{L^2}+\|U\|_{L^2(|x|>n)}\\
&<\frac{\sqrt{\varepsilon}}{2\sqrt{5}}+\frac{\sqrt{\varepsilon}}{2\sqrt{5}}=\frac{\sqrt{\varepsilon}}{\sqrt{5}},
\end{align*}
\begin{align*}
\|\nabla u(\cdot,t_n)\|_{L^2(|x+x(t_n)|>n)}&=\|\nabla u(\cdot-x(t_n),t_n)\|_{L^2(|x|>n)}\\
&\leq\|\nabla(u(\cdot-x(t_n),t_n)-U)\|_{L^2}+\|\nabla U\|_{L^2(|x|>n)}\\
&<\frac{\sqrt{\varepsilon}}{2\sqrt{5}}+\frac{\sqrt{\varepsilon}}{2\sqrt{5}}=\frac{\sqrt{\varepsilon}}{\sqrt{5}},
\end{align*}
\begin{align*}
&\int_{|x+x(t_n)|>n}|v(x,t)u(x,t)^2|dx\\
&\hspace{2cm}\leq\|v(\cdot,t)\|_{L^3(|x+x(t_n)|>n)}\|u(\cdot,t)\|_{L^3(|x+x(t_n)|>n)}^2\\
&\hspace{2cm}=\|v(\cdot-x(t_n),t)\|_{L^3(|x|>n)}\|u(\cdot-x(t_n),t)\|_{L^3(|x|>n)}^2\\
&\hspace{2cm}\leq\left(\|v(\cdot-x(t_n),t)-V\|_{L^3}+\|V\|_{L^3(|x|>n)}\right)\left(\|u(\cdot-x(t_n),t)-U\|_{L^3}+\|U\|_{L^3(|x|>n)}\right)^2\\
&\hspace{2cm}<\left(\frac{\sqrt[3]{\varepsilon}}{2\sqrt[3]{5}}+\frac{\sqrt[3]{\varepsilon}}{2\sqrt[3]{5}}\right)\left(\frac{\sqrt[3]{\varepsilon}}{2\sqrt[3]{5}}+\frac{\sqrt[3]{\varepsilon}}{2\sqrt[3]{5}}\right)^2=\frac{\varepsilon}{5}.
\end{align*}
Therefore,
\[
\int_{|x+x(t_n)|>n}\left(|u(x,t_n)|^2+|v(x,t_n)|^2+|\nabla u(x,t_n)|^2+|\nabla v(x,t_n)|^2+|v(x,t_n)u(x,t_n)^2|\right)dx\leq\varepsilon.
\]
Since this is in contradiction to \eqref{35}, for any $\varepsilon>0$, there exists $R>0$ such that for any $0\leq t<\infty$,
\[
\int_{|x+x(t)|>R}\left(|u(x,t)|^2+|v(x,t)|^2+|\nabla u(x,t)|^2+|\nabla v(x,t)|^2+|v(x,t)u(x,t)^2|\right)dx\leq\varepsilon.
\]
\end{proof}

\begin{proposition}
The solution $(u,v)$ to (NLS) satisfies the following conservation law.
\[
\widetilde{P}(u(t),v(t))\vcentcolon=\text{Im}\int_{\mathbb{R}^5}\left(\nabla u\overline{u}+\nabla v\overline{v}\right)dx=\widetilde{P}(u_0,v_0).
\]
\end{proposition}

\begin{proof}
Applying Lemma \ref{Localized virial identity} as $\chi(x)=x_j$ for $j\in\{1,\cdots,5\}$,
\[
\text{Im}\frac{d}{dt}\int_{\mathbb{R}^5}\left(u_j\overline{u}+v_j\overline{v}\right)dx=0.
\]
Thus,
\[
\text{Im}\frac{d}{dt}\int_{\mathbb{R}^5}\left(\nabla u\overline{u}+\nabla v\overline{v}\right)dx=0
\]
and we have
\[
\widetilde{P}(u(t),v(t))=\text{Im}\int_{\mathbb{R}^5}\left(\nabla u\overline{u}+\nabla v\overline{v}\right)dx=\widetilde{P}(u_0,v_0).
\]
\end{proof}


\begin{proposition}[Galilean transformation]
If $(u,v)$ satisfies (NLS), then
\[
(e^{ix\cdot\xi_0}e^{-it|\xi_0|^2}u(x-2\xi_0t,t),e^{2ix\cdot\xi_0}e^{-2it|\xi_0|^2}v(x-2\xi_0t,t))
\]
also satisfies (NLS).
\end{proposition}

\begin{proof}
Let $(u,v)$ solve (NLS), and\[
(z,w)=(e^{ix\cdot\xi_0}e^{-it|\xi_0|^2}u(x-2\xi_0t,t),e^{2ix\cdot\xi_0}e^{-2it|\xi_0|^2}v(x-2\xi_0t,t)).
\]
We obtain the following formula by a direct calculation:
\begin{align*}
&\partial_tz=-i|\xi_0|^2e^{ix\cdot\xi_0}e^{-it|\xi_0|^2}u(x-2\xi_0t,t)\\
&\hspace{2cm}-2e^{ix\cdot\xi_0}e^{-it|\xi_0|^2}\sum_{j=1}^5\xi_{0,j}(\partial_ju)(x-2\xi_0t,t)+e^{ix\cdot\xi_0}e^{-it|\xi_0|^2}(\partial_tu)(x-2\xi_0t,t).
\end{align*}
From
\begin{align*}
&\partial_j^2z=-\xi_{0,j}^2e^{ix\cdot\xi_0}e^{-it|\xi_0|^2}u(x-2\xi_0t,t)\\
&\hspace{2cm}+2i\xi_{0,j}e^{ix\cdot\xi_0}e^{-it|\xi_0|^2}(\partial_ju)(x-2\xi_0t,t)+e^{ix\cdot\xi_0}e^{-it|\xi_0|^2}(\partial_j^2u)(x-2\xi_0t,t),
\end{align*}
we obtain
\begin{align*}
&\Delta z=-|\xi_0|^2e^{ix\cdot\xi_0}e^{-it|\xi_0|^2}u(x-2\xi_0t,t)\\
&\hspace{2cm}+2ie^{ix\cdot\xi_0}e^{-it|\xi_0|^2}\sum_{j=1}^5\xi_{0,j}(\partial_ju)(x-2\xi_0t,t)+e^{ix\cdot\xi_0}e^{-it|\xi_0|^2}(\Delta u)(x-2\xi_0t,t).
\end{align*}
Thus,
\begin{align*}
i\partial_tz+\Delta z+2w\overline{z}&=|\xi_0|^2e^{ix\cdot\xi_0}e^{-it|\xi_0|^2}u(x-2\xi_0t,t)-2ie^{ix\cdot\xi_0}e^{-it|\xi_0|^2}\sum_{j=1}^5\xi_{0,j}(\partial_ju)(x-2\xi_0t,t)\\
&+ie^{ix\cdot\xi_0}e^{-it|\xi_0|^2}(\partial_tu)(x-2\xi_0t,t)-|\xi_0|^2e^{ix\cdot\xi_0}e^{-it|\xi_0|^2}u(x-2\xi_0t,t)\\
&+2ie^{ix\cdot\xi_0}e^{-it|\xi_0|^2}\sum_{j=1}^5\xi_{0,j}(\partial_ju)(x-2\xi_0t,t)+e^{ix\cdot\xi_0}e^{-it|\xi_0|^2}(\Delta u)(x-2\xi_0t,t)\\
&+2e^{2ix\cdot\xi_0}e^{-2it|\xi_0|^2}v(x-2\xi_0t,t)e^{-ix\cdot\xi_0}e^{it|\xi_0|^2}\overline{u}(x-2\xi_0t,t)=0,
\end{align*}
and
\begin{align*}
i\partial_tw+\frac{1}{2}\Delta w+z^2&=2|\xi_0|^2e^{2ix\cdot\xi_0}e^{-2it|\xi_0|^2}v(x-2\xi_0t,t)-2ie^{2ix\cdot\xi_0}e^{-2it|\xi_0|^2}\sum_{j=1}^5\xi_{0,j}(\partial_jv)(x-2\xi_0t,t)\\
&+ie^{2ix\cdot\xi_0}e^{-2it|\xi_0|^2}(\partial_tv)(x-2\xi_0t,t)-2|\xi_0|^2e^{2ix\cdot\xi_0}e^{-2it|\xi_0|^2}v(x-2\xi_0t,t)\\
&+2ie^{2ix\cdot\xi_0}e^{-2it|\xi_0|^2}\sum_{j=1}^5\xi_{0,j}(\partial_jv)(x-2\xi_0t,t)+\frac{1}{2}e^{2ix\cdot\xi_0}e^{-2it|\xi_0|^2}(\Delta v)(x-2\xi_0t,t)\\
&+e^{2ix\cdot\xi_0}e^{-2it|\xi_0|^2}u(x-2\xi_0t,t)^2=0.
\end{align*}
Therefore, $(z,w)$ solves (NLS).
\end{proof}

\begin{proposition}\label{zero momentum}
Let $(u_c,v_c)$ be the critical solution constructed in Proposition \ref{Existence of a critical solution}. Then $\widetilde{P}(u_c,v_c)=0$.
\end{proposition}

\begin{proof}
We assume $\widetilde{P}(u_c,v_c)\neq0$. Let $(z_c,w_c)=(e^{ix\cdot\xi_0}e^{-it|\xi_0|^2}u_c(x-2\xi_0t,t),e^{2ix\cdot\xi_0}e^{-2it|\xi_0|^2}v_c(x-2\xi_0t,t))$. We obtain the following formula by a direct calculation:
\[
\|z_c\|_{L^2}^2=\|u_c\|_{L^2}^2\,,\ \ \ \|w_c\|_{L^2}^2=\|v_c\|_{L^2}^2.
\]
\begin{align*}
\|\nabla z_c\|_{L^2}^2&=\sum_{j=1}^5\|\partial_jz_c\|_{L^2}^2\\
&=\sum_{j=1}^5\|i\xi_{0,j}e^{ix\cdot\xi_0}e^{-it|\xi_0|^2}u_c(\cdot-2\xi_0t,t)+e^{ix\cdot\xi_0}e^{-it|\xi_0|^2}(\partial_ju_c)(\cdot-2\xi_0t,t)\|_{L^2}^2\\
&=\sum_{j=1}^5|\xi_{0,j}|^2\|u_c\|_{L^2}^2+\sum_{j=1}^5\|\partial_ju_c\|_{L^2}^2+\sum_{j=1}^52\text{Re}(\partial_ju_c,i\xi_{0,j}u_c)_{L^2}\\
&=|\xi_0|^2\|u_c\|_{L^2}^2+\|\nabla u_c\|_{L^2}^2+2\text{Im}\sum_{j=1}^5\xi_{0,j}(\partial_ju_c,u_c)_{L^2},\\
\|\nabla w_c\|_{L^2}^2&=\sum_{j=1}^5\|\partial_jw_c\|_{L^2}^2\\
&=\sum_{j=1}^5\|2i\xi_{0,j}e^{2ix\cdot\xi_0}e^{-2it|\xi_0|^2}v_c(\cdot-2\xi_0t,t)+e^{2ix\cdot\xi_0}e^{-2it|\xi_0|^2}(\partial_jv_c)(\cdot-2\xi_0t,t)\|_{L^2}^2\\
&=\sum_{j=1}^54|\xi_{0,j}|^2\|v_c\|_{L^2}^2+\sum_{j=1}^5\|\partial_jv_c\|_{L^2}^2+\sum_{j=1}^54\text{Re}(\partial_jv_c,i\xi_{0,j}v_c)_{L^2}\\
&=4|\xi_0|^2\|v_c\|_{L^2}^2+\|\nabla v_c\|_{L^2}^2+4\text{Im}\sum_{j=1}^5\xi_{0,j}(\partial_jv_c,v_c)_{L^2}.
\end{align*}
Thus,
\begin{align*}
K(z_c,w_c)&=\|\nabla z_c\|_{L^2}^2+\frac{1}{2}\|\nabla w_c\|_{L^2}^2\\
&=|\xi_0|^2(\|u_c\|_{L^2}^2+2\|v_c\|_{L^2}^2)+\|\nabla u_c\|_{L^2}^2+\frac{1}{2}\|\nabla v_c\|_{L^2}^2+2\xi_0\cdot\widetilde{P}(u_c,v_c)\\
&=|\xi_0|^2M(u_c,v_c)+2\xi_0\cdot\widetilde{P}(u_c,v_c)+K(u_c,v_c),
\end{align*}
and
\[
P(z_c,w_c)=\text{Re}\int_{\mathbb{R}^5}w_c\overline{z_c}^2dx=\text{Re}\int_{\mathbb{R}^5}v_c\overline{u_c}^2dx=P(u_c,v_c).
\]
Therefore,
\begin{align*}
I_\omega(z_c,w_c)&=\frac{\omega}{2}M(z_c,w_c)+\frac{1}{2}K(z_c,w_c)-P(z_c,w_c)\\
&=\frac{\omega}{2}M(u_c,v_c)+\frac{1}{2}|\xi_0|^2M(u_c,v_c)+\xi_0\cdot\widetilde{P}(u_c,v_c)+\frac{1}{2}K(u_c,v_c)-P(u_c,v_c)\\
&=I_\omega(u_c,v_c)+\frac{1}{2}\left(|\xi_0|^2M(u_c,v_c)+2\xi_0\cdot\widetilde{P}(u_c,v_c)\right),\\
K_\omega^{20,8}(z_c,w_c)&=8K(z_c,w_c)-20P(z_c,w_c)\\
&=8|\xi_0|^2M(u_c,v_c)+16\xi_0\cdot\widetilde{P}(u_c,v_c)+8K(u_c,v_c)-20P(u_c,v_c)\\
&=K_\omega^{20,8}(u_c,v_c)+8\left(|\xi_0|^2M(u_c,v_c)+2\xi_0\cdot\widetilde{P}(u_c,v_c)\right).
\end{align*}
Here, if we combine
\[
|\xi_0|^2M(u_c,v_c)+2\xi_0\cdot\widetilde{P}(u_c,v_c)=M(u_c,v_c)\left|\xi_0+\frac{\widetilde{P}(u_c,v_c)}{M(u_c,v_c)}\right|^2-\frac{\widetilde{P}(u_c,v_c)^2}{M(u_c,v_c)}
\]
and
\[
0<I_\omega(u_c,v_c)=I_\omega^c\,,\ \ \ K_\omega(u_c,v_c)>0,
\]
then we can take $\xi_0$ with
\[
I_\omega(z_c,w_c)<I_\omega^c\,,\ \ \ K_\omega(z_c,w_c)>0.
\]
On the other hand, we obtain $\|(z_c,w_c)\|_{S(\dot{H}^\frac{1}{2})\times S(\dot{H}^\frac{1}{2})}=\infty$ from $\|(u_c,v_c)\|_{S(\dot{H}^\frac{1}{2})\times S(\dot{H}^\frac{1}{2})}=\infty$. However, this is in contradiction to the definition of $I_\omega^c$. Therefore, we have $\widetilde{P}(u_c,v_c)=0$.
\end{proof}

\begin{lemma}\label{order of x(t)}
Let $(u,v)$ be a solution to (NLS) defined on $[0,\infty)$ so that $\widetilde{P}(u,v)=0$ and $K=\{(u(\cdot-x(t),t),v(\cdot-x(t),t))\}$ is precompact in $H^1\!\times\!H^1$ for some continuous function $x(\cdot)$. Then
\[
\frac{x(t)}{t}\rightarrow0\ \ \text{as}\ \ t\rightarrow\infty.
\]
\end{lemma}

We prove this lemma by the argument of Fang-Xie-Cazenave \cite{005}.

\begin{proof}
We assume that we do not have Lemma \ref{order of x(t)}. Then, there exist $\delta>0$ and a sequence $t_n\rightarrow\infty$ such that $|x(t_n)|\geq\delta t_n$. Without loss of generality, we assume that $x(0)=0$. We set
\[
\tau_n=\inf\{t\geq0:|x(t)|\geq|x(t_n)|\}.
\]
Since $0<\tau_n\leq t_n$ and $|x(\tau_n)|=|x(t_n)|$, it follows that
\[
\tau_n\rightarrow\infty\ \ \text{as}\ \ n\rightarrow\infty,
\]
\[
|x(t)|<|x(t_n)|\,,\ \ \ 0\leq t<\tau_n,
\]
\[
|x(\tau_n)|\geq\delta\tau_n.
\]
Let $\chi\in C_0^\infty(\mathbb{R}^5)$ be radial with
\begin{equation}
\notag \chi(r)=
\begin{cases}
\hspace{-0.4cm}&\displaystyle{\ \ \ \ 1\ \ \ \ \ \ \,(0\leq r\leq1),}\\
\hspace{-0.4cm}&\displaystyle{smooth\ \ (1\leq r\leq 2),}\\
\hspace{-0.4cm}&\displaystyle{\ \ \ \ 0\ \ \ \ \ \ \,(2\leq r),}
\end{cases}
\end{equation}
where $r=|x|$. Also, let $\chi$ satisfy $|\chi'(r)|\leq2\ \ (r\geq0)$. We define $\chi_R(r)=\chi(\frac{r}{R})$ for $R>0$. Then,
\[
|\chi_R(r)-1|+|r||\chi'_R(r)|\leq1_{\{R\leq r\}}+|r|\left|\frac{1}{R}\chi'\left(\frac{r}{R}\right)\right|\leq1_{\{R\leq r\}}+4\cdot1_{\{R\leq r\leq2R\}}\leq5\cdot1_{\{R\leq r\}},
\]
where we define that a function $1_A$ is $1_A=1$ $(x\in A)$ and $1_A=0$ $(x\notin A)$ for a set A. Also,
\[
|r||\chi_R(r)|\leq2R.
\]
We define
\[
z_R(t)=\int_{\mathbb{R}^5}x\chi_R(|x|)\left(|u(t,x)|^2+2|v(t,x)|^2\right)dx
\]
for $R>0$. By Lemma \ref{Localized virial identity} \eqref{09}, $i$-th component of $z_R'(t)$ for $1\leq i\leq5$ is
\begin{align*}
2\text{Im}\int_{\mathbb{R}^5}\nabla(x_i\chi_R)\cdot(\nabla u\overline{u}+\nabla v\overline{v})&dx=2\text{Im}\sum_{j=1}^5\int_{\mathbb{R}^5}\partial_j(x_i\chi_R)\left(u_j\overline{u}+v_j\overline{v}\right)dx\\
=2\text{Im}&\int_{\mathbb{R}^5}\chi_R\left(u_i\overline{u}+v_i\overline{v}\right)dx+2\text{Im}\sum_{j=1}^5\int_{\mathbb{R}^5}\left(x_i\chi_R'\frac{x_j}{r}\right)\left(u_j\overline{u}+v_j\overline{v}\right)dx\\
=2\text{Im}&\int_{\mathbb{R}^5}\left\{\chi_R\left(u_i\overline{u}+v_i\overline{v}\right)+\left(x_i\chi_R'\frac{x}{r}\right)\cdot\left(\nabla u\overline{u}+\nabla v\overline{v}\right)\right\}dx.
\end{align*}
Thus,
\[
z_R'(t)=2\text{Im}\int_{\mathbb{R}^5}\left\{\chi_R\left(\nabla u\overline{u}+\nabla v\overline{v}\right)+\frac{x\chi_R'}{r}x\cdot\left(\nabla u\overline{u}+\nabla v\overline{v}\right)\right\}dx.
\]
Since $\widetilde{P}(u,v)=0$, it follows that
\[
z_R'(t)=2\text{Im}\int_{\mathbb{R}^5}\left\{\left(\chi_R-1\right)\left(\nabla u\overline{u}+\nabla v\overline{v}\right)+\frac{x\chi_R'}{r}x\cdot\left(\nabla u\overline{u}+\nabla v\overline{v}\right)\right\}dx.
\]
Therefore,
\begin{align*}
|z_R'(t)|&\leq2\int_{\mathbb{R}^5}\left(\left|\chi_R-1\right|+|r||\chi_R'|\right)\left(|\nabla u||\overline{u}|+|\nabla v||\overline{v}|\right)dx\\
&\leq10\int_{\{R\leq |x|\}}\left(|\nabla u||\overline{u}|+|\nabla v||\overline{v}|\right)dx\\
&\leq5\int_{\{R\leq |x|\}}\left(|\nabla u|^2+|u|^2+|\nabla v|^2+|v|^2\right)dx.
\end{align*}
On the other hand, by Lemma \ref{Precompactness of the flow implies uniform localization}, there exists $\rho>0$ such that for any $0\leq t<\infty$,
\[
\int_{\{|x+x(t)|>\rho\}}\left(|\nabla u|^2+|u|^2+|\nabla v|^2+2|v|^2\right)dx\leq\frac{\delta M(u,v)}{10(1+\delta)}.
\]
Let $R_n=|x(\tau_n)|+\rho$. Since for given $0\leq t\leq\tau_n$ and $|x|>R_n$,
\[
|x+x(t)|\geq R_n-|x(t)|\geq R_n-|x(\tau_n)|=\rho,
\]
we obtain
\begin{align}
|z'_{R_n}(t)|\leq5\int_{\{|x+x(t)|>\rho\}}\left(|\nabla u|^2+|u|^2+|\nabla v|^2+|v|^2\right)dx\leq\frac{\delta M(u,v)}{2(1+\delta)}.\label{63}
\end{align}
for any $n\in\mathbb{N}$ and $0\leq t\leq\tau_n$. Also, since $R_n\geq\rho$ and $x(0)=0$,
\begin{align}
|z_{R_n}(0)|&\leq\int_{\{|x|<\rho\}}|x|\chi_{R_n}(|u_0|^2+2|v_0|^2)dx+\int_{\{|x|>\rho\}}|x|\chi_{R_n}(|u_0|^2+2|v_0|^2)dx \notag \\
&=\int_{\{|x|<\rho\}}|x|(|u_0|^2+2|v_0|^2)dx+\int_{\{2R_n>|x+x(0)|>\rho\}}|x|\chi_{R_n}(|u_0|^2+2|v_0|^2)dx \notag \\
&\leq\rho M(u,v)+\frac{\delta M(u,v)}{5(1+\delta)}R_n.\label{64}
\end{align}
Next, we estimate $z_{R_n}(\tau_n)$.
\begin{align*}
z_{R_n}(\tau_n)&=\int_{\{|x+x(\tau_n)|>\rho\}}x\chi_{R_n}\left(|u(\tau_n,x)|^2+2|v(\tau_n,x)|^2\right)dx\\
&\hspace{5cm}+\int_{\{|x+x(\tau_n)|<\rho\}}x\chi_{R_n}\left(|u(\tau_n,x)|^2+2|v(\tau_n,x)|^2\right)dx\\
&=\vcentcolon\Rnum{1}+\Rnum{2}.
\end{align*}
We have
\[
|\Rnum{1}|\leq\frac{\delta M(u,v)}{5(1+\delta)}R_n.
\]
If $|x+x(\tau_n)|<\rho$, then we have $|x|\leq|x+x(\tau_n)|+|x(\tau_n)|\leq\rho+|x(\tau_n)|=R_n$. Thus,
\begin{align*}
-\Rnum{2}&=-\int_{\{|x+x(\tau_n)|<\rho\}}x\left(|u(\tau_n,x)|^2+2|v(\tau_n,x)|^2\right)dx\\
&=-\int_{\{|x+x(\tau_n)|<\rho\}}(x+x(\tau_n))\left(|u(\tau_n,x)|^2+2|v(\tau_n,x)|^2\right)dx\\
&\hspace{3.5cm}+x(\tau_n)\int_{\{|x+x(\tau_n)|<\rho\}}\left(|u(\tau_n,x)|^2+2|v(\tau_n,x)|^2\right)dx\\
&=x(\tau_n)M(u,v)-\int_{\{|x+x(\tau_n)|<\rho\}}(x+x(\tau_n))\left(|u(\tau_n,x)|^2+2|v(\tau_n,x)|^2\right)dx\\
&\hspace{3.5cm}-x(\tau_n)\int_{\{|x+x(\tau_n)|>\rho\}}\left(|u(\tau_n,x)|^2+2|v(\tau_n,x)|^2\right)dx.
\end{align*}
Hence,
\[
|\Rnum{2}|\geq|x(\tau_n)|M(u,v)-\rho M(u,v)-\frac{\delta M(u,v)}{10(1+\delta)}R_n.
\]
Therefore,
\begin{align}
|z_{R_n}(\tau_n)|\geq-|\Rnum{1}|+|\Rnum{2}|\geq|x(\tau_n)|M(u,v)-\rho M(u,v)-\frac{3\delta M(u,v)}{10(1+\delta)}R_n.\label{65}
\end{align}
Combining \eqref{63},\,\eqref{64}, and \eqref{65},
\begin{align*}
\frac{\delta M(u,v)}{2(1+\delta)}\tau_n&=\int_0^{\tau_n}\frac{\delta M(u,v)}{2(1+\delta)}dt\geq\int_0^{\tau_n}|z'_{R_n}(t)|dt\geq\left|\int_0^{\tau_n}z_{R_n}'(t)dt\right|\\
&\geq|z_{R_n}(\tau_n)|-|z_{R_n}(0)|\geq|x(\tau_n)|M(u,v)-2\rho M(u,v)-\frac{\delta M(u,v)}{2(1+\delta)}R_n.
\end{align*}
Substituting $R_n=|x(\tau_n)|+\rho$,
\[
\frac{\delta M(u,v)}{2(1+\delta)}\tau_n\geq|x(\tau_n)|M(u,v)-2\rho M(u,v)-\frac{\delta M(u,v)}{2(1+\delta)}(|x(\tau_n)|+\rho),
\]
\[
\frac{\delta}{2(1+\delta)}\tau_n\geq \frac{2+\delta}{2(1+\delta)}|x(\tau_n)|-\frac{4+5\delta}{2(1+\delta)}\rho,
\]
\[
\frac{\delta}{2+\delta}+\frac{4+5\delta}{2+\delta}\frac{\rho}{\tau_n}\geq \frac{|x(\tau_n)|}{\tau_n}.
\]
Since $0<\tau_n\leq t_n$ and $|x(\tau_n)|=|x(t_n)|$,
\[
\frac{\delta}{2+\delta}+\frac{4+5\delta}{2+\delta}\frac{\rho}{\tau_n}\geq \frac{|x(t_n)|}{t_n}.
\]
We obtain $\frac{4+5\delta}{2+\delta}\frac{\rho}{\tau_n}\leq\frac{\delta}{2}$ for sufficiently large $n\in\mathbb{N}$ by $\tau_n\rightarrow\infty$\ \ as\ \ $n\rightarrow\infty$. Also, since $\frac{\delta}{2+\delta}<\frac{\delta}{2}$ by $\delta>0$, it follows that
\[
\frac{|x(t_n)|}{t_n}<\frac{\delta}{2}+\frac{\delta}{2}=\delta.
\]
This is in contradiction to $|x(t_n)|\geq\delta t_n$. Therefore, Lemma \ref{order of x(t)} holds.
\end{proof}

\begin{lemma}\label{lower bounded K}
Let $(u_c,v_c)$ be the critical solution constructed in Proposition \ref{Existence of a critical solution}. Then, there exists $A>0$ such that
\[
A\,M(u_c,v_c)\leq K(u_c,v_c)
\]
for any $0\leq t<\infty$.
\end{lemma}

\begin{proof}
We assume that Lemma \ref{lower bounded K} does not hold. Then, there exists a sequence $\{t_n\}$ such that
\[
M(u_c,v_c)\geq nK(u_c(t_n),v_c(t_n)).
\]
Since $M(u_c,v_c)>0$ is a constant,
\[
K(u_c(t_n),v_c(t_n))\longrightarrow0\ \ \text{as}\ \ n\rightarrow\infty.
\]
By Proposition \ref{Precompactness of the flow of the critical solution}, we can pass to a subsequence $\{t_n\}$ so that
\[
\|(u_c(t_n),v_c(t_n))\|_{H^1\times H^1}\longrightarrow0\ \ \text{as}\ \ n\rightarrow\infty.
\]
Therefore,
\[
M(u_{c,0},v_{c,0})=M(u_c(t_n),v_c(t_n))\longrightarrow0\ \ \text{as}\ \ n\rightarrow\infty.
\]
However, this leads to contradiction with $\|(u_c,v_c)\|_{S(\dot{H}^\frac{1}{2})\times S(\dot{H}^\frac{1}{2})}=\infty$.
\end{proof}

\subsection{Rigidity}


\begin{theorem}[Rigidity]\label{Rigidity}
Let $(u_0,v_0)\in H^1\!\times\!H^1$ and $(u,v)$ be the time-global solution to (NLS) with initial data $(u_0,v_0)$. Suppose
\[
I_\omega(u_0,v_0)<I_\omega(\phi_\omega,\psi_\omega)\,,\ \ \ K_\omega^{20,8}(u_0,v_0)\geq0\,,\ \ \widetilde{P}(u_0,v_0)=0,
\]
\[
^\exists A>0\text{ such that }0\leq\!^\forall t<\infty,\ \ \ A\,M(u,v)\leq K(u,v)
\]
and there exists a continuous path $x(t)$ such that
\[
K=\{(u(\cdot-x(t),t),v(\cdot-x(t),t)):t\in [0,\infty)\}
\]
is precompact in $H^1\times H^1$. Then, $(u_0,v_0)=(0,0)$.
\end{theorem}


\begin{proof}
In the case $K_\omega^{20,8}(u_0,v_0)=0$, we have $(u_0,v_0)=(0,0)$ by the definition of $\mu_\omega^{20,8}(=I_\omega(\phi_\omega,\psi_\omega))$. Let $K_\omega^{20,8}(u_0,v_0)>0$. We assume that $M(u,v)=\|u\|_{L^2}^2+2\|v\|_{L^2}^2>0$ and lead to contradiction. By Lemma \ref{order of x(t)}, for any $\eta>0$, there exists $T_0=T_0(\eta)>0$ such that
\[
|x(t)|\leq\eta t
\]
for any $t\geq T_0$. Let $\chi\in C_0^\infty(\mathbb{R}^5)$ be radial with
\begin{equation}
\notag \chi(r)=
\begin{cases}
\hspace{-0.4cm}&\displaystyle{\ \ \ \,r^2\ \ \ \,\ \ (0\leq r\leq1),}\\
\hspace{-0.4cm}&\displaystyle{smooth\ \ (1\leq r\leq 3),}\\
\hspace{-0.4cm}&\displaystyle{\ \ \ \ 0\ \ \ \ \ \ \,(3\leq r),}
\end{cases}
\end{equation}
where $r=|x|$.
Also, we assume that $\chi$ satisfies $\chi''(r)\leq2\ \ (r\geq0)$ and define $\chi_R(r)=R^2\chi(\frac{r}{R})$.\\
We set $I(t)=\int_{\mathbb{R}^5}\chi_R\left(|\nabla u|^2+2|\nabla v|^2\right)dx$. Then,
\[
I'(t)=2\text{Im}\int_{\mathbb{R}^5}\chi_R'\left(\frac{x\cdot\nabla u}{r}\overline{u}+\frac{x\cdot\nabla v}{r}\overline{v}\right)dx=2\text{Im}\int_{\mathbb{R}^5}R\chi'\left(\frac{r}{R}\right)\left(\frac{x\cdot\nabla u}{r}\overline{u}+\frac{x\cdot\nabla v}{r}\overline{v}\right)dx
\]
by Lemma \ref{Localized virial identity}\ \eqref{54}. Therefore,
\begin{align*}
|I'(t)|&\leq2R\left|\int_{\mathbb{R}^5}\chi'\left(\frac{r}{R}\right)\left(\frac{x\cdot\nabla u}{r}\overline{u}+\frac{x\cdot\nabla v}{r}\overline{v}\right)dx\right|\\
&\leq2R\int_{|x|\leq3R}\left|\chi'\left(\frac{r}{R}\right)\right|\left(|\nabla u||u|+|\nabla v||v|\right)dx\\
&\leq cR\int_{|x|\leq3R}\left(|\nabla u||u|+|\nabla v||v|\right)dx\\
&\leq cR\left(\frac{1}{2}\|\nabla u\|_{L^2}^2+\frac{1}{2}\|u\|_{L^2}^2+\frac{1}{2}\|\nabla v\|_{L^2}^2+\frac{1}{2}\|v\|_{L^2}^2\right)\\
&\leq cR\left(M(u,v)+2I_\omega(\phi_\omega,\psi_\omega)+4E(u,v)\right)\\
&=\widetilde{c}R
\end{align*}
for any $0\leq t<\infty$. Also, applying Lemma \ref{Localized virial identity}\ \eqref{55},
\begin{align*}
I''(t)&=\int_{\mathbb{R}^5}\left\{\frac{1}{r^2}\chi''\left(\frac{r}{R}\right)-\frac{R}{r^3}\chi'\left(\frac{r}{R}\right)\right\}\left(4|x\cdot\nabla u|^2+2|x\cdot\nabla v|^2\right)dx\\
&\hspace{1.5cm}+\int_{\mathbb{R}^5}\frac{R}{r}\chi'\left(\frac{r}{R}\right)\left(4|\nabla u|^2+2|\nabla v|^2\right)dx\\
&\hspace{1.5cm}-\int_{\mathbb{R}^5}\left\{\frac{1}{R^2}\chi^{(4)}\left(\frac{r}{R}\right)+\frac{8}{Rr}\chi^{(3)}\left(\frac{r}{R}\right)+\frac{8}{r^2}\chi''\left(\frac{r}{R}\right)-\frac{8R}{r^3}\chi'\left(\frac{r}{R}\right)\right\}\left(|u|^2+\frac{1}{2}|v|^2\right)dx\\
&\hspace{1.5cm}-2\text{Re}\int_{\mathbb{R}^5}\left\{\chi''\left(\frac{r}{R}\right)+\frac{4R}{r}\chi'\left(\frac{r}{R}\right)\right\}v\overline{u}^2dx.
\end{align*}
Let
\[
R_1=\int_{\mathbb{R}^5}\left\{\frac{1}{r^2}\chi''\left(\frac{r}{R}\right)-\frac{R}{r^3}\chi'\left(\frac{r}{R}\right)\right\}\left(4|x\cdot\nabla u|^2+2|x\cdot\nabla v|^2\right)dx+\int_{\mathbb{R}^5}\frac{R}{r}\chi'\left(\frac{r}{R}\right)\left(4|\nabla u|^2+2|\nabla v|^2\right)dx,
\]
\[
R_2=-\int_{\mathbb{R}^5}\left\{\frac{1}{R^2}\chi^{(4)}\left(\frac{r}{R}\right)+\frac{8}{Rr}\chi^{(3)}\left(\frac{r}{R}\right)+\frac{8}{r^2}\chi''\left(\frac{r}{R}\right)-\frac{8R}{r^3}\chi'\left(\frac{r}{R}\right)\right\}\left(|u|^2+\frac{1}{2}|v|^2\right)dx,
\]
and
\[
R_3=-2\text{Re}\int_{\mathbb{R}^5}\left\{\chi''\left(\frac{r}{R}\right)+\frac{4R}{r}\chi'\left(\frac{r}{R}\right)\right\}v\overline{u}^2dx.
\]
First, we estimate $R_1$.\\
In the case $\frac{1}{r^2}\chi''\left(\frac{r}{R}\right)-\frac{R}{r^3}\chi'\left(\frac{r}{R}\right)\geq0$, we have $R_1\geq\int_{\mathbb{R}^5}\frac{R}{r}\chi'\left(\frac{r}{R}\right)\left(4|\nabla u|^2+2|\nabla v|^2\right)dx$.\\
In the case $\frac{1}{r^2}\chi''\left(\frac{r}{R}\right)-\frac{R}{r^3}\chi'\left(\frac{r}{R}\right)<0$, we have
\begin{align*}
R_1&\geq\int_{\mathbb{R}^5}\left\{\frac{1}{r^2}\chi''\left(\frac{r}{R}\right)-\frac{R}{r^3}\chi'\left(\frac{r}{R}\right)\right\}\left(4r^2|\nabla u|^2+2r^2|\nabla v|^2\right)dx+\int_{\mathbb{R}^5}\frac{R}{r}\chi'\left(\frac{r}{R}\right)\left(4|\nabla u|^2+2|\nabla v|^2\right)dx\\
&\geq\int_{\mathbb{R}^5}\chi''\left(\frac{r}{R}\right)\left(4|\nabla u|^2+2|\nabla v|^2\right)dx.
\end{align*}
Combining these inequalities,
\[
R_1\geq\int_{|x|\leq R}\left(8|\nabla u|^2+4|\nabla v|^2\right)dx-c\int_{R\leq|x|}\left(4|\nabla u|^2+2|\nabla v|^2\right)dx.
\]
Next, we estimate $R_2$.
\begin{align*}
R_2&=-\int_{\mathbb{R}^5}\left\{\frac{1}{R^2}\chi^{(4)}\left(\frac{r}{R}\right)+\frac{8}{Rr}\chi^{(3)}\left(\frac{r}{R}\right)+\frac{8}{r^2}\chi''\left(\frac{r}{R}\right)-\frac{8R}{r^3}\chi'\left(\frac{r}{R}\right)\right\}\left(|u|^2+\frac{1}{2}|v|^2\right)dx\\
&\geq-\frac{c}{R^2}\int_{R\leq|x|}\left(|u|^2+\frac{1}{2}|v|^2\right)dx.
\end{align*}
Finally, we estimate $R_3$.
\begin{align*}
R_3&=-20\text{Re}\int_{|x|\leq R}v\overline{u}^2dx-2\text{Re}\int_{R\leq|x|\leq3R}\left\{\chi''\left(\frac{r}{R}\right)+\frac{4R}{r}\chi'\left(\frac{r}{R}\right)\right\}v\overline{u}^2dx\\
&\geq-20\text{Re}\int_{|x|\leq R}v\overline{u}^2dx-c\int_{R\leq|x|}|vu^2|dx.
\end{align*}
Therefore,
\begin{align*}
I''(t)&\geq\int_{|x|\leq R}\left(8|\nabla u|^2+4|\nabla v|^2\right)dx-c\int_{R\leq|x|}\left(4|\nabla u|^2+2|\nabla v|^2\right)dx\\
&~~~~~~~~~~~~~~~~~~~~~~~-\frac{c}{R^2}\int_{R\leq|x|}\left(|u|^2+\frac{1}{2}|v|^2\right)dx-20\text{Re}\int_{|x|\leq R}v\overline{u}^2dx-c\int_{R\leq|x|}|vu^2|dx.
\end{align*}
Here, applying Lemma \ref{estimates for K} and Lemma \ref{lower bounded K},
\begin{align*}
&\int_{|x|\leq R}\left(8|\nabla u|^2+4|\nabla v|^2-20\text{Re}\,v\overline{u}^2\right)dx\\
&~~~~~~~=\int_{\mathbb{R}^5}\left(8|\nabla u|^2+4|\nabla v|^2-20\text{Re}\,v\overline{u}^2\right)dx-\int_{R\leq|x|}\left(8|\nabla u|^2+4|\nabla v|^2-20\text{Re}\,v\overline{u}^2\right)dx\\
&~~~~~~~=K_\omega^{20,8}(u,v)-\int_{R\leq|x|}\left(8|\nabla u|^2+4|\nabla v|^2-20\text{Re}\,v\overline{u}^2\right)dx\\
&~~~~~~~\geq\min\{I_\omega(\phi_\omega,\psi_\omega)-I_\omega(u,v),K(u,v)\}-\int_{R\leq|x|}\left(8|\nabla u|^2+4|\nabla v|^2-20\text{Re}\,v\overline{u}^2\right)dx\\
&~~~~~~~\geq\min\{I_\omega(\phi_\omega,\psi_\omega)-I_\omega(u,v),cM(u,v)\}-\int_{R\leq|x|}\left(8|\nabla u|^2+4|\nabla v|^2-20\text{Re}\,v\overline{u}^2\right)dx,
\end{align*}
and hence it follows that
\begin{align*}
I''(t)&\geq\min\{I_\omega(\phi_\omega,\psi_\omega)-I_\omega(u,v),cM(u,v)\}-c\int_{R\leq|x|}\left(4|\nabla u|^2+2|\nabla v|^2\right)dx\\
&\hspace{4cm}-\frac{c}{R^2}\int_{R\leq|x|}\left(|u|^2+\frac{1}{2}|v|^2\right)dx-c\int_{R\leq|x|}|vu^2|dx\\
&\geq\min\{I_\omega(\phi_\omega,\psi_\omega)-I_\omega(u,v),cM(u,v)\}\\
&\hspace{4cm}-c\int_{R\leq|x|}\left(4|\nabla u|^2+2|\nabla v|^2+\frac{1}{R^2}|u|^2+\frac{1}{2R^2}|v|^2+|vu^2|\right)dx.
\end{align*}
By Lemma \ref{Precompactness of the flow implies uniform localization}, there exists $R_0>1$ such that
\[
c\int_{R_0\leq|x+x(t)|}\left(4|\nabla u|^2+2|\nabla v|^2+|u|^2+\frac{1}{2}|v|^2+|vu^2|\right)dx<\frac{1}{2}\min\{I_\omega(\phi_\omega,\psi_\omega)-I_\omega(u,v),cM(u,v)\}
\]
for any $0\leq t<\infty$. If we take $\displaystyle R=R_0+\sup_{t\in[T_0,T_1]}|x(t)|>1$, then $\displaystyle|x+x(t)|\geq|x|-|x(t)|\geq R-\sup_{[T_0,T_1]}|x(t)|=R_0$ for $x$ with $|x|>R$ and $t\in[T_0,T_1]$, where $T_1>T_0$ is chosen later. Thus,
\begin{align*}
&c\int_{R\leq|x|}\left(4|\nabla u|^2+2|\nabla v|^2+\frac{1}{R^2}|u|^2+\frac{1}{2R^2}|v|^2+|vu^2|\right)dx\\
&\hspace{5cm}\leq c\int_{R\leq|x|}\left(4|\nabla u|^2+2|\nabla v|^2+|u|^2+\frac{1}{2}|v|^2+|vu^2|\right)dx\\
&\hspace{5cm}\leq c\int_{R_0\leq|x+x(t)|}\left(4|\nabla u|^2+2|\nabla v|^2+|u|^2+\frac{1}{2}|v|^2+|vu^2|\right)dx\\
&\hspace{5cm}\leq\frac{1}{2}\min\{I_\omega(\phi_\omega,\psi_\omega)-I_\omega(u,v),cM(u,v)\}.
\end{align*}
For such $R>1$ and $t\in[T_0,T_1]$, we have $I''(t)\geq\frac{1}{2}\min\{I_\omega(\phi_\omega,\psi_\omega)-I_\omega(u,v),cM(u,v)\}$. Integrating this inequality in $[T_0,T_1]$,
\begin{align*}
\frac{1}{2}\min\{I_\omega(\phi_\omega,\psi_\omega)-I_\omega(u,v),cM(u,v)\}(T_1-T_0)&\leq I'(T_1)-I'(T_0)\\
&\leq|I'(T_1)|+|I'(T_0)|\\
&\leq2\widetilde{c}R\\
&=2\widetilde{c}\left(R_0+\sup_{t\in[T_0,T_1]}|x(t)|\right)\\
&\leq2\widetilde{c}\left(R_0+T_1\eta\right).
\end{align*}
This inequality is contradiction if we take $\eta>0$ sufficiently small and $T_1>0$ sufficintly large when $M(u,v)>0$. Therefore, $M(u,v)=0$, i.e. $(u_0,v_0)=(0,0)$. However, $(u_0,v_0)$ does not satisfy $K_\omega^{20,8}(u_0,v_0)>(0,0)$.
\end{proof}


Finally, we prove the scattering part of Theorem \ref{Main theorem 1}.
\begin{proof}
We consider $(u_c,v_c)$ constructed in Proposition \ref{Existence of a critical solution}. Then, $(u_c,v_c)$ satisfies the assumption of Theorem \ref{Rigidity} by Proposition \ref{Existence of a critical solution}, Proposition \ref{Precompactness of the flow of the critical solution}, Proposition \ref{zero momentum}, and Lemma \ref{lower bounded K}. By applying Theorem \ref{Rigidity}, we have $(u_{c,0},v_{c,0})=(0,0)$. However, this leads to contradiction with $\|(u_c,v_c)\|_{S(\dot{H}^\frac{1}{2})\times S(\dot{H}^\frac{1}{2})}=\infty$. Therefore, we have $I_\omega^c\geq I_\omega(\phi_\omega,\psi_\omega)$. By definition of $I_\omega^c$, if $I_\omega(u_0,v_0)<I_\omega(\phi_\omega,\psi_\omega)$ and $K_\omega^{20,8}(u_0,v_0)>0$, then $SC(u_0,v_0)$, i.e. $\|(u_0,v_0)\|_{S(\dot{H}^\frac{1}{2})\times S(\dot{H}^\frac{1}{2})}<\infty$. By applying Proposition \ref{Scattering}, the scattering part of Theorem \ref{Main theorem 1} holds.
\end{proof}

\section{Appendix}


\begin{theorem}[Sharp Gagliardo--Nirenberg type inequality]\label{Sharp Gagliardo-Nirenberg inequality}
It follows that
\[
P(u,v)\leq C_{GN}M(u,v)^\frac{1}{4}K(u,v)^\frac{5}{4}
\]
for any $(u,v)\in H^1\!\times\!H^1$ and the best constant $C_{GN}$ is attained by the ground state, i.e.
\[
C_{GN}=\frac{P(\phi_\omega,\psi_\omega)}{M(\phi_\omega,\psi_\omega)^\frac{1}{4}K(\phi_\omega,\psi_\omega)^\frac{5}{4}}.
\]
\end{theorem}

\begin{proof}
Let $\phi_\omega=\omega^\alpha\phi_1(\omega^\beta\cdot),\ \psi_\omega=\omega^\alpha\psi_1(\omega^\beta\cdot)$. We substitute these them into
\begin{equation}
\tag{\text{gNLS}}
\begin{cases}
\hspace{-0.4cm}&\displaystyle{-\Delta\phi_\omega+\omega\phi_\omega=2\psi_\omega\phi_\omega,}\\
\hspace{-0.4cm}&\displaystyle{-\frac{1}{2}\Delta\psi_\omega+2\omega\psi_\omega=\phi_\omega^2,}
\end{cases}
\end{equation}
then
\begin{equation}
\notag
\begin{cases}
\hspace{-0.4cm}&\displaystyle{-\omega^{\alpha+2\beta}(\Delta\phi_1)(\omega^\beta\cdot)+\omega^{\alpha+1}\phi_1(\omega^\beta\cdot)=2\omega^{2\alpha}\psi_1(\omega^\beta\cdot)\phi_1(\omega^\beta\cdot),}\\
\hspace{-0.4cm}&\displaystyle{-\frac{1}{2}\omega^{\alpha+2\beta}(\Delta\psi_1)(\omega^\beta\cdot)+2\omega^{\alpha+1}\psi_1(\omega^\beta\cdot)=\omega^{2\alpha}\phi_1(\omega^\beta\cdot)^2.}
\end{cases}
\end{equation}
Thus, $\alpha+2\beta=\alpha+1=2\alpha$, i.e. $\alpha=1,\,\beta=\frac{1}{2}.$\\
Therefore,
\[
M(\phi_\omega,\psi_\omega)=\|\phi_\omega\|_{L^2}^2+2\|\psi_\omega\|_{L^2}^2=\omega^{-\frac{1}{2}}\|\phi_1\|_{L^2}^2+2\omega^{-\frac{1}{2}}\|\psi_1\|_{L^2}^2=\omega^{-\frac{1}{2}}M(\phi_1,\psi_1),
\]
\[
K(\phi_\omega,\psi_\omega)=\|\nabla \phi_\omega\|_{L^2}^2+\frac{1}{2}\|\nabla\psi_\omega\|_{L^2}^2=\omega^\frac{1}{2}\|\nabla \phi_1\|_{L^2}^2+\frac{1}{2}\omega^\frac{1}{2}\|\nabla\psi_1\|_{L^2}^2=\omega^\frac{1}{2}K(\phi_1,\psi_1)
\]
and
\[
P(\phi_\omega,\psi_\omega)=\text{Re}(\psi_\omega,\phi_\omega^2)_{L^2}=\omega^\frac{1}{2}\text{Re}(\psi_1,\phi_1^2)_{L^2}=\omega^\frac{1}{2}P(\phi_1,\psi_1).
\]
First, we will prove that $0<C_{GN}<\infty$. Using usual Gagliardo--Nirenberg inequality
\[
\|u\|_{L^{p+1}}^{p+1}\leq c_{GN}\|\nabla u\|_{L^2}^\frac{5(p-1)}{2}\|u\|_{L^2}^{2-\frac{3(p-1)}{2}},
\]
we have
\begin{align*}
P(u,v)&\leq\|v\|_{L^3}\|u\|_{L^3}^2\\
&\leq\left(c_{GN}^\frac{1}{3}\|\nabla v\|_{L^2}^\frac{5}{6}\|v\|_{L^2}^\frac{1}{6}\right)\left(c_{GN}^\frac{1}{3}\|\nabla u\|_{L^2}^\frac{5}{6}\|u\|_{L^2}^\frac{1}{6}\right)^2\\
&\leq c_{GN}\left(\sqrt{2}K(u,v)^\frac{1}{2}\right)^\frac{5}{6}\left(\frac{1}{\sqrt{2}}M(u,v)^\frac{1}{2}\right)^\frac{1}{6}\left(\sqrt{2}K(u,v)^\frac{1}{2}\right)^\frac{5}{3}\left(\frac{1}{\sqrt{2}}M(u,v)^\frac{1}{2}\right)^\frac{1}{3}\\
&=2c_{GN}K(u,v)^\frac{5}{4}M(u,v)^\frac{1}{4}
\end{align*}
for any $(u,v)\in H^1\!\times\!H^1$. Thus, we obtain $C_{GN}\leq 2c_{GN}<\infty$ (see \cite{008}). On the other hand, applying Proposition \ref{proposition of ground state}, we have
\begin{align*}
C_{GN}\geq\frac{P(\phi_\omega,\psi_\omega)}{M(\phi_\omega,\psi_\omega)^\frac{1}{4}K(\phi_\omega,\psi_\omega)^\frac{5}{4}}=\frac{2}{5^\frac{5}{4}M(\phi_1,\psi_1)^\frac{1}{2}}>0.
\end{align*}
Thus, we obtain $0<C_{GN}<\infty$.\\
When $P(u,v)\leq0$, the inequality holds. Thus, we assume that $P(u,v)>0$. Next, we consider minimizing problem
\[
\inf_{(u,v)\in \mathcal{N}}J(u,v)\vcentcolon=\inf_{(u,v)\in \mathcal{N}}\frac{M(u,v)^\frac{1}{4}K(u,v)^\frac{5}{4}}{P(u,v)},
\]
where $\mathcal{N}=\{(u,v)\in H^1\!\times\!H^1:P(u,v)>0\}$. $\inf_{(u,v)\in \mathcal{N}}J(u,v)$ is clearly the reciprocal of the best constant $C_{GN}$. We set that $u_{\mu, \lambda}(x)=\mu u(\lambda x)$, $v_{\mu, \lambda}(x)=\mu v(\lambda x)$ for $\mu>0$, $\lambda>0$. Then,
\[
M(u_{\mu, \lambda},v_{\mu, \lambda})=\mu^2\lambda^{-5}M(u,v),\ K(u_{\mu, \lambda},v_{\mu, \lambda})=\mu^2\lambda^{-3}K(u,v),\ P(u_{\mu, \lambda},v_{\mu, \lambda})=\mu^3\lambda^{-5}P(u,v),
\]
and hence
\[
J(u_{\mu, \lambda},v_{\mu, \lambda})=J(u,v).
\]
Let $(\widetilde{u},\widetilde{v})$ attain the infimum of $J(u,v)$. Since the functional $J(u,v)$ is invariance for the above scaling, we can assume that
\[
M(\widetilde{u},\widetilde{v})=1,\ \ K(\widetilde{u},\widetilde{v})=1.
\]
Because $(\widetilde{u},\widetilde{v})$ is a critical point of $J(u,v)$, it follows that
\[
\left.\frac{d}{ds}J(\widetilde{u}+s\phi,\widetilde{v}+s\psi)\right|_{s=0}=0.
\]
for any $(\phi,\psi)\in H^1\!\times\!H^1$. Thus,
\begin{align*}
0&=\left.\frac{d}{ds}J(\widetilde{u}+s\phi,\widetilde{v}+s\psi)\right|_{s=0}\\
&=\left.\frac{d}{ds}\frac{M(\widetilde{u}+s\phi,\widetilde{v}+s\psi)^\frac{1}{4}K(\widetilde{u}+s\phi,\widetilde{v}+s\psi)^\frac{5}{4}}{P(\widetilde{u}+s\phi,\widetilde{v}+s\psi)}\right|_{s=0}\\
&=\frac{\left\{\frac{1}{4}M(\widetilde{u}+s\phi,\widetilde{v}+s\psi)^{-\frac{3}{4}}K(\widetilde{u}+s\phi,\widetilde{v}+s\psi)^\frac{5}{4}\frac{d}{ds}M(\widetilde{u}+s\phi,\widetilde{v}+s\psi)\right.}{P(\widetilde{u}+s\phi,\widetilde{v}+s\psi)^2}\\
&~~~~~~~~~~~~~~\frac{\left.+\frac{5}{4}M(\widetilde{u}+s\phi,\widetilde{v}+s\psi)^\frac{1}{4}K(\widetilde{u}+s\phi,\widetilde{v}+s\psi)^\frac{1}{4}\frac{d}{ds}K(\widetilde{u}+s\phi,\widetilde{v}+s\psi)\right\}P(\widetilde{u}+s\phi,\widetilde{v}+s\psi)}{}\\
&\left.~~~~~~~~~~~~~~\frac{-M(\widetilde{u}+s\phi,\widetilde{v}+s\psi)^\frac{1}{4}K(\widetilde{u}+s\phi,\widetilde{v}+s\psi)^\frac{5}{4}\frac{d}{ds}P(\widetilde{u}+s\phi,\widetilde{v}+s\psi)}{}\right|_{s=0}.
\end{align*}
Taking out the numerator and using \eqref{84}, \eqref{85}
\begin{align*}
&\left\{\frac{1}{4}\left(2\int_{\mathbb{R}^5}\widetilde{u}\phi dx+4\int_{\mathbb{R}^5}\widetilde{v}\psi dx\right)+\frac{5}{4}\left(2\int_{\mathbb{R}^5}\nabla\widetilde{u}\cdot\nabla\phi dx+\int_{\mathbb{R}^5}\nabla\widetilde{v}\cdot\psi dx\right)\right\}C_{GN}\\
&\hspace{10cm}-\int_{\mathbb{R}^5}\left(\widetilde{u}^2\psi+2\widetilde{u}\widetilde{v}\phi\right)dx=0.
\end{align*}
Because $(\phi,\psi)\in H^1\!\times\!H^1$ is arbitrary, it follows that
\begin{equation}
\notag
\begin{cases}
\hspace{-0.4cm}&\displaystyle{\frac{1}{2}\int_{\mathbb{R}^5}\widetilde{u}\phi dx+\frac{5}{2}\int_{\mathbb{R}^5}\nabla\widetilde{u}\cdot\nabla\phi dx=\frac{2}{C_{GN}}\int_{\mathbb{R}^5}\widetilde{u}\widetilde{v}\phi dx,}\\[0.3cm]
\hspace{-0.4cm}&\displaystyle{\int_{\mathbb{R}^5}\widetilde{v}\psi dx+\frac{5}{4}\int_{\mathbb{R}^5}\nabla\widetilde{v}\cdot\nabla\psi dx=\frac{1}{C_{GN}}\int_{\mathbb{R}^5}\widetilde{u}^2\psi dx.}
\end{cases}
\end{equation}
Therefore, $(\widetilde{u},\widetilde{v})$ is a solution to the nonlinear elliptic system:
\begin{equation}
\notag
\begin{cases}
\hspace{-0.4cm}&\displaystyle{-\frac{5}{2}\Delta\widetilde{u}+\frac{1}{2}\widetilde{u}=\frac{2}{C_{GN}}\widetilde{u}\widetilde{v},}\\[0.3cm]
\hspace{-0.4cm}&\displaystyle{-\frac{5}{4}\Delta\widetilde{v}+\widetilde{v}=\frac{1}{C_{GN}}\widetilde{u}^2.}
\end{cases}
\end{equation}
Here, we consider $(\mathring{u},\mathring{v})$ with $(\mathring{u}_{\mu,\lambda},\mathring{v}_{\mu,\lambda})=(\widetilde{u},\widetilde{v})$ for $\mu=C_{GN}/2\omega$, $\lambda=1/\sqrt{5\omega}$. Substituting $(\mathring{u}_{\mu,\lambda},\mathring{v}_{\mu,\lambda})$ into the above nonlinear elliptic system, we obtain
\begin{equation}
\notag
\begin{cases}
\hspace{-0.4cm}&\displaystyle{-\Delta\mathring{u}+\omega\mathring{u}=2\mathring{u}\mathring{v},}\\
\hspace{-0.4cm}&\displaystyle{-\frac{1}{2}\Delta\mathring{v}+2\omega\mathring{v}=\mathring{u}^2.}
\end{cases}
\end{equation}
Hence, $(\mathring{u},\mathring{v})$ solves (gNLS). (gNLS) actually has solutions, so the infimum of $J(u,v)$ is attained by any solution to (gNLS). 
\end{proof}

\begin{remark}
$C_{GN}$ is independent of $\omega>0$. Indeed,
\[
C_{GN}=\frac{P(\phi_\omega,\psi_\omega)}{M(\phi_\omega,\psi_\omega)^\frac{1}{4}K(\phi_\omega,\psi_\omega)^\frac{5}{4}}=\frac{\omega^\frac{1}{2}P(\phi_1,\psi_1)}{\left(\omega^{-\frac{1}{2}}M(\phi_1,\psi_1)\right)^\frac{1}{4}\left(\omega^\frac{1}{2}K(\phi_1,\psi_1)\right)^\frac{5}{4}}=\frac{P(\phi_1,\psi_1)}{M(\phi_1,\psi_1)^\frac{1}{4}K(\phi_1,\psi_1)^\frac{5}{4}}.
\]
\end{remark}


\begin{proposition}\label{equivalent condition}
The following statement holds.
\begin{itemize}
\item[(1)] There exists $\omega>0$ such that $I_\omega(u_0,v_0)<I_\omega(\phi_\omega,\psi_\omega)$\\
\hspace{5cm}$\Longleftrightarrow$\ \ $M(u_0,v_0)E(u_0,v_0)<M(\phi_1,\psi_1)E(\phi_1,\psi_1)$.
\item[(2)] Let there exist $\omega>0$ such that $I_\omega(u_0,v_0)<I_\omega(\phi_\omega,\psi_\omega)$.
\begin{enumerate}
\item $K_\omega^{20,8}(u_0,v_0)\geq0$\ \ $\Longleftrightarrow$\ \ $M(u_0,v_0)K(u_0,v_0)<M(\phi_1,\psi_1)K(\phi_1,\psi_1)$.
\item $K_\omega^{20,8}(u_0,v_0)<0$\ \ $\Longleftrightarrow$\ \ $M(u_0,v_0)K(u_0,v_0)>M(\phi_1,\psi_1)K(\phi_1,\psi_1)$.
\end{enumerate}
\end{itemize}
\end{proposition}

\begin{proof}
\begin{align*}
I_\omega(\phi_\omega,\psi_\omega)&=\frac{\omega}{2}M(\phi_\omega,\psi_\omega)+\frac{1}{2}E(\phi_\omega,\psi_\omega)\\
&=\frac{\omega}{2}M(\phi_\omega,\psi_\omega)+\frac{1}{2}\left(K(\phi_\omega,\psi_\omega)-2P(\phi_\omega,\psi_\omega)\right)\\
&=\frac{\omega^\frac{1}{2}}{2}M(\phi_1,\psi_1)+\frac{1}{2}\left(\omega^\frac{1}{2}K(\phi_1,\psi_1)-2\omega^\frac{1}{2}P(\phi_1,\psi_1)\right)\\
&=\omega^\frac{1}{2}I_1(\phi_1,\psi_1).
\end{align*}
Hence, the condition $I_\omega(u_0,v_0)<I_\omega(\phi_\omega,\psi_\omega)$ is equivalent to $I_\omega(u_0,v_0)<\omega^\frac{1}{2}I_1(\phi_1,\psi_1)$.\ \\
Here, we set $\displaystyle f(\omega)=\omega^\frac{1}{2}I_1(\phi_1,\psi_1)-I_\omega(u_0,v_0)=\omega^\frac{1}{2}I_1(\phi_1,\psi_1)-\frac{\omega}{2}M(u_0,v_0)-\frac{1}{2}E(u_0,v_0)$. Then, we have\\
$\displaystyle f'(\omega)=\frac{1}{2}\omega^{-\frac{1}{2}}I_1(\phi_1,\psi_1)-\frac{1}{2}M(u_0,v_0)$. Solving $f'(\omega_0)=0$, we obtain $\displaystyle\omega_0=\left(\frac{I_1(\phi_1,\psi_1)}{M(u_0,v_0)}\right)^2>0.$\\
Therefore, if there exists $\omega>0$ such that $I_\omega(u_0,v_0)<I_\omega(\phi_\omega,\psi_\omega)$, then $f(\omega_0)>0$ holds.
\begin{align*}
f(\omega_0)=&\omega_0^\frac{1}{2}I_1(\phi_1,\psi_1)-\frac{\omega_0}{2}M(u_0,v_0)-\frac{1}{2}E(u_0,v_0)\\
&=\frac{I_1(\phi_1,\psi_1)^2}{M(u_0,v_0)}-\frac{1}{2}\cdot\frac{I_1(\phi_1,\psi_1)^2}{M(u_0,v_0)}-\frac{1}{2}E(u_0,v_0)\\
&=\frac{1}{2}\cdot\frac{I_1(\phi_1,\psi_1)^2}{M(u_0,v_0)}-\frac{1}{2}E(u_0,v_0)>0,
\end{align*}
i.e. $\displaystyle I_1(\phi_1,\psi_1)^2>M(u_0,v_0)E(u_0,v_0).$ Using Proposition \ref{proposition of ground state},
\begin{align*}
I_1(\phi_1,\psi_1)&=\frac{1}{2}M(\phi_1,\psi_1)+\frac{1}{2}E(\phi_1,\psi_1)\\
&=\frac{1}{2}M(\phi_1,\psi_1)+\frac{1}{2}K(\phi_1,\psi_1)-P(\phi_1,\psi_1)\\
&=\frac{1}{2}M(\phi_1,\psi_1)+\frac{5}{2}M(\phi_1,\psi_1)-2M(\phi_1,\psi_1)\\
&=M(\phi_1,\psi_1),
\end{align*}
and hence it follows that $I_1(\phi_1,\psi_1)^2=M(\phi_1,\psi_1)^2=M(\phi_1,\psi_1)E(\phi_1,\psi_1).$\\
Thus, we obtain $M(\phi_1,\psi_1)E(\phi_1,\psi_1)>M(u_0,v_0)E(u_0,v_0)$ and complete proof of Proposition \ref{equivalent condition} (1).\\
Next, we will prove that $M(u_0,v_0)K(u_0,v_0)\neq M(\phi_1,\psi_1)K(\phi_1,\psi_1)$ under the assumption of Proposition \ref{equivalent condition} (2). Applying Theorem \ref{Sharp Gagliardo-Nirenberg inequality},
\begin{align*}
M(u_0,v_0)E(u_0,v_0)&=M(u_0,v_0)(K(u_0,v_0)-2P(u_0,v_0))\\
&\geq M(u_0,v_0)\left(K(u_0,v_0)-2C_{GN}M(u_0,v_0)^\frac{1}{4}K(u_0,v_0)^\frac{5}{4}\right)\\
&=M(u_0,v_0)K(u_0,v_0)-2C_{GN}M(u_0,v_0)^\frac{5}{4}K(u_0,v_0)^\frac{5}{4}.
\end{align*}
Here, we define a function $f(x)=x-2C_{GN}x^\frac{5}{4}$. Then, $f'(x)=1-\frac{5}{2}C_{GN}x^\frac{1}{4}$. Solving $f'(x)=0$, we obtain $x_0=0$ and $x_1=(\frac{2}{5C_{GN}})^4=M(\phi_1,\psi_1)K(\phi_1,\psi_1)$. Thus, the graph of $f$ has a local minimum at $x_0$ and a local maximum at $x_1$. Also, $f(x_1)=\frac{1}{5}M(\phi_1,\psi_1)K(\phi_1,\psi_1)=M(\phi_1,\psi_1)E(\phi_1,\psi_1)$. Combining these facts, the assumption of Proposition \ref{equivalent condition} (2) and the result of Proposition \ref{equivalent condition} (1), we have $M(u_0,v_0)K(u_0,v_0)\neq M(\phi_1,\psi_1)K(\phi_1,\psi_1)$. Let $K_\omega^{20,8}(u_0,v_0)<0$. Applying Theorem \ref{Sharp Gagliardo-Nirenberg inequality}, 
\begin{align*}
8K(u_0,v_0)-20C_{GN}M(u_0,v_0)^\frac{1}{4}K(u_0,v_0)^\frac{5}{4}&\leq8K(u_0,v_0)-20P(u_0,v_0)\\
&=K_\omega^{20,8}(u_0,v_0)8K(u_0,v_0)-20P(u_0,v_0)<0,
\end{align*}
i.e.
\begin{align*}
M(u_0,v_0)K(u_0,v_0)&>\left(\frac{2}{5}\cdot\frac{1}{C_{GN}}\right)^4\\
&=\left(\frac{2}{5}\right)^4\cdot\frac{M(\phi_1,\psi_1)K(\phi_1,\psi_1)^5}{P(\phi_1,\psi_1)^4}\\
&=M(\phi_1,\psi_1)K(\phi_1,\psi_1).
\end{align*}
Let $K_\omega^{20,8}(u_0,v_0)\geq0$. Since
\begin{align*}
2\sqrt{\frac{\omega}{2}M(u_0,v_0)\frac{1}{10}K(u_0,v_0)}&\leq\frac{\omega}{2}M(u_0,v_0)+\frac{1}{10}K(u_0,v_0)\\
&\leq\frac{\omega}{2}M(u_0,v_0)+\frac{1}{10}K(u_0,v_0)+\frac{1}{20}K_\omega^{20,8}(u_0,v_0)\\
&=\frac{\omega}{2}M(u_0,v_0)+\frac{1}{2}E(u_0,v_0)\\
&=I_\omega(u_0,v_0)<I_\omega(\phi_\omega,\psi_\omega)=\omega^\frac{1}{2}I_1(\phi_1,\psi_1),
\end{align*}
we have
\[
M(u_0,v_0)K(u_0,v_0)<5I_1(\phi_1,\psi_1)^2=5M(\phi_1,\psi_1)^2=M(\phi_1,\psi_1)K(\phi_1,\psi_1).
\]
Combining contraposition of these results and $M(u,v)K(u,v)\neq M(\phi_1,\psi_1)K(\phi_1,\psi_1)$, we complete the proof of Proposition \ref{equivalent condition} (2).
\end{proof}

\begin{corollary}
Let $(u_0,v_0)\in H^1\!\times\!H^1$ and $(u,v)$ be the solution to (NLS) with a initial data $(u_0,v_0)$. Moreover, we assume that $M(u_0,v_0)E(u_0,v_0)<M(\phi_1,\psi_1)E(\phi_1,\psi_1)$. 
\begin{itemize}
\item[(1).] If $M(u_0,v_0)K(u_0,v_0)<M(\phi_1,\psi_1)K(\phi_1,\psi_1)$, then $M(u,v)K(u,v)<M(\phi_1,\psi_1)K(\phi_1,\psi_1)$ for any $t\in\mathbb{R}$.
\item[(2).] If $M(u_0,v_0)K(u_0,v_0)>M(\phi_1,\psi_1)K(\phi_1,\psi_1)$, then $M(u,v)K(u,v)>M(\phi_1,\psi_1)K(\phi_1,\psi_1)$ for any $t\in (T_\ast,T^\ast)$.
\end{itemize}
\end{corollary}

\begin{proof}
Considering the graph of $f$ of the proof for Proposition \ref{equivalent condition}, this corollary holds.
\end{proof}

\begin{proposition}\label{extend negative energy}
Let $(u,v)\in H^1\!\times\!H^1\setminus\{(0,0)\}$. If $(u,v)$ satisfies $E(u,v)\leq0$, then
\[
I_\omega(u,v)<I_\omega(\phi_\omega,\psi_\omega)\ \ \ \text{and}\ \ \ K_\omega^{20,8}(u,v)<0.
\]
\end{proposition}

\begin{proof}
By the assumption $E(u,v)\leq0$,
\[
M(u,v)E(u,v)\leq0<M(\phi_1,\psi_1)E(\phi_1,\psi_1).
\]
Applying Proposition \ref{equivalent condition} (1), it follows that there exists $\omega>0$ such that
\[
I_\omega(u,v)<I_\omega(\phi_\omega,\psi_\omega).
\]
Also, by the assumption $E(u,v)\leq0$, we obtain $K(u,v)\leq2P(u,v)$, and hence
\[
K_\omega^{20,8}(u,v)=8K(u,v)-20P(u,v)\leq-2K(u,v)\leq0.
\]
Here, $K_\omega^{20,8}(u,v)\neq0$ by $I_\omega(u,v)<I_\omega(\phi_\omega,\psi_\omega)$. Therefore, we have $K_\omega^{20,8}(u,v)<0$.
\end{proof}

\begin{corollary}
Let $(u_0,v_0)\in H^1\!\times\!H^1\setminus\{(0,0)\}$ and $(u,v)$ be the solution to (NLS) with a initial data $(u_0,v_0)$. If $(xu_0,xv_0)\in L^2\!\times\!L^2$ and $E(u_0,v_0)\leq0$, then the solution $(u,v)$ blows up.
\end{corollary}

\begin{proof}
Combinig Theorem \ref{Global versus blow-up dichotomy} and Proposition \ref{extend negative energy}, this corollary holds.
\end{proof}

\end{document}